\documentclass[11pt,reqno]{article}

\usepackage{epsfig}
\usepackage{amsmath, amsthm, amsfonts,mathtools}
\usepackage{graphicx}
\usepackage{xcolor}
\usepackage{hyperref}
\usepackage{tikz}
\usepackage[margin=1in]{geometry}

\usetikzlibrary{arrows.meta}

\numberwithin{equation}{section}

\newcommand{\R}{{\mathbb R}}
\newcommand{\C}{{\mathbb C}}
\newcommand{\N}{{\mathbb N}}

\newcommand{\Z}{{\mathbb Z}}

\renewcommand{\d}{\partial}
\renewcommand{\Re}{{\operatorname{Re\,}}}
\renewcommand{\Im}{{\operatorname{Im\,}}}

\newcommand{\F}{F^{\text{st}}}

\newcommand{\Res}{{\operatorname{Res}\,}}
\newcommand{\al}{\alpha}
\newcommand{\be}{\beta}
\newcommand{\ga}{\gamma}
\newcommand{\Ga}{\Gamma}

\newcommand{\la}{\lambda}
\newcommand{\ep}{\varepsilon}
\newcommand{\de}{\delta}
\newcommand{\De}{\Delta}
\newcommand{\f}{\phi}
\newcommand{\sg}{\sigma}
\newcommand{\Sg}{\Sigma}

\newcommand{\Om}{\Omega}

\newcommand{\Arg}{{\operatorname{Arg}\,}}

\newcommand{\ZZ}{ \widetilde {\mathcal Z}}
\newcommand{\m}{\mathfrak m}
\newcommand{\mt}{\mathfrak m\mathfrak t}
\newcommand{\dd}{\mathrm d}
\newcommand{\B}{\mathcal B}

\newcommand{\eq}{\begin{equation}}
\newcommand{\eeq}{\end{equation}}

\newcommand{\E}{{\mathbb E}}

\newcommand{\bigO}{{\mathcal O}}

\newcommand{\ii}{\mathbf i}
\newcommand{\eps}{\varepsilon}

\newtheorem{theo}{{Theorem}}[section]
\newtheorem{cor}[theo]{{Corollary}}
\newtheorem{lem}[theo]{{Lemma}}
\newtheorem{prop}[theo]{{Proposition}}
\newtheorem{conjecture}[theo]{Conjecture}

\DeclareMathOperator{\sh}{sh}

\theoremstyle{remark}
\newtheorem{definition}[theo]{Definition}
\newtheorem{remark}[theo]{Remark}

\title{Boundary statistics for the six-vertex model with DWBC}

\author{ Vadim Gorin\thanks{Departments of Statistics and Mathematics, University of California at Berkeley, USA \href{mailto:vadicgor@gmail.com}{\nolinkurl{vadicgor@gmail.com}}.}  \and Karl Liechty\thanks{Department of Mathematical Sciences, DePaul University, Chicago, IL, 60614 USA \href{mailto:kliechty@depaul.edu}{\nolinkurl{kliechty@depaul.edu}}.} }

\begin{document}
\maketitle

\begin{abstract}
 We study the behavior of configurations in the symmetric six-vertex model with $a,b,c$ weights in the $n\times n$ square with Domain Wall Boundary Conditions as $n\to\infty$. We prove that when $\Delta=\frac{a^2+b^2-c^2}{2ab}<1$,  configurations near the boundary have fluctuations of order $n^{1/2}$ and are asymptotically described by the GUE-corners process of the random matrix theory. On the other hand, when $\Delta>1$, the fluctuations are of finite order and configurations are asymptotically described by the stochastic six-vertex model in a quadrant. In the special case $c=0$ (which implies $\Delta>1$), the limit is expressed as the $q$-exchangeable random permutation of infinitely many letters, distributed according to the infinite Mallows measure.
\end{abstract}

\tableofcontents

\section{Introduction}

\sloppy

\subsection{Background and motivations}

The six-vertex model is one of the most fundamental and important exactly solvable models in statistical physics, and we refer to \cite{baxter2016exactly,Bleher-Liechty14,Gorin_Nicoletti_lectures,LiebWu,reshetikhin2010lectures} for general reviews. The model can be defined on a domain in the two-dimensional square lattice by drawing arrows on the edges of the lattice which satisfy the {\it ice-rule}: for each vertex there are exactly two adjacent arrows pointing inward, and two pointing outward. There are therefore six possible arrow configurations at each vertex, which we can label as Type 1, 2, 3, 4, 5, and 6. By simple bijections, one can identify these six configurations with local configurations of ice molecules $H_2O$ or of up-right paths which are allowed to touch each other but not intersect, as in Figure \ref{Figure_six_vertices}.

\begin{figure}[t]
\begin{center}
   \includegraphics[width=0.7\linewidth]{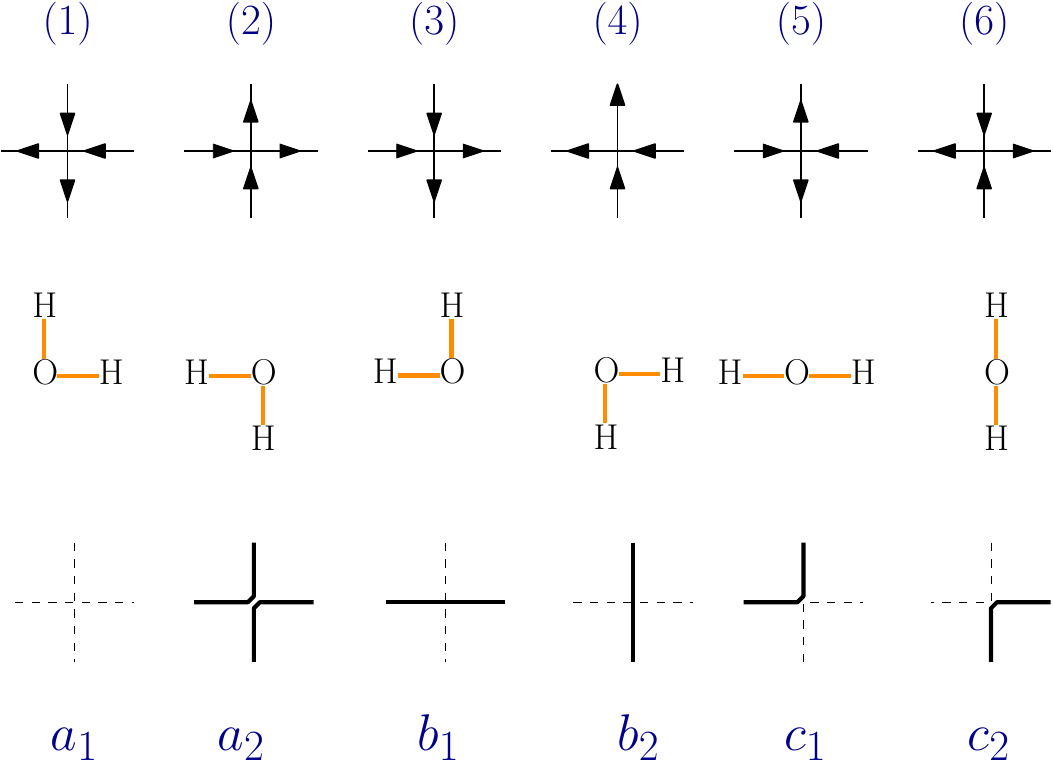}
\end{center}
        \caption{\label{Figure_six_vertices} The six types of vertices in three different representations}
\end{figure}

We introduce a weight $w_i>0$ for each vertex type $i = 1,2,\dots, 6$ and define the weight of a configuration $\sigma$ in a finite domain as
\[
w(\sg) = \prod_{i=1}^6 w_i^{N_i(\sg)},
\]
where $N_i(\sg)$ is the number of vertices of type $i$ in the configuration $\sg$. The Boltzmann-Gibbs measure is then defined by assigning to configurations probabilities computed as the normalized weights:
\begin{equation}
\label{eq_Gibbs_measure}
\mu(\sg) = \frac{w(\sg)}{\mathcal Z}, \qquad \mathcal Z = \sum_{\sg} w(\sg).
\end{equation}
The quantity $\mathcal Z$ is the partition function for the model, and the sum defining it is over all allowable configurations $\sg$ in the particular domain we deal with. In the current work we consider the six-vertex model with {\it Domain Wall Boundary Conditions} (DWBC), in which the vertices are constrained to a $n\times n$ rectangle, all arrows on the left and right boundary point inward, and all arrows on the top and bottom boundaries point outward. See Figure \ref{Figure_DWBC} for an example of an arrow configuration with DWBC on the $5\times 5$ lattice and corresponding family of non-intersecting paths.

\begin{figure}[t]
\begin{center}
   \includegraphics[width=0.43\linewidth]{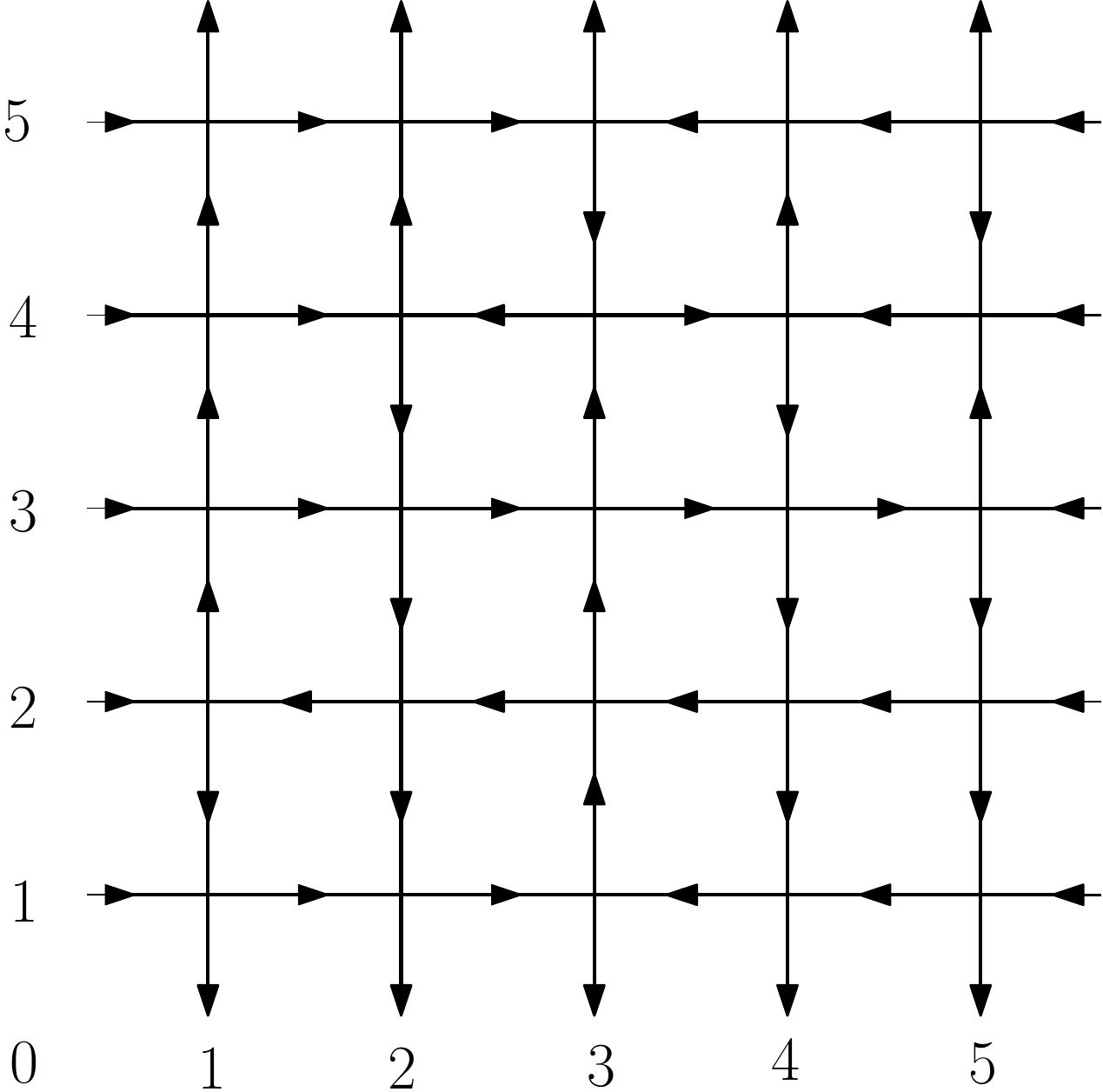} \hfill \includegraphics[width=0.43\linewidth]{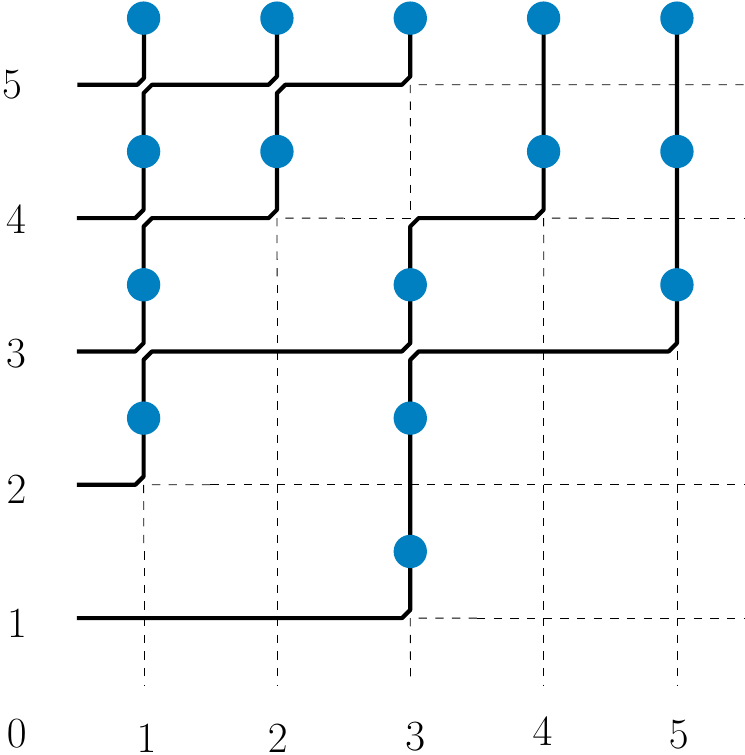}
\end{center}
        \caption{\label{Figure_DWBC} A sample configuration for the Domain Wall Boundary Conditions in $5\times 5$ square in arrows and paths representations. We additionally put blue dots on the vertical segments of paths.}
\end{figure}

It is customary to split the six types of vertices into three groups: we refer to Type 1 and Type 2 as type $a$ vertices, Types 3 and 4 are type $b$ vertices, and Types 5 and 6 are type $c$ vertices. The weights also reflect this split: rather than using $w_1,\dots,w_6$, we prefer the weights $a_1,a_2,b_1,b_2,c_1,c_2$, as in Figure \ref{Figure_six_vertices}. When dealing with the six-vertex model in a finite domain with fixed boundary conditions (as for our case of DWBC), there are four conservation laws: the simplest one is that the total number of vertices is deterministically fixed, and there are three others, see e.g.\ \cite[Lemma 2.1]{Gorin_Nicoletti_lectures}. As a corollary, one can change the six weights in four  ways without changing the Gibbs measure, but only getting factors which cancel between numerator and denominator in \eqref{eq_Gibbs_measure}. Hence, the Gibbs measure actually depends only on two parameters, rather than all six weights. Therefore, there is no loss of generality in considering the symmetric weights $a_1=a_2=a$, $b_1=b_2=b$, $c_1=c_2=c$, keeping in mind that all three coordinates of the triplet $(a:b:c)$ can still be multiplied by a common factor without changing \eqref{eq_Gibbs_measure}; see Figure \ref{Figure_DWBC_3} below for the products of symmetric weights corresponding to the all seven possible configurations for $n=3$.

In this paper we focus on the asymptotic analysis of the $(a,b,c)$--random configurations of the six-vertex model with Domain Wall Boundary Conditions as $n\to\infty$. The DWBC are the simplest possible fixed boundary conditions on the plane and in this role they have been intensively studied by various researchers. The interest is based on two components: on one hand, these boundary conditions lead to very rich  asymptotic behavior, so that we can use them for predictions for many other classes of domains; on the other hand, DWBC are \emph{integrable} boundary conditions, meaning that there are many more exact formulas available as compared to more general domains. The first manifestation of this integrability is the evaluation of the partition function $\mathcal Z$ as the Izergin-Korepin determinant \cite{Izergin87,korepin1982calculation}, which will also play a role in our developments.

From the behaviors side, one of the most interesting features of the six-vertex model in finite domains is the phenomenon of spatial phase separation. There are four different phases or kinds of local behaviors in the model: in the frozen phase only one type of vertex is seen (there are four distinct frozen phases corresponding to the first four types of vertices); in the rough (or liquid) phase all types of vertices are present, correlations decay polynomially in distance, the fluctuations are significant and are described by conformal objects, such as the Gaussian Free Field; in the smooth (or gaseous) phase all types of vertices are present, but the fluctuations are very small and correlations decay exponentially in distance; in the KPZ phase all types of vertices are again present, but the correlations  are inhomogeneous, in the sense that their decay significantly depends on the direction, and the fluctuations are described by the objects from random matrix theory and interacting particle systems, such as the Tracy--Widom distribution and Airy$_2$ process. It is expected that whenever $\Delta=\frac{a^2+b^2-c^2}{2ab}<-1$, the first three phases might occur, for $-1<\Delta<1$, only frozen and rough (liquid) phases appear, and for $\Delta>1$ frozen, rough, and KPZ phases are possible. Interestingly, several different phases might coexist\footnote{Outside the $\Delta=0$ case, this is mostly conjectural, as rigorous results are very limited.} in the same domain, being separated by certain curves. As a sample result, consider the DWBC model with $a=b=1$, $c=\sqrt{2}$, resulting in $\Delta=0$. It is known \cite{Jockusch-Propp-Shor98} that as $n\to\infty$, outside the inscribed circle the proportion of $c$--type vertices tends to $0$: near the bottom--right corner there are only Type 1 vertices, near the bottom-left corner there are only Type 3 vertices, near the top-left corner there are only Type 2 vertices, and near the top--right corner there are only Type 4 vertices. On the other hand, inside the inscribed circle all six types of vertices appear with positive proportions. For other values of $(a,b,c)$, the circle is replaced by more complicated curves, see Figure \ref{Figure_DWBC_simulations}.

\begin{figure}[t]
\begin{center}
   \includegraphics[width=0.46\linewidth]{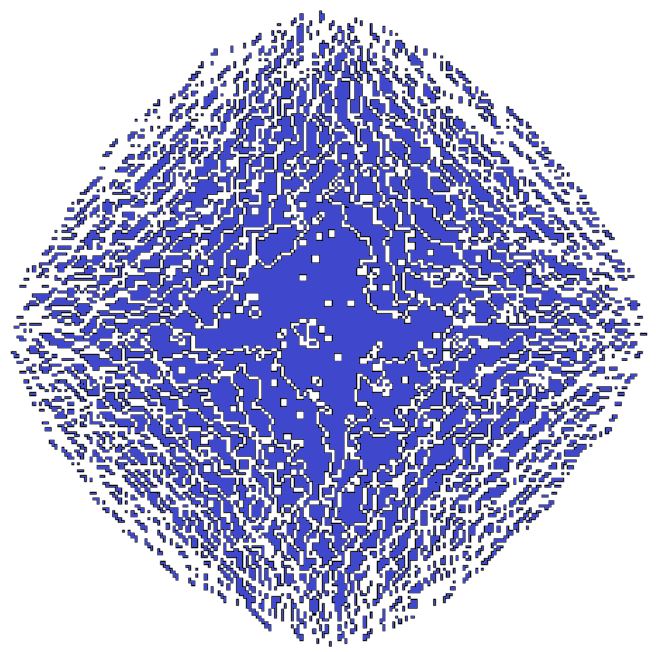} \hfill \includegraphics[width=0.46\linewidth]{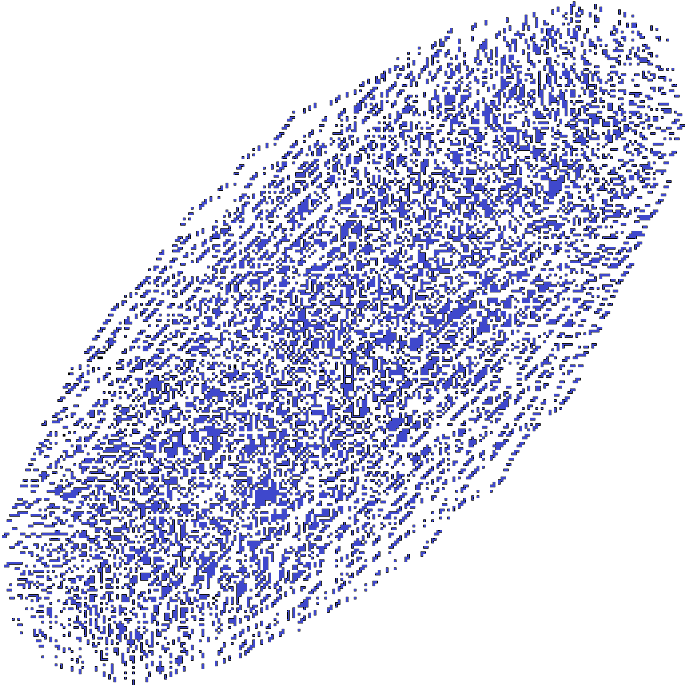}
\end{center}
        \caption{\label{Figure_DWBC_simulations} A random configuration for the Domain Wall Boundary Conditions with only $c$--type vertices shown. Left panel: $n=200$, $a=b=1$, $c=\sqrt{8}$; $\Delta=-3$. Right panel: $n=256$, $a=c=2$, $b=1$; $\Delta=\tfrac{1}{4}$.  We are grateful to David Keating for helping with these pictures based on \cite{keating2018random}.}
\end{figure}

The problem of the explicit identification of the curve separating different phases in the model with DWBC attracted significant attention of researchers. In the theoretical physics literature, several approaches were developed in \cite{Colomo-Pronko10,Colomo-Pronko-ZinnJustin10,Colomo-Sportiello16} which resulted in exact formulas for this curve. The mathematical justifications for these formulas are known in $\Delta=0$ case \cite{Cohn-Elkies-Propp96, Jockusch-Propp-Shor98} and for $a=b=c=1$ \cite{Aggarwal20}; the rigorous approaches might extend to other values of $(a,b,c)$ as well. The next step is to understand the fluctuations of this curve. At $\Delta=0$ for generic points of the separation curve the fluctuations were shown in \cite{Johansson02, Johansson05} to be of size $n^{1/3}$, with scaling limit governed by the Tracy-Widom distribution and the Airy$_2$ process\footnote{Despite the same limiting objects showing up, the way they appear is quite different from the asymptotic analysis of the KPZ phase in \cite{aggarwal2018current,borodin2016stochastic,dimitrov2023two}.}. A result of a similar flavor was also recently achieved for $a=b=c=1$ case in \cite{ayyer2021goe}. Beyond these results and, in particular, for other values of $(a,b,c)$ we are only aware of computer-simulated studies, see, e.g., \cite{lyberg2023fluctuation} or \cite{praehofer2023domain}.

Another area of interest is to investigate the special points where the separation curve touches the boundary of the square; because near all other points of the boundary the configurations are frozen, one could argue that the model interacts with the boundaries only through these special points, and therefore they are of central importance. For the $\Delta=0$ case the fluctuations near such points (there are four of them, corresponding to four sides of the square) were shown in \cite{Johansson-Nordenstam06} to be governed by the Gaussian Unitary Ensemble of random matrix theory and the related GUE-corners process. Later a similar result was obtained in \cite{Gorin14,gorin2015asymptotics} for the $a=b=c=1$ case. Our paper continues this line of research and addresses the question: does the GUE-corners asymptotics near the boundary extend to all values of $(a,b,c)$ or are there alternative limiting objects?

\subsection{Main results}

\label{Section_main_results}

The six-vertex model with DWBC has many symmetries: one can rotate the configurations in the domain by $\pi/2$ or reflect them with respect to the vertical axis, horizontal axis, or diagonals. Because of these symmetries, we only study the behavior of the model near the bottom boundary of the domain. Let us introduce an encoding of the configurations by arrays of integers.

\begin{definition}
\label{Definition_monotone_triangle}
 Consider a random $(a,b,c)$--weighted configuration of the six-vertex model with DWBC in the $n\times n$ square in the paths representation. For $k=1,2,\dots,n$, let random variables $(1\le \lambda_1^k< \lambda_2^k<\dots<\lambda_k^k\le n)$ denote the positions of vertical segments of paths connecting vertices in rows $k$ and $k+1$; they are blue dots in Figure \ref{Figure_DWBC}.
\end{definition}

In particular, the configuration of Figure \ref{Figure_DWBC} has: $\lambda_1^1=3$, $(\lambda_1^2,\lambda_2^2)=(1,3)$, $(\lambda_1^3,\lambda_2^3,\lambda_3^3)=(1,3,5)$, $(\lambda_1^4,\lambda_2^4,\lambda_3^4,\lambda_4^4)=(1,2,4,5)$, and $(\lambda_1^5,\lambda_2^5, \lambda_3^5, \lambda_4^5,\lambda_5^5)=(1,2,3,4,5)$. Note that for any configuration, interlacement inequalities
$\lambda_{i}^{k+1}\le \lambda_i^k \le \lambda_{i+1}^{k+1}$ hold and the top row is deterministically fixed: $(\lambda_1^n,\lambda_2^n,\dots,\lambda_n^n)=(1,2,\dots,n)$. Arrays of integers satisfying these conditions are called \emph{monotone triangles}.

We also need to define a limiting object:
\begin{definition}
\label{Definition_GUE_corners}
 Let $X$ be an $N\times N$ matrix of i.i.d.\ complex Gaussian random variables $\mathcal N(0,1)+\ii \mathcal N(0,1)$ and let the Hermitian matrix $\mathcal M$ be $\frac{1}{2}(X+X^*)$. The \emph{GUE-corners process} of rank $N$ is an array $\{g_i^k\}_{1\le i \le k \le N}$ of real numbers, such that $g_i^k$ is the $i$th eigenvalue of the top--left $k\times k$ corner of $\mathcal M$.
\end{definition}
The eigenvalues in the GUE-corners process interlace, i.e.\ $g_{i}^{k+1}\le g_i^k\le g_{i+1}^{k+1}$. Their distribution can be understood through the joint density, cf.\ \cite[Section 4]{Baryshnikov01}, or through the correlation functions \cite{Johansson-Nordenstam06}, \cite{okounkov2006birth}. Since each top-left square corner of $N\times N$ matrix $\mathcal M$ is again a matrix of the same type, we can omit $N$ and think about the GUE--corners process as an infinite array $(g_i^k)_{1\le i \le k}$, so that when we restrict to $k\le N$, we are back to Definition \ref{Definition_GUE_corners}. Here is our first main result:
\begin{theo} \label{Theorem_GUE_corners}
 Choose $a,b,c>0$ such that $\Delta=\frac{a^2+b^2-c^2}{2ab}<1$. Let $(\lambda_i^k)$ be a random monotone triangle corresponding to $(a,b,c)$--random configuration of the six-vertex model with DWBC, as in Definition \ref{Definition_monotone_triangle}. For explicit constants ${\mathfrak m}(a,b,c)$ and ${\mathfrak s}(a,b,c)$ given in Section \ref{Section_GUE_proof}, we have
 $$
  \lim_{n\to\infty} \left( \frac{\lambda_i^k- {\mathfrak m}(a,b,c)\cdot n}{{\mathfrak s}(a,b,c) \sqrt{n}} \right)_{1\le i \le k} = \bigl(g_i^k\bigr)_{1\le i \le k}
 $$
 in the sense of convergence of finite-dimensional distributions.
\end{theo}

For the special cases $\Delta=0$ and $a=b=c=1$, Theorem \ref{Theorem_GUE_corners} was proven in \cite{Johansson-Nordenstam06} and \cite{Gorin14,gorin2015asymptotics}, respectively. For other values of $(a,b,c)$ the result is new. In the six-vertex model with other boundary conditions (and special choices of $a,b,c$ in the $\Delta>1$ situation), GUE-corners also appeared in \cite{dimitrov2020six,dimitrov2022gue}. In a more general context, Theorem \ref{Theorem_GUE_corners} can be viewed as a part of the program for establishing the universal appearance of the GUE-corners process in random matrix theory and $2d$ statistical mechanics: see \cite{aggarwal2022gaussian} for a universality result in the context of random lozenge tilings, \cite{cuenca2021universal,meckes2020random} for investigations in the context of orbital measures on random matrices, \cite{okounkov2006birth} for reasons to expect universality, and \cite{Gorin_Nicoletti_lectures} for a general discussion in the context of the six-vertex model. Along these lines, in particular, Theorem \ref{Theorem_GUE_corners} is the first result in which the GUE--corners process appeared in the six-vertex model with $\Delta<-1$.

\bigskip


The $\Delta<1$ restriction in Theorem \ref{Theorem_GUE_corners} is crucial for the result and for $\Delta>1$ the asymptotic behavior changes dramatically, see Figure \ref{Figure_DWBC_Delta_large}. We need to introduce a new limiting object.

\begin{figure}[t]
\begin{center}
   \includegraphics[width=0.45\linewidth]{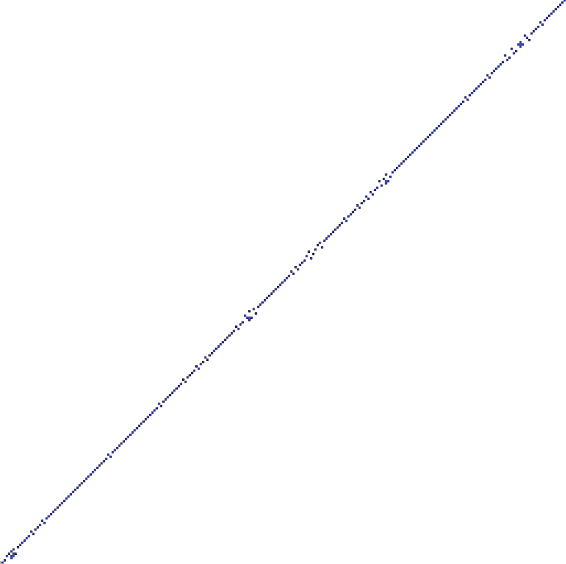} \hfill  \includegraphics[width=0.45\linewidth]{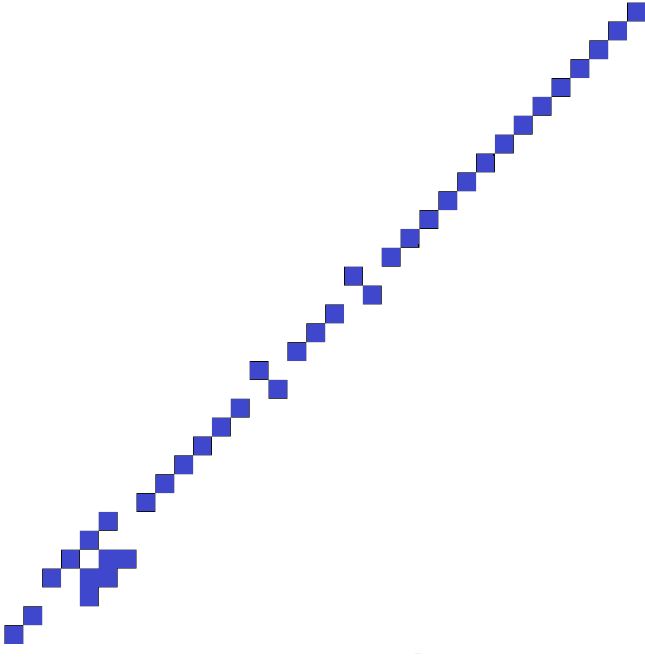}
\end{center}
        \caption{\label{Figure_DWBC_Delta_large} A random configuration for the Domain Wall Boundary Conditions with only $c$--type vertices shown. Left panel: $n=256$, $a=3$, $b=c=1$; $\Delta=\tfrac{3}{2}$. Right panel: corner of the same picture enlarged.  We are grateful to David Keating for helping with these pictures based on \cite{keating2018random}.}
\end{figure}

\begin{figure}[t]
\begin{center}
   \includegraphics[width=0.6\linewidth]{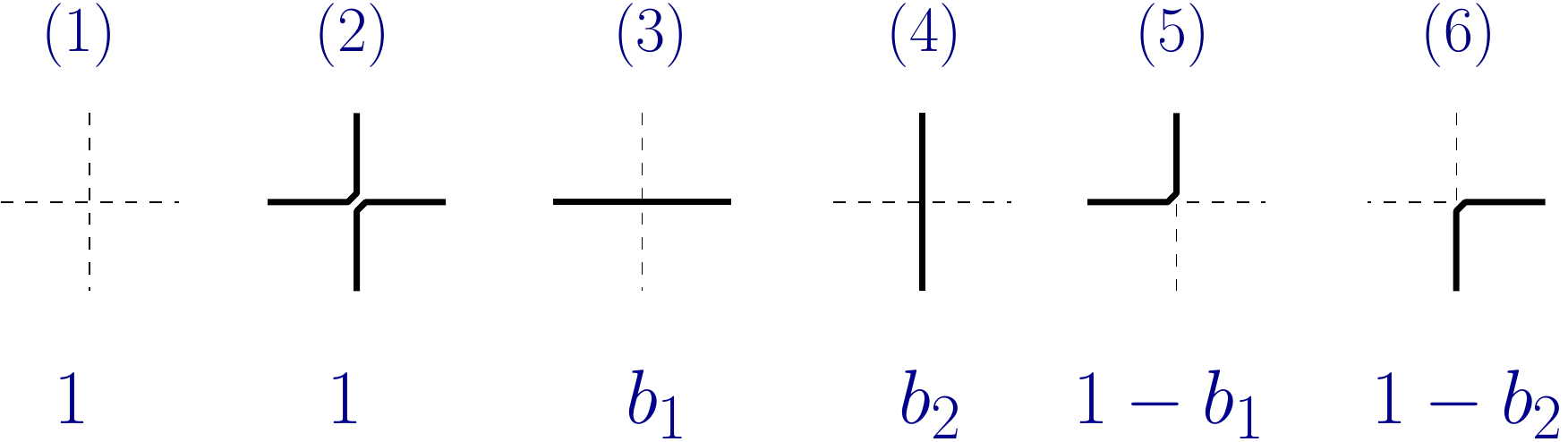}
\end{center}
        \caption{\label{Figure_stochastic} The stochastic weights}
\end{figure}

We return to asymmetric weights and set them to be $1,1,b_1,b_2,c_1,c_2$ subject to \emph{stochasticity condition} $b_1+c_1=b_2+c_2=1$, as shown in Figure \ref{Figure_stochastic}. The {\it stochastic six-vertex model} in the quadrant $\mathbb Z_{>0}\times \mathbb Z_{>0}$ has boundary conditions which look like the bottom-left corner of DWBC (i.e.\ paths entering from the left and not entering from below) and should be thought as $n=\infty$ version of DWBC. The vertex types are sampled recursively: we first define the type of the vertex at $(1,1)$, then types at $(1,2)$ and $(2,1)$, then types at $(1,3)$, $(2,2)$, and $(3,1)$, etc. When we need to sample a vertex at $(x,y)$, we already know whether a path should be entering into it from below, and whether a path should be entering from the left. Hence, there are four cases:
\begin{itemize}
 \item If no paths are entering from both directions, then the vertex at $(x,y)$ must have Type 1.
 \item If paths are entering from both directions, then the vertex at $(x,y)$ must have Type 2.
 \item If a path is only entering from the left, then we flip an independent biased coin and set the vertex to be of Type 3 with probability $b_1$ and of Type 5 with probability $1-b_1$.
\item If a path is only entering from below, then we flip an independent biased coin and set the vertex to be of Type 4 with probability $b_2$ and of Type 6 with probability $1-b_2$.
\end{itemize}
The stochastic six-vertex model on the torus first appeared in \cite{gwa1992six}, and the version in the quadrant, which we use, originates in \cite{borodin2016stochastic}. Our next result connects it to DWBC.

\begin{theo} \label{Theorem_convergence_to_stochastic}
Choose parameters $a,b,c>0$ such that $\Delta=\frac{a^2+b^2-c^2}{2ab}>1$ and $a>b$. Let $(\lambda_i^k)$ be a random monotone triangle corresponding to $(a,b,c)$--random configuration of the six-vertex model with DWBC, as in Definition \ref{Definition_monotone_triangle}. Then
$$
  \lim_{n\to\infty} \left( \lambda_i^k \right)_{1\le i \le k} = \bigl(\mu_i^k\bigr)_{1\le i \le k},
$$
in the sense of convergence of finite-dimensional distributions, where the array $\bigl(\mu_i^k\bigr)$ is obtained by the procedure of Definition \ref{Definition_monotone_triangle} from the configuration of the stochastic six-vertex model with
$$
b_1= \frac{a^2+b^2-c^2-\sqrt{(a^2+b^2-c^2)^2 - 4a^2b^2}}{2a^2}, \qquad
b_2= \frac{a^2+b^2-c^2+\sqrt{(a^2+b^2-c^2)^2 - 4a^2b^2}}{2a^2}.
$$
\end{theo}
\begin{remark} 
 The conditions of the theorem guarantee that $0<b_1<b_2<1$. 
\end{remark}
\begin{remark} \label{Remark_symmetry} 
 The case $a<b$ can be reduced to $a>b$ by the following procedure. We first locally XOR the six types of vertices with horizontal line, i.e.\ with the Type 3 vertex. Then we reflect the entire configuration with DWBC boundary conditions with respect to the vertical axis. As a result of both steps, we have swapped the vertex types: Type 1 $\leftrightarrow$ Type 3 and Type 2 $\leftrightarrow$ Type 4 --- which means that we have swapped the  weights $a\leftrightarrow b$. Simultaneously, we have interchanged the bottom--left and bottom--right corners. Hence, in the version of Theorem \ref{Theorem_convergence_to_stochastic} for $a<b$ we should be looking at $N+1-\lambda_i^k$ rather than $\lambda_i^k$, and the limit of the former is described by the stochastic six-vertex model, but with modified formulas for $b_1$ and $b_2$ obtained by swap $a\leftrightarrow b$.
 
 Also note that the final case $a=b$ is impossible when $\Delta>1$.
\end{remark}

Theorem \ref{Theorem_convergence_to_stochastic} should be contrasted with the results of \cite{dimitrov2020six,dimitrov2022gue}, where GUE--corners process was found for the boundary asymptotics in the six-vertex model for $\Delta>1$ situation with boundary conditions very different from DWBC. This indicates that in the $\Delta>1$ situation the possible boundary limits of the six-vertex model are richer than for $\Delta<1$.

\bigskip

In Theorems \ref{Theorem_GUE_corners} and \ref{Theorem_convergence_to_stochastic} we had $a,b,c>0$, but we can also investigate the boundary cases as one of the parameters tends to $0$. One interesting situation is achieved by setting $c=\eps>0$ and then sending $\eps\to 0$ (which implies $\Delta\ge 1$) keeping the size of the square $n$ fixed: in the limit the Gibbs measure is supported on $n!$ configurations with minimal number of $c$--type vertices, which can be identified with $n\times n$ permutation matrices. For the $n\to\infty$ limiting object in this situation we introduce the $q$-shuffle procedure of \cite[Section 4]{gnedin2010q}, closely related to the Mallows measure on permutations\footnote{The Mallows measure assigns to a permutation $\tau:\{1,\dots,n\}\to\{1,\dots,n\}$ a weight proportional to $q^{\mathrm{inv}(\tau)}$, where $\mathrm{inv}(\tau)$ is the number of inversions in $\tau$. It was first introduced in \cite{mallows1957non} and has been intensively studied in probability, statistics, and theoretical physics, see, e.g.\ \cite{he2022cycles} for one recent result and many references to previous work.}.

\begin{definition} \label{Def_q_shuffle}
The $q$-exchangeable random bijection $\tau:\mathbb Z_{>0}\to\mathbb Z_{>0}$ is defined by starting from a sequence $\xi_1,\xi_2,\dots$ of i.i.d.\ Geometric random variables with parameter $0<q<1$: $\mathrm{Prob} (\xi_i=k)={(1-q) q^{k-1}}$, $k=1,2,\dots$. We set $\tau(1)=\xi_1$. More generally, for $k>1$ we let $\tau(k)$ be the $\xi_k$-th letter in the infinite word $1234\dots$ from which we removed $\tau(1),\tau(2),\dots,\tau(k-1)$.
\end{definition}
The $\tau$ of Definition \ref{Def_q_shuffle} is, indeed, a bijection as \cite[Lemma 4.2]{gnedin2010q} explains.

\begin{theo} \label{Theorem_Mallows}
 Choose parameters $a>b>0$ and $c=0$. Let $(\lambda_i^k)$ be a random monotone triangle corresponding to $(a,b,c)$--random configuration of the six-vertex model with DWBC, as in Definition \ref{Definition_monotone_triangle}. Then
$$
  \lim_{n\to\infty} \left( \lambda_i^k \right)_{1\le i \le k} = \bigl(\mu_i^k\bigr)_{1\le i \le k},
$$
in the sense of convergence of finite-dimensional distributions, where the array $\bigl(\mu_i^k\bigr)$ is created out of the $q$-exchangeable random bijection $\tau$ of positive integers of Definition \ref{Def_q_shuffle} with $q=\frac{b^2}{a^2}$ by requesting that for each $k\ge 1$ the $k$ numbers $(\mu_1^k<\mu_2^k<\dots,<\mu_k^k)$ are the reordering of $(\tau(1),\tau(2),\dots,\tau(k))$ in the increasing order.
\end{theo}
\begin{remark}
 As in Remark \ref{Remark_symmetry}, the $0<a<b$ case can be obtained from the $a>b>0$ case by applying vertical symmetry.
\end{remark}

\subsection{Further directions}

\begin{figure}[t]
\begin{center}
   \includegraphics[width=0.9\linewidth]{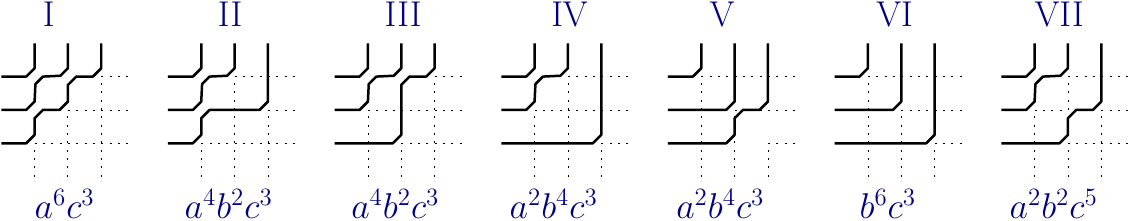}
\end{center}
        \caption{\label{Figure_DWBC_3} Seven configurations for $n=3$ and corresponding $(a,b,c)$--weights.}
\end{figure}

In the space of all possible $(a:b:c)$ parameters there are several additional points not covered by our results, yet potentially leading to interesting boundary limits for the six-vertex model with DWBC. First, there are more ways to send some of the $a,b,c$ parameters to $0$. In order to see that, it is helpful to draw all seven configurations in $3\times 3$ square and their corresponding weights, see Figure \ref{Figure_DWBC_3}. If we send $c\to 0$, as in Theorem \ref{Theorem_Mallows}, then we get a probability measure concentrated on the first six configurations, with weights obtained by removing $c^3$ factor. Alternatively, we could send $b\to 0$, getting a measure concentrated on configuration I; or we could send $a\to 0$, getting a measure concentrated on configuration VI; or we could simultaneously send $a,b\to 0$, getting a measure concentrated on configuration VII. We do note address this in our text, yet one can hope that there are ways to perform some of these limits simultaneously with $n\to\infty$ and get some non-trivial limits near the boundary of the square.

\medskip

Another intriguing point is $\Delta=1$, not covered by any of Theorems \ref{Theorem_GUE_corners}, \ref{Theorem_convergence_to_stochastic}, or \ref{Theorem_Mallows}. We expect a rich family of new limits of the six-vertex model with DWBC near boundaries as $n\to\infty$ simultaneously with $\Delta\to 1$. For comparison, in the stochastic six-vertex model (which always has $\Delta>1$) two such limits were previously found, leading to Kardar--Parisi--Zhang SPDE in \cite{corwin2020stochastic} and  to the stochastic telegraph equation in \cite{borodin2019stochastic, shen2019stochastic}. For the DWBC with $\Delta$ close to $1$,  we are only aware of the investigations of the partition function in \cite{bleher2014calculation}.
 We leave the study of all possible limit regimes in the neighborhood of $\Delta=1$ to a future research and only provide one (relatively simple) result in this direction.

\begin{prop} \label{Proposition_c_0_double}
 Let $(\lambda_i^k)$ be a random monotone triangle corresponding to $(a,b,c)$--random configuration of the six-vertex model with DWBC, as in Definition \ref{Definition_monotone_triangle}. Fix $\theta\in\mathbb R$, set $c=0$ and suppose that $a$ and $b$ depend on $n$ in such a way that
 \begin{equation}
 \label{eq_q_to_1}
  \lim_{n\to\infty}\left[ n \ln \left(\frac{b^2}{a^2}\right)\right]=\theta.
 \end{equation}
 Then
$$
  \lim_{n\to\infty} \left( \frac{1}{n}\lambda_i^k \right)_{1\le i \le k} = \bigl(\zeta_i^k\bigr)_{1\le i \le k},
 $$
in the sense of convergence of finite-dimensional distributions, where the array $\bigl(\zeta_i^k\bigr)$ is obtained from a sequence $\eta_i$, $i=1,2,\dots$, of i.i.d.\ random variables of density
\begin{equation}
\label{eq_eta_density}
 \rho_\eta(x)=\frac{\theta}{e^{\theta}-1} e^{\theta x}, \qquad x\in [0,1],
\end{equation}
by requiring that for each $k\ge 1$, the vector $(\zeta_1^k,\zeta_2^k,\dots,\zeta_k^k)$ is a reordering of $(\eta_1,\eta_2,\dots,\eta_k)$ in the increasing order.
\end{prop}
\begin{remark}
 In this asymptotic regime $$\Delta=\frac{a^2+b^2-c^2}{2ab}=\frac{a}{2b}+\frac{b}{2a}=1+\frac{\theta^2}{8 n^2}+o\left(\frac{1}{n^2}\right).$$
\end{remark}

\subsection{An outline of the proofs}

The key ingredient for our proofs of Theorems \ref{Theorem_GUE_corners} and \ref{Theorem_convergence_to_stochastic} is an asymptotic analysis of the inhomogeneous partition function for the DWBC six-vertex model. The most general of such partition functions, $\mathcal Z_n(\chi_1,\dots,\chi_n; \psi_1,\dots,\psi_n; \gamma)$, depends on $n$ parameters $\chi_1,\dots,\chi_n$ attached to the columns of the square, $n$ parameters $\psi_1,\dots,\psi_n$ attached to the rows of the square, and an additional parameter $\gamma$. When all $\chi_i$s are equal to each other and all $\psi_i$s are equal to each other, we are back to the homogeneous $(a,b,c)$--weighting in \eqref{eq_Gibbs_measure}. The most relevant for our studies is the (normalized version of) the partition function $\mathcal Z_n(0^n; t+\xi_1,\dots,t+\xi_k, t^{n-k}; \gamma)$, which has $k$ inhomogeneities and this number does not grow with $n$. The asymptotic analysis of this partition function as $n\to\infty$ occupies the largest part of our text. In Section \ref{Section_inhomogeneous_partition_function} we define it, state asymptotic results, and introduce our main tool --- a representation in terms of a $\beta=2$ log-gas system in Theorem \ref{Theorem_partition_through_Bessel}. Section \ref{Section_asymptotics_through_log_gas} provides a short (but non-rigorous) argument for extracting the asymptotics of $\mathcal Z_n(0^n; t+\xi_1,\dots,t+\xi_k, t^{n-k}; \gamma)$ from a combination of Theorem \ref{Theorem_partition_through_Bessel} with asymptotics of linear statistics for log-gases and known asymptotic expansions of multivariate Bessel functions. In order to make this argument rigorous, one would have to establish certain concentration bounds on fluctuations of linear statistics for log-gases with non-smooth potential. The required bounds go beyond of what's currently available in the literature (see Section \ref{Section_asymptotics_through_log_gas} for the discussion), and rather than trying to prove them, we designed an alternative approach circumventing this difficulty.

For the second, rigorous, approach we discover in Section \ref{Section_k1_asymptotics} that the $k=1$ inhomogeneous partition function can be expressed as a contour integral involving polynomials orthogonal with respect to the weight of the aforementioned log-gas. These polynomials previously appeared in the literature on the fully homogeneous partition function (see \cite{Bleher-Liechty14} for the review), and their asymptotics can be obtained by the Riemann--Hilbert analysis. The desired asymptotic expansion of the polynomials has never been stated in the literature, so we need to carefully extract it from there by massaging the previously existing results and making additional computations: see Section \ref{sec:OPasy} for the statements and Appendix \ref{app:OP_asymptotics} for the proofs. In the rest of Section \ref{Section_k1_asymptotics} we insert the expansions of the polynomials into contour integrals, perform the steepest descent asymptotic analysis of these integrals, and achieve the desired asymptotic expansion for the normalized $\mathcal Z_n(0^n; t+\xi, t^{n-1}; \gamma)$. Next, we generalize the $k=1$ case to general $k$ in Section \ref{Section_general_reduction}, based on an exact expression (see Theorem \ref{Theorem_partition_general_k_reduction}) of the partition function with $k$ inhomogeneities as a $k\times k$ determinant involving partition functions with one inhomogeneity.

Section \ref{Section_probabilistic_proofs} extracts the limits of Theorems \ref{Theorem_GUE_corners} and \ref{Theorem_convergence_to_stochastic} from the asymptotic expansions of the partition functions developed in the previous sections. The idea here is to interpret the inhomogeneous partition functions as generating functions for the marginal probabilities of the six-vertex model along a horizontal section of the domain. In discrete settings, these generating functions are written in terms of inhomogeneous Hall-Littlewood symmetric polynomials (see \cite{borodin2017integrable} for a review of properties of these polynomials), while for the continuous limits we use $\beta=2$ multivariate Bessel functions (they are evaluations of the Harish Chandra/Itzykson-Zuber \cite{Harish-Chandra57, Itzykson-Zuber80} spherical integral). On our path we need to develop theorems connecting convergence of generating functions to convergence of random variables, which can be viewed as non-standard multi-dimensional versions of the classical relation between convergence of characteristic functions or Laplace transforms on one side and convergence in distribution for random variables on $\mathbb R$ on the other side.

In Section \ref{Section_c_0} we investigate the special case $c=0$ and prove Theorem \ref{Theorem_Mallows} and Proposition \ref{Proposition_c_0_double}. This section is essentially independent from the rest of the text and does not use the advanced analytic machinery of the previous sections. The argument is based on a curious observation that the random configurations of the six-vertex model at $c=0$ can be identified with random permutations distributed according to the Mallows measure.

\subsection*{Acknowledgements}
We would like to thank Pavel Bleher and Filippo Colomo for fruitful communications related to this project. We are grateful to David Keating for providing simulations.

\medskip

\noindent Part of this work began during the Fall 2021 program {\it Universality and Integrability in Random Matrix Theory and Interacting Particle Systems} at the Mathematical Sciences Research Institute (MSRI) in Berkeley, CA. That program was supported by the National Science Foundation under Grant No. DMS-1928930.

\medskip

\noindent The work of V.G.\ was partially supported by NSF grant DMS - 2152588.
The work of K.L. was partially supported by a Faculty Summer Research Grant from the College of Science and Health at DePaul University.

\section{Inhomogeneous partition functions}

\label{Section_inhomogeneous_partition_function}

Our proofs of Theorems \ref{Theorem_GUE_corners} and \ref{Theorem_convergence_to_stochastic} are based on the asymptotic analysis of inhomogeneous partition functions of the six-vertex model, which we present in this section.

\subsection{Izergin--Korepin determinant}

In our main results stated in Section \ref{Section_main_results}, the weights $(a:b:c)$ of the vertices did not depend on the position inside the $n\times n$ square. However, for the proofs we introduce inhomogeneities and make the weight of a vertex at $(x,y)$ explicitly depend on $1\le x,y\le n$. We introduce four different parameterizations of the weights by real numbers $\gamma$ and $t$. The cases correspond to different values of $\Delta=\frac{a^2+b^2-c^2}{2ab}$:
\begin{itemize}
\item {\bf Ferroelectric phase}. For $\Delta>1$ we parameterize the weights as
\begin{equation}\label{eq_case1}
a\equiv a(t, \ga)=\sinh(t-\ga), \quad
b\equiv b(t, \ga)=\sinh(t+\ga), \quad
c\equiv c(\ga)=\sinh(2\ga), \qquad
0<|\ga|<t.
\end{equation}

\item {\bf Disordered phase}. For $-1<\Delta<1$ we parameterize the weights as
\begin{equation}\label{eq_case2}
a\equiv a(t, \ga)=\sin(\ga-t), \quad
b\equiv b(t, \ga)=\sin(\ga+t), \quad
c\equiv c(\ga)=\sin(2\ga), \qquad
|t|<\ga<\pi/2.
\end{equation}
\item {\bf Antiferroelectric phase}. For $\Delta<-1$ we parameterize the weights as
\begin{equation}\label{eq_case3}
a\equiv a(t, \ga)=\sinh(\ga-t), \quad
b\equiv b(t, \ga)=\sinh(\ga+t), \quad
c\equiv c(\ga)=\sinh(2\ga), \qquad
|t|<\ga.
\end{equation}

\item {\bf Boundary phase.} For $\Delta=-1$, we parameterize\footnote{The choice \eqref{eq_case4} comes with an overparameterization: multiplication of both $\gamma$ and $t$ by the same constant does not change the probability measure on configuration.}  the weights as
\begin{equation}\label{eq_case4}
a\equiv a(t, \ga)=\gamma-t, \quad
b\equiv b(t, \ga)=\gamma+t, \quad
c\equiv c(\ga)=2 \gamma, \qquad
|t|<\gamma.
\end{equation}
\end{itemize}

\noindent The $\Delta=-1$ case is obtained from either of the parameterizations \eqref{eq_case2} or \eqref{eq_case3} as $t,\ga\to 0$ limit:
$$
\bigl(\gamma-t,\, \gamma+t,\, 2\gamma \bigr) = \lim_{\eps \to 0} \frac{1}{\eps} \bigl(a(t\eps, \ga\eps),\, b(t\eps,\ga\eps),\, c(\ga\eps)\bigr).
$$
There is another boundary case at $\Delta=1$, but we are not addressing it in this paper.

\begin{definition} \label{Definition_inhomogeneities}
 Fix $n=1,2,\dots$, and $2n+1$ complex parameters $\chi_1,\dots,\chi_n; \psi_1,\dots,\psi_n; \gamma$, and one of the four cases \eqref{eq_case1}, \eqref{eq_case2}, \eqref{eq_case3}, or \eqref{eq_case4}. For a configuration $\sigma$ of the six-vertex model in an $n\times n$ square with DWBC, we define the weight of the vertex at $(x,y)$ to be
 $$
  \omega(x,y;\sigma)=
  \begin{cases}
    a(\psi_y-\chi_x, \ga), &\quad\text{if }\sigma(x,y)\text{ is of Type 1 or 2,}\\
    b(\psi_y-\chi_x, \ga), &\quad\text{if }\sigma(x,y)\text{ is of Type 3 or 4,}\\
    c(\ga), &\quad\text{if }\sigma(x,y)\text{ is of Type 5 or 6.}
  \end{cases}
 $$
 The inhomogeneous partition function is then defined as
 \begin{equation}
  \mathcal Z_n(\chi_1,\dots,\chi_n; \psi_1,\dots,\psi_n; \gamma)=\sum_{\sigma} \prod_{x=1}^n \prod_{y=1}^n \omega(x,y;\sigma),
 \end{equation}
 where the sum goes over all configurations in the $n\times n$ square with DWBC.
\end{definition}

\begin{theo}[Izergin-Korepin determinant] \label{Theorem_IK_det} In each of the four phases \eqref{eq_case1}-\eqref{eq_case4} we have
\begin{multline}
\label{eq_IK_det}
  \mathcal Z_n(\chi_1,\dots,\chi_n;\, \psi_1,\dots,\psi_n;\, \gamma)\\= \frac{\prod\limits_{i,j=1}^n \bigl(a(\psi_j-\chi_i, \ga) b(\psi_j-\chi_i, \ga)\bigr)}{\prod\limits_{i<j} \bigl(b(\chi_i-\chi_j,0)b(\psi_i-\psi_j,0)\bigr)} \det \left[ \frac{c(\ga)}{a(\psi_j-\chi_i, \ga) b(\psi_j-\chi_i, \ga)} \right]_{i,j=1}^n.
\end{multline}
\end{theo}

\noindent The proof of Theorem \ref{Theorem_IK_det} is due to \cite{Izergin87, Izergin-Coker-Korepin92}; see also \cite[Chapter 5]{Bleher-Liechty14} or \cite[Theorem 2.2]{Gorin_Nicoletti_lectures}.

\subsection{Normalized partition functions}

Choosing the inhomogeneities $\chi_1,\dots,\chi_n; \psi_1,\dots,\psi_n$ in various ways, one can extract lots of asymptotic information out of $\mathcal Z_n(\chi_1,\dots,\chi_n; \psi_1,\dots,\psi_n; \gamma)$. Here is one particular choice which allows us to understand the limiting behavior near boundaries.

\begin{definition} \label{Definition_inhom_norm_part} Fix $n=1,2,\dots$, and $1\le k\le n$, one of the four cases \eqref{eq_case1}, \eqref{eq_case2}, \eqref{eq_case3}, or \eqref{eq_case4}, and real parameters $\gamma$, $t$ satisfying the inequalities for this case. For complex $\xi_1,\dots,\xi_k$, we define
\begin{equation}
 \ZZ_n(\xi_1,\dots,\xi_k;\, t,\gamma)=\frac{ \mathcal Z_n(0^n;\, t+\xi_1,\dots,t+\xi_k, t^{n-k};\, \gamma)}{\mathcal Z_n(0^n;\, t^n;\, \gamma)}.
\end{equation}
\end{definition}

For normalized partition functions $\ZZ_n$ our results are summarized in the following four theorems.

\begin{theo} \label{Theorem_asymptotics_case1}
For weights in the ferroelectric phase $\Delta>1$ parameterized as \eqref{eq_case1}, for any $k=1,2,\dots$, and complex $\xi_1, \dots, \xi_k$,  satisfying $-\tfrac{t}{2}+\tfrac{|\ga|}{2} < \Re \xi_j $, and $|\Im \xi | < \pi/2$, $1\le j \le k$, we have
\begin{equation}
\ZZ_n(\xi_1,\dots,\xi_k;\, t,\gamma) =  \prod_{j=1}^k\left(\frac{\sinh(t+|\ga|+\xi_j)}{\sinh(t+|\ga|)}\right)^n
\cdot \exp\left[-\sum_{j=1}^k \xi_j +\bigO\left(\frac{1}{\sqrt{n}}\right) \right], \qquad n\to\infty,
\end{equation}
where  the remainder $\bigO\left(\frac{1}{\sqrt{n}}\right)$ is uniform over $\xi_j$ belonging to compact subsets of the set $\left\{ \xi\in \C : -\tfrac{t}{2}+\tfrac{|\ga|}{2} < \Re \xi, \ |\Im \xi | < \pi/2\right\}$.
\end{theo}

\begin{theo} \label{Theorem_asymptotics_case2}
For weights in the disordered phase $|\Delta|<1$ parameterized as \eqref{eq_case2}, for any $k=1,2,\dots,$ and any $\xi_1, \dots, \xi_k \in \C$, we have as $n\to\infty$
\begin{align}  \notag
\ZZ_n\left(\frac{\xi_1}{\sqrt{n}},  \dots, \frac{\xi_k}{\sqrt{n}};\, t,\gamma\right) =  \exp\Biggl[ \sqrt{n}  &\sum_{j=1}^k \xi_j\left(\cot(\ga+t)-\cot(\ga-t)+\frac{\pi}{2\ga}\tan\left(\frac{\pi t}{2\ga}\right)\right) \\
 \notag - &\sum_{j=1}^k \frac{\xi_j^2}{2}\left(\frac{5}{3}- \frac{\pi^2}{6\ga^2}-\frac{\pi^2\tan^2\left(\frac{\pi t}{2\ga}\right)}{4\ga^2}+\cot^2(\ga+t)+\cot^2(\ga-t)\right)\\+&\bigO\left(\frac{1}{\sqrt{n}}\right) \Biggr],\label{eq_Zn_expansion_case2}
\end{align}
where the remainder $\bigO\left(\frac{1}{\sqrt{n}}\right)$ is uniform over $\xi_j$ belonging to compact subsets of $\C$.
\end{theo}

For the $\Delta<-1$ case we use four Jacobi elliptic theta functions, each with elliptic nome $q=e^{-\pi^2/(2\ga)}$, as in \cite[Section 7.2]{Bleher-Liechty14}:
\begin{equation}
\label{eq_elliptic_functions}
\vartheta_\ell(u)\equiv \vartheta_\ell( u; e^{-\pi^2/(2\ga)}), \quad \ell=1,2,3,4.
\end{equation}

\begin{theo} \label{Theorem_asymptotics_case3}
For weights in the antiferroelectric phase $\Delta<-1$ parameterized as \eqref{eq_case3}, for any $k=1,2,\dots,$ and any $\xi_1, \dots, \xi_k \in \C$, we have as $n\to\infty$,
\begin{align}  \notag
\ZZ_n&\left(\frac{\xi_1}{\sqrt{n}}, \dots, \frac{\xi_k}{\sqrt{n}};\, t,\gamma\right) =  \exp\Biggl[ \sqrt{n}  \sum_{j=1}^k \xi_j\left(\coth(\ga+t)-\coth(\ga-t)-\frac{\pi}{2\ga}\frac{\vartheta_2'\left(\frac{\pi t}{2\ga}\right)}{\vartheta_2\left(\frac{\pi t}{2\ga}\right)}
\right) \\
 \notag &+ \sum_{j=1}^k \frac{\xi_j^2}{2}\left(\frac{5}{3}- \frac{\pi^2}{12\ga^2}\left(\frac{\vartheta_2'\left(\frac{\pi t}{2\ga}\right)}{\vartheta_2\left(\frac{\pi t}{2\ga}\right)}\right)^2 + \frac{\pi^2}{12\ga^2} \sum_{\ell=1}^4 \left(\frac{\vartheta_\ell'\left(\pi \frac{t+\gamma}{4\gamma}\right)}{\vartheta_\ell\left(\pi \frac{t+\gamma}{4\gamma}\right)}\right)^2-\coth^2(\ga+t)-\coth^2(\ga-t)\right)
 \\&+\bigO\left(\frac{1}{\sqrt{n}}\right) \Biggr],\label{eq_Zn_expansion_case3}
\end{align}
where the remainder $\bigO\left(\frac{1}{\sqrt{n}}\right)$ is uniform over $\xi_j$ belonging to compact subsets of $\C$.
\end{theo}

\begin{theo} \label{Theorem_asymptotics_case4} For weights in the boundary phase $\Delta=-1$ parameterized as \eqref{eq_case4} with $\ga=1$, for any $k=1,2,\dots,$ and any $\xi_1, \dots, \xi_k \in \C$, we have as $n\to\infty$
\begin{align}  \notag
\ZZ_n\left(\frac{\xi_1}{\sqrt{n}}, \dots, \frac{\xi_k}{\sqrt{n}};\, t,1\right) =  \exp\Biggl[& \sqrt{n}  \sum_{j=1}^k \xi_j\left(\frac{1}{1+t} - \frac{1}{1-t}+\frac{\pi}{2}\tan\left(\frac{\pi t}{2}\right)\right) \\
 \notag &+ \sum_{j=1}^k \frac{\xi_j^2}{2}\left(\frac{1}{(1-t)^2}+\frac{1}{(1+t)^2}- \frac{\pi^2}{6}-\frac{\pi^2\tan^2\left(\frac{\pi t}{2}\right)}{4}\right)
 \\&+\bigO\left(\frac{1}{\sqrt{n}}\right) \Biggr],\label{eq_Zn_expansion_case4}
\end{align}
where the remainder $\bigO\left(\frac{1}{\sqrt{n}}\right)$ is uniform over $\xi_j$ belonging to compact subsets of $\C$.
\end{theo}
\begin{remark}
 Multiplying both $\xi_i$ and $t$ by $\gamma$ and then sending $\gamma\to 0$, the expression \eqref{eq_Zn_expansion_case4} can be obtained from either of \eqref{eq_Zn_expansion_case2} or \eqref{eq_Zn_expansion_case3}. Hence there is no discontinuity in the asymptotic expansions at $\Delta=-1$.
\end{remark}

\subsection{Reduction to log-gas}

Our approach to proving Theorems \ref{Theorem_asymptotics_case1}-\ref{Theorem_asymptotics_case4} starts from rewriting the normalized partition functions $\ZZ_n$ in terms of a log-gas. The general idea that a log-gas can be relevant for the six-vertex model with DWBC goes back to the work of Zinn-Justin \cite{Zinn_Justin00}, and our main contribution in this section is in finding a new twist of that idea, providing access to $\ZZ_n$ which we need.

Using the $a$, $b$, $c$ functions of \eqref{eq_case1}--\eqref{eq_case4}, we denote
\begin{equation}
\label{eq_phi_def}
\f(t;\,  \ga) := \frac{c(\ga)}{a(t,\ga)b(t,\ga)}
\end{equation}
and represent $\f$ as a Laplace transform in the variable $t$ of a measure, following \cite{Zinn_Justin00} and \cite{Bleher-Liechty14}.
\begin{lem} We have
\begin{equation}
\label{eq_Laplace_rep}
 \f(t;\, \ga)=\int_{-\infty}^{\infty} e^{tx} \m(\dd x),
\end{equation}
where the measure $\m$ depends on the phase and is given as follows.
\begin{enumerate}
 \item For weights \eqref{eq_case1}, $\m$ is supported on negative even integers and is given by
 \begin{equation}
 \label{eq_measure_case1}
    \m=\sum_{x\in 2 \Z_{<0}} 2 \sinh(-\ga x)\, \delta_x.
 \end{equation}
 \item For weights \eqref{eq_case2}, $\m$ has density with respect to the Lebesgue measure:
 \begin{equation}
 \label{eq_measure_case2}
  \m=\frac{\sinh\left(\frac{x(\pi - 2\ga)}{2}\right)}{\sinh(\pi x/2)}\dd x.
 \end{equation}
  \item For weights \eqref{eq_case3}, $\m$ is supported on even integers and is given by
  \begin{equation}
 \label{eq_measure_case3}
   \m=\sum_{x\in 2\Z} 2 e^{-\ga|x|}\, \delta_x.
 \end{equation}
  \item For weights \eqref{eq_case4}, $\m$ has density with respect to the Lebesgue measure:
 \begin{equation}
 \label{eq_measure_case4}
  \m=e^{-\gamma |x|}\dd x.
 \end{equation}
\end{enumerate}
\end{lem}
\begin{remark}
 The total mass of $\m$ equals $\f(0;\,  \ga)$ and is not $1$. Hence, if we want to get a probability measure, we need to divide $\m$ by this number.
\end{remark}

We further introduce an $n$-dimensional probabilistic version of $\m$.

\begin{definition}
 The probability measure $\mathcal M^{n,t,\gamma}$ on ordered $n$-tuples $x_1< x_2<\dots < x_n$ of real numbers is defined for each of the phases \eqref{eq_case1}--\eqref{eq_case4} by its density with respect to $\m^{\otimes n}$:
 \begin{equation}
 \label{eq_M_def}
   \frac{1}{h_0 h_1 \cdots h_{n-1}} \prod_{1\le i < j \le n} (x_i-x_j)^2 \prod_{i=1}^n e^{tx_i} \m(\dd x_i),
 \end{equation}
 where $h_0 h_1\cdots h_{n-1}$ is a normalization constant making the total mass of $\mathcal M^{n,t,\gamma}$ equal to $1$.
\end{definition}
Note that depending on the phase, $\mathcal M^{n,t,\gamma}$ is supported on $n$-tuples of negative even integers, or on $n$-tuples of real numbers, or on $n$-tuples of even integers.

We next relate $\ZZ_n$ to expectations of certain random variables with respect to the measure $\mathcal M^{n,t,\gamma}$. For that we need one additional definition.

\begin{definition} The $\beta=2$ multivariate Bessel function is a function of complex variables $(z_1,\dots,z_n)$ labeled by an $n$-tuple of real numbers $(x_1,x_2,\dots, x_n)$ and given by
\begin{equation}
\label{eq_Bessel_1}
 \B_{x_1,\dots,x_n}(z_1,\dots,z_n)= 1! 2! \cdots (n-1)! \frac{\det\bigl[ e^{x_i z_j}\bigr]_{i,j=1}^n}{\prod\limits_{1\le i<j\le n} (x_i-x_j)(z_i-z_j)}.
\end{equation}
\end{definition}

\smallskip

\noindent These functions play an important role in random matrix theory. In particular, they can be identified with values of the Harish Chandra/Itzykson--Zuber spherical integral  \cite{Harish-Chandra57,Itzykson-Zuber80} :
\begin{equation}
\label{eq_Bessel_2}
 \B_{x_1,\dots,x_n}(z_1,\dots,z_n)=\int_{\mathbb U(n)} \exp\bigl[\mathrm{Trace} \bigl(\mathrm{diag}(x_1,\dots,x_n) U \mathrm{diag}(z_1,\dots,z_n) U^*\bigr)\bigr] \dd U,
\end{equation}
where the integration goes over the Haar measure on the group of $n\times n$ unitary matrices and $\mathrm{diag}(x_1,\dots,x_n)$ is the diagonal matrix with elements $x_1,\dots,x_n$. The two definitions \eqref{eq_Bessel_1} and \eqref{eq_Bessel_2} imply that $\B$ is symmetric both in $(x_1,\dots,x_n)$ and in $(z_1,\dots,z_n)$. In addition,
\begin{equation}
\label{eq_Bessel_norm}
 \B_{x_1,\dots,x_n}(0^n)=1\quad \text{ and }\quad \B_{x_1,\dots,x_n}(t+z_1,\dots,t+z_n)=\exp\left(t\sum_{i=1}^n x_i\right) \B_{x_1,\dots,x_n}(z_1,\dots,z_n).
\end{equation}

\begin{theo} \label{Theorem_partition_through_Bessel} In each of the four phases \eqref{eq_case1}--\eqref{eq_case4} we have for $1\le k \le n$:
\begin{align}
 \notag \ZZ_n(\xi_1,\dots,\xi_k;\, t,\gamma)=&
 \prod\limits_{j=1}^k \left[\left(\frac{a(t+\xi_j, \ga) b(t+\xi_j, \ga)}{a(t, \ga) b(t, \ga)}\right)^n  \left(\frac{\xi_j}{ b(\xi_j,0)}\right)^{n-k}\right]  \prod_{1\le i<j \le k} \frac{\xi_i-\xi_j}{ b(\xi_i-\xi_j,0)}
 \\ &\times \E_{\mathcal M^{n,t,\gamma}}\bigl[  \B_{x_1,\dots,x_n}(\xi_1,\dots,\xi_k,0^{n-k})\bigr], \label{eq_partition_through_Bessel}
\end{align}
where in $\B$ under the expectation, the label $(x_1,\dots,x_n)$ is $\mathcal M^{n,t,\gamma}$--random.
\end{theo}
\begin{proof} It is sufficient to prove \eqref{eq_partition_through_Bessel} for $k=n$, as the $k<n$ case would immediately follow by substitution $\xi_{k+1}=\dots=\xi_n=0$.
 We transform the determinant in \eqref{eq_IK_det} using \eqref{eq_Laplace_rep}:
\begin{equation}
\label{eq_x1}
 \det \bigl[ \f(\psi_j-\chi_i, \ga) \bigr]_{i,j=1}^n= \iiint \sum_{\sigma\in S(n)} (-1)^{\sigma}\prod_{i=1}^n \left[ \exp\bigl((\psi_{\sigma(i)}-\chi_i) x_i \bigr) \m(\dd x_i)\right],
\end{equation}
 where each $x_i$ is integrated over $\mathbb R$. There are $n!$ possible ordering of $x_i$ in the last integral (if some of $x_i$ coincide, then the integrand vanishes), and we can instead rewrite \eqref{eq_x1} as an integral over $x_1\le x_2\le \dots\le x_n$, resulting in
\begin{align}
 &\iiint\limits_{x_1\le \dots\le x_n} \sum_{\sigma,\sigma'\in S(n)} (-1)^{\sigma} (-1)^{\sigma'}\prod_{i=1}^n \left[ \exp\bigl((\psi_{\sigma(i)}-\chi_{\sigma'(i)}) x_i \bigr) \m(\dd x_i)\right]\notag
 \\&= \iiint\limits_{x_1\le \dots\le x_n} \det[e^{\chi_j x_i}]_{i,j=1}^n \det [e^{\psi_l x_i}]_{i,l=1}^n \prod_{i=1}^n  \m(\dd x_i)\label{eq_x2}
  \\&=\frac{ \prod\limits_{i<j} (\chi_i-\chi_j)(\psi_i-\psi_j)}{1!^2 2!^2\cdots (n-1)!^2} \iiint\limits_{x_1\le \dots\le x_n} \B_{x_1,\dots,x_n}(\chi_1,\dots,\chi_n) \B_{x_1,\dots,x_n}(\psi_1,\dots,\psi_n) \prod_{i<j} (x_i-x_j)^2  \prod_{i=1}^n  \m(\dd x_i).\notag
\end{align}
Multiplying the last line of \eqref{eq_x2} by the prefactor of the determinant in \eqref{eq_IK_det}, setting $\chi_1=\chi_2=\dots=\chi_n=0$ using \eqref{eq_Bessel_norm} (relying on $\lim\limits_{\chi_i,\chi_j\to 0} \frac{\chi_i-\chi_j}{b(\chi_i-\chi_j,0)}=1$ to resolve the singularity), setting $\psi_i=t+\xi_i$, $1\le i \le n$, again using \eqref{eq_Bessel_norm}, and dividing the general $\xi$ version by $\xi_1=\xi_2=\dots=\xi_n=0$ version, we conclude that
\begin{align*}
 \frac{ \mathcal Z_n(0^n;\, t+\xi_1,\dots,t+\xi_n;\, \gamma)}{\mathcal Z_n(0^n;\, t^n;\, \gamma)}=&\prod\limits_{j=1}^n \left[\frac{a(t+\xi_j, \ga) b(t+\xi_j, \ga)}{a(t, \ga) b(t, \ga)}\right]^n \prod\limits_{1\le i<j\le n}\frac{\xi_i-\xi_j}{ b(\xi_i-\xi_j,0)} \\ &\times \E_{\mathcal M^{n,t,\gamma}}\bigl[  \B_{x_1,\dots,x_n}(\xi_1,\dots,\xi_n)\bigr]. \qedhere
\end{align*}
\end{proof}

\smallskip

For $k=1$, there is an alternative representation of $\ZZ_n(\xi;\, t,\gamma)$ as an expectation, this time using both $\mathcal M^{n-1,t,\gamma}$ and $\mathcal M^{1,t,\gamma}$.
\begin{theo}  \label{Theorem_partition_k1} In each of the four phases \eqref{eq_case1}--\eqref{eq_case4} we have
\begin{equation}
\label{eq_partition_k1}
 \ZZ_n(\xi;\, t,\gamma)=(n-1)! \frac{h_0}{ h_{n-1}} \left(\frac{a(t+\xi, \ga) b(t+\xi, \ga)}{a(t, \ga) b(t, \ga)}\right)^n  \left(\frac{1}{ b(\xi,0)}\right)^{n-1}  \E_{\mathcal M^{1,t,\gamma}} \left[e^{\xi x} P_{n-1}(x)\right],
\end{equation}
where $x$ under the expectation is $\mathcal M^{1,t,\gamma}$--random, and $P_{n-1}(x)$ is a degree $n-1$ polynomial defined as
$$
 P_{n-1}(x):=\E_{\mathcal M^{n-1,t,\gamma}} \left[\prod_{i=1}^{n-1} (x-x_i)\right], \qquad \qquad (x_1,\dots,x_{n-1})\text { is }\mathcal M^{n-1,t,\gamma}\text{-random,}
$$
and $h_0,h_{n-1}$ are normalization constants from \eqref{eq_M_def}.
\end{theo}
\begin{proof}
One can establish Theorem \ref{Theorem_partition_k1} directly by a modification of the argument of Theorem \ref{Theorem_partition_through_Bessel}, but we instead prove it as a corollary of \eqref{eq_partition_through_Bessel}. We rely on the formula for Bessel functions:
\begin{equation}
\label{eq_Bessel_as_sum}
    \B_{x_1,\dots,x_n}(\xi,0^{n-1})=\frac{(n-1)!}{\xi^{n-1}}\sum_{i=1}^n \frac{e^{\xi x_i}}{\prod_{j\ne i} (x_i- x_j)} .
\end{equation}
The formula \eqref{eq_Bessel_as_sum} is proven\footnote{The formula \eqref{eq_Bessel_as_sum} is a residue expansion of the contour integral representation previously used in \cite[(3.7)-(3.8)]{Colomo-Pronko-ZinnJustin10}, \cite[Theorem 5.1]{cuenca2021universal}, \cite[Theorem 1.1]{gorin2015asymptotics}.} by expanding the determinant in \eqref{eq_Bessel_1} in the row containing $z_1$, observing that the terms in the expansion are (up to simple factors) again Bessel functions in the smaller number of variables, and then setting $z_1=\xi$,\, $z_2=\dots=z_n=0$ using \eqref{eq_Bessel_norm}.

Let $\widetilde {\mathcal M}^{n,t,\gamma}$ denote the extension of the probability measure $\mathcal M^{n,t,\gamma}$ from ordered $n$-tuples ${x_1\le \dots\le x_n}$ to arbitrary $n$-tuples, given by the same formula \eqref{eq_M_def}, but with an additional prefactor $\tfrac{1}{n!}$ in front. Then, using the symmetry of the Bessel functions, we have
\begin{align*}
 \E_{\mathcal M^{n,t,\gamma}}&\bigl[  \B_{x_1,\dots,x_n}(\xi,0^{n-1})\bigr]=  \E_{\widetilde{\mathcal M}^{n,t,\gamma}}\bigl[  \B_{x_1,\dots,x_n}(\xi,0^{n-1})\bigr]\\&= \frac{(n-1)!}{n!}\frac{1}{\xi^{n-1}  h_0 \cdots h_{n-1} } \int_{\mathbb R^{n}} \left[\sum_{i=1}^n \frac{e^{\xi x_i}}{\prod_{j\ne i} (x_i- x_j)}\right]  \prod_{1\le i < j \le n} (x_i-x_j)^2 \prod_{i=1}^n e^{tx_i} \m(\dd x_i).
\end{align*}
Note that (by symmetry), each of the terms in $\sum_{i=1}^n$ leads to the same integral. Hence, we only deal with $i=n$ term, and we rename $x_n$ into $x$, getting
\begin{equation}
\label{eq_x3}
 \frac{1}{\xi^{n-1} h_0 \cdots h_{n-1} } \int_{\mathbb R} e^{\xi x}  \left[\int_{\mathbb R^{n-1}} \prod_{i=1}^{n-1} (x- x_i)  \prod_{1\le i < j \le n-1} (x_i-x_j)^2 \prod_{i=1}^{n-1} e^{tx_i} \m(\dd x_i) \right] e^{tx} \m(\dd x) .
\end{equation}
Replacing in the $k=1$ version of \eqref{eq_partition_through_Bessel} the expectation $\E_{\mathcal M^{n,t,\gamma}}\bigl[  \B_{x_1,\dots,x_n}(\xi,0^{n-1})\bigr]$ with \eqref{eq_x3}, we get \eqref{eq_partition_k1}.
\end{proof}

Taking $\ZZ_n(\xi;\, t,\gamma)$ as a basic building block, one can also express $\ZZ_n(\xi_1,\dots,\xi_k;\, t,\gamma)$ as a $k\times k$ determinant, as was first noticed in \cite[Sections 6.2, 6.3]{Colomo-Pronko08}:

\begin{theo}  \label{Theorem_partition_general_k_reduction} In each of the four phases \eqref{eq_case1}--\eqref{eq_case4} we have for $1\le k \le n$:
\begin{equation}\label{eq_general_k_reduction}
  \ZZ_n(\xi_1, \xi_2, \dots, \xi_k;\, t,\gamma )= \frac{\displaystyle\det \left[  \ZZ_{n-k+i}(\xi_j;\, t,\gamma) \left(\frac{a(t+\xi_j, \ga) b(t+\xi_j, \ga)}{a(t, \ga) b(t, \ga)}\right)^{k-i}  \left(b(\xi_j,0)\right)^{i-1} \right]_{i,j=1}^k}{\displaystyle \prod\limits_{1\le i<j\le k} b(\xi_j-\xi_i,0) }.
 \end{equation}
\end{theo}
\begin{proof}
 Note that if we substitute $\xi_k=0$ into \eqref{eq_general_k_reduction}, the left-hand side turns into ${\ZZ_n(\xi_1, \xi_2, \dots, \xi_{k-1};\, t,\gamma)}$, while in the determinant in the right-hand side the last $j=k$ column has a single non-zero element $i=1$. Hence, the right-hand side becomes
$$
 (-1)^{k-1}\frac{\displaystyle \det \left[  \ZZ_{n-k+i+1}(\xi_j;\, t,\gamma) \left(\frac{a(t+\xi_j, \ga) b(t+\xi_j, \ga)}{a(t, \ga) b(t, \ga)}\right)^{k-i-1}  \left(b(\xi_j,0)\right)^{i} \right]_{i,j=1}^{k-1}}{\displaystyle \prod\limits_{1\le i<j\le k-1} b(\xi_j-\xi_i,0) \prod\limits_{i=1}^{k-1} b(-\xi_i,0) }.
$$
Using $b(-\xi,0)=-b(\xi,0)$ and cancelling $\prod_{i=1}^{k-1} b(\xi_i)$ between numerator and denominator, we arrive at the right-hand side of \eqref{eq_general_k_reduction} with $k$ decreased by $1$. Therefore, it is sufficient to prove \eqref{eq_general_k_reduction} for $k=n$, which we do in the rest of the argument.

 We recall the notation \eqref{eq_phi_def} and the formula \eqref{eq_IK_det}. By a version of L'H\^{o}pital's rule, we have

 \begin{align}
\notag  \mathcal Z_n(0^n;\, \psi_1,\dots,\psi_n;\, \gamma)&= \frac{\prod\limits_{j=1}^n \bigl(a(\psi_j, \ga) b(\psi_j, \ga)\bigr)^n}{\prod\limits_{i<j} b(\psi_j-\psi_i,0)} \lim_{\chi_1,\dots\chi_n\to 0}\frac{\det \bigl[ \f(\psi_j-\chi_i, \ga) \bigr]_{i,j=1}^n}{\prod\limits_{i<j} b(\chi_j-\chi_i,0)}\\&= \frac{\prod\limits_{j=1}^n \bigl(a(\psi_j, \ga) b(\psi_j, \ga)\bigr)^n}{\prod\limits_{i<j} b(\psi_j-\psi_i,0)} \frac{\det \biggl[ \f^{(i-1)}(\psi_j, \ga) \biggr]_{i,j=1}^n}{1! 2!\cdots (n-1)!}, \label{eq_x4}
\end{align}
where $\f^{(i-1)}(z,\ga)$ is the $(i-1)$st derivative of $\f(z,\ga)$ in $z$. Take a sequence of monic polynomials $P_m(x)=x^m+\dots$, and let $P_m(\partial_z)$ denote the $m$-order differential operator, obtained by plugging the derivation operator $\partial_z$ as a variable into $P_m(x)$. Then, by elementary row operations in the determinant, we can rewrite \eqref{eq_x4} as
 \begin{align}
 \mathcal Z_n(0^n;\, t+\xi_1,\dots,t+\xi_n;\, \gamma)= \frac{\prod\limits_{j=1}^n \bigl(a(t+\xi_j, \ga) b(t+\xi_j, \ga)\bigr)^n}{\prod\limits_{i<j} b(\xi_j-\xi_i,0)} \frac{\det \biggl[ P_{i-1}(\partial_z) \f(z, \ga)\bigr|_{z=t+\xi_j}  \biggr]_{i,j=1}^n}{1! 2!\cdots (n-1)!}. \label{eq_x5}
\end{align}
Note that differentiating \eqref{eq_Laplace_rep}, we get
$$
 P_{i-1}(\partial_z) \f(z, \ga)\bigr|_{z=t+\xi} =  h_0 \E_{\mathcal M^{1,t,\gamma}} \left[e^{\xi x} P_{i-1}(x)\right].
$$
Thus, if we choose $P_{i-1}(x)$ to be the polynomials from Theorem \ref{Theorem_partition_k1}, then the expression under the determinant in \eqref{eq_x5} can be transformed using \eqref{eq_partition_k1} into
\begin{equation}
\label{}
 P_{i-1}(\partial_z) \f(z, \ga)\bigr|_{z=t+\xi}= \ZZ_i(\xi;\, t,\gamma)  \cdot \frac{h_{i-1}}{(i-1)!} \left(\frac{a(t+\xi, \ga) b(t+\xi, \ga)}{a(t, \ga) b(t, \ga)}\right)^{-i}  \left(b(\xi,0)\right)^{i-1}.
\end{equation}
Therefore, \eqref{eq_x5} becomes
 \begin{multline}
 \mathcal Z_n(0^n;\, t+\xi_1,\dots,t+\xi_n;\, \gamma)\\= \frac{h_0\cdots h_{n-1}}{(1!2!\cdots(n-1)!)^2} \bigl(a(t, \ga) b(t, \ga)\bigr)^{n^2} \frac{\det \left[  \ZZ_i(\xi_j;\, t,\gamma) \left(\frac{a(t+\xi_j, \ga) b(t+\xi_j, \ga)}{a(t, \ga) b(t, \ga)}\right)^{n-i}  \left(b(\xi_j,0)\right)^{i-1} \right]_{i,j=1}^n}{ \prod\limits_{i<j} b(\xi_j-\xi_i,0) }. \label{eq_x6}
\end{multline}
Setting $\xi_j=\eps \zeta_j$ and sending $\eps\to 0$, we see that the determinant in the numerator turns into the Vandermonde determinant and cancels with denominator. Hence,
\begin{equation}
 \mathcal Z_n(0^n;\, t^n;\, \gamma)=\frac{h_0\cdots h_{n-1}}{(1!2!\cdots(n-1)!)^2}  \bigl(a(t, \ga) b(t, \ga)\bigr)^{n^2}.
 \label{eq_x7}
\end{equation}
Dividing \eqref{eq_x6} by \eqref{eq_x7}, we arrive at \eqref{eq_general_k_reduction} with $k=n$.
\end{proof}

\section{Theorems \ref{Theorem_asymptotics_case1}--\ref{Theorem_asymptotics_case4} through concentration for log-gases}

\label{Section_asymptotics_through_log_gas}

In this section we demonstrate how Theorem \ref{Theorem_partition_through_Bessel} combined with the known asymptotic expansions for the multivariate Bessel functions and conjectural concentration results for the log-gases (showing that their empirical distribution is closely approximated by the equilibrium measure), leads to quick proofs of Theorems \ref{Theorem_asymptotics_case1}--\ref{Theorem_asymptotics_case4}. The difficulty in making this approach rigorous is to justify these concentration results: while for the log-gases with smooth potentials similar theorems are well-known, we need to deal with a non-smooth case, where we are not aware of ``black-box'' results, giving the concentration at the necessary precision. Comparing with the best available results, \cite[Proposition 4.3]{dadoun2023asymptotics} requires the potential to be three times differentiable, \cite[Theorem 1]{bekerman2018clt} requires six times differentiability, and \cite[Theorem 1.2]{lambert2019quantitative} requires eight times differentiability; in contrast, our potentials are non-differentiable at $x=0$, which is the best seen in \eqref{eq_measure_case4}. Another complication is that for $|\Delta|>1$ our log-gases are discrete and the theory is even less developed for those.

\medskip

For the Bessel functions we use the following asymptotic expansion.

\begin{theo} \label{Theorem_Bessel_limit}
 Let $\mathbf x(n)=(x_1(n),\dots,x_n(n))\in\R^{n}$, $n=1,2,\dots$ be a sequence of $n$-tuples of reals. Denote
 $$
  \mu(n)=\frac{1}{n} \sum_{i=1}^n \frac{x_i(n)}{n},\qquad \sigma^2(n)=\left[\frac{1}{n} \sum_{i=1}^n \frac{x_i(n)^2}{n^2}\right] - \bigl(\mu(n)\bigr)^2,\qquad \kappa^3(n)=\frac{1}{n}\sum_{i=1}^n \left|\frac{x_i(n)}{n}-\mu(n)\right|^3.
 $$
Suppose that there exists a constant $C>0$, such that for all $n=1,2,\dots$ we have $C^{-1}< \sigma^2(n)< C$ and $\kappa^3(n)<C$. Then for each $k=1,2,\dots$ and $\xi_1,\dots,\xi_k\in\C$, as $n\to\infty$,
\begin{equation}
\label{eq_x52}
 \B_{x_1(n),\dots,x_n(n)}\left(\tfrac{\xi_1}{\sqrt{n}},\cdots, \tfrac{\xi_k}{\sqrt{n}}, 0^{n-k}\right)=\exp\left(\sqrt{n} \mu(n)\sum_{j=1}^k \xi_j+\frac{\sigma^2(n)}{2} \sum_{j=1}^k (\xi_j)^2+\bigO\left(\tfrac{1}{\sqrt{n}}\right)\right),
\end{equation}
where $\bigO\left(\tfrac{1}{\sqrt{n}}\right)$ is uniform over $\xi_i$ belonging to compact subsets of $\mathbb C$ and over sequences $\mathbf x(n)$ satisfying the inequalities with the same constant $C$.
\end{theo}
Theorem \ref{Theorem_Bessel_limit} can be found in \cite{cuenca2021universal},\cite{gorin2015asymptotics},\cite{novak2015lozenge}; for $k=1$ it can be established by the steepest descent analysis applied to the contour integral representation of the sum \eqref{eq_Bessel_as_sum}.

We would like to combine Theorem \ref{Theorem_Bessel_limit} with Theorem \ref{Theorem_partition_through_Bessel} and for that we need to understand the asymptotic behavior of $\mu(n)$ and $\sigma^2(n)$ for $\mathcal M^{n,t,\gamma}$--distributed $(x_1,\dots,x_n)$. We recall that for a probability measure $\nu$ on $\mathbb R$, its Stieltjes transform is given by
\begin{equation}\label{eq:Stieltjes}
 G_\nu(z)= \int_{\mathbb R} \frac{\nu(\dd x)}{z-x}.
\end{equation}
If $\nu$ has compact support, then the integral converges for all large $z\in\mathbb C$ and $G_\nu(z)$ decays as $\frac{1}{z}$ when $z\to\infty$.

\begin{theo} \label{Theorem_LLN_log_gas}
 In each of the four phases \eqref{eq_case1}--\eqref{eq_case4}, with $\mathcal M^{n,t,\gamma}$ given by \eqref{eq_M_def}, the random  $\mathcal M^{n,t,\gamma}$--distributed $(x_1,\dots,x_n)$ satisfies the Law of Large Numbers as $n\to\infty$:
 \begin{equation}
 \label{eq_empirical_convergence}
   \lim_{n\to\infty} \frac{1}{n}\sum_{i=1}^n \delta_{x_i/n} = \nu, \qquad \text{ weakly, in probability,}
 \end{equation}
 where $\nu$ is a deterministic probability measure, depending on the phase and $(\gamma,t)$.
  \begin{itemize}
\item {\bf Ferroelectric phase} \eqref{eq_case1}: $\nu$ has Stieltjes transform\footnote{In ferroelectric phase $\nu$ is supported on interval $[\beta,0]$ and has density $\tfrac{1}{2}$ on $[\alpha,0]$.}
\begin{equation}
 \label{eq_case1_G}
 G_\nu(z)= \frac{|\ga|-t}{2}  -\frac{1}{2} \ln\left[\frac{\left(\sqrt{-\al(z-\be)}- \sqrt{-\be(z-\al)}\right)^2}{z(\alpha-\be)}\right],
\end{equation}
\begin{equation}
 \label{eq_case1_alpha_beta}
 \text{ where }\qquad
 \al= -2\frac{e^{t-|\ga|}-1}{e^{t-|\ga|}+1}, \qquad \be = -2\frac{e^{t-|\ga|}+1}{e^{t-|\ga|}-1}.
\end{equation}
Here and below we always choose the principal branch for $\sqrt{\cdot}$ and $\ln(\cdot)$ functions, mapping large positive reals to large positive reals.

\item {\bf Disordered phase} \eqref{eq_case2}:
$\nu$ has Stieltjes transform\footnote{In disordered phase $\nu$ is supported on interval $[\alpha,\beta]$.}
\begin{equation}
 \label{eq_case2_G}
 G_\nu(z)= \frac{\ga-t}{2}+ \ii \frac{\gamma}{\pi} \ln\left[\frac{\left(\sqrt{-\alpha(z-\beta)}-\sqrt{-\beta(z-\alpha)}\right)^2}{z(\alpha-\beta)}\right],
\end{equation}
\begin{equation}
\label{eq_case2_alpha_beta}
 \text{ where }\qquad
\alpha= -\frac{\pi}{\gamma} \tan\left(\frac{\pi}{4}\left(1-\frac{t}{\gamma}\right)\right),\qquad \beta= \frac{\pi}{\gamma} \tan\left(\frac{\pi}{4}\left(1+\frac{t}{\gamma}\right)\right).
\end{equation}
\item {\bf Antiferroelectric phase} \eqref{eq_case3}: $\nu$ has Stieltjes transform\footnote{In antiferroelectric phase $\nu$ is supported on interval $[\alpha,\beta]$ and has density $\tfrac{1}{2}$ on the interval $[\alpha',\beta']\subset [\alpha,\beta]$.}
\begin{equation}\label{eq_case3_G}
 G_\nu(z)= \int_{z}^{\infty} \frac{\dd \mathfrak z}{\sqrt{(\mathfrak z-\alpha)(\mathfrak z-\alpha')(\mathfrak z-\beta)(\mathfrak z-\beta')}},
\end{equation}
\begin{equation}
\label{eq_case3_alpha_beta}
 \text{ where }\qquad
\alpha= -\frac{\pi}{\gamma} \frac{\vartheta'_1(\omega)}{\vartheta_1(\omega)},\quad \alpha'= -\frac{\pi}{\gamma} \frac{\vartheta'_4(\omega)}{\vartheta_4(\omega)}, \quad \beta= -\frac{\pi}{\gamma} \frac{\vartheta'_2(\omega)}{\vartheta_2(\omega)},
 \quad \beta'= -\frac{\pi}{\gamma} \frac{\vartheta'_3(\omega)}{\vartheta_3(\omega)},
\end{equation}
$\omega=\pi \frac{t+\gamma}{4\gamma}$;\,   $\vartheta_\ell(u)$ and $\vartheta_\ell'(u)$ are Jacobi elliptic theta functions \eqref{eq_elliptic_functions} and their $u$-derivatives.

\item {\bf Boundary phase} \eqref{eq_case4}:
$\nu$ is given  by the formulas \eqref{eq_case2_G}, \eqref{eq_case2_alpha_beta} as in the disordered phase.
\end{itemize}
In addition, there exists $C>0$ depending on the phase and $(\gamma,t)$, such that
\begin{equation}
\label{eq_log_gas_support}
 \lim_{n\to\infty} \mathrm{Prob} \left(-C\le \frac{x_1}{n}\le \frac{x_n}{n}\le C\right)=1.
\end{equation}
\end{theo}
\begin{remark} A unifying form for the formulas \eqref{eq_case1_G} and \eqref{eq_case2_G} is by writing
\eq\label{eq:Gp}
 \partial_z G_\nu(z)=-\frac{1}{z\sqrt{(z-\alpha)(z-\beta)}}.
\eeq
Similarly, \eqref{eq_case3_G} implies that
\eq\label{eq:GpAF}
 \partial_z G_\nu(z)=-\frac{1}{\sqrt{(z-\alpha)(z-\alpha')(z-\beta)(z-\beta')}}.
\eeq
Yet, let us emphasize that the precise expressions \eqref{eq_case1_alpha_beta}, \eqref{eq_case2_alpha_beta}, and \eqref{eq_case3_alpha_beta} for $\alpha$ and $\beta$ depend on the phase.
\end{remark}

\medskip

The fact that convergence of the form \eqref{eq_empirical_convergence} holds for some measure $\nu$ is well-known for a very general class of log-gases, which are distributions of the kind \eqref{eq_M_def}, but with $e^{tx_i} \m(\dd x_i)$ replaced by a more general potential, see \cite{Anderson-Guionnet-Zeitouni10,arous1997large,Deift99,feral2008large,guionnet2019asymptotics}. The measure $\nu$ is known as the \emph{equilibrium measure}, and is it is typically determined as a solution to a variational problem. Explicit formulas for $\nu$ in our situation first appeared in \cite{Zinn_Justin00}, see also \cite{Bleher-Liechty14} for additional details and \cite{Bleher-Bothner12} for the boundary case $\De = -1$.

\smallskip

Decoding what convergence in Theorem \ref{Theorem_LLN_log_gas} means, we have:
\begin{cor} \label{Corollary_linear_statistics}
 Take a real function $f(x)$, which is continuous on the segment $[-C,C]$, where $C$ comes from \eqref{eq_log_gas_support}. Then, in the notations of Theorem \ref{Theorem_LLN_log_gas},
 \begin{equation}
 \label{eq_linear_statistics}
   \lim_{n\to\infty} \frac{1}{n} \sum_{i=1}^n f\left(\frac{x_i}{n}\right)= \int_{\mathbb R} f(x) \nu(\dd x), \qquad \text{ in probability.}
 \end{equation}
\end{cor}
We would like to apply this corollary to the expectation in Theorem \ref{Theorem_partition_through_Bessel}. Taking into account Theorem \ref{Theorem_Bessel_limit}, for that we need some control on the speed of convergence in \eqref{eq_linear_statistics}. In fact we expect the difference between left and right sides of \eqref{eq_linear_statistics} to be $\bigO\left(\tfrac{1}{n}\right)$.

\begin{conjecture} \label{Conjecture_concentration}
 Take a polynomial $f(x)$. Then, in the notations of Theorem \ref{Theorem_LLN_log_gas}, the random variable
 \begin{equation}
 \label{eq_linear_statistics_remainder}
   \xi_n=\left[\sum_{i=1}^nf\left(\frac{x_i}{n}\right)\right]- n \int_{\mathbb R} f(x) \nu(\dd x),
 \end{equation}
 is exponentially bounded as $n\to\infty$, which means that there exist two constant $c_1,c_2>0$, such that $\E \exp(c_1|\xi_n|)<c_2$ for all $n=1,2,\dots$.
\end{conjecture}
Results of the flavor of Conjecture \ref{Conjecture_concentration} are known for log-gases with \emph{smooth} potential, see, e.g., \cite{Johansson98} for the continuous setting or \cite{borodin2017gaussian} for the discrete setting. As mentioned at the beginning of this section, our potentials in \eqref{eq_measure_case1}--\eqref{eq_measure_case4} are more exotic, and therefore the existing results in the literature do not apply. Below we show how Conjecture \ref{Conjecture_concentration} implies Theorems \ref{Theorem_asymptotics_case1}--\ref{Theorem_asymptotics_case4}. We do not attempt to prove Conjecture \ref{Conjecture_concentration} in this text; instead in Sections \ref{Section_k1_asymptotics} and \ref{Section_general_reduction} we present an alternative approach to Theorems \ref{Theorem_asymptotics_case1}--\ref{Theorem_asymptotics_case4} based on the asymptotic analysis of orthogonal polynomials with orthogonality measures \eqref{eq_measure_case1}--\eqref{eq_measure_case4}. We remark that \cite{Breuer-Duits13} explains a link between orthogonal polynomials and asymptotic behavior of random variables of the kind \eqref{eq_linear_statistics_remainder}; hence, it is plausible that the orthogonal polynomial approach can also be used to prove Conjecture \ref{Conjecture_concentration}.

\begin{proof}[Proof of Theorems \ref{Theorem_asymptotics_case2}, \ref{Theorem_asymptotics_case3}, and \ref{Theorem_asymptotics_case4} conditional on Conjecture \ref{Conjecture_concentration}]

Theorem \ref{Theorem_partition_through_Bessel} yields
\begin{align}
 \notag \ZZ_n\left(\frac{\xi_1}{\sqrt{n}},\dots,\frac{\xi_k}{\sqrt{n}};\, t,\gamma\right)=&
 \prod\limits_{j=1}^k \left[\left(\frac{a\left(t+\frac{\xi_j}{\sqrt{n}}, \ga\right) b\left(t+\frac{\xi_j}{\sqrt{n}}, \ga\right)}{a(t, \ga) b(t, \ga)}\right)^n  \left(\frac{\xi_j}{ \sqrt{n}\, b\left(\frac{\xi_j}{\sqrt{n}},0\right)}\right)^{n-k}\right]
  \\ &\times \prod_{1\le i<j \le k} \frac{\xi_i-\xi_j}{\sqrt{n}\, b\left(\frac{\xi_i-\xi_j}{\sqrt{n}},0\right)}
 \,\, \E_{\mathcal M^{n,t,\gamma}}\left[  \B_{x_1,\dots,x_n}\left(\frac{\xi_1}{\sqrt{n}},\dots,\frac{\xi_k}{\sqrt{n}},0^{n-k}\right)\right]. \label{eq_x10}
\end{align}
Theorem \ref{Theorem_Bessel_limit} and Conjecture \ref{Conjecture_concentration} provide an asymptotic expansion for the expectation in the last line:\footnote{Formally, one needs an additional argument to explain that $\bigO\left(\tfrac{1}{\sqrt{n}}\right)$ term in \eqref{eq_x52} when plugged under expectation results in $\bigO\left(\tfrac{1}{\sqrt{n}}\right)$ in \eqref{eq_x11}. We omit this justification which requires making the tail bound \eqref{eq_log_gas_support} and the error term in \eqref{eq_x52} more precise.}
\begin{equation}
\label{eq_x11}
 \E_{\mathcal M^{n,t,\gamma}}\left[  \B_{x_1,\dots,x_n}\left(\frac{\xi_1}{\sqrt{n}},\dots,\frac{\xi_k}{\sqrt{n}},0^{n-k}\right)\right]=
 \exp\left(\sqrt{n} \mu_1\sum_{j=1}^k \xi_j+\frac{\mu_2-(\mu_1)^2}{2} \sum_{j=1}^k (\xi_j)^2+\bigO\left(\tfrac{1}{\sqrt{n}}\right)\right),
\end{equation}
where $\mu_1$ and $\mu_2$ are the first two moments of the measure $\nu$ of Theorem \ref{Theorem_LLN_log_gas}. They are readily found as coefficients of the expansion of $G_\nu(z)$ in powers of $\tfrac{1}{z}$, and using \eqref{eq_case2_G}, \eqref{eq_case3_G}, we get
\begin{itemize}
\item {\bf Disordered phase} \eqref{eq_case2}: using \eqref{eq_case2_alpha_beta}
\begin{equation}
 \label{eq_case2_moments}
  \mu_1=\frac{\alpha+\beta}{4}, \qquad \mu_2=\frac{3\alpha^2+2\alpha\beta+3\beta^2}{24}.
\end{equation}
\item {\bf Antiferroelectric phase} \eqref{eq_case3}: using \eqref{eq_case3_alpha_beta}
\begin{equation}
 \label{eq_case3_moments}
  \mu_1=\frac{\alpha+\alpha'+\beta+\beta'}{4}, \qquad \mu_2=\frac{(\alpha)^2+(\alpha')^2+(\beta)^2+(\beta')^2}{12}+\frac{(\alpha+\alpha'+\beta+\beta')^2}{24}.
\end{equation}
\item {\bf Boundary phase} \eqref{eq_case4}: $\mu_1$ and $\mu_2$ are given by \eqref{eq_case2_moments}.
\end{itemize}

Next, the double product $\prod_{1\le i<j \le k}$ in \eqref{eq_x10}  is $1+\bigO\left(\tfrac{1}{\sqrt{n}}\right)$ directly from the definition of $b(\cdot)$ in \eqref{eq_case2}--\eqref{eq_case4}. For the single product $\prod_{j=1}^k$, we use the $n\to\infty$ asymptotic expansions:
$$
 \frac{\sin\left(x+\frac{y}{\sqrt{n}}\right)}{\sin(x)}=1+\frac{y}{\sqrt{n}} \cot(x)- \frac{y^2}{2n} +\bigO\left(n^{-\frac{3}{2}}\right), \qquad
 \frac{y}{\sqrt{n} \sin\left(\frac{y}{\sqrt{n}}\right)}=1+\frac{y^2}{6n}+\bigO\left(n^{-\frac{3}{2}}\right),
$$
$$
 \frac{\sinh\left(x+\frac{y}{\sqrt{n}}\right)}{\sinh(x)}=1+\frac{y}{\sqrt{n}} \coth(x)+ \frac{y^2}{2n} +\bigO\left(n^{-\frac{3}{2}}\right),\qquad
 \frac{y}{\sqrt{n} \sinh\left(\frac{y}{\sqrt{n}}\right)}=1-\frac{y^2}{6n}+\bigO\left(n^{-\frac{3}{2}}\right).
$$
Hence, using also $\ln(1+z)=z-\frac{z^2}{2}+\bigO(z^3)$, in the disordered phase  \eqref{eq_case2}, we have
\begin{align*}
 &\left(\frac{a\left(t+\frac{\xi_j}{\sqrt{n}}, \ga\right) b\left(t+\frac{\xi_j}{\sqrt{n}}, \ga\right)}{a(t, \ga) b(t, \ga)}\right)^n  \left(\frac{\xi_j}{ \sqrt{n}\, b\left(\frac{\xi_j}{\sqrt{n}},0\right)}\right)^{n-k}=
 \left(1-\frac{\xi_j}{\sqrt n}\cot(\gamma-t)-\frac{\xi_j^2}{2n}+\bigO\left(\frac{1}{n^{\frac{3}{2}}}\right)\right)^n
 \\ &\times \left(1+\frac{\xi_j}{\sqrt n}\cot(\gamma+t)-\frac{\xi_j^2}{2n}+\bigO\left(\frac{1}{n^{\frac{3}{2}}}\right)\right)^n \left(1+\frac{\xi_j^2}{6n}+\bigO\left(n^{-\frac{3}{2}}\right)\right)^{n-k}
 \\ &=\exp\left( \sqrt{n}\xi_j \bigl(\cot(\gamma+t)-\cot(\gamma-t)\bigr)+ \xi_j^2 \left(-\frac{5}{6}- \frac{\cot^2(\gamma-t)}{2}-\frac{\cot^2(\gamma+t)}{2} \right) +\bigO\left(n^{-\frac{1}{2}}\right) \right).
\end{align*}
Multiplying by \eqref{eq_x11} and simplifying using
$$
 \mu_1=\frac{\alpha+\beta}{4}=  \frac{\pi}{4\gamma} \tan\left(\frac{\pi}{4}\left(1+\frac{t}{\gamma}\right)\right)-\frac{\pi}{\gamma} \tan\left(\frac{\pi}{4}\left(1-\frac{t}{\gamma}\right)\right)=\frac{\pi}{2\gamma} \tan \left(\frac{\pi t}{2\gamma}\right),
$$
$$
\mu_2-(\mu_1)^2=\frac{3\alpha^2-2\alpha\beta +3\beta^2}{48}= \frac{\pi^2}{\gamma^2} \left( \frac{1}{4}\tan^2 \left(\frac{\pi t}{2\gamma}\right)+\frac{1}{6}\right),
$$
we get \eqref{eq_Zn_expansion_case2} of Theorem \ref{Theorem_asymptotics_case2}. In the antiferroelectric phase \eqref{eq_case3} we have
\begin{align*}
 &\left(\frac{a\left(t+\frac{\xi_j}{\sqrt{n}}, \ga\right) b\left(t+\frac{\xi_j}{\sqrt{n}}, \ga\right)}{a(t, \ga) b(t, \ga)}\right)^n  \left(\frac{\xi_j}{ \sqrt{n}\, b\left(\frac{\xi_j}{\sqrt{n}},0\right)}\right)^{n-k}=
 \left(1-\frac{\xi_j}{\sqrt n}\coth(\gamma-t)+\frac{\xi_j^2}{2n}+\bigO\left(\frac{1}{n^{\frac{3}{2}}}\right)\right)^n
 \\ &\times \left(1+\frac{\xi_j}{\sqrt n}\coth(\gamma+t)+\frac{\xi_j^2}{2n}+\bigO\left(\frac{1}{n^{\frac{3}{2}}}\right)\right)^n \left(1-\frac{\xi_j^2}{6n}+\bigO\left(n^{-\frac{3}{2}}\right)\right)^{n-k}
 \\ &=\exp\left( \sqrt{n}\xi_j \bigl(\coth(\gamma+t)-\coth(\gamma-t)\bigr)+ \xi_j^2 \left(\frac{5}{6}- \frac{\coth^2(\gamma-t)}{2}-\frac{\coth^2(\gamma+t)}{2} \right) +\bigO\left(n^{-\frac{1}{2}}\right) \right).
\end{align*}
We  multiply by \eqref{eq_x11} and simplify using
\begin{align*}
 \mu_1=\frac{\alpha+\alpha'+\beta+\beta'}{4}=-\frac{\pi}{4\gamma} \sum_{\ell=1}^4 \frac{\vartheta'_\ell\left(\pi \frac{t+\gamma}{4\gamma}\right)}{\vartheta_\ell\left(\pi \frac{t+\gamma}{4\gamma}\right)}
 =-\frac{\pi}{2\gamma}\frac{\vartheta'_2\left(\frac{\pi t}{2\gamma}\right)}{\vartheta'_2\left(\frac{\pi t}{2\gamma}\right)}.
\end{align*}
The last identity simplifying $\sum_{\ell=1}^4$ is obtained from\footnote{See \cite{Watson-Whittaker96} or \cite[Section 7.2]{Bleher-Liechty14} for the properties of Theta functions.} the duplication formula for Theta functions $\vartheta_1(2z) \vartheta'(0)= \prod_{\ell=1}^4 \vartheta_i(z)$ by taking the logarithmic derivative in $z$ and using $\vartheta_1(2z+\tfrac{\pi}{2})=\vartheta_2(2z)$. Further,
\begin{align*}
 \mu_2&=\frac{(\alpha)^2+(\alpha')^2+(\beta)^2+(\beta')^2}{12}+\frac{(\alpha+\alpha'+\beta+\beta')^2}{24}
 \\&= \frac{\pi^2}{12\gamma^2} \sum_{\ell=1}^4 \left(\frac{\vartheta'_\ell\left(\pi \frac{t+\gamma}{4\gamma}\right)}{\vartheta_\ell\left(\pi \frac{t+\gamma}{4\gamma}\right)}\right)^2
 + \frac{\pi^2}{6\gamma^2}  \left(\frac{\vartheta'_2\left(\frac{\pi t}{2\gamma}\right)}{\vartheta'_2\left(\frac{\pi t}{2\gamma}\right)}\right)^2.
\end{align*}
Plugging all ingredients, we get \eqref{eq_Zn_expansion_case3} of Theorem \ref{Theorem_asymptotics_case3}. The boundary case of Theorem \ref{Theorem_asymptotics_case4} is obtained as a limiting case of either of the above two computations.
\end{proof}

The scaling in Theorem \ref{Theorem_asymptotics_case1} is different and therefore the proof proceeds in a slightly different (although conceptually similar) way. For weights \eqref{eq_case1}, we introduce an auxiliary measure $\widehat \m$, which is a relative of $\m$ in \eqref{eq_measure_case1}: $\widehat \m$ is supported on negative even integers and is given by
 \begin{equation}
 \label{eq_measure_case1_aux}
 \widehat \m=\sum_{x\in 2 \Z_{<0}} \exp(-|\ga| x)\, \delta_x.
 \end{equation}

\begin{definition}
 The probability measure $\widehat{\mathcal M}^{n,t,\gamma}$ on ordered $n$-tuples $x_1\le x_2\le \dots \le x_n$ of negative even integers  is defined by its density with respect to $\widehat \m^{\otimes n}$:
 \begin{equation}
 \label{eq_M_def_2}
   \frac{1}{h_0 h_1 \cdots h_{n-1}} \prod_{1\le i < j \le n} (x_i-x_j)^2 \prod_{i=1}^n e^{tx_i} \widehat \m(\dd x_i),
 \end{equation}
 where $h_0 h_1\cdots h_{n-1}$ is a normalization constant making the total mass of $\widehat{\mathcal M}^{n,t,\gamma}$ equal to $1$.
\end{definition}

\begin{conjecture} \label{Conjecture_concentration_case1}
 In the setting of Theorem \ref{Theorem_asymptotics_case1}, we have as $n\to\infty$:
 \begin{equation}
 \label{eq_x13}
   \E_{\mathcal M^{n,t,\gamma}}\bigl[  \B_{x_1,\dots,x_n}(\xi_1,\dots,\xi_k,0^{n-k})\bigr]=\E_{\widehat{\mathcal M}^{n,t,\gamma}}\bigl[  \B_{x_1,\dots,x_n}(\xi_1,\dots,\xi_k,0^{n-k})\bigr] (1+o(1)).
 \end{equation}
\end{conjecture}
We expect the validity of Conjecture \ref{Conjecture_concentration_case1}, because it should be possible to represent (similar to how we did in the argument for Theorems \ref{Theorem_asymptotics_case2}--\ref{Theorem_asymptotics_case4}) the asymptotic of the expectations in two sides of \eqref{eq_x13} in terms of the equilibrium measures for $\mathcal M^{n,t,\gamma}$ and $\widehat{\mathcal M}^{n,t,\gamma}$. But these two equilibrium measures actually coincide because $\m$ and $\widehat \m$ are very similar, cf.\ \cite{Zinn_Justin00} or \cite{Bleher-Liechty14}.

\begin{proof}[Proof of Theorem \ref{Theorem_asymptotics_case1} conditional on Conjecture \ref{Conjecture_concentration_case1}] Our first task is to evaluate $\E_{\widetilde{\mathcal M}^{n,t,\gamma}}\bigl[  \B_{x_1,\dots,x_n}(\xi_1,\dots,\xi_k,0^{n-k})\bigr]$ explicitly. For that, we take $n$ complex numbers $z_1,\dots,z_n$, such that $\Re (z_i+t-|\gamma|)>0$ and compute
\begin{align}
\label{eq_x12}
 \sum_{x_1< x_2<\dots < x_n\le -2}\,  \prod_{1\le i < j \le n} (x_i-x_j)^2 \prod_{i=1}^n e^{(t-|\gamma|)x_i}  \frac{\det\bigl[ e^{x_i z_j}\bigr]_{i,j=1}^n}{\prod\limits_{1\le i<j\le n} (x_i-x_j)(z_i-z_j)},
\end{align}
where all $x_i$ are even negative integers. Introducing an auxiliary parameter $r$ and using the Vandermonde determinant evaluation  we rewrite \eqref{eq_x12} as
\begin{align}
\label{eq_x13a}
   \lim_{r\to 1} \left( \frac{(1-r)^{-n(n-1)/2}}{\prod\limits_{i<j} (z_i-z_j)}  \sum_{x_1< x_2<\dots < x_n\le -2}\,  \det\bigl[r^{k x_i}\bigr]_{i,k=1}^n \det\bigl[ e^{x_i (z_j+t-|\gamma|)}\bigr]_{i,j=1}^n\right).
\end{align}
The Cauchy--Binet formula computes the last sum in the closed form, transforming \eqref{eq_x13a} into
\begin{multline}
\label{eq_x14}
 \lim_{r\to 1} \left( \frac{(1-r)^{-n(n-1)/2}}{\prod\limits_{i<j} (z_i-z_j)}  \det\left[ \sum_{x=-2,-4,\dots} r^{kx} e^{x (z_j+t-|\gamma|)}\right]_{k,j=1}^n\right)
  \\= \lim_{r\to 1} \left( \frac{(1-r)^{-n(n-1)/2}}{\prod\limits_{i<j} (z_i-z_j)}  \det\left[ \frac{ r^{-2k} e^{-2 (z_j+t-|\gamma|)}}{1- r^{-2k} e^{-2(z_j+t-|\gamma|)}}\right]_{k,j=1}^n\right).
\end{multline}
We further use the Cauchy determinant evaluation in the form
$$
 \det\left[\frac{1}{1-a_i b_j}\right]_{i,j=1}^n = \frac{\prod\limits_{1\le i<j \le n} (a_i-a_j)(b_i-b_j)}{\prod_{i,j=1}^n(1-a_i b_j)},
$$
which transforms \eqref{eq_x14} into
\begin{multline}
\label{eq_x15}
 \lim_{r\to 1} \left( \frac{(1-r)^{-n(n-1)/2}}{\prod\limits_{i<j} (z_i-z_j)} \cdot \frac{\prod\limits_{i<j}(r^{-2i}-r^{-2j})(e^{-2(z_i+t-|\gamma|)}-e^{-2(z_j+t-|\gamma|)})}{\prod_{i,j=1}^n \bigl(1- r^{-2i} e^{-2(z_j+t-|\gamma|)}\bigr) } \prod_{i=1}^n \bigl[ r^{-2i} e^{-2 (z_i+t-|\gamma|)}\bigr]\right)
 \\ = \prod\limits_{i<j}\left[ (2j-2i)\frac{e^{-2(z_i+t-|\gamma|)}-e^{-2(z_j+t-|\gamma|)}}{z_i-z_j}\right] \cdot \prod_{j=1}^n\left[ \bigl(1-  e^{-2(z_j+t-|\gamma|)}\bigr)^{-n}    e^{-2 (z_i+t-|\gamma|)}\right]
 \\ = \prod\limits_{i<j}\left[ (2i-2j)\frac{e^{2(z_i+t-|\gamma|)}-e^{2(z_j+t-|\gamma|)}}{z_i-z_j}\right] \cdot \prod_{j=1}^n\left[ \bigl(1-  e^{2(z_j+t-|\gamma|)}\bigr)^{-n}  \right].
\end{multline}
Computing the ratio of \eqref{eq_x15} at $z_1=\xi_1,\dots,z_k=\xi_k$, $z_{k+1}=\dots=z_n=0$ and at $z_1=\dots=z_n=0$, we get
\begin{equation}
 \E_{\widehat{\mathcal M}^{n,t,\gamma}}\bigl[  \B_{x_1,\dots,x_n}(\xi_1,\dots,\xi_k,0^{n-k})\bigr]= \prod\limits_{1\le i <j \le k} \frac{e^{2\xi_i}-e^{2\xi_j}}{2(\xi_i-\xi_j)} \prod_{j=1}^k \Biggl[ \frac{e^{2\xi_j}-1}{2\xi_j} \Biggr]^{n-k}  \Biggl[ e^{-\xi_j} \frac{\sinh(t-|\gamma|)}{\sinh(\xi_j+t-|\gamma|)}\Biggr]^n.
\end{equation}
We combine the last formula with Theorem \ref{Theorem_partition_through_Bessel} and Conjecture \ref{Conjecture_concentration_case1}. We get
\begin{align*}
 \notag \ZZ_n(\xi_1,\dots,\xi_k;\, t,\gamma)&= \prod_{1\le i<j \le k} \frac{\xi_i-\xi_j}{ \sinh(\xi_i-\xi_j)}
 \prod\limits_{j=1}^k \left[\frac{\sinh(\xi_j+t-\ga) \sinh(\xi_j+t+\ga)}{\sinh(t-\ga) \sinh(t+\ga)}\right]^n  \left[\frac{\xi_j}{ \sinh(\xi_j)}\right]^{n-k}
 \\ &\quad \times \prod\limits_{1\le i <j \le k} \frac{e^{2\xi_i}-e^{2\xi_j}}{2(\xi_i-\xi_j)} \prod_{j=1}^k \Biggl[ \frac{e^{2\xi_j}-1}{2\xi_j} \Biggr]^{n-k}  \Biggl[e^{-\xi_j} \frac{\sinh(t-|\gamma|)}{\sinh(\xi_j+t-|\gamma|)}\Biggr]^n \cdot (1+o(1))
 \\ &=  \prod\limits_{j=1}^k \left(e^{-\xi_j} \left[\frac{\sinh(\xi_j+t+|\ga|)}{ \sinh(t+|\ga|)}\right]^n \right) \cdot (1+o(1)). \qedhere
\end{align*}
\end{proof}

\section{$k=1$ in Theorems \ref{Theorem_asymptotics_case1} -- \ref{Theorem_asymptotics_case4} through orthogonal polynomials}

\label{Section_k1_asymptotics}

In this section and the next one we provide a rigorous alternative to the arguments of Section \ref{Section_asymptotics_through_log_gas}. First, in the current section we prove Theorems \ref{Theorem_asymptotics_case1}--\ref{Theorem_asymptotics_case4}  for the case $k=1$,  relying on Theorem \ref{Theorem_partition_k1} and associated orthogonal polynomials for the asymptotic analysis. Section \ref{Section_general_reduction} then extends the arguments to $k>1$.

Theorem \ref{Theorem_partition_k1} expresses $\ZZ_n(\xi; t, \ga)$ in terms of the polynomial $P_{n-1}(x) = \E_{\mathcal M^{n-1,t,\gamma}} \left[\prod_{i=1}^{n-1} (x-x_i)\right] $.
 Our asymptotic analysis of $\ZZ(\xi; t, \ga)$ is based on the observation that $P_k$, $k\ge 0$, are orthogonal polynomials.
 \begin{lem}
Let $P_0(x) = 1$ and $P_k(x) := \E_{\mathcal M^{k,t,\gamma}} \left[\prod_{i=1}^{k} (x-x_i)\right] $ for $k>1$. These are the unique monic polynomials satisfying the orthogonality condition
\eq\label{eq:OP-def}
\int_{\R} P_j(x) P_k(x) e^{tx} \m(\dd x)= h_k \de_{j=k}, \quad j,k=0,1,2,\dots,
\eeq
where $h_k$ are the constants defined in \eqref{eq_M_def}.
\end{lem}
\begin{proof}
By definition
 \eq\label{eq:P_n1_Heine_def}
 \begin{aligned}
 \E_{\mathcal M^{k,t,\gamma}} \left[\prod_{i=1}^{k} (x-x_i)\right] &= \frac{1}{h_0h_1 \cdots h_{k-1}}\int_{x_1<x_2<\cdots<x_k}  \prod_{i=1}^{k} (x-x_i)\prod_{1\le i < j \le k} (x_j-x_i)^2 \prod_{i=1}^{k} e^{tx_i} \m(\dd x_i) \\
& =\frac{1}{h_0h_1 \cdots h_{k-1}k!}\int_\R \cdots \int_\R  \prod_{i=1}^{k} (x-x_i)\prod_{1\le i < j \le k} (x_j-x_i)^2 \prod_{i=1}^{k} e^{tx_i} \m(\dd x_i),
 \end{aligned}
 \eeq
 where in the second line we used the fact that the integrand is symmetric in all $k$ variables. This is the well-known Heine formula for orthogonal polynomials, see, e.g., \cite[Proposition 3.8]{Deift99} for a proof of orthogonality  for the corresponding system orthonormal polynomials. The formula given there differs from \eqref{eq:P_n1_Heine_def} by a multiplicative constant, and the polynomials  $ P_{k}(x) $ defined here are monic by construction.

The normalizing constant $h_0h_1\cdots h_{k-1}$ is given as
\eq\label{eq:h0hk}
h_0h_1\cdots h_{k-1} =\frac{1}{k!}  \int_\R\cdots\int_{\R} \prod_{1\le i < j \le k} (x_j-x_i)^2 \prod_{i=1}^{k} e^{tx_i} \m(\dd x_i),
\eeq
which is the same as the quantity $D_{k-1}$ in  \cite[Proposition 3.8]{Deift99}. The orthonormal polynomial of degree $k$ has leading coefficient $\sqrt{\frac{D_{k-1}}{D_k}}$ (see \cite[Equation (3.11)]{Deift99}), which implies the monic orthonormal polynomials satisfy
\eq\label{eq:OP-Dk}
\int_{\R} P_j(x) P_k(x) e^{tx} \m(\dd x)= \frac{D_k}{D_{k-1}} \de_{j=k}.
\eeq
Combining \eqref{eq:h0hk} and \eqref{eq:OP-Dk} we prove \eqref{eq:OP-def}.
\end{proof}

Since
\eq
\E_{\mathcal M^{1,t,\gamma}} \left[e^{\xi x} P_{n-1}(x)\right]= \frac{1}{h_0} \int_\R e^{\xi x} P_{n-1}(x) e^{tx} \m(\dd x),
\eeq
the formula \eqref{eq_partition_k1} can be written as
\eq\label{eq:Ztilde_OP_formula}
\ZZ_n(\xi; t, \ga) =(n-1)!\left(\frac{a(t+\xi,\ga)b(t+\xi,\ga)}{a(t,\ga)b(t,\ga)}\right)^n \left(\frac{1}{b(\xi,0)}\right)^{n-1} \int_\R e^{\xi x} \frac{P_{n-1}(x)}{h_{n-1}} e^{tx} \m(\dd x),
\eeq
where $P_{n-1}(x)$ and $h_{n-1}$ are the monic orthogonal polynomial and normalizing constant, respectively, defined by the orthogonality system \eqref{eq:OP-def}.

In order to evaluate \eqref{eq:Ztilde_OP_formula} asymptotically as $n\to\infty$, it is convenient to rescale the integral by a factor of $n$ and to absorb the $e^{tx}$ factor into the definition of the measure. We therefore introduce an auxiliary rescaled measure $\m_n$ by
\[
\m_n(A) = \frac{1}{n} \m(nA),
\]
for $A\subseteq \R$. And then we use $\m_n$ to define a new measure $\mt_n$ through
\[
\mt_n(\dd x) = e^{ntx} \m_n(\dd x).
\]
The rescaled polynomials
\eq\label{eq:OPs_rescaled}
p_{n,k}(x) = \left(\frac{1}{n}\right)^k P_k(nx),
\eeq
satisfy the orthogonality relation under the measure $\mt_n$:
\eq\label{eq:rescaled_orth}
\int_{\R} p_{n,j}(x) p_{n,k}(x)\mt_n(\dd x) = h_{n,k} \de_{j=k},  \qquad h_{n,k} = \left(\frac{1}{n}\right)^{2k+1} h_k.
\eeq
In terms of the rescaled polynomials, the formula \eqref{eq:Ztilde_OP_formula} becomes
\eq\label{eq:Ztilde_OP_formula_rescaled}
\ZZ_n(\xi; t, \ga) = (n-1)!\left(\frac{a(t+\xi,\ga)b(t+\xi,\ga)}{a(t,\ga)b(t,\ga)}\right)^n \left(\frac{1}{n b(\xi,0)}\right)^{n-1} \int_\R e^{n \xi x} \frac{p_{n,n-1}(x)}{h_{n,n-1}}  \mt_n(\dd x).
\eeq
It is convenient to split the last integral into two. For any $r, s>0$, we define the functions $I_{\rm in}(\xi;r,s)$ and $I_{\rm out}(\xi;r,s)$ as
\eq\label{eq:I_inout}
I_{\rm in}(\xi; r,s) = \int_{-r}^s e^{n \xi x}\frac{p_{n,n-1}(x)}{h_{n,n-1}}\mt_n(\dd x),\qquad I_{\rm out}(\xi; r,s) = \int_{\R \setminus (-r, s)} e^{n \xi x}\frac{p_{n,n-1}(x)}{h_{n,n-1}}\mt_n(\dd x),
\eeq
so that
\eq\label{eq:Ztilde_in_out}
\ZZ_n(\xi; t, \ga) = (n-1)!\left(\frac{a(t+\xi,\ga)b(t+\xi,\ga)}{a(t,\ga)b(t,\ga)}\right)^n \left(\frac{1}{n b(\xi,0)}\right)^{n-1} \bigl(I_{\rm in}(\xi; r,s)+I_{\rm in}(\xi; r,s)\bigr).
\eeq
The advantage of splitting the integral in \eqref{eq:Ztilde_OP_formula_rescaled} into these two pieces is that $I_{\rm in}(\xi; r,s)$ may then be expressed as a complex contour integral. For that we need one more definition. For a H\"older continuous function $f(x)\in L^1(\mt_n)$, we define the Cauchy transform with respect to  $\mt_n$ as
\eq\label{eq:Cauchy_def}
C^{\mt_n}(f)(z) = \frac{1}{2\pi \ii} \int_{\R} \frac{f(x) \mt_n(\dd x)}{x- z}.
\eeq
The integral \eqref{eq:Cauchy_def} is obviously convergent for $z$ outside the support of $\mt_n$. If the measure $\mt_n$ is a discrete one (as it is in the ferroelectric and antiferroelectric phase regions), then $C^{\mt_n}(f)(z)$ is a meromorphic function on $\C$. If the measure $\mt_n$ is absolutely continuous with respect to Lebesgue measure on $\R$ (as it is in the disordered and boundary phase regions), then the integral \eqref{eq:Cauchy_def} is singular for $z\in \R$, but H\"older continuity of $f$ implies that it takes limiting values as $z$ approaches the real line from the upper and lower half-planes. We therefore define for $x\in \R$ the functions $C^{\mt_n}(f)_\pm(x)$ as
\[
C^{\mt_n}(f)_{\pm}(x) =  \lim_{\eps\to 0}  \frac{1}{2\pi \ii} \int_{\R} \frac{f(y) \mt_n(\dd y)}{y- (x\pm \ii \eps) },
\]
so that $C^{\mt_n}(f)_{+}(z)$ (resp. $C^{\mt_n}(f)_{-}(z)$) is analytic in the upper (resp. lower) half-plane and continuous up to the boundary $\R$.

\begin{prop} \label{Prop:cauchy}
We can express $I_{\rm in}(\xi; r,s)$ in terms of the Cauchy transform \eqref{eq:Cauchy_def} as
\eq\label{eq:I_in_contour}
I_{\rm in}(\xi; r,s) = - \int_\Om e^{n \xi z} \frac{C^{\mt_n}(p_{n,n-1})(z) }{h_{n,n-1}}\, \dd z,
\eeq
where $\Om$ is a positively oriented contour which encloses the interval $(-r, s)$ and crosses the real axis at the points $z=-r$ and $z=s$. If $\mt_n$ is a discrete measure, then we additionally assume that $-r$ and $s$ are outside its support.
\end{prop}
\begin{remark}
In the disordered and boundary phase regions there is a minor ambiguity in the formula \eqref{eq:I_in_contour}, since $C^{\mt_n}(f)(z)$ is defined to be two-valued on the real line and the contour $\Om$ necessarily crosses the real line. However, since it crosses the real line at only two points, the particular value of $C^{\mt_n}(f)(z)$ at these two points does not affect the value of the integral and we may choose the value of $C^{\mt_n}(f)(-r)$ and $C^{\mt_n}(f)(s)$ to be either the limit from above or below.
\end{remark}
\begin{proof}[Proof of Proposition \ref{Prop:cauchy}] In the case that $\mt_n$ is a discrete measure, the integrals \eqref{eq:I_inout} are in fact sums and \eqref{eq:I_in_contour} follows from the residue theorem.

 If $\mt_n$ is absolutely continuous with respect to Lebesgue measure and supported on $\R$  then we can use the Plemelj formula (see, e.g., \cite[Lemma 7.2.1]{Ablowitz-Fokas97}), which relates the limiting values $C^{\mt_n}(f)_\pm(x)$ as
\eq\label{eq:Cauchy_jump}
C^{\mt_n}(f)_+(x) - C^{\mt_n}(f)_-(x) = f(x)w(x), \quad x\in \R,
\eeq
where $w(x)$ is the density for the measure $\mt_n$:
\eq
\mt_n(\dd x) = w(x) \dd x.
\eeq
It follows that
\eq
I_{\rm in}(\xi; r,s) = -\int_{\Ga_+ \cup \Ga_-} e^{n \xi z} \frac{C^{\mt_n}(p_{n,n-1})(z) }{h_{n,n-1}}\, \dd z,
\eeq
where $\Ga_\pm$ is the upper (resp. lower) cusp of the interval $(-r, s)$, oriented right-to-left (resp. left-to-right). Since $\Ga_\pm$ can be deformed into the upper (resp. lower) half-plane, the union $\Ga_+ \cup \Ga_-$ can be deformed into a positively oriented closed contour $\Om$ which encloses $(-r, s)$.
\end{proof}

\subsection{Asymptotics of the orthogonal polynomials}\label{sec:OPasy}

Since the measure $\m$ (and therefore $\mt_n$) is different in each of the phase regions, so are the orthogonal polynomials \eqref{eq:OP-def}, and asymptotic analysis of \eqref{eq:Ztilde_in_out} is slightly different in each of the phase regions. In each phase region, the large-$n$ asymptotics for the polynomial $p_{n,n-1}(z)$ can be obtained for $z\in \C$ by the Riemann--Hilbert method. This analysis has been implicitly carried out in a series of earlier papers \cite{Bleher-Bothner12, Bleher-Fokin06, Bleher-Liechty09, Bleher-Liechty10} (see also \cite{Bleher-Liechty14}), though explicit asymptotic formulas for $p_{n,n-1}(z)$ are not presented in any of those papers (the focus of those papers was on the asymptotic expansion of the normalizing constants $h_n$) and we need to extract these formulas.

Below we describe the relevant systems of orthogonal polynomials in each of the phase regions and present their asymptotic properties.
%

\begin{itemize}
 \item For  the ferroelectric phase parameterized by \eqref{eq_case1}, the orthogonality condition \eqref{eq:rescaled_orth} is
 \eq\label{def:OPsrescaled_F}
 \frac{1}{n}\sum_{x\in \frac{2}{n} \Z_{<0}} p_{n,k}^{\rm F}(x)p_{n,j}^{\rm F}(x)\, w_n^{\rm F}(x) = h^{\rm F}_{n,k}\de_{j=k}, \qquad w_n^{\rm F}(x) = e^{n(t-|\ga|)x} - e^{n(t+|\ga|)x}.
 \eeq
 \item For  the disordered phase parameterized by \eqref{eq_case2},  the orthogonality condition \eqref{eq:rescaled_orth} is
\eq\label{eq:rescaledOP}
\int_{-\infty}^\infty p_{n,j}^{\rm D}(x) p_{n,k}^{\rm D}(x) w_n^{\rm D}(x)\dd x = h_{n,k}^{\rm D} \de_{j=k}, \qquad w_n^{\rm D}(x) = e^{ntx}\frac{\sinh\left(\frac{nx(\pi - 2\ga)}{2}\right)}{\sinh(n\pi x/2)}.
\eeq
\item For the antiferroelectric  phase parameterized by \eqref{eq_case3},  the orthogonality condition \eqref{eq:rescaled_orth} is\footnote{There is an additional factor of 2 in the original measure \eqref{eq_measure_case3}. Changing the measure by a constant factor does not change the monic orthogonal polynomials, and only changes the normalizing constants $h_{n,k}$ by the same factor. Since the formula \eqref{eq:Ztilde_OP_formula_rescaled} depends only on the ratio of the measure and the normalizing constant $h_{n,n-1}$, it is unaffected by this change.}
\eq\label{def:OPsrescaled_AF}
\frac{1}{n}\sum_{x\in \frac{2}{n}\Z} p_{n,k}^{\rm AF}(x) p_{n,j}^{\rm AF}(x) w_n^{\rm AF}(x) = h_{n,k}^{\rm AF} \de_{j=k}, \qquad w_n^{\rm AF}(x) = e^{-n(\ga|x| - tx)}.
\eeq
  \item For  the boundary case parameterized by \eqref{eq_case4},  the orthogonality condition \eqref{eq:rescaled_orth} is
\eq\label{eq:rescaledOP_B}
\int_{-\infty}^\infty p^{\rm B}_{n,k}(x)p^{\rm B}_{n,j}(x) w^{\rm B}_n(x)\dd x = h^{\rm B}_{n,k} \de_{j=k}, \qquad w_n^{\rm B}(x) = e^{-n(\ga|x| - tx)}.
\eeq
\end{itemize}

Large-$n$ asymptotics of the orthogonal polynomials and their Cauchy transforms can be described in terms of the {\it equilibrium measure} described in Theorem \ref{Theorem_LLN_log_gas}.  Note that in each of the phase regions, the equilibrium measure $\nu$ is supported on an interval $[\al,\be]$ except in the ferroelectric region, where it is supported on the interval $[\be, 0]$ having density $\frac{1}{2}$ on the interval $[\al,0]$ with $\be<\al<0$. Since the points $\al$ and $\be$ describing the equilibrium measure are different in each phase region, we introduce subscripts for each phase region to avoid confusion: we use $\al_{\rm F}$ and $\be_{\rm F}$ to denote $\al$ and $\be$ as defined in \eqref{eq_case1_alpha_beta}; we use $\al_{\rm D}$ and $\be_{\rm D}$ to denote $\al$ and $\be$ as defined in \eqref{eq_case2_alpha_beta}; we use $\al_{\rm AF}$, $\al_{\rm AF}'$, $\be_{\rm AF}'$ and $\be_{\rm AF}$ to denote $\al$, $\al'$, $\be'$, and $\be$ as defined in \eqref{eq_case3_alpha_beta}; and we use $\al_{\rm B} \equiv \al_{\rm D}$ and $\be_{\rm B} \equiv \be_{\rm D}$ when referring to the boundary phase $\De = -1$.

The asymptotics of the orthogonal polynomials are given in terms of the log-transform of the equilibrium measure. For $* \in \{ {\rm D, B, AF}\}$, denote
\eq\label{def:g-function}
g_{*}(z):= \int_{\al_{*}}^{\be_{*}} \ln(z-x)\nu(\dd x), \quad z\in \C \setminus [\al_{*}, \be_{*}],
\eeq
and for the ferroelectric phase denote
\eq\label{def:g-functionF}
g_{{\rm F}}(z):= \int_{\be_{\rm F}}^0 \ln(z-x)\nu(\dd x), \quad z\in \C \setminus [\be_{\rm F}, 0],
\eeq
where we take the principal branch of the logarithm so that for $* \in \{ {\rm D, B, AF}\}$, $g_{*}(z)$ has a cut on $(-\infty, \be_{*}]$, and $g_F(z)$ has a cut on $(-\infty,0]$.
The derivative of the $g$-function is exactly the Stieltjes transform \eqref{eq:Stieltjes}, for which explicit formulas are given in each of the phase regions in Theorem \ref{Theorem_LLN_log_gas}. Explicit formulas for the $g$-function are then obtained by the conditions
\eq\label{eq:g_gp}
g_*'(z) = G_{\nu}(z), \qquad g_*(z) \sim \ln(z)+\bigO(1/z) \ \textrm{as} \ z\to\infty.
\eeq

%
The $g$-function satisfies several additional properties which we collect below. For $* \in \{ {\rm D, B, AF, F}\}$ introduce the notation
\eq
V_*(x):= -\lim_{n\to\infty} \frac{1}{n}  \ln w_n^*(x),
\eeq
so that
\eq\label{def:V}
V_{\rm D}(x)=V_{\rm B}(x) = V_{\rm AF}(x) = \ga|x| - t x, \qquad  V_{\rm F}(x) = -(t-|\ga|) x.
\eeq
The $g$-function satisfies the following properties, see e.g., \cite[Sections 2.4.2 and 3.4]{Bleher-Liechty14}.
\begin{itemize}
\item For $* \in \{ {\rm D, B, AF}\}$, $g_*(z)$ is analytic on $\C \setminus (-\infty , \be_*]$, and for $* = {\rm F}$ it is analytic on $\C \setminus (-\infty , 0]$. On the cuts  $(-\infty,  \be_*]$ or $(-\infty, 0]$  it takes limiting values from the upper and lower half planes, which we may denote
$g_{*\pm}(x)$. For $x$ on these cuts, the functions $g_{*+}(x)$ and $g_{*-}(x)$ differ by a purely imaginary term, thus $\Re g(z)$ is continuous on $\C$.
\item There exists a constant $l_*$ such that the functions $g_{*\pm}(x)$ satisfy
\eq\label{eq:eq_condition_band}
\begin{aligned}
g_{\rm D +}(x) + g_{\rm D -}(x) &= 2\Re(g_{\rm D \pm}(x))= V_{\rm D}(x) + l_{\rm D}, \quad x\in [\al_{\rm D}, \be_{\rm D}], \\
g_{\rm AF +}(x) + g_{\rm AF -}(x) &= 2\Re(g_{\rm AF \pm}(x)) = V_{\rm AF}(x) + l_{\rm AF}, \quad x\in [\al_{\rm AF}, \al_{\rm D}'] \cup [\be_{\rm AF}', \be_{\rm AF}], \\
g_{\rm F +}(x) + g_{\rm F -}(x) &=2\Re(g_{\rm F \pm}(x)) = V_{\rm F}(x) + l_{\rm F}, \quad x\in [ \be_{\rm F}, \al_{\rm F}]. \\
\end{aligned}
\eeq
\item For the same constant $l_*$, for $*\in \{{\rm D, B, AF}\}$,
\eq\label{eq:equilibrium_lessthan}
g_{ *+}(x) + g_{ *-}(x) = 2\Re(g_{\rm * \pm}(x)) < V_{\rm *}(x) + l_{\rm *}, \quad x \notin  {\rm supp}(\nu),
\eeq
and for $*={\rm F}$,
\eq\label{eq:equilibrium_lessthanF}
g_{ {\rm F}+}(x) + g_{ {\rm F}-}(x) = 2\Re(g_{\rm F \pm}(x)) < V_{\rm F}(x) + l_{\rm F}, \quad x< \be_{\rm F} .
\eeq
\item For the same constant $l_*$, for $*\in \{{\rm F, AF}\}$,
\eq \label{eq:eq_condition_sat}
\begin{aligned}
g_{ {\rm AF}+}(x) + g_{ {\rm AF}-}(x) &=  2\Re(g_{\rm AF \pm}(x))> V_{\rm AF}(x) + l_{\rm AF}, \quad x \in (\al_{\rm AF}', \be_{\rm AF}'), \\
g_{ {\rm F}+}(x) + g_{ {\rm F}-}(x) &=  2\Re(g_{\rm F \pm}(x)) > V_{\rm F}(x) + l_{\rm F}, \quad x \in ( \al_{\rm F},0].
\end{aligned}
\eeq
\end{itemize}

The following propositions  give asymptotic formulas for the systems of orthogonal polynomials as well as their Cauchy transforms  in terms of the relevant $g$-function.
\begin{prop}\label{prop:OP_asy}
For $*\in \{{\rm D},  {\rm B}, {\rm AF}\}$, let $p^*_{n,n-1}(z)$ and $h^*_{n,n-1}$ be the orthogonal polynomials and normalizing constants, respectively, defined in \eqref{eq:rescaledOP}, \eqref{def:OPsrescaled_AF}, and \eqref{eq:rescaledOP_B}.
\begin{itemize}
\item
For $* \in \{{\rm D}, {\rm B}\}$, the following asymptotic formulas are uniformly valid in closed subsets of $\C$ which are disjoint from $[\al_{\rm *},\be_{\rm *}]$.
\item For $* = {\rm AF}$ the following asymptotic formulas are uniformly valid in closed subsets of $\C$ which are disjoint from $[\al_{\rm AF}, \be_{\rm AF}]$ and whose distance from the lattice $\frac{2}{n}\Z$ is at least $\ep/n$ for a fixed $\ep>0$.
\end{itemize}
For $z$ in the sets described above we have as $n\to\infty$,
 \eq\label{eq:leading_Cauchy}
\frac{C^{\mt_n}(p^*_{n,n-1})(z)}{h^*_{n,n-1}} =-\frac{1}{2\pi \ii}e^{-ng_*(z)}M_*(z)(1+\bigO(1/n)),
\eeq
where $M_*(z)$ is an explicit function which is analytic on $\C\setminus [\al_*, \be_*]$ and satisfies
\eq\label{eq:Mstar_infinity}
M_*(z) = 1 + \bigO(1/z) \quad \textrm{as} \ z\to\infty.
\eeq
For $z$ in the same sets we have
\eq\label{eq:leading_OPs}
\begin{aligned}
\frac{p^*_{n,n-1}(z)}{h^*_{n,n-1}} &=e^{n(g_*(z)-l_*)}N_*(z)(1+\bigO(1/n)), \\
\end{aligned}
\eeq
where $N_*(z)$ is an explicit function which is analytic on $\C\setminus [\al_*, \be_*]$ and satisfies
\eq\label{eq:Nstar_infinity}
N_*(z) = \bigO(1/z) \quad \textrm{as} \ z\to\infty.
\eeq
\end{prop}

\begin{remark}
In the disordered and boundary phases, the Cauchy transform on the left-hand-side of \eqref{eq:leading_Cauchy} has a discontinuity across the real axis which does not appear in the asymptotic formula on the right-hand-side of \eqref{eq:leading_Cauchy}. In fact the discontinuity is exponentially small in $n$ compared to the leading order behavior $e^{-ng_*(z)}$, which can be seen as follows. Denote $V_n(x) = -\frac{1}{n}\ln w_n(x)$, and for $x\in \R$ write $C^{\mt_n}(p_{n,n-1})_\pm(x)$ for  for the the limit $\lim_{\ep\to 0} C^{\mt_n}(p_{n,n-1})(x\pm \ii \ep)$. The Plemelj formula gives
\[
\frac{C^{\mt_n}(p_{n,n-1})_+(x)}{h_{n,n-1}e^{-ng(z)}} - \frac{C^{\mt_n}(p_{n,n-1})_-(x)}{h_{n,n-1}e^{-ng(z)}} = \frac{p_{n,n-1}(x)e^{-nV_n(x)}}{h_{n,n-1}e^{-ng(z)}}.
\]
Assuming the result \eqref{eq:leading_OPs}, for $x\in \R \setminus [\al,\be]$, this is
\[
e^{n(2g(x)-l-V_n(x))}N(z)(1+\bigO(1/n)),
\]
which is exponentially small in $n$ by \eqref{eq:equilibrium_lessthan}, assuming $V_n(x) = V(x) + \bigO(1/n)$, which is easily checked for $x\in \R \setminus [\al,\be]$.
\end{remark}

In the ferroelectric phase there is a similar result.
\begin{prop}\label{prop:OP_asyF}
For $z\in \C\setminus\{\frac{2}{n}\Z_{<0}\cup [\be_{\rm F}, 0]\}$, the orthogonal polynomials $p^{\rm F}_{n,n-1}(z)$ satisfy
\eq\label{eq:leading_OPsF}
\frac{p^{\rm F}_{n,n-1}(z)}{h^{\rm F}_{n,n-1}} =e^{n(g_{\rm F}(z+1/n)-l_{\rm F})}N_{\rm F}(z+1/n)(1+\bigO(1/n)), \\
\eeq
and
 \eq\label{eq:leading_CauchyF}
\frac{C^{\mt_n}(p^{\rm F}_{n,n-1})(z)}{h^{\rm F}_{n,n-1}} =-\frac{1}{2\pi \ii}e^{-ng_{\rm F}(z+1/n)}M_{\rm F}(z+1/n)(1+\bigO(1/n)),
\eeq
where $M_{\rm F}(z)$ and $N_{\rm F}(z)$ are explicit functions analytic on $\C\setminus [ \be_{\rm F}, \al_{\rm F}]$ satisfying \eqref{eq:Mstar_infinity} and \eqref{eq:Nstar_infinity}, respectively,  with
$M_{\rm F}(z)$  given explicitly as
\eq\label{def:M}
M_{\rm F}(z) = \frac{1}{2}\left(\left(\frac{z-\al_{\rm F}}{z-\be_{\rm F}}\right)^{1/4} +  \left(\frac{z-\be_{\rm F}}{z-\al_{\rm F}}\right)^{1/4}\right),
\eeq
with a cut on $[\be_{\rm F}, \al_{\rm F}]$.

The asymptotic formulas \eqref{eq:leading_OPsF} and \eqref{eq:leading_CauchyF} are uniform on closed subsets of $\C\setminus  [\be_{\rm F}, 0]$ whose distance from $\frac{2}{n}\Z_{<0}$ is at least $\ep/n$.
\end{prop}

Proposition \ref{prop:OP_asy} follows from the Riemann--Hilbert analysis performed in \cite{Bleher-Bothner12,Bleher-Fokin06,Bleher-Liechty10}. We present a proof in Appendix \ref{app:OP_asymptotics} based on the analysis of those papers. For the orthogonal polynomials appearing in Proposition \ref{prop:OP_asyF}, the Riemann--Hilbert analysis has not been carried out before, though it is very similar to the analysis of the Meixner polynomials performed in \cite{Wang-Wong11}. We prove Proposition \ref{prop:OP_asyF} in Section \ref{app:A2} by presenting the steepest descent analysis of the Riemann--Hilbert problem for the polynomials \eqref{def:OPsrescaled_F} based on the approach of \cite{Bleher-Liechty11}, \cite[Chapter 3]{Bleher-Liechty14}. The shifts by $1/n$ in the right-hand sides of \eqref{eq:leading_OPsF} and \eqref{eq:leading_CauchyF} can be removed at the expense of adjusting the definitions $N_{\rm F}$ and $M_{\rm F}$ --- we keep the shifts in order to closer match the notations of the previous papers.

\medskip

In the following subsections  we prove the $k=1$ versions of Theorems \ref{Theorem_asymptotics_case1} -- \ref{Theorem_asymptotics_case4} by first inserting the asymptotic expressions from Propositions \ref{prop:OP_asy} and \ref{prop:OP_asyF} into the contour integral formulas for the integrals $I_{\rm in}(\xi; r, s)$ and $I_{\rm out}(\xi; r, s)$. The contour integrals can then be evaluated as $n\to\infty$ by the method of steepest descent, see \cite{Copson65,Erdelyi56} for the general introduction to the steepest descent. This method works for all $\xi$ in certain neighborhoods of the origin except for $\xi = 0$; in that case the critical point is at $\infty$ and is unapproachable by a steepest descent contour. We therefore prove the $k=1$ versions of Theorems \ref{Theorem_asymptotics_case1} -- \ref{Theorem_asymptotics_case4} first for $|\xi|> \ep$ for a fixed $\ep>0$ via steepest descent analysis. We then extend the result to $|\xi| \le \ep$ using the following lemma.

\begin{lem}\label{lem:small_xi}
Fix $\ep>0$ and suppose $F_n(\xi)$ is an $n$-dependent function of the complex variable $\xi$ which is analytic in the disc $\{\xi : |\xi|< 2\ep\}$ for all $n\in \N$ and which satisfies the asymptotic formula as $n\to\infty$
\eq\label{eq:2ep_lemma}
F_n(\xi) = \Psi(\xi)^{n} \psi(\xi) (1+\bigO(1/\sqrt{n})),
\eeq
uniformly for $ |\xi| = 2\ep$, where $\Psi(\xi)$ and $\psi(\xi)$ are both analytic and nonvanishing in the disc $\{\xi : |\xi|< 2\ep\}$ and independent of $n$. Then $F_n(\xi)$ satisfies the asymptotic formula \eqref{eq:2ep_lemma} uniformly for $|\xi| \le \ep$.
\end{lem}
\begin{proof}
This lemma follows easily from Cauchy's formula. Assume $|\xi| \le \ep$. The function $F_n(\xi) \Psi(\xi)^{-n}\psi(\xi)^{-1}$ can be expressed as
\[
\frac{F_n(\xi)}{ \Psi(\xi)^{n}\psi(\xi)}= \frac{1}{2\pi \ii} \oint_{|\xi| = 2\ep} \frac{F_n(\xi')  \dd\xi'}{\Psi(\xi')^{n}\psi(\xi')(\xi' - \xi)}.
\]
Inserting the asymptotics \eqref{eq:2ep_lemma} for $|\xi'| = 2\ep$, this gives
\begin{multline}
\frac{F_n(\xi)}{ \Psi(\xi)^{n}\psi(\xi)} = \frac{1}{2\pi \ii} \oint_{|\xi| = 2\ep} \frac{ (1+\bigO(1/\sqrt{n}))\dd\xi'}{\xi' - \xi} \\
=  \frac{1}{2\pi \ii} \oint_{|\xi| = 2\ep} \frac{\dd\xi'}{\xi' - \xi} + \frac{1}{2\pi \ii\sqrt{n} } \oint_{|\xi| = 2\ep} \frac{\bigO(1)\dd\xi'}{\xi' - \xi}
= 1 + \bigO(1/\sqrt{n}),
\end{multline}
where we have used $|\xi'-\xi|\ge \ep$.
\end{proof}


\subsection{Steepest descent analysis for $\De > 1$ and the proof of Theorem \ref{Theorem_asymptotics_case1} for $k=1$}\label{subsec:F_steepest_descent}
In this subsection we prove the $k=1$ version of Theorem \ref{Theorem_asymptotics_case1}.
Throughout this subsection, the orthogonal polynomials and all related quantities are those described in \eqref{def:OPsrescaled_F}
corresponding to the ferroelectric phase region, and we omit the sub/superscript F throughout. In the asymptotic analysis of integrals below, we assume $\xi$ is a complex number satisfying $\Re \xi > - (t - |\ga|)/2$, $|\Im \xi | < \pi/2$, and $|\xi| > \ep$ for a small but fixed positive number $\ep$. We then extend the asymptotics to the full domain  $\Re \xi > - (t - |\ga|)/2$, $|\Im \xi | < \pi/2$ by applying Lemma \ref{lem:small_xi}.

The expression \eqref{eq:Ztilde_in_out} is the sum of two terms, and we show below that $I_{\rm in}(\xi;r,s)$ provides the dominant contribution to the asymptotics of $\ZZ_n(\xi; t,\ga)$. We first evaluate
\eq\label{eq:I_in_shifted}
I_{\rm in}(\xi;r,s)= - \int_{\Om}\frac{C^{\mt_n}\left(p_{n,n-1}\right)(z)}{h_{n,n-1}}e^{n\xi z}\,\dd z = - \int_{\Om}\frac{C^{\mt_n}\left(p_{n,n-1}\right)(z-1/n)}{h_{n,n-1}} e^{n\xi z}e^{-\xi}\,\dd z,
\eeq
where in the second equality above we have shifted $z$ by $1/n$ in order to use Proposition \ref{prop:OP_asyF} and slightly abused notation by using $\Om$ for the shifted contour as well. Assuming the (shifted) contour $\Om$ remain at a distance $\ep/n$ from the discrete set $\frac{2}{n}\Z_{< 0}$, we can apply  the asymptotic formula \eqref{eq:leading_CauchyF}, we obtain
\eq\label{eq:S_in_F}
I_{\rm in}(\xi;r,s)=\frac{e^{-\xi}}{2\pi \ii} \int_{\Om} M(z) e^{-n(g(z)-\xi z)}(1+\bigO(1/n))\,\dd z.
\eeq
%
%
%

We will choose the contour $\Omega$ so that it passes through the critical point $z_{\rm cr}$ which solves
\eq\label{eq:case2_critF}
g'(z_{\rm cr})-\xi = 0.
\eeq
To solve this equation we use the exact formula \eqref{eq_case1_G} for $g'(z) \equiv G_\nu(z)$, from which we see that \eqref{eq:case2_critF} is equivalent to
\[
\frac{(\sqrt{-\al(z_{\rm cr}-\be)} - \sqrt{-\be(z_{\rm cr}-\al)})^2}{z_{\rm cr}(\al-\be)} =e^{- t + |\ga| -2\xi},
\]
assuming $\Re \xi> -(t-|\ga|)/2$, $|\Im \xi|< \pi/2$.
The above equation is straightforward to solve for $z_{\rm cr}$. Using the expressions \eqref{eq_case1_alpha_beta} for $\al$ and $\be$, we find the unique solution
\eq\label{eq:x0F}
z_{\rm cr}= \frac{\sinh(t-|\ga|)}{\sinh(\xi)\sinh(t-|\ga|+\xi)}.
\eeq

\begin{prop}\label{prop:Omega_F}
Let $g(z)$ be the function defined in \eqref{def:g-functionF}.
Assume $\xi$ is a complex parameter satisfying $\Re \xi > (t-|\ga|)/2$, $|\Im \xi|<\pi/2$, and $|\xi| > \ep$ for fixed $\ep>0$. There exists a positively oriented contour $\Om$ which encloses the interval $[\be,0]$ and satisfies the following properties:
\begin{enumerate}
\item $\Omega$ crosses $z_{\rm cr}$ in the direction of steepest descent for the function $\Re\left(-g(z) + \xi z\right)$.
\item For all $z\in \Om \setminus \{ z_{\rm cr}\}$,
\eq\label{eq:Omega_inequality}
\Re\left(-g(z_{\rm cr})+\xi z_{\rm cr} \right) > \Re\left(-g(z)+\xi z \right) , \quad z\in \Om \setminus \{ z_{\rm cr}\},
\eeq
\item $\Om$ encloses the interval $[\be, 0]$ and crosses the negative real axis at a point $-r<\be$ such that $\Re(g(x) -V(x)+ \xi x) = \Re(g(x) + (t-|\ga|+\xi) x)$ is strictly increasing on $(-\infty, -r)$.\label{prop:Omega_Fc}
\end{enumerate}
\end{prop}
Proposition \ref{prop:Omega_F} is proven in Section \ref{app:Omega}.
As long as the critical point $z_{\rm cr}$ itself is not within a distance of $\ep/n$ from the lattice $\frac{2}{n} \Z_{<0}$, it is possible to choose the contour $\Om$ from the above proposition such that it remains at a distance of at least $\ep/n$ from the lattice $\frac{2}{n} \Z_{<0}$. We will first proceed under the assumption that the critical point is not near $\frac{2}{n} \Z_{<0}$, and describe at the end of this subsection how the analysis may be adjusted in the case that it is.

Using \eqref{eq:Omega_inequality}, the integral \eqref{eq:S_in_F} is localized near $z_{\rm cr}$, where it is computed by Taylor expanding $g(z)$ near $z_{\rm cr}$ and recognizing a Gaussian integral after a change of variables $z=z_{\rm cr}+\frac{w}{\sqrt{n}}$. The result is
\eq\label{eq:I_in_locF}
I_{\rm in}(\xi;r,s)=  \frac{1}{\sqrt{2\pi n }} \frac{1}{\ii \sqrt{g''(z_{\rm cr})}} e^{-n(g(z_{\rm cr})-z_{\rm cr}\xi)}M(z_{\rm cr})(1+\bigO(1/\sqrt{n})),
\eeq
where the branch of $\sqrt{g''(z_{\rm cr})}$ is specified by the steepest descent direction of the contour $\Omega$ at the point $z_{\rm cr}$. Specifically,  $\Arg \sqrt{g''(z_{\rm cr})} = - \phi$, where $\phi$ is the local (oriented) direction of $\Om$ at the point $z_{\rm cr}$. We note that this direction depends continuously on $\xi$.

Turning to $I_{\rm out}(\xi;r,s)$ and using the asymptotic formula \eqref{eq:leading_OPs}, we have
\eq\label{eq:SoutF1}
I_{\rm out}(\xi;r,s)  =  \frac{1}{n} \sum_{\substack{x\in \frac{2}{n}\Z_{<0} \\ x < -r}} e^{n(g(x) - l  +\xi x)}w_n(x) N(x)(1+\bigO(1/n)), \\
\eeq
where $w_n(x)\equiv w_n^{\rm F}(x)$ is given in \eqref{def:OPsrescaled_F} with $N(x)$ uniformly $\bigO(1)$ throughout the sum. Note that for $x<-r$,
\[
w_n(x) = e^{n(t-|\ga|)x}(1-e^{2n|\ga| x}) = e^{-nV(x)}(1+\bigO(e^{-2nr|\ga|})),
\]
where $V(x) \equiv V_F(x)$ is defined in \eqref{def:V}.
Then Proposition \ref{prop:Omega_F}\eqref{prop:Omega_Fc} implies that \eqref{eq:SoutF1} satisfies the estimate
\eq\label{eq:I_out_F}
I_{\rm out}( \xi; r,s)   = \bigO(e^{n(g(-r) - l - V(-r) -\xi r) }). \\
\eeq
Combining \eqref{eq:I_in_locF} and \eqref{eq:I_out_F} we have
\begin{multline}\label{eq:in_plus_out}
I_{\rm in}( \xi; r,s) +I_{\rm out}( \xi; r,s)  = \\ \frac{1}{\sqrt{2\pi n }} \frac{M(z_{\rm cr})}{\ii \sqrt{g''(z_{\rm cr})}} e^{-n(g(z_{\rm cr})-z_{\rm cr}\xi)}  \left(1+\bigO(e^{n(g(-r)+g(z_{\rm cr}) - l - V(-r) -\xi(z_{\rm cr}+ r)) }\right)+\bigO(1/\sqrt{n})).
\end{multline}
Since $-r$ is on the contour $\Omega$, the estimate \eqref{eq:Omega_inequality} implies
\eq
\Re g(z_{\rm cr}) \le \Re(g(-r) + \xi (r+z_{\rm cr})).
\eeq
Using this estimate, the $\bigO$-term in \eqref{eq:in_plus_out} can be written as
\eq
e^{n(g(-r)+g(z_{\rm cr}) - l - V(-r) -\xi(z_{\rm cr}+ r)) } = \bigO(e^{n(2g(-r)-V(-r) - l )}).
\eeq
According to \eqref{eq:equilibrium_lessthan}, this exponent is negative, so \eqref{eq:in_plus_out} is simply
\eq
I_{\rm in}(\xi; r,s)+I_{\rm out}(\xi; r,s)  = \frac{1}{\sqrt{2\pi n }} \frac{M(z_{\rm cr})}{\ii \sqrt{g''(z_{\rm cr})}} e^{-n(g(z_{\rm cr})-z_{\rm cr}\xi)}  \left(1+\bigO(1/\sqrt{n})\right).
\eeq
The inhomogeneous partition function \eqref{eq:Ztilde_in_out} is therefore
\begin{multline}\label{eq:Zn_F_1}
\ZZ_n(\xi; t,\ga) = (n-1)!\left[\frac{\sinh(t-\ga+\xi)}{\sinh(t-\ga)}\frac{\sinh(t+\ga+\xi)}{\sinh(t+\ga)}\right]^n
 \left(\frac{1}{n\sinh(\xi)}\right)^{n-1}  \frac{e^{-\xi}}{\sqrt{2\pi n }} \frac{M(z_{\rm cr})}{\ii \sqrt{g''(z_{\rm cr})}}\\
 \times e^{-n(g(z_{\rm cr})-z_{\rm cr}\xi)}(1+\bigO(1/\sqrt{n})).
\end{multline}
We need to evaluate $g(z_{\rm cr})-z_{\rm cr}\xi$.  
The definition of $g(z)$ implies  that
\[
g(z) = z g'(z) + \ln\left[(\sqrt{z-\al} + \sqrt{z-\be})^2\right]- 1 - \ln 4.
\]
Plugging in $z=z_{\rm cr}$ and using $g'(z_{\rm cr})=\xi$ we find
\begin{multline}
g(z_{\rm cr}) -z_{\rm cr}\xi = z_{\rm cr} g'(z_{\rm cr}) - z_{\rm cr}\xi + \ln[(\sqrt{z_{\rm cr}-\al}+ \sqrt{z_{\rm cr}-\be})^2] - 1 - \ln 4\\
=  \ln[(\sqrt{z_{\rm cr}-\al}+ \sqrt{z_{\rm cr}-\be})^2]- 1 - \ln 4 \\
=\ln\left(2z_{\rm cr} - \al-\be+2\sqrt{(z_{\rm cr}-\al)(z_{\rm cr}-\be)}\right) - 1 - \ln 4.
\end{multline}
A direct calculation shows
\eq\label{eq:sqrt_prod}
2\sqrt{(z_{\rm cr}-\al)(z_{\rm cr}-\be)} =  \frac{2\sinh(t-|\ga|+2\xi)}{\sinh(\xi)\sinh(t-|\ga|+\xi)} = 2\coth(t-|\ga|+\xi)  + 2\coth(\xi),
\eeq
\eq
-\al-\be =4\coth(t-|\ga|).
\eeq
Combining these expressions with the formula \eqref{eq:x0F} for $z_{\rm cr}$, we find (employing the addition formulas for $\sinh$ and $\cosh$),
\eq\label{eq:sqrt_diff_squared}
2z_{\rm cr} - \al-\be+2\sqrt{(z_{\rm cr}-\al)(z_{\rm cr}-\be)} = \frac{4\sinh(t-|\ga|+\xi)}{\sinh(\xi)\sinh(t-|\ga|)},
\eeq
so that
\[
g(z_{\rm cr}) -z_{\rm cr}\xi = \ln\left(\frac{\sinh(t-|\ga|+\xi)}{\sinh(\xi)\sinh(t-|\ga|)}\right) - 1.
\]
We therefore have
\[
e^{-n(g(z_{\rm cr})-z_{\rm cr}\xi)}  = \left(\frac{e\sinh(t-|\ga|)\sinh(\xi)}{\sinh(t-|\ga|+\xi)}\right)^n,
\]
and using also Stirling's formula for $(n-1)!$,  \eqref{eq:Zn_F_1} becomes
\begin{multline}\label{eq:ZZ_pen}
\ZZ_n(\xi; t,\ga) = \frac{n!}{\sqrt{2\pi n}}\left(\frac{e}{n}\right)^n\left[\frac{\sinh(t+|\ga|+\xi)}{\sinh(t+|\ga|)}\right]^n
 \frac{e^{-\xi}\sinh(\xi)M(z_{\rm cr})}{\sqrt{- g''(z_{\rm cr})}} (1+\bigO(1/\sqrt{n}))  \\
 = \left[\frac{\sinh(t+|\ga|+\xi)}{\sinh(t+|\ga|)}\right]^n
 \frac{e^{-\xi}\sinh(\xi)M(z_{\rm cr})}{\ii \sqrt{ g''(z_{\rm cr})}} (1+\bigO(1/\sqrt{n})).
\end{multline}
It remains to calculate
$
\frac{M(z_{\rm cr})}{\ii \sqrt{g''(z_{\rm cr})}}.
$
As noted after \eqref{eq:I_in_locF}, the denominator $\ii \sqrt{g''(z_{\rm cr})}$ is continuous in $\xi$ throughout $\Re \xi > -(t-|\ga|)/2$, $|\Im \xi|< \pi/2$, and $|\xi|>\ep$.
Since $M(z)$ is analytic on $\C \setminus [\be, \al]$, and $z_{\rm cr}$ also depends continuously on $\xi$ and does not cross the interval $[\be, \al]$, the numerator is continuous in $\xi$ on the same set.

Using the formula \eqref{eq:Gp} for $g''(z)$ along with the formula \eqref{def:M} for $M(z)$, we obtain
\[
\frac{M(z_{\rm cr})}{\ii \sqrt{g''(z_{\rm cr})}} = \frac{M(z_{\rm cr})}{\sqrt{- g''(z_{\rm cr})}} = \pm \frac{\sqrt{z_{\rm cr}}\left(\sqrt{z_{\rm cr} - \al}+\sqrt{z_{\rm cr} - \be}\right)}{2},
\]
where the $\pm$ in the above formula accounts for the two possible branches of the square root function. While it is not immediately clear which branch is the correct one, we keep this ambiguity in the calculations for now.
The quantity $(\sqrt{z_{\rm cr}-\al}+ \sqrt{z_{\rm cr}-\be})^2$ is computed in \eqref{eq:sqrt_diff_squared}. Combining it with the formula \eqref{eq:x0F} for $z_{\rm cr}$, we find
\[
\frac{z_{\rm cr}(\sqrt{z_{\rm cr}-\al}+ \sqrt{z_{\rm cr}-\be})^2}{4} = \frac{1}{\sinh^2(\xi)}.
\]
Taking the square root yields
\[
\frac{M(z_{\rm cr})}{\ii \sqrt{g''(z_{\rm cr})} } =\pm \frac{1}{\sinh \xi},
\]
where once again the $\pm$ symbol refers to an ambiguity coming from the branch of the square root. To resolve this ambiguity, we first note that because $\frac{M(z_{\rm cr})}{\ii \sqrt{g''(z_{\rm cr})} }$ is continuous in $\xi$, the same sign must be applied for all $\xi$ in the region $\Re \xi > -(t-|\ga|)/2$, $|\Im \xi|< \pi/2$, and $|\xi|>\ep$. We can therefore determine the correct branch by looking at the case $z_{\rm cr}>0$, which is equivalent to $0 < \xi$. In this case the direction of steepest descent is $\phi = \pi/2$, so the argument of $\sqrt{g''(z_{\rm cr})}$ is $-\phi = -\pi/2$, and the denominator of $\frac{M(z_{\rm cr})}{\ii \sqrt{g''(z_{\rm cr})} }$ is positive. The numerator is positive as well, since $M(z)>0$ for $z>\al$, which implies that
\eq\label{eq:M_over_gpp}
\frac{M(z_{\rm cr})}{\ii \sqrt{g''(z_{\rm cr})} } = \frac{1}{\sinh \xi}
\eeq
is the correct branch of the square root for $\xi>0$, and by continuity for all $\Re \xi > -(t-|\ga|)/2$, $|\Im \xi| < \pi/2$, $|\xi|>\ep$.

Plugging \eqref{eq:M_over_gpp} into \eqref{eq:ZZ_pen} yields
\eq\label{ZZ_k1_F}
\ZZ_n(\xi; t,\ga) =  \left[\frac{\sinh(t+|\ga|+\xi)}{\sinh(t+|\ga|)}\right]^n
e^{-\xi}(1+\bigO(1/\sqrt{n})).
\eeq
This proves the $k=1$ version of Theorem \ref{Theorem_asymptotics_case1} for $\Re \xi > (t-|\ga|)/2$, $|\Im \xi|<\pi/2$, and $|\xi| > \ep$ for fixed $\ep>0$. The extension to $|\xi| \le \ep$ follows from Lemma \ref{lem:small_xi} with $\Psi(\xi) = \frac{\sinh(t+|\ga|+\xi)}{\sinh(t+|\ga|)}$ and $\psi(\xi) = e^{-\xi}$.

\medskip

In the steepest descent analysis above, we proved equation \eqref{eq:I_in_locF} under the assumption that the critical point $z_{\rm cr}$ is at a distance at least $\ep/n$ from the lattice $\frac{2}{n} \Z_{<0}$, so that the contour $\Omega$ can be taken at the same distance and the asymptotic formula \eqref{eq:leading_CauchyF} can be used throughout $\Om$. We complete the proof of the $k=1$ version of Theorem \ref{Theorem_asymptotics_case1} by explaining how the analysis should be adjusted when the critical point $z_{\rm cr}$ lies within a distance $\ep/n$ of lattice $\frac{2}{n} \Z_{<0}$.
If the critical point $z_{\rm cr}$ is closer than a distance $\ep/n$ from the lattice $\frac{2}{n} \Z$, then we must deform the contour $\Om$ in a small neighborhood of $z_{\rm cr}$ in order for the asymptotic formula \eqref{eq:leading_CauchyF} to hold throughout the contour of integration. In this case, denote by $x_0$ the lattice point of $\frac{2}{n}\Z_{< 0}$ which is nearest to $z_{\rm cr}$, and define the deformed contour $\widetilde \Om$ in the following way:
\eq\label{eq:Omega_deformed}
\widetilde \Om = \{\Om \setminus D(x_0, \ep/n)\} \cup \left\{x_0 + \frac{\ep}{n} e^{\ii\theta} \ | \ \phi_1 < \theta < \phi_2\right\},
\eeq
where $D(x_0, \ep/n)$ is the disc of radius $\ep/n$ centered at $x_0$, and $x_0+ \frac{\ep}{n} e^{\ii\phi_1}$ and $x_0 + \frac{\ep}{n} e^{\ii\phi_2}$ are the points at which $\Om$ intersects $\d D(x_0, \ep/n)$, see Figure \ref{fig:Om_tilde}.

\begin{figure}\
\begin{center}
\begin{tikzpicture}[scale=1.3]
\draw[red,thick,fill=red] (-1.8,0.3) circle (0.5mm)
(-1.6,0.2) node {$z_{\rm cr}$};

\draw[blue,thick,fill=blue] (-1,0) circle (0.5mm)
(-0.8,-0.1) node {$x_0$};

\draw[black,thick,fill=black] (-1.8,1.27) circle (0.5mm)
(-1.05,1.2) node {$x_0+\frac{\ep}{n}e^{{\bf i} \phi_1}$};

\draw[black,thick,fill=black] (-1.8,-1.27) circle (0.5mm)
(-1.15,-1.1) node {$x_0+\frac{\ep}{n}e^{{\bf i} \phi_2}$};

\draw[red,dashed] (-1.8,1.27) --(-1.8,-1.27);

\draw[red,very thick] (-1.8,1.27) -- (-1.8,2.27);

\draw[red,very thick]  (-1.8,-1.27) -- (-1.8,-2.27);

\draw[thick,dotted] (-1,0) circle [radius=1.5];

\draw[red,very thick] (-1.8,1.27) arc
    [
        start angle=122.231,
        end angle=237.769,
radius=1.5
    ] ;
    \end{tikzpicture}
\end{center}
\caption{The segment of the contour $\widetilde \Om$ in a small neighborhood of $z_{\rm cr}$ is shown in (solid) red. The dotted circle is the circle of radius $\ep/n$ around $x_0$, the point in $\frac{2}{n} \Z_{<0}$ which is closest to $z_{\rm cr}$. The red dashed line is the piece of the original contour $\Om$ which has been deformed to remain at a distance $\ep/n$ from $x_0$.}\label{fig:Om_tilde}
\end{figure}
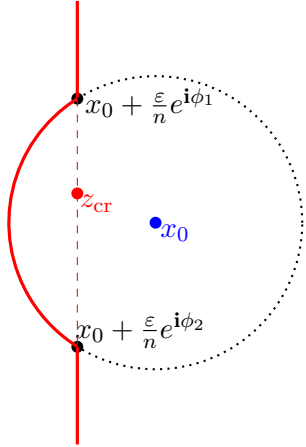

Replacing $\Om$ with $\widetilde \Om$ in \eqref{eq:I_in_shifted}, we can apply the asymptotic formula \eqref{eq:leading_CauchyF} to obtain
\eq\label{eq:Om_tilde_integral}
I_{\rm in}(\xi;r,s)=\frac{e^{-\xi}}{2\pi \ii} \int_{\widetilde \Om} M(z) e^{-n(g(z)-\xi z)}(1+\bigO(1/n))\,\dd z.
\eeq
The contour $\widetilde \Om$ differs from $\Om$ only in a $\bigO(1/n)$ neighborhood of $z_{\rm cr}$, and Proposition \ref{prop:Omega_F} implies that the integral localizes in a neighborhood of $z_{\rm cr}$:
\eq\label{eq:Om_tilde_integral_loc}
I_{\rm in}(\xi;r,s)=\frac{e^{-\xi}}{2\pi \ii} \int_{\widetilde \Om \cap D(z_{\rm cr}, \ep)} M(z) e^{-n(g(z)-\xi z)}(1+\bigO(1/n))\,\dd z.
\eeq
Let $\phi$ denote the direction of steepest descent for the contour $\Om$ at the point $z_{\rm cr}$ and denote by $\mathcal{S}$ the line segment from $z_{\rm cr} - \ep e^{\ii \phi}$ to  $z_{\rm cr} - \ep e^{\ii \phi}$. Also let $\mathcal{S}_n \subset \mathcal{S}$ be the line segment from $x_0+ \frac{\ep}{n} e^{\ii\phi_1}$ to $x_0+ \frac{\ep}{n} e^{\ii\phi_2}$.  Also denote by $\mathcal{A}$ the arc
\[
\mathcal{A} = \left\{x_0 + \frac{\ep}{n} e^{\ii\theta} \ | \ \phi_1 < \theta < \phi_2\right\},
\]
so that the only difference between $\Om$ and $\widetilde\Om$ is that $\mathcal{S}_n \subset \Om$ is deformed to $\mathcal{A} \subset \widetilde \Om$.
Note that both $\mathcal{S}_n$ and $\mathcal{A}$ have length of order $\bigO(1/n)$.
We can split the integral \eqref{eq:Om_tilde_integral_loc} as
\begin{multline}\label{eq:I_in_split}
I_{\rm in}(\xi;r,s)=\frac{e^{-\xi}}{2\pi \ii} \int_{\mathcal{S} \setminus \mathcal{S}_n} M(z) e^{-n(g(z)-\xi z)}(1+\bigO(1/n))\,\dd z \\
+ \frac{e^{-\xi}}{2\pi \ii} \int_{\mathcal{A}} M(z) e^{-n(g(z)-\xi z)}(1+\bigO(1/n))\,\dd z.
\end{multline}
The integral over $\mathcal{A}$ lies entirely in an $\bigO(1/n)$ neighborhood of $z_{\rm cr}$.  Since $z_{\rm cr}$ is a critical point for $g(z)-\xi z$, the Taylor expansions of $g(z)-\xi z$ and $M(z)$ at $z=z_{\rm cr}$ yield
\eq\label{eq:Om_tilde_estimate}
M(z) e^{-n(g(z)-\xi z)} = M(z_{\rm cr})e^{-n(g(z_{\rm cr})-\xi z_{\rm cr})}\left(1+ \bigO(1/n)\right), \quad \textrm{for} \ |z - z_{\rm cr}| = \bigO(1/n).
\eeq
Since the length of $\mathcal{A}$ is $\bigO(1/n)$, we therefore have
\begin{multline}\label{eq:SminusSn_int}
 \frac{e^{-\xi}}{2\pi \ii} \int_{\mathcal{A}} M(z) e^{-n(g(z)-\xi z)}(1+\bigO(1/n))\,\dd z =  \frac{e^{-\xi}}{2\pi \ii} \int_{\mathcal{A}} M(z_{\rm cr}) e^{-n(g(z_{\rm cr})-\xi z_{\rm cr})}(1+\bigO(1/n))\,\dd z  \\
=   e^{-n(g(z_{\rm cr})-\xi z_{\rm cr})}\bigO(1/n).
\end{multline}

In the integral over $\mathcal{S} \setminus \mathcal{S}_n$ we make the usual change of variables
\[
x=\sqrt{n g''(z_{\rm cr})}(z-z_{\rm cr}) \quad \textrm{with} \ \Arg \sqrt{n g''(z_{\rm cr})} = -\phi,
\]
which gives
\begin{multline}
\frac{e^{-\xi}}{2\pi \ii} \int_{\mathcal{S} \setminus \mathcal{S}_n} M(z) e^{-n(g(z)-\xi z)}(1+\bigO(1/n))\,\dd z  \\
= \frac{e^{-\xi}}{2\pi \ii  \sqrt{n g''(z_{\rm cr})}} M(z_{\rm cr}) e^{-n(g(z_{\rm cr})-\xi z_{\rm cr})}\int_{(-\ep \sqrt{n|g''(z_{\rm cr})|},\ep \sqrt{n|g''(z_{\rm cr})|})\setminus (\ep_n^{(1)}, \ep_n^{(2)})}e^{-x^2/2} (1+\bigO(1/n))\,\dd x,
\end{multline}
where $\ep_n^{(1)}$ and $\ep_n^{(2)}$ are the images of the points $x_0+ \frac{\ep}{n} e^{\ii\phi_1}$ and $x_0+ \frac{\ep}{n} e^{\ii\phi_2}$ under the change of variables $x=\sqrt{n g''(z_{\rm cr})}(z-z_{\rm cr})$. Since both $x_0+ \frac{\ep}{n} e^{\ii\phi_1}$ and $x_0+ \frac{\ep}{n} e^{\ii\phi_2}$ are within a distance of order $\bigO(1/n)$ from $z_{\rm cr}$, the numbers  $\ep_n^{(1)}$ and $\ep_n^{(2)}$ are both $\bigO(1/\sqrt{n})$. The integral in the second line of \eqref{eq:SminusSn_int} differs from a standard Gaussian integral by $\bigO(1/\sqrt{n})$ and we obtain
\[
\frac{e^{-\xi}}{\sqrt{2\pi} \ii  \sqrt{n g''(z_{\rm cr})}} M(z_{\rm cr}) e^{-n(g(z_{\rm cr})-\xi z_{\rm cr})}\left(1+\bigO(1/\sqrt{n})\right).
\]
Combining with \eqref{eq:I_in_split} and \eqref{eq:SminusSn_int} we obtain \eqref{eq:I_in_locF}.

\subsection{Steepest descent analysis for $\De < 1$ and the proof of Theorems \ref{Theorem_asymptotics_case2} -- \ref{Theorem_asymptotics_case4} for $k=1$}
In this section we assume $\De<1$ so that the six-vertex weights are parametrized by \eqref{eq_case2}, \eqref{eq_case3}, or \eqref{eq_case4}. We will prove the $k=1$ versions of Theorems \ref{Theorem_asymptotics_case2} -- \ref{Theorem_asymptotics_case4} in the following way: we first use \eqref{eq:Ztilde_in_out} and the contour integral formula \eqref{eq:I_in_contour} to derive an asymptotic formula for $\ZZ(\xi; t,\ga)$ which is uniform in an annulus $\ep \le |\xi| < 2\ep$; then we use Lemma \ref{lem:small_xi} to extend this asymptotic formula to the full disc $|\xi| \le 2\ep$; finally we replace $\xi$ with $\xi/\sqrt{n}$ to obtain Theorems \ref{Theorem_asymptotics_case2} -- \ref{Theorem_asymptotics_case4}.


As in the case $\De>1$ described in the previous subsection, the primary contribution to \eqref{eq:Ztilde_in_out} comes from $I_{\rm in}(\xi;r,s)$, and we will use the formula \eqref{eq:I_in_contour} for $I_{\rm in}(\xi;r,s)$. Substituting the asymptotic formula \eqref{eq:leading_Cauchy} into the contour integral \eqref{eq:I_in_contour} yields
\eq\label{eq:case2_int_asy}
I_{\rm in}(\xi; r,s)= \frac{1}{2 \pi i} \int_{\Omega} e^{n(-g(z) + z\xi)}M(z)(1+\bigO(1/n)) \dd z,
\eeq
so we once again seek a critical point $z_{\rm cr}$ which solves the equation
\eq\label{eq:case2_crit}
g'(z)= \xi.
\eeq
In order to prove Theorems \ref{Theorem_asymptotics_case2} -- \ref{Theorem_asymptotics_case4} for $k=1$, we will not need an explicit formula for this critical point; indeed, in the antiferroelectric phase $\De<-1$ the function $g(z)$ is described by elliptic integrals and an explicit solution to \eqref{eq:case2_crit} becomes rather complicated. To avoid this difficulty we prove the existence of a unique solution to \eqref{eq:case2_crit} for $\xi$ in a small annulus around the origin, as described in the following proposition.
\begin{prop}\label{prop:crit_pt_ex}
For any choice of six-vertex weights \eqref{eq_case2} -- \eqref{eq_case4}, there exists $\ep>0$ such that the equation \eqref{eq:case2_crit} has a unique solution $z_{\rm cr} \equiv z_{\rm cr}(\xi)$ for all $\xi\in \C$ satisfying $0< |\xi| \le 2\ep$. Furthermore, the set $\{ z_{\rm cr}(\xi) : \ep\le |\xi| \le 2\ep\}$ is a compact set in $\C$ which does not intersect the support of the equilibrium measure, i.e., it is a compact subset of $\C \setminus [\al,\be]$ when $\De<1$.
\end{prop}
\begin{proof}
Note that
\[
g'(z) = \int_\al^\be \frac{\nu(\dd x)}{z-x}= \frac{1}{z}\int_\al^\be \frac{\nu(\dd x)}{1-x/z},
\]
which can be expanded as a geometric series in $x/z$, giving the expansion at infinity,
\eq\label{eq:gp_infinity}
g'(z) = \frac{1}{z} + \frac{\mu_1}{z^2} + \frac{\mu_2}{z^3} + \frac{\mu_3}{z^4} + \cdots, \qquad \mu_j = \int_\al^\be x^j \nu(\dd x),
\eeq
thus $g'(z)$ is analytic at $z=\infty$ with value $g'(\infty) = 0$. Making the change of variable $w=1/z$, the expansion \eqref{eq:gp_infinity} becomes
\[
g'(z(w))=g'(1/w) = w + \mu_1w^2 + \mu_2 w^3 + \mu_3 w^4 + \cdots
\]
which is clearly invertible in a neighborhood of $w=0$. We can denote this inverse function as $w_{\rm cr}(\xi)$, so that for some fixed $\ep>0$,
\[
g'(1/w_{\rm cr}(\xi)) = \xi, \qquad |\xi|\le 2\ep.
\]
Since $g'(\infty) = g'(w(0)) = 0$, we have $w_{\rm cr}(0) = 0$ and $|w_{\rm cr}(\xi)|>0$ for $\ep\le|\xi|\le 2\ep$.
Writing $z_{\rm cr}(\xi) = \frac{1}{w_{\rm cr}(\xi)}$, we then have that for a fixed $\ep>0$, $z_{\rm cr}(\xi)$ is well defined on the annulus $\ep\le|\xi|\le 2\ep$ and maps that region to a bounded closed subset of $\C$. To see that this set does not intersect the support of the equilibrium measure, we can look at the equilibrium conditions \eqref{eq:eq_condition_band} and \eqref{eq:eq_condition_sat}, which imply (see the formulas \eqref{def:V} for $V(x)$) that
\eq\label{eq:gp_V_bound}
|g'(x)| \ge \frac{V'(x)}{2} \ge \frac{\ga - |t|}{2},
\eeq
for $x$ in the support of the equilibrium measure. Thus if we take $\ep< \frac{\ga - |t|}{4}$ we cannot have $g'(x) = \xi$ for any $x$ in the support of the equilibrium measure and $\xi$ satisfying $|\xi|\le 2\ep$.

The above argument shows that for each $0<\xi\le \ep$ there is a unique solution to \eqref{eq:case2_crit} lying in a neighborhood of infinity. To complete the proof we must show additionally that there are no solutions outside a neighborhood of infinity. The estimate \eqref{eq:gp_V_bound} implies that for small enough $\ep>0$, there are no solutions to \eqref{eq:case2_crit} in a small complex neighborhood of the interval $[\al,\be]$.
To consider other possible solutions, let us split $g'(z)$ into its real and imaginary parts and write \eqref{eq:case2_crit} as
\[
\int_\al^\be \frac{(\Re z -x) \nu(\dd x)}{(\Re z -x)^2 +(\Im z)^2} - \ii \, \Im z \int_\al^\be \frac{\nu(\dd x)}{(\Re z -x)^2 +(\Im z)^2} = \xi.
\]
Since $|\xi|<\ep$, both the real and imaginary parts of $g'(z)$ must be smaller than $\ep$ in absolute value as well. Since we are looking for a solution outside a neighborhood of infinity, consider first the case $|\Re z| < C$ and $C^{-1} < |\Im z| < C$ for a fixed $C>\max\{|\al|, \be\}$. Then the imaginary part of $g'(z)$ is estimated as
\[
|\Im g'(z) | =\left|  \Im z \int_\al^\be \frac{\nu(\dd x)}{(\Re z -x)^2 +(\Im z)^2}\right| \ge \frac{C^{-1} }{D^2+C^2}, \quad D = \max\{C-\al, \be+C\},
\]
which is larger than $\ep$ for small enough $\ep$. We are left only to consider the case $|\Im z|$ is small and $\Re z$ is bounded away from both $[\al,\be]$ and $\infty$. For concreteness, consider the case $\be +C^{-1}< \Re z < C$ and $|\Im z |< C^{-1}$ for a fixed $C>\be$ (the case $\Re z < \al - C^{-1}$ is similar). Then the real part of $g'(z)$ is an integral with a positive integrand and may be estimated as
\[
\Re g'(z) = \int_\al^\be \frac{(\Re z -x) \nu(\dd x)}{(\Re z -x)^2 +(\Im z)^2}  \ge  \frac{C^{-1}}{(C -\al)^2 +C^{-2}},
\]
which is again bounded away from zero for a fixed $C$, thus greater than a small enough $\ep>0$. It follows that for small enough $|\xi|$ the only solution to \eqref{eq:case2_crit} lies in a neighborhood of $\infty$ and is therefore unique by the inverse function argument at the beginning of this proof.
\end{proof}

In the rest of this subsection we assume that $0<|\xi|\le2\ep \le \frac{\ga-|t|}{2}$ and denote the unique solution to \eqref{eq:case2_crit} as $z_{\rm cr}$. The asymptotic analysis of $\ZZ(\xi; t,\ga)$ depends on the following proposition which describes a contour for which the steepest descent analysis of \eqref{eq:Ztilde_in_out} is possible.
\begin{prop}\label{prop:Omega_DAF}
Let $g(z)$ be the function defined in \eqref{def:g-function} associated with six-vertex weights parametrized by one of the cases \eqref{eq_case2} -- \eqref{eq_case4}.
Assume $\xi$ is a complex parameter satisfying $\xi \in \C$ and $\ep\le|\xi|\le2\ep \le \frac{\ga-|t|}{2}$. There exists a positively oriented contour $\Om$ which encloses the interval $[\al,\be]$ and satisfies the following properties:
\begin{enumerate}
\item $\Omega$ crosses $z_{\rm cr}$ in the direction of steepest descent for the function $\Re\left(-g(z) + \xi z\right)$.
\item For all $z\in \Om \setminus \{ z_{\rm cr}\}$,
\eq\label{eq:Omega_inequality_DAF}
\Re\left(-g(z_{\rm cr})+\xi z_{\rm cr} \right) > \Re\left(-g(z)+\xi z \right) , \quad z\in \Om \setminus \{ z_{\rm cr}\}.
\eeq
\item $\Om$ encloses the interval $[\al, \be]$ and crosses the real axis at the points $-r<\be$ and $s>\al$ such that $\Re(g(x) -V(x)+ \xi x) = \Re(g(x) + (t+\xi) x - \ga|x|)$ is strictly increasing on $(-\infty, -r)$ and strictly decreasing on $(s,\infty)$.\label{prop:Omega_DAFc}
\end{enumerate}
\end{prop}
This proposition is proved in Section \ref{app:Omega}.

\medskip

We now let the contour $\Om$ in \eqref{eq:I_in_contour} be as described in Proposition \ref{prop:Omega_DAF}. For $\De\ge -1$, we can then apply the asymptotic formula \eqref{eq:leading_Cauchy} to obtain \eqref{eq:case2_int_asy}. If $\De<-1$ then we can only apply \eqref{eq:leading_Cauchy} if the contour $\Om$ is separated from the lattice $\frac{2}{n} \Z$ by a distance of at least $\ep/n$. It is possible to choose the contour $\Om$ of Proposition \ref{prop:Omega_DAF} in this way unless the critical point $z_{\rm cr}$ itself is within a distance $\ep/n$ of $\frac{2}{n} \Z$. In that case the contour $\Om$ can be deformed in an $\bigO(1/n)$-neighborhood of $z_{\rm cr}$ to a contour $\widetilde\Om$ as defined in \eqref{eq:Omega_deformed} without affecting the asymptotic analysis, as described in \eqref{eq:Om_tilde_estimate} --- this part of the argument is exactly the same as the one at the end of Section \ref{subsec:F_steepest_descent}.

We therefore proceed assuming that \eqref{eq:case2_int_asy} holds, in which case the integral $I_{\rm in}(\xi;r,s)$ is localized near $z_{\rm cr}$, and the method of steepest descent yields
\eq\label{eq:hn_case2_asy2}
I_{\rm in}(\xi; r,s)  =  \frac{e^{-n(g(z_{\rm cr})- z_{\rm cr}\xi)}}{\sqrt{2\pi n}}\frac{M(z_{\rm cr})}{\ii \sqrt{g''(z_{\rm cr})}}
(1+\bigO(1/\sqrt{n})),
\eeq
where the branch of $\sqrt{g''(z_{\rm cr})}$ is again given as $\Arg \sqrt{g''(z_{\rm cr})} = -\phi$, where $\phi$ is the local (oriented) steepest descent direction for the contour $\Om$ at the point $z_{\rm cr}$. In particular, $\sqrt{g''(z_{\rm cr})}$ is continuous, and thus analytic, in $\xi$ on the set $\ep < |\xi| < 2\ep$.

Now consider $I_{\rm out}(\xi;r,s)$, given in \eqref{eq:I_inout}. Inserting the asymptotic formula \eqref{eq:leading_OPs} we obtain
\eq\label{eq:I_out_asy}
I_{\rm out}(\xi;r,s) = \int_{\R \setminus (-r,s)} e^{n(g(z)-V(x) - l +\xi x )}(1+\bigO(1/\sqrt{n}))\dd x,
\eeq
and Proposition \ref{prop:Omega_DAF}\eqref{prop:Omega_DAFc} implies
\[
I_{\rm out}(\xi;r,s) = \bigO\left(e^{n(g(-r)-V(-r) - l -\xi r)}\right) + \bigO\left(e^{n(g(s)-V(s) - l +\xi s)}\right).
\]
Combining with \eqref{eq:hn_case2_asy2} we obtain
\begin{multline}\label{eq:in_plus_out2}
I_{\rm in}(\xi;r,s)+I_{\rm out}(\xi;r,s)  =   \frac{e^{-n(g(z_{\rm cr})- z_{\rm cr}\xi)}}{\sqrt{2\pi n}}\frac{M(z_{\rm cr})}{\ii \sqrt{g''(z_{\rm cr})}} \\
\times\left(1+\bigO(e^{n(g(-r)+g(z_{\rm cr})-V(-r) - l +\xi (-r-z_{\rm cr}))}) + \bigO(e^{n(g(s)+g(z_{\rm cr})-V(s) - l +\xi (s-z_{\rm cr}))})+\bigO(1/\sqrt{n})\right).
\end{multline}
Since $-r$ and $s$ are on the contour $\Omega$, the estimate \eqref{eq:Omega_inequality_DAF} implies
\eq
\Re g(z_{\rm cr}) \le \Re(g(-r) - \xi (-r-z_{\rm cr})), \quad \Re g(z_{\rm cr}) \le \Re(g(s) - \xi (s-z_{\rm cr})).
\eeq
Using this estimate, the $\bigO$-terms in \eqref{eq:in_plus_out} can be written as
\begin{align}
\bigO\left(e^{n(g(-r)+g(z_{\rm cr})-V(-r) - l +\xi (-r-z_{\rm cr}))}\right) &= \bigO\left(e^{n(2g(-r)-V(-r) - l )}\right),  \\ \bigO\left(e^{n(g(s)+g(z_{\rm cr})-V(s) - l +\xi (s-z_{\rm cr}))}\right) &= \bigO\left(e^{n(2g(s)-V(s) - l )}\right).
\end{align}
According to \eqref{eq:equilibrium_lessthan}, both of these exponents are negative, so \eqref{eq:in_plus_out2} is simply
\eq
I_{\rm in}(\xi;r,s)+I_{\rm out}(\xi;r,s)  =
 \frac{e^{-n(g(z_{\rm cr})- z_{\rm cr}\xi)}}{\sqrt{2\pi n}}\frac{M(z_{\rm cr})}{\ii \sqrt{g''(z_{\rm cr})}}\left(1+\bigO(1/\sqrt{n})\right).
\eeq
Inserting this asymptotic into \eqref{eq:Ztilde_in_out} we have
\begin{multline}\label{eq:hn_case2_asy3}
\ZZ(\xi; t,\ga) = (n-1)!\left(\frac{a(t+\xi,\ga)b(t+\xi,\ga)}{a(t,\ga)b(t,\ga)}\right)^n \left(\frac{1}{n b(\xi,0)}\right)^{n-1} \\
\times  \frac{e^{-n(g(z_{\rm cr})- z_{\rm cr}\xi)}}{\sqrt{2\pi n}}\frac{M(z_{\rm cr})}{\ii \sqrt{g''(z_{\rm cr})}}(1+\bigO(1/\sqrt{n})) \\
= \left(\frac{a(t+\xi,\ga)b(t+\xi,\ga)e^{-(1+g(z_{\rm cr})- z_{\rm cr}\xi)}}{a(t,\ga)b(t,\ga)b(\xi,0)}\right)^n \frac{M(z_{\rm cr})b(\xi,0)}{\ii \sqrt{g''(z_{\rm cr})}}
(1+\bigO(1/\sqrt{n})).
\end{multline}
where we have used Stirling's formula in the second line.
We have proved the following proposition.
\begin{prop}\label{prop:le_1_annulus}
Fix $0<\ep<\frac{\ga-|t|}{4}$ such that Proposition \ref{prop:crit_pt_ex} holds, and denote $z_{\rm cr} \equiv z_{\rm cr}(\xi)$ for all $\xi\in \C$ satisfying $\ep\le|\xi|\le 2\ep$. Then the following asymptotic formula holds uniformly for $\ep\le|\xi|\le 2\ep$:
\eq\label{eq:Psipsi}
\ZZ(\xi; t,\ga) = \Psi(\xi)^n \psi(\xi)(1+\bigO(1/\sqrt{n})),
\eeq
where
\eq\label{def:Psipsi}
\Psi(\xi) = \frac{a(t+\xi,\ga)b(t+\xi,\ga)e^{-(1+g(z_{\rm cr})- z_{\rm cr}\xi)}}{a(t,\ga)b(t,\ga)b(\xi,0)} , \quad \psi(\xi) = \frac{M(z_{\rm cr})b(\xi,0)}{\ii \sqrt{ g''(z_{\rm cr})}}.
\eeq
In the above, the branch of the square root $\sqrt{g''(z_{\rm cr})}$ is determined by the local oriented direction of the steepest descent contour, and $\sqrt{g''(z_{\rm cr})}$ is analytic in $\xi$ on the set $\ep<|\xi|<2\ep$.
\end{prop}
We would like to extend these asymptotics to all $|\xi|< 2\ep$ using Lemma \ref{lem:small_xi}. In order to do so we need only to show that $\Psi(\xi)$ and $\psi(\xi)$ are analytic and nonvanishing in the disc $|\xi| < 2\ep$. Consider first $\Psi(\xi)$. It is clear that  $\frac{a(t+\xi,\ga)b(t+\xi,\ga)}{a(t,\ga)b(t,\ga)}$ is analytic and nonvanishing in a neighborhood of $\xi=0$, so we need to consider the factor  $\frac{e^{-(1+g(z_{\rm cr})- z_{\rm cr}\xi)}}{b(\xi,0)}$. The quantity $z_{\rm cr}\equiv z_{\rm cr}(\xi)$ is well-defined and analytic for $\xi\ne 0$, but approaches $\infty$ as $\xi\to 0$. Similarly $g(z_{\rm cr}(\xi))$ and $1/b(\xi,0)$ are well-defined and analytic for $\xi\ne 0$, but both approach $\infty$ as $\xi\to 0$. In particular, for any of the parametrizations \eqref{eq_case2} -- \eqref{eq_case4}, as $\xi\to 0$,
\eq\label{eq:b_xi0}
b(\xi, 0) = \xi + \bigO(\xi^2).
\eeq
Therefore in order to show that $\Psi(\xi)$ is analytic and nonvanishing in the full disc $|\xi|<2\ep$, we need to show that it has a non-zero limit as $\xi\to 0$. To see this, we first use the expansion \eqref{eq:gp_infinity} to obtain the relation
\[
\xi z_{\rm cr}(\xi) = 1+\bigO(1/z_{\rm cr}), \quad \textrm{as} \ z_{\rm cr} \to \infty.
\]
Since $z_{\rm cr} \to \infty$ is equivalent to $\xi\to 0$, we have
\eq\label{eq:xi_zcr}
\lim_{\xi \to 0} \xi z_{\rm cr}(\xi) = 1.
\eeq
Now consider $\frac{e^{-g(z_{\rm cr})}}{b(\xi,0)}$ as $\xi\to 0$. Since $g(z) = \ln(z) + \bigO(1/z)$ as $z\to\infty$, and $b(\xi,0) = \xi + \bigO(\xi^2)$ as $\xi\to 0$,
\eq
\frac{e^{-g(z_{\rm cr})}}{b(\xi,0)} = 1/(\xi z_{\rm cr})\left(1+\bigO(\xi) + \bigO(1/z_{\rm cr})\right),
\eeq
which in light of \eqref{eq:xi_zcr} implies
\eq\label{eq:exp_b}
\lim_{\xi\to 0} \frac{e^{-g(z_{\rm cr})}}{b(\xi,0)} = 1.
\eeq
Combining  \eqref{eq:xi_zcr} and \eqref{eq:exp_b} we find that $\lim_{\xi\to 0} \Psi(\xi) = 1$.

Now consider $\psi(\xi)$. The expansion \eqref{eq:Mstar_infinity} implies that $M(z_{\rm cr})$ is analytic and nonzero for large enough $|z_{\rm cr}|$, or equivalently for $\xi$ in a fixed small neighborhood of zero.  We need to check that $\frac{b(\xi,0)}{\ii \sqrt{g''(z_{\rm cr})}}$ is also analytic and non-zero. Recall that $\sqrt{g''(z_{\rm cr})}$ is analytic in $\xi$ on $\ep<|\xi|<2\ep$, and that $z_{\rm cr} \to \infty$ as $\xi \to 0$. Expanding \eqref{eq:Gp} or \eqref{eq:GpAF} at $z=\infty$ we find as $z_{\rm cr} \to \infty$,
\[
g''(z_{\rm cr}) = -\frac{1}{z_{\rm cr}^2} + \bigO(z_{\rm cr}^{-3}),
\]
thus
\eq\label{eq:sqrt_gp_pm}
\sqrt{g''(z_{\rm cr})} = \pm \frac{\ii }{z_{\rm cr}} + \bigO(z_{\rm cr}^{-2}),
\eeq
where the $\pm$ symbol refers to the fact that we must choose the correct branch of the square root. Since $\sqrt{g''(z_{\rm cr})}$ is analytic $\xi$ on $\ep<|\xi|<2\ep$, the same sign choice must apply for all $\xi$ in this set. To fix the sign then, we consider real positive $\xi$, $\ep<\xi<2\ep$, in which case $z_{\rm cr}>\be$ and $g''(z_{\rm cr})<0$. Locally near $z_{\rm cr}$, the contour $\Om$ is oriented in the direction $\phi = \pi/2$, and so the argument of $\sqrt{g''(z_{\rm cr})}$ is $-\phi = -\pi /2$, which implies that \eqref{eq:sqrt_gp_pm} should come with a minus-sign, and that
\eq
\ii \sqrt{g''(z_{\rm cr})} = \frac{1}{z_{\rm cr}} + \bigO(z_{\rm cr}^{-2}),
\eeq
which we can use as the local behavior near $\xi=0$ of $\ii \sqrt{g''(z_{\rm cr})}$ into the set $|\xi|<2\ep$.
It follows that as $\xi\to 0$ (and equivalently $z_{\rm cr} \to\infty$),
\eq\label{eq:b_over_sqrt}
\frac{b(\xi,0)}{\ii \sqrt{g''(z_{\rm cr}(\xi))}} = \xi z_{\rm cr}\bigl(1+\bigO(\xi)+\bigO(1/z_{\rm cr})\bigr),
\eeq
and therefore
\[
\lim_{\xi\to 0} \frac{b(\xi,0)}{\ii \sqrt{g''(z_{\rm cr}(\xi))}}=\lim_{\xi\to 0} \frac{b(\xi,0)M(z_{\rm cr})}{\ii \sqrt{g''(z_{\rm cr}(\xi))}} =\lim_{\xi\to 0} \psi(\xi) =1,
\]
and we find that $\psi(\xi)$ is analytic and nonvanishing in $|\xi| < \ep$.

We therefore may apply Lemma \ref{lem:small_xi} to the result of Proposition \ref{prop:le_1_annulus} to obtain the following strengthened version of Proposition \ref{prop:le_1_annulus}.
\begin{prop}\label{prop:le_1_disc}
Fix $0<\ep<\frac{\ga-|t|}{4}$ such that Proposition \ref{prop:crit_pt_ex} holds, and denote $z_{\rm cr} \equiv z_{\rm cr}(\xi)$ for all $\xi\in \C\setminus \{0\}$ satisfying $|\xi|\le 2\ep$. Let $\Psi(\xi)$ and $\psi(\xi)$ be defined by \eqref{def:Psipsi} for $\xi \ne 0$, and for $\xi=0$ as $\Psi(0) = \psi(0) = 1$. Then the asymptotic formula \eqref{eq:Psipsi} holds uniformly for $|\xi|\le 2\ep$.
\end{prop}

To complete the proof of Theorems \ref{Theorem_asymptotics_case2} -- \ref{Theorem_asymptotics_case4}, we now replace $\xi$ with $\xi/\sqrt{n}$ in the asymptotic formula \eqref{eq:Psipsi}. Since $z_{\rm cr} \to \infty$ as $\xi\to 0$, we can use the expansion \eqref{eq:gp_infinity} of $g'(z)$ at $\infty$ to find that $z_{\rm cr}(\xi/\sqrt{n})$ satisfies
\[
 \frac{1}{z} + \frac{\mu_1}{z^2} + \frac{\mu_2}{z^3} + \dots =  \frac{\xi}{\sqrt{n}},
 \]
and we find that
\eq\label{eq:crit_pt_exp}
z_{\rm cr}(\xi/\sqrt{n}) =\frac{\sqrt{n}}{\xi}+\mu_1+\frac{\xi}{\sqrt{n}}\left(\mu_2-\mu_1^2\right)+\bigO(1/n).
\eeq
From \eqref{eq:crit_pt_exp}, \eqref{eq:b_over_sqrt}, and \eqref{eq:Mstar_infinity} we immediately obtain
\eq\label{eq:psi_asy}
\psi(\xi/\sqrt{n}) = 1 + \bigO(1/\sqrt{n}).
\eeq
 To evaluate $\Psi(\xi/\sqrt{n})^n$, we use the expansion for $g(z)$ at infinity
 \eq\label{eq:g_infinity}
g(z) = \ln z - \frac{\mu_1}{z}-\frac{\mu_2}{2z^2}-\frac{\mu_3}{3z^3}+\bigO(1/z^4).
\eeq
to see
\[
-1-g(z_{\rm cr}(\xi/\sqrt{n})) + (\xi/\sqrt{n})z_{\rm cr}(\xi/\sqrt{n})) = \ln (\xi/\sqrt{n}) + \frac{\xi \mu_1}{\sqrt{n}} + \frac{\xi^2}{2n}(\mu_2 - \mu_1^2) + \bigO(n^{-3/2}).
\]
Then Proposition \ref{prop:le_1_disc} along with \eqref{eq:psi_asy} give that
\begin{multline}\label{eq:ZZ_Psi}
\ZZ(\xi/\sqrt{n}; t,\ga) = \Psi(\xi/\sqrt{n})^n \left(1 + \bigO(1/\sqrt{n})\right) \\
 =\left(\frac{a(t+\xi/\sqrt{n},\ga)b(t+\xi/\sqrt{n},\ga)\xi}{a(t,\ga)b(t,\ga)b(\xi/\sqrt{n},0)\sqrt{n}}\right)^n \exp\left(\xi \mu_1 \sqrt{n}+\frac{\xi^2}{2}(\mu_2 - \mu_1^2)\right)(1+\bigO(1/\sqrt{n})). \\
\end{multline}
Using  \eqref{eq:b_xi0}, we can write the above formula in a slightly different way:
\begin{multline}\label{eq:ZZ_Psi2}
\ZZ(\xi/\sqrt{n}; t,\ga) = \left(\frac{a(t+\xi/\sqrt{n},\ga)b(t+\xi/\sqrt{n},\ga)}{a(t,\ga)b(t,\ga)}\right)^n\left(\frac{\xi}{b(\xi/\sqrt{n},0)\sqrt{n}}\right)^{n-1} \\
\times \exp\left(\xi \mu_1 \sqrt{n}+\frac{\xi^2}{2}(\mu_2 - \mu_1^2)\right)(1+\bigO(1/\sqrt{n})),
\end{multline}
which is exactly the formula given in equations \eqref{eq_x10} and \eqref{eq_x11} with $k=1$. The rest of the proof of Theorems \ref{Theorem_asymptotics_case2} -- \ref{Theorem_asymptotics_case4} for $k=1$ is then identical to the one which follows those formulas, though we emphasize that \eqref{eq:ZZ_Psi2} was obtained using the rigourously proven asymptotics for orthogonal polynomials and steepest descent analysis of an integral, and does not rely on Conjecture \ref{eq_linear_statistics_remainder}.

\subsection{Existence of the contour $\Omega$}\label{app:Omega}
In this subsection we prove Propositions \ref{prop:Omega_F} and \ref{prop:Omega_DAF}. 
In order to describe the steep descent contour for the function $\Re (-g(z) +\xi z)$ uniformly in all phase regions, we introduce the notation
\eq
\mathcal{I} = \left\{
\begin{aligned}
& [\al, \be] \quad \textrm{for} \   \Delta < 1,\\
  &[\be, 0] \quad \textrm{for} \ \Delta > 1.
 \end{aligned}
 \right.
\eeq
to denote the intervals on which the function $g'(z)$ has a cut in each of the phase regions.

Locally near $z_{\rm cr}$, the level curve $\Re (-g(z) + \xi z) = \Re (-g(z_{\rm cr}) + \xi z_{\rm cr})$ consists of two smooth contours intersecting at $z_{\rm cr}$. Since $z_{\rm cr}$ is the unique critical point for $(-g(z) + \xi z)$, these curves cannot intersect at any other point in $\C\setminus \mathcal{I}$. For concreteness in this section, we assume that $\Re \xi\ge 0$ and $\Im \xi \ge 0$, in which case $\Im z_{\rm cr}\le 0$. The structure of $\Om$ for other values of $\xi$ is similar and given by straightforward symmetries. We will show the existence of a contour $\Omega$ for which we can apply the steepest descent analysis by examining the possible ways the four ends of these curves terminate in the following lemma.
\begin{lem}\label{lem_contour_Omega}
For six-vertex weights with $\De<1$ parametrized by \eqref{eq_case2} -- \eqref{eq_case4}, assume ${\ep<|\xi|<2\ep}$ for some fixed $\ep< \frac{\ga-|t|}{4}$. For six-vertex weights with $\De>1$ parametrized by \eqref{eq_case1}, assume $\Re \xi > (t-|\ga|)/2$, $|\Im \xi|<\pi/2$, and $|\xi| > \ep$ for a fixed $\ep>0$. In all cases, assume $\Re \xi \ge0$, and $\Im \xi\ge 0$.
\begin{enumerate}
\item
\begin{enumerate}
\item For $-1\le \De<1$, $\Re (-g(x) + \xi x)$ is strictly increasing in $x\in \R$ on $(-\infty, 0)$, decreases linearly on $(0, \be)$, and is convex on $(\be,\infty)$.
\label{lem:contour_Omega_a1}
\item For $\De<-1$, there exists $x_0 \in [\al',\be']$ such that $\Re (-g(x) + \xi x)$ is strictly increasing in $x\in \R$ on $(-\infty, x_0)$, strictly decreasing on $(x_0, \be)$, and convex on $(\be,\infty)$.
\label{lem:contour_Omega_a2}
\item For $\De>1$, $\Re (-g(x) + \xi x)$ is strictly increasing in $x\in \R$ on $(-\infty, \al)$, concave on $(\al,0)$, and convex on $(0,\infty)$.
\label{lem:contour_Omega_a3}
\end{enumerate}
\label{lem:contour_Omega_a}
\item Exactly one end of the level curve $\Re (-g(z) + \xi z) = \Re (-g(z_{\rm cr}) + \xi z_{\rm cr})$  approaches infinity at the angle $\arctan((\Re \xi)/(\Im \xi)) \in[0,\pi/2]$\footnote{If $\Im \xi = 0$, the level curve will end at $\infty$ at angle $\pi/2$, so we adopt the convention $\arctan(a/0) = \pi/2$ for $a>0$.}
, and exactly one end of the level curve $\Re (-g(z) + \xi_ z) = \Re (-g(z_{\rm cr}) + \xi_nz_{\rm cr})$ approaches infinity at the angle $\arctan((\Re \xi)/(\Im \xi)) -\pi\in[-\pi,-\pi/2]$. These are the only two ends of this level curve which approach infinity.\label{lem:contour_Omega_c}
\item If any part of the level curve $\Re (-g(z) + \xi z) = \Re (-g(z_{\rm cr}) +\xi z_{\rm cr})$ encloses a bounded region of the complex plane whose intersection with $\mathcal{I}$ is nonempty, then $\Re (-g(z) + \xi z) > \Re (-g(z_{\rm cr}) + \xi z_{\rm cr})$ for all $z$ within this region.\label{lem:contour_Omega_d}
\end{enumerate}
\end{lem}

This lemma is proved at the end of this subsection, and determines the structure of the level curve $\Re (-g(z) + \xi z) = \Re (-g(z_{\rm cr}) +\xi z_{\rm cr})$ precisely. Recall that we have assumed that $\Im z_{\rm cr}\le 0$, and let $\mathcal{C}$ be the component of the level curve $\Re (-g(z) + \xi z) = \Re (-g(z_{\rm cr}) + \xi z_{\rm cr})$ which contains $z_{\rm cr}$, and therefore consists of four curves emanating from $z_{\rm cr}$ which can end either at $\infty$ or on the cut $\mathcal{I}$. Part \eqref{lem:contour_Omega_c} of the above lemma implies that that at least two of the ends of $\mathcal{C}$ terminate at the cut, and we claim that in fact exactly two of the ends of $\mathcal{C}$ terminate on the cut $\mathcal{I}$ and the other two terminate  at $\infty$. Note that Part \ref{lem:contour_Omega_a} of Lemma  \ref{lem_contour_Omega} implies that any level curve for the function $\Re (-g(z) + \xi z) = \Re (-g(z_{\rm cr}) + \xi z_{\rm cr})$ can cross the real axis at most at 3 distinct points, of which at most 2 can be on the cut $\mathcal{I}$.

First consider the case $\xi \in \R$, for which $z_{\rm cr}\in \R$ as well. Then the conjugate symmetry $- g(\overline{z}) + \xi \overline{z} = \overline{ - g( z) + \xi  z}$ implies that either two or four of the ends terminate on the cut.  If all four ends of $\mathcal{C}$ go to the cut, then Part \ref{lem:contour_Omega_a} of the lemma implies they can do so only at two points, so two of them must go to the upper cusp of the cut. Then according to Part  \ref{lem:contour_Omega_c} of  Lemma \ref{lem_contour_Omega} there is a component of the level curve $\Re (-g(z) + \xi z) = \Re (-g(z_{\rm cr}) + \xi z_{\rm cr})$ which ends at infinity at the angle $\pi/2$. The other end of this component can only end on the cut $\mathcal{I}$ or at infinity with angle $-\pi/2$. Either way it must cross the real axis at a new point, giving three distinct real points such that $\Re (-g(x) + \xi x) = \Re (-g(z_{\rm cr}) +\xi z_{\rm cr})$, which cannot happen by Part \ref{lem:contour_Omega_a} of Lemma \ref{lem_contour_Omega}.

Now consider the case $\Im \xi >0$, so that $\Im z_{\rm cr}<0$. If all four ends of $\mathcal{C}$ go to the cut, then Part \ref{lem:contour_Omega_a} of the lemma implies they can do so only at two points, so two of them must go to the upper cusp of the cut. Since $\Im z_{\rm cr}< 0$, they would each need to cross the real axis outside of the cut. This gives four points of the real axis for which $\Re (-g(x) + \xi x) = \Re (-g(z_{\rm cr}) +\xi z_{\rm cr})$, which is impossible by Part \ref{lem:contour_Omega_a} of the lemma.

If three of the ends of $\mathcal{C}$ go to the cut (which we already established cannot happen for $\xi\in \R$), then two would go to the lower cusp and one to the upper cusp (since we just determined that two ends cannot go to the upper cusp). Part \ref{lem:contour_Omega_a}  of Lemma \ref{lem_contour_Omega} indicates that there are at most two points on $(-\infty, \be)$ where $\Re (-g(x) + \xi x) = \Re (-g(z_{\rm cr}) +\xi z_{\rm cr})$, so the end of $\mathcal{C}$ which goes to the upper cusp of the cut would necessarily cross the real axis to the right of the point $x=\be$. The fourth end would necessarily go to infinity at the angle $\arctan((\Re \xi)/(\Im \xi))-\pi \in (-\pi, -\pi/2)$ (if it went at the angle $\arctan((\Re \xi)/(\Im \xi)) \in (0, \pi/2)$ it would need to cross the real axis). Since there is always a component of the level curve $\Re (-g(x) + \xi x) = \Re (-g(z_{\rm cr}) +\xi z_{\rm cr})$ going to infinity at the angle $\arctan((\Re \xi)/(\Im \xi)) \in (0, \pi/2)$, this end coming from infinity must go to the upper cusp of the cut and must end at the same point of the cut as one of the ends of $\mathcal{C}$ on the lower cusp. The level curve $\Re (-g(x) + \xi x) = \Re (-g(z_{\rm cr}) +\xi z_{\rm cr})$ then divides the complex plane into 3 regions: a finite region which encloses part of the cut as well as two regions which contain infinity, see Figure \ref{fig:Omegas}. Taking into account the behavior of $\Re (-g(x+iy) + \xi (x+iy))$ as $x\to \infty$ or $y\to-\infty$ (recall that $g(z)\sim \ln(z)$ as $z\to\infty$ and both $\Re \xi$ and $\Im \xi$ are non-negative, with at least one of them positive), we find that the region which is to the right of or below the level curve must satisfy $\Re (-g(z) + \xi z) > \Re (-g(z_{\rm cr}) + \xi z_{\rm cr})$, and therefore the finite region enclosed by the level curve must satisfy $\Re (-g(z) + \xi z) < \Re (-g(z_{\rm cr}) + \xi z_{\rm cr})$, a scenario not allowed by part \ref{lem:contour_Omega_d} of Lemma  \ref{lem_contour_Omega}

  \begin{figure}[ht]
      \includegraphics[width=0.33\linewidth]{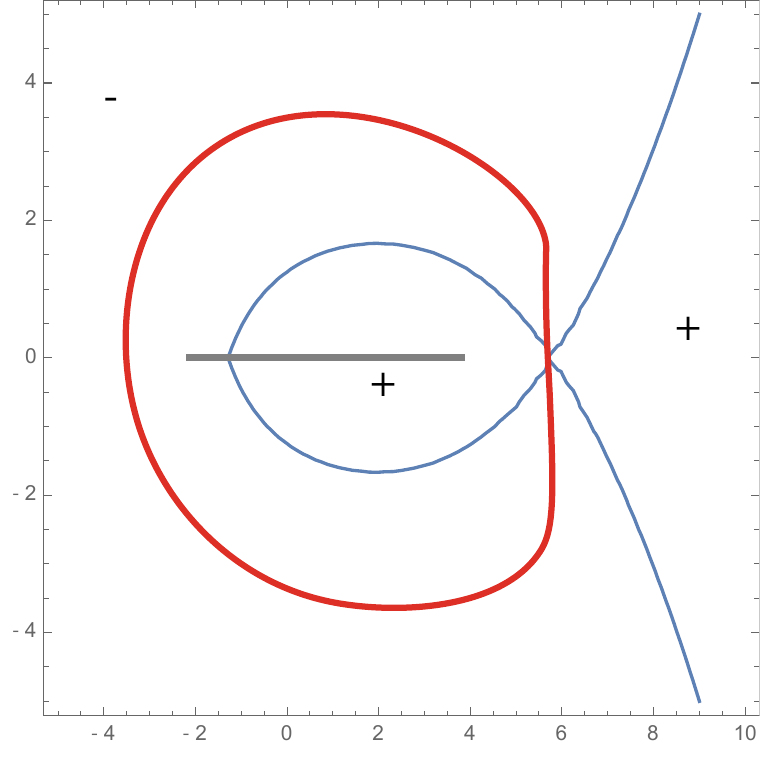}
\includegraphics[width=0.33\linewidth]{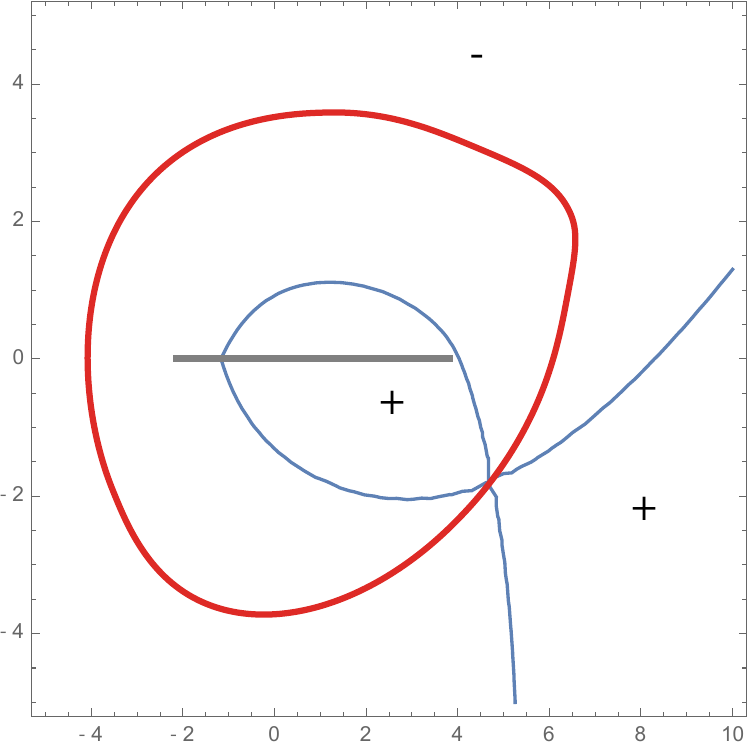}
\includegraphics[width=0.33\linewidth]{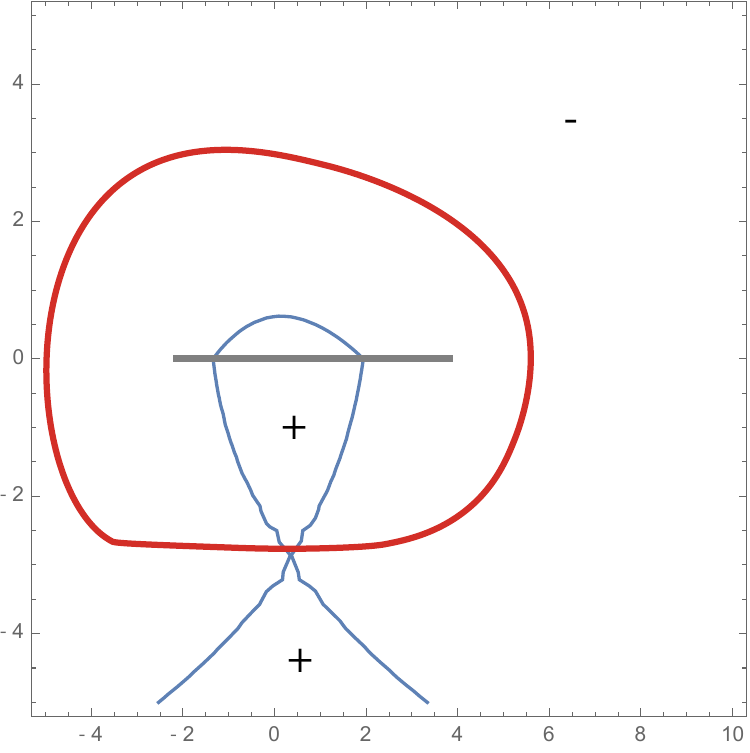}
\caption{\label{fig:Omegas} Plots of the level curve $\Re (-g(z) +\xi z) = \Re (-g(z_{\rm cr}) + \xi z_{\rm cr})$ are shown in blue for various values of $\xi$. Each plot uses six-vertex weights in the disordered phase with $t=0.2$, $\ga =1.1$. The different values of $\xi$ are, from left to right, $\xi = 0.2, 0.2+0.1\ii,$ and $0.3\ii.$ The $+$ and $-$ symbols denote the sign of $\Re (-g(z) +\xi z) - \Re (-g(z_{\rm cr}) + \xi z_{\rm cr})$ in each region separated by the level curve.
The thick grey line is the support of the equilibrium measure $[\al,\be]$, and a possible contour $\Om$ is shown in red.}
  \end{figure}

We therefore conclude that exactly two ends of $\mathcal{C}$ go to infinity asymptotic to a line with slope $(\Re \xi)/(\Im \xi) \ge 0$, and two of them go to the cut $\mathcal{I}$. The two pieces of $\mathcal{C}$ that go to infinity form a continuous path in $\C\setminus \mathcal{I}$ crossing $z_{\rm cr}$. Denote this path $\mathcal{C}_\infty$. Then $\{\mathcal{C}\setminus \mathcal{C}_\infty\} \cup \{z_{\rm cr}\}$ is a closed path which passes through the cut, and part \eqref{lem:contour_Omega_d} of the Lemma indicates that this closed path encloses a region for which $\Re (-g(z) + \xi) z) < \Re (-g(z_{\rm cr}) + \xi z_{\rm cr})$. As noted in the previous paragraph, the infinite region which is to the right of and/or below $\mathcal{C}_\infty$ must satisfy $\Re (-g(z) +\xi z) > \Re (-g(z_{\rm cr}) + \xi z_{\rm cr})$. These two facts force the interval $\mathcal{I}$ to lie to the left of and/or above $\mathcal{C}_\infty$.

We can now choose $\Omega$ to be a closed contour which stays to the left of and/or above $\mathcal{C}_\infty$, encircling $\{\mathcal{C}\setminus \mathcal{C}_\infty\}$ as well as the interval(s) $\mathcal{I}$, passing through the point $z_{\rm cr}$ in the direction of steepest descent. Such a contour will satisfy \eqref{eq:Omega_inequality}, \eqref{eq:Omega_inequality_DAF}.

We still need to prove Part \ref{prop:Omega_DAFc} of Propositions \ref{prop:Omega_F} and \ref{prop:Omega_DAF}. We describe the proof here for the case $\Delta<1$, i.e. we prove Part \eqref{prop:Omega_DAFc} of Proposition \ref{prop:Omega_DAF} in detail. Since $g(x)\sim \ln x$ for large $x$-values, the function $\Re (g(x) -V(x)+\xi x-l)$ is decreasing for large positive values of $x$, and increasing for large negative values of $x$. Its second derivative is $\frac{\dd^2}{\dd x^2} \Re (g(x) -V(x)+\xi x-l)  = \Re g''(x)$, which by  \eqref{eq:Gp} and \eqref{eq:GpAF} is negative for $x>\be$ and for $x<\al$. (recall $g'(z) = G_\nu(z)$).  In order to prove Part \ref{prop:Omega_DAFc} of Proposition \ref{prop:Omega_DAF}, we need only show that $\Re (g'(s) -V'(s)+\xi)< 0$ and $\Re (g'(-r) -V'(-r)+\xi)> 0$.

In the case $\Re \xi \ge 0$ described above, the contour $\mathcal{C}_\infty$ may only intersect $\R$ to the right of $\mathcal{I}$, and we are free to take $-r$, the point at which $\Om$ crosses the negative real axis, as large as we like. Since $\Re (g(x) -V(x)+\xi x-l)$ is increasing for large negative values of $x$, we can choose $r>0$ large enough that so that $\Re (g'(-r) -V'(-r)+\xi)> 0$. The value $s$ at which the contour $\Om$ crosses the positive real axis has the constraint that it must not lie to the right of the point at which $\mathcal{C}_\infty$ crosses $(\be,\infty)$. That is, it we denote by $\mathfrak{s}$ the minimal point at which $\mathcal{C}_\infty$ crosses $(\be,\infty)$, we must have $s<\mathfrak{s}.$  (If $\Re \xi = 0$, then $\mathcal{C}_\infty$ may not cross $\R$ at all, and we are free to choose $s$ as large as necessary so that $\Re (g'(s) -V'(s)+\xi)< 0$.) Since $s$ may be taken very close to $\mathfrak{s}$, showing that  $\Re (g'(\mathfrak{s}) -V'(\mathfrak{s})+\xi)< 0$ is sufficient to show that the same inequality holds for $s$.

Since $\mathfrak{s}$ is on the level curve $\Re (-g(z) + \xi z) = \Re (-g(z_{\rm cr}) + \xi z_{\rm cr})$ with the region $\Re (-g(z) + \xi z) > \Re (-g(z_{\rm cr}) + \xi z_{\rm cr})$ to the right of $\mathfrak{s}$, $\Re (-g(x) + \xi x)$ must be nondecreasing at $\mathfrak{s}$, i.e.,
\[
\Re (-g'(\mathfrak{s}) + \xi ) \ge 0,
\]
or equivalently
\eq
\Re (g'(\mathfrak{s})-V'(\mathfrak{s})+\xi )\le \Re(\xi-V'(\mathfrak{s})+\xi) = -(\ga - t) + 2\Re \xi < 0,
\eeq
under the assumption that $|\Re \xi| < (\ga - |t|)/2$. Taking $s$ slightly less than $\mathfrak{s}$ (or equal to $\mathfrak{s}$ if $\mathfrak{s} = z_{\rm cr}$), we have that $\Re (g(x) -V(x)+\xi x-l)$ is decreasing at $x=s$, and therefore decreasing on $(s,\infty)$ since it is concave on $(\al,\infty)$.

This proves Part \ref{prop:Omega_DAFc} of Proposition \ref{prop:Omega_DAF} when $\Re \xi \ge 0$. The case $\Re \xi < 0$ is similar, but the contour $\mathcal{C}_\infty$ crosses the real axis to the left of $\mathcal{I}$ rather than to the right. The proof of Part \eqref{prop:Omega_DAFc} of Proposition \ref{prop:Omega_F} is the same, but we only need to check that $\Re (g'(-r) -V'(-r)+\xi)> 0$ in the case $\Re \xi < 0$.

\medskip

{\bf {Proof of Lemma \ref{lem_contour_Omega}}}\begin{proof}

We first prove Parts \ref{lem:contour_Omega_a1} and \ref{lem:contour_Omega_a2} of Lemma \ref{lem_contour_Omega}. We will use the equilibrium condition \eqref{eq:eq_condition_band}, and it will be convenient to define $\al' = \be' \equiv 0$ for $-1\le \De<1$, while $\al'$ and $\be'$ are as defined in \eqref{eq_case3_alpha_beta} for $\De<-1$. Then \eqref{eq:eq_condition_band} states
\eq
\begin{aligned}
\Re (-g(x) + \xi x)&=-V(x)/2 + x\Re \xi - l/2  \\
&=\begin{cases}
x(-(\ga-t)/2 + \Re\xi) - l/2 \quad &\textrm{for} \ x\in [\al,\al'], \\
x((\ga+t)/2 +\Re\xi) - l/2 \quad &\textrm{for} \  x\in [\be',\be].\\
\end{cases}
\end{aligned}
\eeq
Since we have assumed $|\xi|< (\ga-|t|)/2$ (and so $|\Re \xi|<(\ga-|t|)/2$), this expression is increasing on the interval $[\al,\al']$ and decreasing on $[\be',\be]$. We claim that $\Re (-g(x) + \xi x)$ is increasing on $(-\infty, \al)$ as well. See see this, first note that for $x<\al$,
\eq
\frac{\dd^2}{\dd x^2} \Re (-g(x) +\xi x) = -\Re g''(x) = \begin{cases}
- \frac{1}{ x\sqrt{(\al-x)(\be-x)}} >0 \quad \textrm{for} \ -1\le \Delta < 1, \\
 \frac{1}{\sqrt{(\al-x)(\al'-x)(\be'-x)(\be-x)}} >0 \quad \textrm{for} \ \Delta < -1,
 \end{cases}
\eeq
where we have used \eqref{eq:Gp} and \eqref{eq:GpAF} (recall $g'(z) = G_\nu(z)$), so $\frac{\dd}{\dd x} \Re (-g(x) + \xi x)$ is increasing on $(-\infty, \al)$.
As $x\to -\infty$ its value is
\eq
\frac{\dd}{\dd x} \Re (-g(x) +\xi x) = \Re \xi-\frac{1}{x}+\bigO(1/|x|^2) > 0,
\eeq
where we have used the asymptotic formula \eqref{eq:gp_infinity} for $g'(z)$.
It follows that $\Re (-g(x) + \xi x)$ is strictly positive on
 on $(-\infty, \al)$. This proves Lemma \ref{lem_contour_Omega}\eqref{lem:contour_Omega_a} for the disordered phase $|\Delta|<1$ as well as the boundary phase $\De = -1$. For the antiferroelectric phase $\Delta <-1$, we additionally use the facts that $\Re (-g(x) + \xi x)$ is continuous on $\R$, and concave on the interval $(\al',\be')$. The latter fact follows from \eqref{eq:GpAF}, which gives that for $\al'<x<\be'$, $\frac{\dd^2}{\dd x^2} \Re (-g(x) + (\xi/\sqrt{n}) x)= -\Re g''(x)$ is negative. The statements  in Parts \ref{lem:contour_Omega_a1} and \ref{lem:contour_Omega_a2} that $\Re (-g(x) + \xi x)$ is convex on $(\be,\infty)$ follow immediately from \eqref{eq:Gp} and \eqref{eq:GpAF} as well.

The proof of Part \ref{lem:contour_Omega_a3} is similar. The equilibrium condition \eqref{eq:eq_condition_band} implies
\eq
\begin{aligned}
\Re (-g(x) + \xi  x)&=-V(x)/2 + x\Re \xi - l/2  \\
&=x((t-|\ga|)/2+\Re \xi) - l/2 \quad &\textrm{for} \ x\in \mathcal{I}, \\
\end{aligned}
\eeq
which is increasing in $x$, as we have assumed $\Re \xi > -(t-|\ga|)/2$. To see that $\Re (-g(x) + \xi  x)$ is increasing on $(-\infty,\be)$ as well, note that
\[
\frac{\dd^2}{\dd x^2} \Re (-g(x) + \xi  x)= -g''(x),
\]
which by \eqref{eq:Gp} is positive for $x<\be$, so $\frac{\dd}{\dd x} \Re (-g(x) + \xi  x)$ is increasing on $(-\infty, \be)$. As $x\to-\infty$ it behaves as
\[
 \Re (-g'(x) + \xi  x) = \xi - \frac{1}{x} + \bigO(1/x^2) >0.
 \]
Since $\frac{\dd}{\dd x} \Re (-g(x) + \xi  x)$ increases on $(-\infty, \be)$ from a positive value near $x=-\infty$, it is positive on the entire interval $(-\infty, \be)$, so $\Re (-g(x) + \xi  x)$ increases on that interval as well. The statement that $\Re (-g(x) + \xi x)$ is concave on $(\al,0)$ and convex on $(0,\infty)$ follows immediately from  \eqref{eq:Gp}.



To prove Part \ref{lem:contour_Omega_c} of Lemma \ref{lem_contour_Omega}, we can use the asymptotic expansion $g(z) \sim \ln z+ \bigO(1/z)$. Writing $z=x+\ii y$ and $\xi=a+\ii b$, this expansion  implies that if $\Re(-g(x+\ii y) + (a+\ii b) (x+\ii y)) =\Re (-g(z_{\rm cr}) + \xi z_{\rm cr})$ we must have
\eq\label{eq:abxy}
ax - by = \Re (-g(z_{\rm cr}) + \xi z_{\rm cr}) + \bigO(\ln |z|).
\eeq
For $z$ to approach infinity requires $|x|\to\infty$ or $y\to\infty$ (or both).  Fixing $n$ and taking $|x|\to\infty$, \eqref{eq:abxy} is
\[
a - by/x =  \bigO(\ln |x|)/|x|.
\]
Solving for the slope $y/x$ gives
\[
y/x = \frac{a}{b}+ \bigO(\ln |x|)/|x|.
\]
Similarly, taking $y\to\infty$ in \eqref{eq:abxy} gives
\[
\frac{y}{x} = \frac{a}{b}+ \bigO(\ln |y|)/|y|,
\]
proving that any part of the level curve $\Re (-g(z) + \xi z) =\Re (-g(z_{\rm cr}) +\xi z_{\rm cr})$ ending at infinity must be asymptotic to the slope $\frac{\Re \xi}{\Im \xi}$. Existence and uniqueness of such an end follows from the fact that $\Re (-g(z) + \xi z)$ approaches a linear function in $x$ as $y\to\infty$ or a linear function in $y$ as $x\to\infty$.

Finally we prove Part \ref{lem:contour_Omega_d} of the lemma. First assume $\De<1$.
Suppose there is a bounded region $R$ such that $\Re (-g(z) + \xi z) = \Re (-g(z_{\rm cr}) + \xi z_{\rm cr})$ for $z\in \partial R$, and such that $R \cap \mathcal{I} \ne \emptyset$. Suppose further that, contrary to the statement of the lemma, $\Re (-g(z) +\xi z) < \Re (-g(z_{\rm cr}) + \xi z_{\rm cr})$ for $z\in R$. The maximum principle indicates that the minimum value of $\Re (-g(z) + \xi z)$ in the region $R$ is achieved at some point $x_{\min}$ on the interval $R \cap \mathcal{I}$ which is not also on $\partial R$. According to Part \ref{lem:contour_Omega_a} of the lemma the continuous function $\Re (-g(x) +\xi x)$ has a unique maximum on $\mathcal{I}$ so its minimum value  on $R \cap \mathcal{I}$ can only occur at an endpoint of $R \cap \mathcal{I}$. This implies that $x_{\min}$ is either $\al$ or $\be$, since otherwise it would lie on $\partial R$.

If $x_{\min} = \be$, then $\be+\ep\in R$ for some $\ep>0$ as well, since $x_{\min}$ cannot be on $\partial R$. But in fact, since as a function of a real variable, $g'(x)$ is continuous at $x=\be$, $\Re (-g(x) + \xi x)$ is decreasing at $x=\be$ so $\Re (-g(\be + \ep) + \xi (\be +\ep)) < \Re (-g(\be) + \xi (\be))$ for small enough $\ep>0$, contradicting the implication of the maximum principle that the minimum value of $\Re (-g(z) + \xi z)$ in the region $R$ is achieved on $\mathcal{I}$.
If $x_{\min} = \al$, a similar argument shows that $\Re (-g(\al -\ep) + \xi (\al -\ep)) < \Re (-g(\al) +\xi (\al))$, again contradicting the implication of the maximum principle that the minimum value of $\Re (-g(z) + \xi z)$ in the region $R$ is achieved on $\mathcal{I}$.

The proof of Part \ref{lem:contour_Omega_d} of Lemma \ref{lem_contour_Omega} in the ferroelectric phase $\De>1$ is similar, except that Part \eqref{lem:contour_Omega_a} of the Lemma does not state that as a function of a real variable $x\in \R$, $-g(x) + \xi x$ is decreasing at $x=0$. In fact, from \eqref{eq_case1_G}, one sees that $\frac{\dd}{\dd x} (-g(x) + \xi x)$ has a logarithmic singularity at $x=0$, and approaches $-\infty$ as $x\to 0$. The singularity is integrable, so $-g(x) + \xi x$ is still continuous at $x=0$, and is clearly decreasing there since its derivative is negative in a neighborhood of $x=0$. The rest of the proof of Part \ref{lem:contour_Omega_d} of the Lemma for $\De>1$ is identical to the $\De<1$ case, with the left and right endpoints of $\mathcal{I}$ being $\be$ and $0$ instead of $\al$ and $\be$.
\end{proof}

\section{Reduction of general $k$ to $k=1$ in Theorems \ref{Theorem_asymptotics_case1} -- \ref{Theorem_asymptotics_case4}}

\label{Section_general_reduction}

In this section we explain how Theorem \ref{Theorem_partition_general_k_reduction} can be used to reduce general $k$ in Theorems \ref{Theorem_asymptotics_case1}--\ref{Theorem_asymptotics_case4} to $k=1$ case. Thus, we finish the proofs of these theorems.

\begin{proof}[Proof of Theorem \ref{Theorem_asymptotics_case1}]
 We plug $k=1$ asymptotics (established in Section \ref{Section_k1_asymptotics}) into \eqref{eq_general_k_reduction}:
\begin{align}
\notag  \ZZ_n&(\xi_1, \xi_2, \dots, \xi_k;\, t,\gamma )\\\notag &= \frac{\det \left[  \left(\frac{\sinh(t+|\ga|+\xi_j)}{\sinh(t+|\ga|)} \right)^{n-k+i} e^{-\xi_j} \left(\frac{a(t+\xi_j, \ga) b(t+\xi_j, \ga)}{a(t, \ga) b(t, \ga)}\right)^{k-i}  \left(b(\xi_j,0)\right)^{i-1} \left(1+\bigO\left(\tfrac{1}{\sqrt{n}}\right)\right)\right]_{i,j=1}^k}{\displaystyle \prod\limits_{1\le i<j\le k} b(\xi_j-\xi_i,0) }
  \\\notag &=  \prod_{j=1}^k\left(\frac{\sinh(t+|\ga|+\xi_j)}{\sinh(t+|\ga|)} \right)^n \, \exp\left(-\sum_{j=1}^k \xi_j\right)
    \\ &\quad \times \frac{\det \left[  \left(\frac{\sinh(t+|\ga|+\xi_j)}{\sinh(t+|\ga|)} \right)^{i-k} \left(\frac{a(t+\xi_j, \ga) b(t+\xi_j, \ga)}{a(t, \ga) b(t, \ga)}\right)^{k-i}  \left(b(\xi_j,0)\right)^{i-1} \left(1+\bigO\left(\tfrac{1}{\sqrt{n}}\right)\right)\right]_{i,j=1}^k}{\displaystyle \prod\limits_{1\le i<j\le k} b(\xi_j-\xi_i,0) }. \label{eq_x8}
\end{align}
The proof of Theorem \ref{Theorem_asymptotics_case1} would be complete, if we show that the last line is $ \left(1+\bigO\left(\tfrac{1}{\sqrt{n}}\right)\right)$ as $n\to\infty$. Using the definitions of the weights \eqref{eq_case1}, we rewrite the matrix entry under determinant as
\begin{align*}
  \left(\frac{\sinh(t-|\ga|+\xi_j)}{\sinh(t-|\ga|)} \right)^{k-i}   \sinh(\xi_j)^{i-1} \left(1+\bigO\left(\tfrac{1}{\sqrt{n}}\right)\right)=&
   \left(\frac{\sinh(t-|\ga|+\xi_j)}{\sinh(t-|\ga|)} \right)^{k-1} \\ &\times \left(\frac{ \sinh(\xi_j)\sinh(t-|\ga|) }{\sinh(t-|\ga|+\xi_j)}\right)^{i-1}  \left(1+\bigO\left(\tfrac{1}{\sqrt{n}}\right)\right).
\end{align*}
Hence, the determinant in the last line of \eqref{eq_x8} evaluates as the Vandermonde determinant. We conclude that the last line of \eqref{eq_x8} is
\begin{multline*}
 \prod_{j=1}^k   \left(\frac{\sinh(t-|\ga|+\xi_j)}{\sinh(t-|\ga|)} \right)^{k-1} \prod_{1\le i<j \le k} \frac{\frac{ \sinh(\xi_j)\sinh(t-|\ga|) }{\sinh(t-|\ga|+\xi_j)}-\frac{ \sinh(\xi_i)\sinh(t-|\ga|) }{\sinh(t-|\ga|+\xi_i)}}{ \sinh(\xi_j-\xi_i)} \left(1+\bigO\left(\tfrac{1}{\sqrt{n}}\right)\right)
 \\= \prod_{1\le i<j \le k} \frac{\sinh(\xi_j) \sinh(t-|\ga|+\xi_i)-\sinh(\xi_i)\sinh(t-|\ga|+\xi_j)}{ \sinh(t-|\ga|) \sinh(\xi_j-\xi_i)} \left(1+\bigO\left(\tfrac{1}{\sqrt{n}}\right)\right).
\end{multline*}
It remains to notice that the numerator and the denominator in the last formula are the same.
\end{proof}

\begin{remark} An additional care is required if $\xi_i=\xi_j$, because the denominator $b(\xi_j-\xi_i,0)$ vanishes in \eqref{eq_x8}. One way to resolve this is by expressing the $\xi_i=\xi_j$ situation through $\xi_i\ne \xi_j$ by the Cauchy integral formula. Same applies in the next proof.
\end{remark}

\begin{proof}[Proof of Theorem \ref{Theorem_asymptotics_case2}]
We write the $k=1$ asymptotics (established in Section \ref{Section_k1_asymptotics}) in the form
\begin{align}  \notag
\widetilde Z_n\left(\frac{\xi}{\sqrt{n}};\, t,\gamma\right) =  \exp\Biggl[ \sqrt{n} \,  \xi \theta_1 + \xi^2 \theta_2+\bigO\left(\frac{1}{\sqrt{n}}\right) \Biggr], \qquad n\to\infty,
\end{align}
with constants $\theta_1$ and $\theta_2$ not depending on $n$, plug into \eqref{eq_general_k_reduction}, and use $\sqrt{n-k+i}=\sqrt{n}+\bigO\left(\tfrac{1}{\sqrt{n}}\right)$ to get
\begin{align}
 \notag \ZZ_n&\left(\frac{\xi_1}{\sqrt{n}}, \frac{\xi_2}{\sqrt{n}},\dots, \frac{\xi_k}{\sqrt{n}};\, t,\gamma \right)
  \\\notag &= \frac{\displaystyle\det \left[   \exp\Biggl[ \sqrt{n} \,  \xi_j \theta_1 + (\xi_j)^2 \theta_2+\bigO\left(\frac{1}{\sqrt{n}}\right) \Biggr] \left(\frac{a\left(t+\tfrac{\xi_j}{\sqrt{n}}, \ga\right) b\left(t+\tfrac{\xi_j}{\sqrt{n}}, \ga\right)}{a(t, \ga) b(t, \ga)}\right)^{k-i}  \left(b\left(\tfrac{\xi_j}{\sqrt{n}},0\right)\right)^{i-1} \right]_{i,j=1}^k}{\displaystyle \prod\limits_{1\le i<j\le k} b\left(\tfrac{\xi_j-\xi_i}{\sqrt{n}},0\right) }
  \\&=  \exp\Biggl[ \sqrt{n} \,  \sum_{j=1}^k \xi_j \theta_1 + \sum_{j=1}^k(\xi_j)^2 \theta_2\Biggr]
   \frac{\displaystyle\det \left[     \left(b\left(\tfrac{\xi_j}{\sqrt{n}},0\right)\right)^{i-1} \left(1+\bigO\left(\tfrac{1}{\sqrt{n}}\right)\right) \right]_{i,j=1}^k}{\displaystyle \prod\limits_{1\le i<j\le k} b\left(\tfrac{\xi_j-\xi_i}{\sqrt{n}},0\right) } \label{eq_x9}.
 \end{align}
The definition of the weight $b$ in \eqref{eq_case2} readily implies that $b\left(\tfrac{\xi}{\sqrt{n}};0\right)=\frac{\xi}{\sqrt{n}} \left(1+\bigO\left(\tfrac{1}{\sqrt{n}}\right)\right)$. Hence, the numerator in the last line of \eqref{eq_x9} evaluates as the Vandermonde determinant and cancels with denominator. We conclude that
$$
  \ZZ_n\left(\frac{\xi_1}{\sqrt{n}}, \frac{\xi_2}{\sqrt{n}},\dots, \frac{\xi_k}{\sqrt{n}};\, t,\gamma \right)= \exp\Biggl[ \sqrt{n} \,  \sum_{j=1}^k \xi_j \theta_1 + \sum_{j=1}^k(\xi_j)^2 \theta_2+\bigO\left(\tfrac{1}{\sqrt{n}}\right)\Biggr].\qedhere
$$
\end{proof}

The reduction of general $k$ cases in Theorems \ref{Theorem_asymptotics_case3} and \ref{Theorem_asymptotics_case4} is exactly the same as for Theorem \ref{Theorem_asymptotics_case2} and is omitted.

\section{Proofs of Theorems \ref{Theorem_GUE_corners} and \ref{Theorem_convergence_to_stochastic}}

\label{Section_probabilistic_proofs}

The goal of this section is to prove Theorems \ref{Theorem_GUE_corners} and \ref{Theorem_convergence_to_stochastic}. The two inputs are the limits of inhomogeneous partition functions of Theorems \ref{Theorem_asymptotics_case1} -- \ref{Theorem_asymptotics_case4} and preservation of the Gibbs property of the six-vertex model in limit transitions.

\subsection{Random variables and generating functions}

\label{Section_F_functions}

Our main results, Theorems \ref{Theorem_GUE_corners} and \ref{Theorem_convergence_to_stochastic} are stated in terms of the random variables $\lambda_k^i$ of Definition \ref{Definition_monotone_triangle}. Let us connect the distribution of these random variables to the inhomogeneous partition functions $\ZZ_n(\xi_1,\dots,\xi_k;\, t,\gamma)$ of Definition \ref{Definition_inhom_norm_part}.

\smallskip

For that we introduce for each $n=1,2,\dots$, $0<k\le n$, each of the four cases \eqref{eq_case1}--\eqref{eq_case4}, and each $k$-tuple $\nu$ of positive integers, $0<\nu_1<\nu_2<\dots<\nu_k\le n$, a function $F^{\text{sym}}_\nu(\xi_1,\dots,\xi_k;\, t,\gamma,n)$, which is a function\footnote{While we are not going to use it directly, $F^{\text{sym}}_\nu$ is a \emph{symmetric} function of its $k$ arguments $\xi_1,\dots,\xi_k$, as can be shown using the Yang-Baxter relation, see e.g.\ \cite[Section 5.2]{Bleher-Liechty14} or \cite[Section 2]{Gorin_Nicoletti_lectures} and references therein for arguments of this type.}   of $k$ complex variable $\xi_1$,\dots, $\xi_k$.

\begin{figure}[t]
\begin{center}
\includegraphics[width=0.45\linewidth]{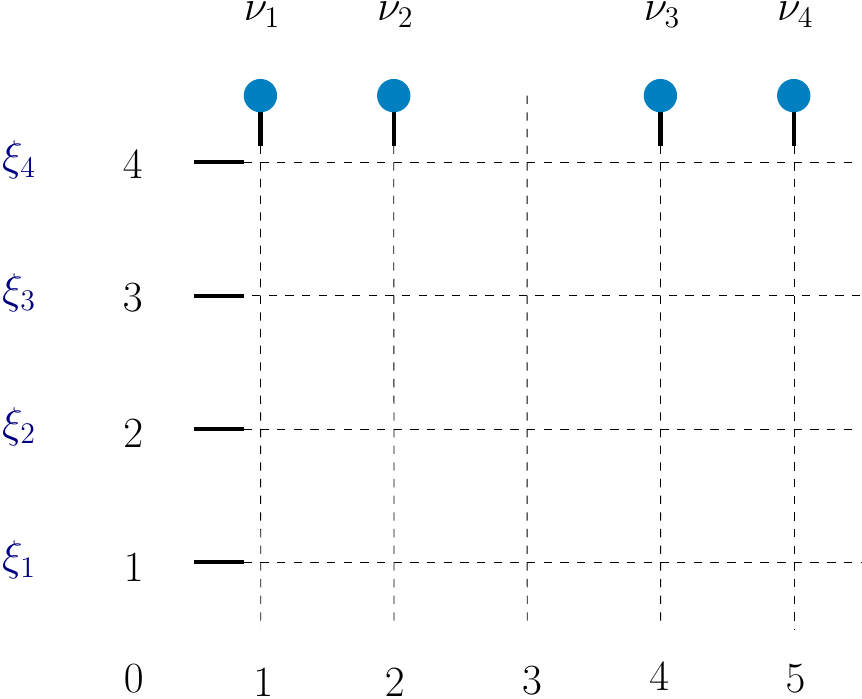} \hfill
\includegraphics[width=0.45\linewidth]{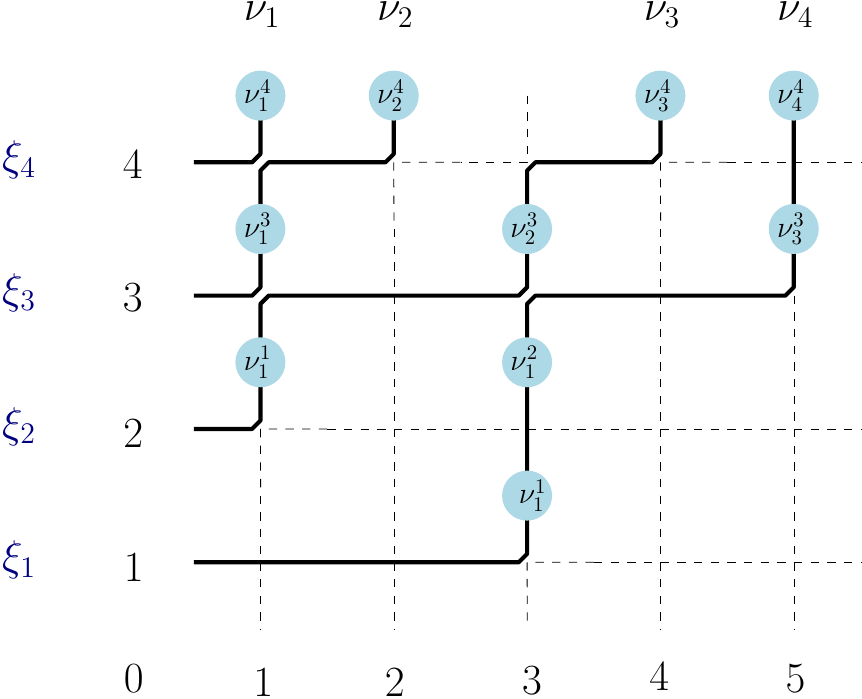}
 \caption{Left panel: $F^{\text{sym}}_\nu(\xi_1,\dots,\xi_4;\, t,\gamma,n)$ is a partition function for this domain with $n=5$, $\nu=(1,2,4,5)$. Right panel: one configuration in the partition function sum and corresponding monotone triangle $(\nu_i^k)_{1\le i\le k \le 4}$: $(3; 1<3; 1<3<5; 1<2<4<5)$.\label{Fig_domain_for_symfunctions}}
\end{center}
\end{figure}

The function is obtained as a partition function of the six-vertex model in a rectangular domain $\{1,2,\dots,n\}\times \{1,2,\dots,k\}$ with $k$ paths entering on the left at positions $1,2,\dots,k$ and exiting on top at positions $\nu_1,\nu_2,\dots,\nu_k$, as in Figure \ref{Fig_domain_for_symfunctions}. The vertex weights depend on the vertical coordinate, as in Definition \ref{Definition_inhomogeneities}, using the formulas \eqref{eq_case1}--\eqref{eq_case4}:
 $$
  \omega(x,y;\sigma)=
  \begin{cases}
    a(t+\xi_y, \ga), &\quad\text{if }\sigma(x,y)\text{ is of Type 1 or 2,}\\
    b(t+\xi_y, \ga), &\quad\text{if }\sigma(x,y)\text{ is of Type 3 or 4,}\\    c(\ga), &\quad\text{if }\sigma(x,y)\text{ is of Type 5 or 6.}
  \end{cases}
 $$
We define
\begin{equation}
\label{eq_rectangular_function}
 F^{\text{sym}}_\nu(\xi_1,\dots,\xi_k;\, t,\gamma,n):=\sum_{\sigma} \prod_{x=1}^n \prod_{y=1}^k \omega(x,y;\sigma),
\end{equation}
where the summation goes over all the configurations of the six-vertex model in the rectangle agreeing with the just prescribed boundary conditions. Let us emphasize that \eqref{eq_rectangular_function} depends on the choice of one of the four cases \eqref{eq_case1}--\eqref{eq_case4} and, therefore, we have defined four different functions here. The superscript $\text{sym}$ indicates that we use symmetric weights in \eqref{eq_rectangular_function}. If $k=n$ and ${\nu=(1<2<\dots<n)}$, then $F^{\text{sym}}_\nu$ is the partition function for DWBC, as before, so that:
\begin{equation}
\label{eq_x16}
 \ZZ_n(\xi_1,\dots,\xi_n;\, t,\gamma)=\frac{F^{\text{sym}}_\nu(\xi_1,\dots,\xi_n;\, t,\gamma,n)}{F^{\text{sym}}_\nu(0^n;\, t,\gamma,n)}, \qquad \nu=(1,2,\dots,n).
\end{equation}
We can rewrite $ F^{\text{sym}}_\nu$ as a sum over monotone triangles. As in Definition \ref{Definition_monotone_triangle} and Figure \ref{Fig_domain_for_symfunctions}, we identify each configuration $\sigma$ in \eqref{eq_rectangular_function} with a collection of positive integers $\{\nu_i^j\}_{1\le i \le j \le k}$, such that:
\begin{itemize}
 \item $\nu^k_i=\nu_i$, for each $i=1,2,\dots,k$;
 \item $\nu^j_1<\nu^j_2<\dots<\nu^j_j$, for each $j=1,2,\dots,k$;
 \item $\nu^{j+1}_i\le \nu^j_i\le \nu^{j+1}_i$, for each $1\le i \le j <k$.
\end{itemize}
Formally, $\nu_i^j$ is the position of the $i$--th vertical path segment going from vertices on level $y=j$ to those on level $y=j+1$.
In terms of such $\{\nu_i^j\}_{1\le i \le j \le k}$, we have
\begin{multline}
\label{eq_rectangular_function_2}
  F^{\text{sym}}_\nu(\xi_1,\dots,\xi_k;\, t,\gamma,n):=\prod_{j=1}^k \Bigl[a(t+\xi_j,\gamma)^{n-2j+1}  c(\gamma)^{2j-1}  \Bigr]\\ \times  \sum_{\nu_i^j} \prod_{j=1}^k \Biggl[ \left(\frac{b(t+\xi_j,\gamma)}{a(t+\xi_j,\gamma)}\right)^{\sum\limits_{i=1}^j \nu_i^j - \sum\limits_{i'=1}^{j-1} \nu_{i'}^{j-1}-j}
 \cdot  \prod_{i=1}^j  \left(\frac{b^2(t+\xi_j,\gamma)}{c^2(\gamma)}\right)^{\mathbf 1(\nu_i^j=\nu_{i-1}^{j-1})} \left(\frac{a^2(t+\xi_j,\gamma)}{c^2(\gamma)}\right)^{\mathbf 1(\nu_i^j=\nu_{i}^{j-1})}  \Biggr].
\end{multline}
The simplest way to see that \eqref{eq_rectangular_function_2} is equivalent to \eqref{eq_rectangular_function} is by first looking at the ``generic'' situation when all indicators $\mathbf 1(\nu_i^j=\nu_{i}^{j-1})$ and $\mathbf 1(\nu_i^j=\nu_{i-1}^{j-1})$ vanish (which implies having the maximal possible number, $2j-1$, of type $c$ vertices in row $j$), and then checking that the corrections to these situation are precisely described by $ \prod_{i=1}^j $ factor in the second line of \eqref{eq_rectangular_function_2}.
\begin{remark}
\label{Remark_F_n_indepent}
 One corollary of the form \eqref{eq_rectangular_function_2} is that $ F^{\text{sym}}_\nu(\xi_1,\dots,\xi_k;\, t,\gamma,n) \prod\limits_{j=1}^k a(t+\xi_j,\gamma)^{-n}$ is stable, which means that it does not depend on $n$ as long as $n\ge \nu_k$.
\end{remark}

Here is a generalization of the identity \eqref{eq_x16} to arbitary $k\le n$:
\begin{prop} \label{Proposition_inhom_as_generating}
 For each $k\le n$, each of the four cases \eqref{eq_case1}--\eqref{eq_case4}, and $k$-tuple of complex numbers $\xi_1,\dots,\xi_k$, we have using the monotone triangle of Definition \ref{Definition_monotone_triangle}:
\begin{equation}\label{eq_inhom_as_generating}
 \ZZ_n(\xi_1,\dots,\xi_k;\, t,\gamma)=\sum_{\nu=(\nu_1<\dots<\nu_k)} \mathrm{Prob}\bigl( (\lambda_1^k,\dots,\lambda_k^k)=\nu\bigr)\, \frac{ F^{\text{sym}}_\nu(\xi_1,\dots,\xi_k;\, t,\gamma,n)}{ F^{\text{sym}}_\nu(0^k;\, t,\gamma,n)}.
\end{equation}
\end{prop}
\begin{proof}
 Plugging Definition \ref{Definition_inhom_norm_part} of $\ZZ_n$, \eqref{eq_inhom_as_generating} is equivalent to
\begin{multline}\label{eq_inhom_as_generating_2}
  \mathcal Z_n(0^n;\, t+\xi_1,\dots,t+\xi_k, t^{n-k};\, \gamma)\\=\sum_{\nu=(\nu_1<\dots<\nu_k)} \mathrm{Prob}\bigl( (\lambda_1^k,\dots,\lambda_k^k)=\nu\bigr)\, \mathcal Z_n(0^n;\, t^n;\, \gamma)\, \frac{ F^{\text{sym}}_\nu(\xi_1,\dots,\xi_k;\, t,\gamma,n)}{ F^{\text{sym}}_\nu(0^k;\, t,\gamma,n)}.
\end{multline}
Noting that $\mathrm{Prob}\bigl( (\lambda_1^k,\dots,\lambda_k^k)=\nu\bigr)\, \mathcal Z_n(0^n;\, t^n;\, \gamma)$ is a product of partition functions of the six-vertex model in two rectangles, corresponding to rows $1,\dots,k$ and rows $k+1,\dots,n$, we see that
\eqref{eq_inhom_as_generating_2} is merely a rewriting of the definition of $\mathcal Z_n$, in which we combined together all the configurations satisfying $ (\lambda_1^k,\dots,\lambda_k^k)=\nu$.
\end{proof}

In words, \eqref{eq_inhom_as_generating} says that $\ZZ_n(\xi_1,\dots,\xi_k;\, t,\gamma)$ is a generating function for the distribution of $(\lambda_1^k,\dots,\lambda_k^k)$, written in terms of the (normalized) functions $F^{\text{sym}}_\nu$. In principle, this generating function contains complete information about the distribution of $(\lambda_1^k,\dots,\lambda_k^k)$. In the next two subsections we demonstrate how this information can be extracted in the asymptotic regimes of Theorems \ref{Theorem_GUE_corners} and \ref{Theorem_convergence_to_stochastic}. Note that in our asymptotic theorems $k$ is fixed. It is an important open problem to figure out how to extract asymptotic information from generating functions in the regime of $k$ growing with $n$, cf.\ \cite{bufetov2019fourier} for a such theory in the situation when $F^{\text{sym}}_\nu$ are replaced by the Schur symmetric functions, related to the $\Delta=0$ case of the six-vertex model.

\subsection{Proof of Theorem \ref{Theorem_GUE_corners}}
\label{Section_GUE_proof}

We start by detailing the constants in Theorem \ref{Theorem_GUE_corners}:
\begin{itemize}
\item {\bf Disordered phase} $|\Delta|<1$, parameterization \eqref{eq_case2}:
\begin{align}
 \label{eq_case2_constants}
 {\mathfrak m}(a,b,c)&=\frac{\cot(\ga+t)+\frac{\pi}{2\ga}\tan\left(\frac{\pi t}{2\ga}\right)}{\cot(\gamma-t)+\cot(\gamma+t)}, \\ {\mathfrak s}^2(a,b,c)&=\frac{-\frac{2}{3}+ \frac{\pi^2}{6\ga^2}
   -\left(\cot(\ga+t)+\frac{\pi}{2\ga}\tan\left(\frac{\pi t}{2\ga}\right) \right)\left( \cot(\gamma-t)-\frac{\pi}{2\ga}\tan\left(\frac{\pi t}{2\ga}\right)\right) }{(\cot(\gamma-t)+\cot(\gamma+t))^2}.\notag
\end{align}
\item {\bf Antiferroelectric phase} $\Delta<-1$, parameterization \eqref{eq_case3}:

\begin{align}
 \label{eq_case3_constants}
  {\mathfrak m}(a,b,c)&=\frac{\coth(\ga+t)
-\frac{\pi}{2\ga}\frac{\vartheta_2'\left(\frac{\pi t}{2\ga}\right)}{\vartheta_2\left(\frac{\pi t}{2\ga}\right)}}{ \coth(\gamma-t)+\coth(\gamma+t)},
  \\{\mathfrak s}^2(a,b,c)&= \frac{\frac{2}{3}- \frac{\pi^2}{12\ga^2}\left(\frac{\vartheta_2'\left(\frac{\pi t}{2\ga}\right)}{\vartheta_2\left(\frac{\pi t}{2\ga}\right)}\right)^2 + \frac{\pi^2}{12\ga^2} \sum\limits_{\ell=1}^4 \left(\frac{\vartheta_\ell'\left(\pi \frac{t+\gamma}{4\gamma}\right)}{\vartheta_\ell\left(\pi \frac{t+\gamma}{4\gamma}\right)}\right)^2
  }{(\coth(\gamma-t)+\coth(\gamma+t))^2} \notag
 \\ & \quad +
  \frac{
\frac{\pi}{2\ga}\frac{\vartheta_2'\left(\frac{\pi t}{2\ga}\right)}{\vartheta_2\left(\frac{\pi t}{2\ga}\right)} \bigl(\coth(\gamma-t)-\coth(\gamma+t)\bigr)-\coth(\ga+t)\coth(\ga-t)}{(\coth(\gamma-t)+\coth(\gamma+t))^2}\notag
. \notag
\end{align}
\item {\bf Boundary phase} $\Delta=-1$, parameterization \eqref{eq_case4}:
\begin{align}
 \label{eq_case4_constants}
 {\mathfrak m}(a,b,c)&=\frac{\ga-t}{2\ga}+ \pi \frac{\ga^2-t^2}{4\ga^2}\tan\left(\frac{\pi t}{2\ga}\right), \\ {\mathfrak s}^2(a,b,c)&=\frac{ \frac{\pi^2}{6\ga^2}
   -\left(\frac{1}{\ga+t}+\frac{\pi}{2\ga}\tan\left(\frac{\pi t}{2\ga}\right) \right)\left( \frac{1}{\gamma-t}-\frac{\pi}{2\ga}\tan\left(\frac{\pi t}{2\ga}\right)\right) }{\left(\frac{1}{\gamma-t}+\frac{1}{\gamma+t}\right)^2}.\notag
\end{align}
\end{itemize}
Next, we state and prove Theorem \ref{Theorem_GUE_corners} for $k=1$.
\begin{prop} \label{Proposition_Gaussian_limit}
 Choose parameters $a,b,c>0$ such that $\Delta=\frac{a^2+b^2-c^2}{2ab}<1$. Let $(\lambda_i^k)$ be a random monotone triangle corresponding to $(a,b,c)$--random configuration of the six-vertex model with DWBC, as in Definition \ref{Definition_monotone_triangle}. We have
 $$
  \lim_{n\to\infty} \left( \frac{\lambda_1^1- {\mathfrak m}(a,b,c)\cdot n}{{\mathfrak s}(a,b,c) \sqrt{n}} \right) \stackrel{d}{=} \mathcal N(0,1).
 $$
\end{prop}
\begin{proof}
 At $k=1$, the function $F^{\text{sym}}_\nu$ of the previous section simplifies. Using \eqref{eq_rectangular_function_2}, we have
\begin{equation}
    F^{\text{sym}}_{\nu_1}(\xi_1;\, t,\gamma,n)=a(t+\xi_1,\gamma)^{n-1}  c(\gamma) \left(\frac{b(t+\xi_1,\gamma)}{a(t+\xi_1,\gamma)}\right)^{\nu_1-1}.
\end{equation}
Hence, the result of Proposition \ref{Proposition_inhom_as_generating} is restated as
\begin{equation}
\label{eq_x17}
 \ZZ_n\left(\frac{\xi_1}{\sqrt{n}};\, t,\gamma\right)= \frac{a\left(t+\frac{\xi_1}{\sqrt{n}},\gamma\right)^{n-1}  }{a(t,\gamma)^{n-1}  }\, \E \left[ \left(\frac{b\left(t+\frac{\xi_1}{\sqrt{n}},\gamma\right)a(t,\gamma)}{b(t,\gamma) a\left(t+\frac{\xi_1}{\sqrt{n}},\gamma\right)}\right)^{\lambda_1^1-1}\right].
\end{equation}
We write the random variables $\lambda_1^1$ as
$$
 \lambda_1^1=n \, {\mathfrak m}+ \sqrt{n}\, {\mathfrak s}\eta, \quad \text{ where }\qquad   {\mathfrak m}={\mathfrak m}(a,b,c), \quad {\mathfrak s}={\mathfrak s}(a,b,c),
$$
for a random variable $\eta$ and plug into \eqref{eq_x17} asymptotic expansions of all factors as $n\to\infty$. For $\ZZ_n\left(\frac{\xi_1}{\sqrt{n}};\, t,\gamma\right)$ we directly use Theorems \ref{Theorem_asymptotics_case2}--\ref{Theorem_asymptotics_case4}. In the disordered phase \eqref{eq_case2}, using the same computations as in the conditional proof of Theorem \ref{Theorem_asymptotics_case2} in Section \ref{Section_asymptotics_through_log_gas}, we further have
\begin{align*}
  \frac{a\left(t+\frac{\xi_1}{\sqrt{n}},\gamma\right)^{n-1}}{a\left(t,\gamma\right)^{n-1}} &= \left(\frac{\sin\left(\gamma-t-\frac{\xi_1}{\sqrt{n}}\right)}{\sin(\gamma-t)}\right)^{n-1}=\left(1-\frac{\xi_1}{\sqrt{n}} \cot(\gamma-t)- \frac{(\xi_1)^2}{2n}+o(n^{-1})\right)^{n-1}
  \\&=\exp\left(-\sqrt{n}\, \xi_1 \cot(\gamma-t) - \frac{(\xi_1)^2}{2} \left( 1+ \cot^2(\gamma-t)  \right)+o(1)\right)
\end{align*}
\begin{multline*}
  \left(\frac{b\left(t+\frac{\xi_1}{\sqrt{n}},\gamma\right)a(t,\gamma)}{b(t,\gamma) a\left(t+\frac{\xi_1}{\sqrt{n}},\gamma\right)}\right)^{\lambda_1^1}= \exp\Biggl( {\mathfrak m} \sqrt{n} \xi_1 (\cot(\gamma-t)+\cot(\gamma+t)) \\ \qquad\qquad\qquad\qquad + {\mathfrak m} \frac{(\xi_1)^2}{2} \left(\cot^2(\gamma-t)-\cot^2(\gamma+t)  \right)  +o(1)  \\
  +{\mathfrak s} \eta\left( \xi_1(\cot(\gamma-t)+\cot(\gamma+t)) + \frac{(\xi_1)^2}{2\sqrt{n}} \left(\cot^2(\gamma-t)-\cot^2(\gamma+t)\right)+o(n^{-1/2})  \right)
  \Biggr),
\end{multline*}
where all $o(\cdot)$ terms are real and deterministic. Combining the expansions, we transform \eqref{eq_x17} into
\begin{align*}
 \E &\exp\Biggl(  {\mathfrak s} \eta\left( \xi_1(\cot(\gamma-t)+\cot(\gamma+t)) + \frac{(\xi_1)^2}{2\sqrt{n}} \left(\cot^2(\gamma-t)-\cot^2(\gamma+t)\right)+o(n^{-1/2})  \right)
  \Biggr)\\&=\exp\Biggl (\sqrt{n} \xi_1\left(\cot(\ga+t)+\frac{\pi}{2\ga}\tan\left(\frac{\pi t}{2\ga}\right)-{\mathfrak m} \cdot \bigl(\cot(\gamma-t)+\cot(\gamma+t)\bigr) \right)
  \\&\quad\quad\quad -  \frac{\xi_1^2}{2}\left(\frac{2}{3}- \frac{\pi^2}{6\ga^2}-\frac{\pi^2\tan^2\left(\frac{\pi t}{2\ga}\right)}{4\ga^2}+\cot^2(\ga+t)
  + {\mathfrak m} \left(\cot^2(\gamma-t)-\cot^2(\gamma+t)\right)\right)+o(1)\Biggr).
\end{align*}
Note that $\mathfrak m$ was chosen in \eqref{eq_case2_constants}, in such a way that $\sqrt{n}$ terms in the second line of the last formula cancel out. Thus, changing the variable $\xi_1= \frac{\xi}{{\mathfrak s} (\cot(\gamma-t)+\cot(\gamma+t))}$, and recalling the definition of $\mathfrak s$ from \eqref{eq_case2_constants}, we finally get
\begin{align}
 \label{eq_x18}&\lim_{n\to\infty} \E \exp\Biggl(  \eta\left( \xi  + \frac{\xi^2}{2{\mathfrak s}  \sqrt{n}} \cdot \frac{\cot(\gamma-t)-\cot(\gamma+t)}{\cot(\gamma-t)+\cot(\gamma+t)}+o(n^{-1/2})  \right)
  \Biggr)\\&=\exp\Biggl ( -  \frac{\xi^2}{2 {\mathfrak s}^2} \cdot \frac{\frac{2}{3}- \frac{\pi^2}{6\ga^2}-\frac{\pi^2\tan^2\left(\frac{\pi t}{2\ga}\right)}{4\ga^2}+\cot^2(\ga+t)
  + {\mathfrak m} \left(\cot^2(\gamma-t)-\cot^2(\gamma+t)\right)}{(\cot(\gamma-t)+\cot(\gamma+t))^2}\Biggr)=\exp\left(\frac{\xi^2}{2}\right).\notag
\end{align}
This asymptotic identity means that the Laplace transform of $\eta$ converges as $n\to\infty$ to that of standard normal $\mathcal N(0,1)$. Hence, $\eta$ converges to $\mathcal N(0,1)$ in distribution.

In the antiferroelectric phase \eqref{eq_case3}, we similarly have
\begin{align*}
  \frac{a\left(t+\frac{\xi_1}{\sqrt{n}},\gamma\right)^{n-1}}{a\left(t,\gamma\right)^{n-1}}
  =\exp\left(-\sqrt{n}\, \xi_1 \coth(\gamma-t) - \frac{(\xi_1)^2}{2} \left(-1+ \coth^2(\gamma-t)  \right)+o(1)\right),
\end{align*}
\begin{multline*}
  \left(\frac{b\left(t+\frac{\xi_1}{\sqrt{n}},\gamma\right)a(t,\gamma)}{b(t,\gamma) a\left(t+\frac{\xi_1}{\sqrt{n}},\gamma\right)}\right)^{\lambda_1^1}= \exp\Biggl( {\mathfrak m} \sqrt{n} \xi_1 (\coth(\gamma-t)+\coth(\gamma+t)) \\ \qquad\qquad\qquad\qquad + {\mathfrak m} \frac{(\xi_1)^2}{2} \left(\coth^2(\gamma-t)-\coth^2(\gamma+t)  \right)  +o(1)  \\
  +{\mathfrak s} \eta\left( \xi_1(\coth(\gamma-t)+\coth(\gamma+t)) + \frac{(\xi_1)^2}{2\sqrt{n}} \left(\coth^2(\gamma-t)-\coth^2(\gamma+t)\right)+o(n^{-1/2})  \right)
  \Biggr),
\end{multline*}
Combining the expansions with \eqref{eq_Zn_expansion_case3}, we transform \eqref{eq_x17} into
\begin{multline*}
 \E \exp\Biggl(  {\mathfrak s} \eta\left( \xi_1(\coth(\gamma-t)+\coth(\gamma+t)) + \frac{(\xi_1)^2}{2\sqrt{n}} \left(\coth^2(\gamma-t)-\coth^2(\gamma+t)\right)+o(n^{-1/2})  \right)
  \Biggr)\\=\exp\Biggl[\sqrt{n} \xi_1\left(\coth(\ga+t)
-\frac{\pi}{2\ga}\frac{\vartheta_2'\left(\frac{\pi t}{2\ga}\right)}{\vartheta_2\left(\frac{\pi t}{2\ga}\right)} -{\mathfrak m} \cdot \bigl(\coth(\gamma-t)+\coth(\gamma+t)\bigr) \right)
  \\   \frac{\xi_1^2}{2}\Biggl(\frac{2}{3}- \frac{\pi^2}{12\ga^2}\left(\frac{\vartheta_2'\left(\frac{\pi t}{2\ga}\right)}{\vartheta_2\left(\frac{\pi t}{2\ga}\right)}\right)^2 + \frac{\pi^2}{12\ga^2} \sum_{\ell=1}^4 \left(\frac{\vartheta_\ell'\left(\pi \frac{t+\gamma}{4\gamma}\right)}{\vartheta_\ell\left(\pi \frac{t+\gamma}{4\gamma}\right)}\right)^2-\coth^2(\ga+t)
  \\- {\mathfrak m} \bigl(\coth^2(\gamma-t)-\coth^2(\gamma+t)\bigr)\Biggr)+o(1)\Biggr].
\end{multline*}
Comparing with the definition of $\mathfrak m$ in \eqref{eq_case3_constants}, we see that the second line vanishes. Thus, changing the variable $\xi_1= \frac{\xi}{{\mathfrak s} (\coth(\gamma-t)+\coth(\gamma+t))}$, and recalling the definition of $\mathfrak s$ from \eqref{eq_case3_constants}, we finally get
\begin{multline*}
 \lim_{n\to\infty}\E \exp\Biggl(  \eta\left( \xi  + \frac{\xi^2}{2{\mathfrak s}  \sqrt{n}} \cdot \frac{\coth(\gamma-t)-\coth(\gamma+t)}{\coth(\gamma-t)+\coth(\gamma+t)}+o(n^{-1/2})  \right)
  \Biggr)=\exp\Biggl[ \frac{\xi^2}{2 {\mathfrak s}^2}\\ \times \frac{\frac{2}{3}- \frac{\pi^2}{12\ga^2}\left(\frac{\vartheta_2'\left(\frac{\pi t}{2\ga}\right)}{\vartheta_2\left(\frac{\pi t}{2\ga}\right)}\right)^2 + \frac{\pi^2}{12\ga^2} \sum\limits_{\ell=1}^4 \left(\frac{\vartheta_\ell'\left(\pi \frac{t+\gamma}{4\gamma}\right)}{\vartheta_\ell\left(\pi \frac{t+\gamma}{4\gamma}\right)}\right)^2-\coth^2(\ga+t)
  - {\mathfrak m} \left(\coth^2(\gamma-t)-\coth^2(\gamma+t)\right)}{(\coth(\gamma-t)+\coth(\gamma+t))^2}\Biggr]\\=\exp\left(\frac{\xi^2}{2}\right).
\end{multline*}
We again conclude that $\eta$ converges to $\mathcal N(0,1)$ in distribution. In the boundary case \eqref{eq_case4} the computation is similar and we omit it.
\end{proof}

For general $k>1$ the proof of Theorem \ref{Theorem_GUE_corners} has the same structure as for $k=1$ case and the arithmetics of the rescaling constants is exactly the same. The new feature is that the left-hand side of the analogue of \eqref{eq_x18} (and similarly in $\Delta\le -1$ cases) is no-longer a simple one-dimensional Laplace transform of the desired random variable; instead it becomes a multidimensional version, which we refer to as the \emph{Bessel generating function}.

\begin{definition}
 Fix $k\ge 1$ and let $\mathbf x=(x_1\le x_2\le \dots\le x_k)$ be a random vector. Its \emph{Bessel generating function} (or BGF) is a deterministic symmetric function of $k$ complex variables $z_1,\dots,z_k$ given in terms of Bessel functions \eqref{eq_Bessel_1} by
 \begin{equation}
 \label{eq_BGF}
  \mathcal G( z_1,\dots, z_k;\, \mathbf x)=\E_{\mathbf x}\bigl[ \B_{x_1,\dots,x_k}(z_1,\dots,z_k)\bigr].
 \end{equation}
\end{definition}

\smallskip

\noindent When $k=1$, the BGF \eqref{eq_BGF} turns into the usual Laplace transform. For general $k$ its properties are parallel to those of the Laplace transform. The relation to GUE eigenvalues is based on the following:

\begin{prop} \label{Proposition_GUE_BGF}
 Suppose that a random vector $\mathbf x$ is such that
 \begin{equation}
   \mathcal G( z_1,\dots, z_k;\, \mathbf x)= \exp\left(\frac{1}{2}\sum_{j=1}^k (z_j)^2 \right)
 \end{equation}
  for all $z_1,\dots,z_k$ in an open real neighborhood of $(0,\dots,0)$. Then $\mathbf x$ is equal in distribution to $k$ GUE eigenvalues, i.e.\ to $k$-dimensional vector $(g_1^k,g_2^k,\dots,g_k^k)$ with $g_i^k$ from Definition \ref{Definition_GUE_corners}.
\end{prop}
\begin{proof} Consider $k\times k$ Hermitian matrix $X$, whose eigenvalues are given by $\mathbf x$, and whose eigenvectors are independent of $\mathbf x$ and chosen uniformly at random. Such a matrix can be constructed as
$$
 X= U^* \mathrm{diag}(\mathbf x) U,
$$
where $U$ is a uniformly random (i.e.\ Haar-distributed) $k\times k$ unitary matrix independent of $\mathbf x$ and $\mathrm{diag}(\mathbf x)$ is a diagonal matrix with eigenvalues given by the coordinates of $\mathbf x$. The statement of Proposition \ref{Proposition_GUE_BGF} would follow if we prove that the distribution of $X$ is that of GUE random matrix. For that we compute the Laplace transform of $X$:
\begin{equation}
\label{eq_x30}
 f(A)=\E \exp\bigl( \mathrm{Trace} (XA) \bigr),
\end{equation}
where $A$ is an arbitrary $k\times k$ (deterministic) Hermitian matrix of variables. Notice that due to invariance of the law of $X$ with respect to unitary conjugations, we can diagonalize $A$ in \eqref{eq_x30} by such conjugations. Hence, it is sufficient to compute  \eqref{eq_x30} with $A$ replaced by a diagonal matrix consisting of eigenvalues of $A$, which we denote $a_1,\dots,a_k$. Then, using \eqref{eq_Bessel_2} (in which we moved $U^*$ under the trace to the left) we have
\begin{multline*}
 f(A)=\E_{U,\mathbf x} \exp\bigl( \mathrm{Trace} (U^* \mathrm{diag}(\mathbf x) U \mathrm{diag}(a_1,\dots,a_k) ) \bigr)=\E_{\mathbf x} \B_{\mathbf x}(a_1,\dots,a_k)\\=\mathcal G(a_1,\dots,a_k)=\exp\left(\frac{1}{2}\sum_{j=1}^k (a_j)^2 \right).
\end{multline*}
On the other hand, if $Y$ is a $k\times k$ GUE matrix, then its law is also invariant under unitary conjugations, hence, the Laplace transform of $Y$ again depends only on the eigenvalues of the matrix of variables $A$. If $A$ is diagonal, then the Laplace transform of $Y$ evaluated on such a matrix matches the Laplace transform of diagonal elements of $Y$. They are i.i.d.\ $\mathcal N(0,1)$, and, therefore, the desired Laplace transform is again $\exp\left(\frac{1}{2}\sum_{j=1}^k (a_j)^2 \right)$. We conclude that the Laplace transforms of $X$ and $Y$ coincide (in a small neighborhood of $0$); therefore, by the uniqueness theorem for the Laplace transforms, the distribution of $X$ matches that of GUE.
\end{proof}

Multivariate Bessel functions appear naturally in our problem as a scaling limit of the functions $F^{\text{sym}}_\nu$ of Section \ref{Section_F_functions}. In order to see that, we  need to analyze the sum in the definition \eqref{eq_rectangular_function_2} of $F^{\text{sym}}_\nu$. The asymptotic analysis is done by comparing the following two probability measures.

\begin{definition}
\label{Def_abc_measure}
 Fix $a,b,c>0$ and a $k$-tuple of integers $\nu_1<\nu_2<\dots<\nu_k$. We define a random $\frac{k(k+1)}{2}$--tuple of integers $\Upsilon^{\nu; a,b,c}=(\upsilon_{i}^j)_{1\le i \le j \le k}$, satisfying the constraints:
\begin{itemize}
 \item $\upsilon^k_i=\nu_i$, for each $i=1,2,\dots,k$;
 \item $\upsilon^j_1<\upsilon^j_2<\dots<\upsilon^j_j$, for each $j=1,2,\dots,k$;
 \item $\upsilon^{j+1}_i\le \upsilon^j_i\le \upsilon^{j+1}_i$, for each $1\le i \le j <k$.
\end{itemize}
The probability weights of $\Upsilon^{\nu; a,b,c}$ are proportional to\footnote{The formula \eqref{eq_abc_measure} is obtained from \eqref{eq_rectangular_function_2} by setting $\xi_1=\dots=\xi_k=0$ and omitting the factors which depend only on $\nu_i$, but not on $\nu_i^j$ with $j<k$.}:
\begin{equation}
\label{eq_abc_measure}
 \mathrm{Prob} \Bigl( (\upsilon_i^j)_{1\le i\le j \le k}=(\nu_i^j)_{1\le i \le j \le k}\Bigr) \sim
 \prod_{j=1}^k \prod_{i=1}^j \Biggl[  \left(\frac{b^2}{c^2}\right)^{\mathbf 1(\nu_i^j=\nu_{i-1}^{j-1})} \left(\frac{a^2}{c^2}\right)^{\mathbf 1(\nu_i^j=\nu_{i}^{j-1})}  \Biggr].
\end{equation}
\end{definition}
\begin{definition}
\label{Def_continuous_measure}
 Fix a $k$-tuple of reals $ \nu_1\le \nu_2\le \dots\le \nu_k$. We define a random $\frac{k(k+1)}{2}$--tuple of reals $\widetilde \Upsilon^{ \nu}=(\widetilde \upsilon_{i}^j)_{1\le i \le j \le k}$, as a \emph{uniformly random} array of real numbers satisfying the constraints:
\begin{itemize}
 \item $\widetilde \upsilon^k_i=\nu_i$, for each $i=1,2,\dots,k$;
 \item $\widetilde \upsilon^{j+1}_i\le \widetilde \upsilon^j_i\le \widetilde \upsilon^{j+1}_i$, for each $1\le i \le j <k$.
\end{itemize}
\end{definition}

\medskip

There are two differences between Definitions \ref{Def_abc_measure} and \ref{Def_continuous_measure}: first, the former deals with discrete random variables, while the latter deals with continuous ones; second, the former has non-trivial weights \eqref{eq_abc_measure}, while for the latter the weights (densities) are constant. However, these differences disappear asymptotically for large $\nu$.

\begin{prop} \label{Proposition_Gibbs_approximation}
 Fix $a,b,c>0$, $k=1,2,\dots$, and let $\mathfrak F$ denote the space of all $1$--Lipschitz real functions on $\mathbb R^{k(k+1)/2}$. Then for each $C>0$, we have
 \begin{equation} \label{eq_Wesserstein distance}
  \lim_{\eps\to 0}\, \sup_{\begin{smallmatrix} \nu=(\nu_1<\dots<\nu_k)\in\mathbb Z^k\\ |\nu_i|\le \eps^{-1} C\end{smallmatrix}}\, \sup_{f\in \mathfrak F} \left[\E f\bigl(\eps \Upsilon^{\nu; a,b,c}\bigr)- \E f\bigl(\eps \widetilde \Upsilon^{\nu}\bigr)  \right].
 \end{equation}
\end{prop}
\begin{remark} Equivalently, \eqref{eq_Wesserstein distance} says that the Wesserstein distance between the distributions of random vectors $\eps \Upsilon^{\nu; a,b,c}$ and $\eps \widetilde \Upsilon^{\nu}$ tends to $0$ as $\eps\to 0$, uniformly over $\nu$ with bounded $\eps^{-1} \nu_i$.
\end{remark}
\begin{proof} We argue by contradiction. If \eqref{eq_Wesserstein distance} fails, then there exists a sequence $\nu^{(l)}$, $l=1,2,\dots$, of $k$-tuples of integers, a sequence $\eps^{(l)}$ tending to $0$ and a sequence of $1$--Lipschitz functions $f^{(l)}$, such that $\E f^{(l)}\bigl(\eps^{(l)} \Upsilon^{\nu^{(l)}; a,b,c}\bigr)- \E f^{(l)}\bigl(\eps^{(l)} \widetilde \Upsilon^{\nu^{(l)}}\bigr)$ is bounded away from $0$ as $l\to\infty$.

Suppose first that
\begin{equation}
\label{eq_x19}
 \lim_{l\to\infty} \min_{i=1,\dots,k-1} \left|\nu^{(l)}_{i+1}-\nu^{(l)}_i\right| = +\infty.
\end{equation}
In this case we notice that for large $l$ the equalities $\upsilon_i^j=\upsilon_{i-1}^{j-1}$ or $\upsilon_i^j=\upsilon_{i-1}^{j-1}$ hold only on a vanishing fraction of possible configurations in the support of the distribution of $\Upsilon^{\nu; a,b,c}$. Hence, for most configurations the factors in \eqref{eq_abc_measure} are all equal to $1$, and the distribution of $\Upsilon^{\nu; a,b,c}$ is close to being uniform over all $\nu_i^j$ satisfying the interlacement constraints spelled in Definition \ref{Def_abc_measure}. Therefore, up to small error $\E f^{(l)}\bigl(\eps^{(l)} \Upsilon^{\nu^{(l)}; a,b,c}\bigr)$ is a Riemann sum (with step $\eps^{(l)}$) for the integral computed by $\E f^{(l)}\bigl(\eps^{(l)} \widetilde \Upsilon^{\nu^{(l)}}\bigr)$. Because $f^{(l)}$ is $1$--Lipschitz, the Riemenn sum should be approximating the integral as $\eps^{(l)}\to 0$, contradicting that $\E f^{(l)}\bigl(\eps^{(l)} \Upsilon^{\nu^{(l)}; a,b,c}\bigr)- \E f^{(l)}\bigl(\eps^{(l)} \widetilde \Upsilon^{\nu^{(l)}}\bigr)$ is bounded away from $0$.

\smallskip

If \eqref{eq_x19} fails, then some of the differences $\nu^{(l)}_{i+1}-\nu^{(l)}_i$ should not be converging to infinity. In order to shorten the notations, we only study the case when the first difference $ \nu^{(l)}_{2}-\nu^{(l)}_1$ does not converge to infinity, while all others differences do; all other cases are very similar. Passing to the subsequence, if necessary, we can assume that the first difference is constant, i.e.\ we now deal with the situation
\begin{equation}
\label{eq_x20}
 \lim_{l\to\infty} \min_{i=2,\dots,k-1} |\nu^{(l)}_{i+1}-\nu^{(l)}_i| = +\infty\quad \text{ and }\quad \nu^{(l)}_{2}-\nu^{(l)}_1= d\text{ for all }l=1,2,\dots.
\end{equation}
Then $\upsilon_1^{k-1}$ can take $d+1$ different values, $\nu^{(l)}_1,\dots,\nu^{(l)}_2$; we condition on each of these $d+1$ cases and then use the same approximation argument as for the case \eqref{eq_x19}. Note that the dependence on the value of $\upsilon_1^{k-1}$ is negligible, because $d\eps_l$ tends to $0$ as $l\to \infty$. Conditionally on $\upsilon_1^{k-1}$, $\E \left[ f^{(l)}\bigl(\eps^{(l)} \Upsilon^{\nu^{(l)}; a,b,c}\bigr)\mid \upsilon_1^{k-1}=\nu_1^{k-1}\right]- \E f^{(l)}\bigl(\eps^{(l)} \widetilde \Upsilon^{\nu^{(l)}}\bigr)$ is again small as $\eps_l\to 0$, because this is an approximation error for Riemann integral via Riemann sum. Summing over all choices of $\nu_1^{k-1}$ with corresponding probabilities we again reach a contradiction with $\E f^{(l)}\bigl(\eps^{(l)} \Upsilon^{\nu^{(l)}; a,b,c}\bigr)- \E f^{(l)}\bigl(\eps^{(l)} \widetilde \Upsilon^{\nu^{(l)}}\bigr)$ being bounded away from $0$.
\end{proof}

\begin{prop} \label{Proposition_F_to_Bessel}
 Fix one of the four cases \eqref{eq_case1}--\eqref{eq_case4}, corresponding $t$ and $\gamma$, an integer $k=1,2,\dots$, and a compact set $\mathfrak A$ in $\mathbb R^k$. Let $\eps>0$ be a small parameter. Denote\footnote{Using Remark \ref{Remark_F_n_indepent}, we use $n$-independence of \eqref{eq_F_normalized} and therefore omit $n$ from the notations.}
 \begin{equation}
 \label{eq_F_normalized}
  \widetilde F_\nu(\xi_1,\dots,\xi_k)=\frac{F^{\text{sym}}_\nu(\xi_1,\dots,\xi_k;\, t,\gamma,n) \prod_{j=1}^k a(t+\xi_j,\gamma)^{-n}}{F^{\text{sym}}_\nu(0,\dots,0;\, t,\gamma,n) a(t,\gamma)^{-kn}}.
 \end{equation}
 Then for each $\nu=(\nu_1<\dots<\nu_k)$, each $M\in\mathbb R$, and each $(\xi_1,\dots,\xi_k)\in \mathfrak A$,  we have as $\ep\to 0$,
 \begin{equation}
 \label{eq_F_to_Bessel}
  \widetilde F_\nu\left(\eps \xi_1,\dots,\eps\xi_k\right)=\left[ \prod_{j=1}^k \left(\frac{b\left(t+\eps \xi_j,\gamma\right)a(t,\gamma)}{b(t,\gamma) a\left(t+\eps \xi_j,\gamma\right)}\right)^{M}\right] \B_{\sigma \eps(\nu-M)}(\xi_1,\dots,\xi_k) \cdot \bigl(1+ o(1) \bigr),
 \end{equation}
 where
 \begin{equation}
 \label{eq_sigma_formula}
  \sigma=\begin{cases} \cot(\gamma-t)+\cot(\gamma+t), &  |\Delta|<1,\\
                       \coth(\gamma-t)+\coth(\gamma+t), & |\Delta|>1,\\
                       \frac{1}{\gamma-t}-\frac{1}{\gamma+t}, &\Delta=-1.
  \end{cases}
 \end{equation}
 The $o(1)$ term tends to $0$ as $\eps \to 0$, uniformly in $(\xi_1,\dots,\xi_k)\in \mathfrak A$, and over $\nu$, $M$, such that $\sigma \eps(\nu-M)$ belongs to a compact subset of the closed Weyl chamber $\mathbb W_k=\{x_1\le x_2\le \dots\le x_k\}$.
\end{prop}
\begin{remark}
 While we are not focusing on $\Delta=1$ case in this paper, it should be possible to obtain an analogue of Proposition \ref{Proposition_F_to_Bessel} for it. 
\end{remark}
\begin{proof}[Proof of Proposition \ref{Proposition_F_to_Bessel}]
 We use the formula \eqref{eq_rectangular_function_2} and aim to prove \eqref{eq_F_to_Bessel}. We have
 \begin{multline}
\label{eq_rectangular_function_3}
  F^{\text{sym}}_\nu(\eps \xi_1,\dots,\eps \xi_k;\, t,\gamma,n) \prod_{j=1}^k a(t+\eps \xi_j,\gamma)^{-n} =\prod_{j=1}^k \left[\frac{b(t+\eps \xi_j,\gamma)^{M}}{a(t+\eps \xi_j,\gamma)^{M+2j-1}}  c(\gamma)^{2j-1}  \right]\\ \times  \sum_{\nu_i^j} \prod_{j=1}^k \Biggl[ \left(\frac{b(t+\eps \xi_j,\gamma)}{a(t+\eps \xi_j,\gamma)}\right)^{\sum\limits_{i=1}^j (\nu_i^j-M-\frac{j+1}{2}) - \sum\limits_{i'=1}^{j-1} (\nu_{i'}^{j-1}-M-\frac{j}{2})}
\\ \times \prod_{i=1}^j  \left(\frac{b^2(t+\xi_j,\gamma)}{c^2(\gamma)}\right)^{\mathbf 1(\nu_i^j=\nu_{i-1}^{j-1})} \left(\frac{a^2(t+\xi_j,\gamma)}{c^2(\gamma)}\right)^{\mathbf 1(\nu_i^j=\nu_{i}^{j-1})} \Biggr].
\end{multline}
Choosing $a=a(t,\gamma)$, $b=b(t,\gamma)$, $c=c(\gamma)$ and using Definition \ref{Def_abc_measure}, we can further write
 \begin{multline}
\label{eq_rectangular_function_expectation}
  \frac{F^{\text{sym}}_\nu(\eps \xi_1,\dots,\eps \xi_k;\, t,\gamma,n) \prod_{j=1}^k a(t+\eps \xi_j,\gamma)^{-n}}
   {F^{\text{sym}}_\nu(0,\dots,0;\, t,\gamma,n) \prod_{j=1}^k a(t+\eps \xi_j,\gamma)^{-n}}
   =\prod_{j=1}^k \left[\frac{b(t+\eps \xi_j,\gamma)^{M} a(t,\gamma)^{M+2j-1}}{b(t,\gamma)^M a(t+\eps \xi_j,\gamma)^{M+2j-1}  }    \right]\\ \times  \E_{\Upsilon^{\nu-M;a,b,c,}}  \Biggl[ \prod_{j=1}^k\Biggl( \left(\frac{b(t+\eps \xi_j,\gamma)a(t,\gamma)}{b(t,\gamma) a(t+\eps \xi_j,\gamma)}\right)^{\sum\limits_{i=1}^j (\upsilon_i^j-\frac{j+1}{2}) - \sum\limits_{i'=1}^{j-1} (\upsilon_{i'}^{j-1}-\frac{j}{2})}
\\ \times \prod_{i=1}^j  \left(\frac{b^2(t+\eps \xi_j,\gamma) }{b^2(t,\gamma) }\right)^{\mathbf 1(\nu_i^j=\nu_{i-1}^{j-1})} \left(\frac{a^2(t+\eps \xi_j,\gamma)}{a^2(t,\gamma)}\right)^{\mathbf 1(\nu_i^j=\nu_{i}^{j-1})} \Biggr)\Biggr].
\end{multline}
Note that (for real $\xi_1,\dots,\xi_k$ and small-enough $\eps$) the expression under expectation is positive and the factors in the last line are $1+o(1)$. Therefore, they can be omitted. Further, using the same asymptotic expansion as in the proof of Proposition \ref{Proposition_Gaussian_limit}, we note
\begin{equation}
\label{eq_x22}
  \frac{b(t+\eps \xi_j,\gamma)a(t,\gamma)}{b(t,\gamma) a(t+\eps \xi_j,\gamma)}=\exp(\eps \sigma + o(\eps)).
\end{equation}
Hence, \eqref{eq_rectangular_function_expectation} as $\eps\to 0$ becomes
 \begin{equation}
\label{eq_rectangular_function_expectation_2}
(1+o(1))\prod_{j=1}^k \left[\frac{b(t+\eps \xi_j,\gamma)^{M} a(t,\gamma)^{M}}{b(t,\gamma)^M a(t+\eps \xi_j,\gamma)^{M}  }    \right] \E_{\Upsilon^{\nu-M;a,b,c}}   \exp\left( \sum_{j=1}^k \eps \sigma \xi_j \left(\sum\limits_{i=1}^j \upsilon_i^j - \sum\limits_{i'=1}^{j-1} \upsilon_{i'}^{j-1}\right) \right).
\end{equation}
Using Proposition \ref{Proposition_Gibbs_approximation}, we replace $\Upsilon^{\nu-M;a,b,c,}$ with $\widetilde \Upsilon^{\nu-M}$, transforming \eqref{eq_rectangular_function_2} into
\begin{equation}
 \label{eq_rectangular_function_expectation_3}
 (1+o(1))\prod_{j=1}^k \left[\frac{b(t+\eps \xi_j,\gamma)^{M} a(t,\gamma)^{M}}{b(t,\gamma)^M a(t+\eps \xi_j,\gamma)^{M}  }    \right] \E_{\widetilde \Upsilon^{\nu-M}}   \exp\left(\sum_{j=1}^k \eps \sigma  \xi_j \left(\sum\limits_{i=1}^j \upsilon_i^j - \sum\limits_{i'=1}^{j-1} \upsilon_{i'}^{j-1}\right) \right).
\end{equation}
At this point we can recognize in $ \E_{\widetilde \Upsilon^{\nu-M}}$ an integral representation\footnote{The integral representation of the multivariate Bessel function $\B$ as an integral over triangular arrays of reals is a folklore statement, which can be proven by a straightforward induction in $k$; in an explicit form it is contained, e.g.\ in \cite{guhr2002recursive}.} for the multivariate Bessel functions. Therefore, \eqref{eq_rectangular_function_expectation_3} equals
\begin{equation}
 \label{eq_rectangular_function_expectation_4}
 (1+o(1))\prod_{j=1}^k \left[\frac{b(t+\eps \xi_j,\gamma)^{M} a(t,\gamma)^{M}}{b(t,\gamma)^M a(t+\eps \xi_j,\gamma)^{M}  }    \right] \B_{\nu-M}(\eps \sigma \xi_1,\dots,\eps\sigma\xi_k).
\end{equation}
Replacing $\B_{\nu-M}(\eps \sigma \xi_1,\dots,\eps\sigma\xi_k)$ with equal $\B_{\eps \sigma(\nu-M)}(\xi_1,\dots,\xi_k)$, we arrive at \eqref{eq_F_to_Bessel}.

\medskip

Note that in order to make $o(1)$ in \eqref{eq_rectangular_function_expectation_4} uniform, we need to assume that $\eps \nu$ is bounded (equivalently, belongs to a compact subset of $\mathbb W_k$), because that was the assumption in Proposition \ref{Proposition_Gibbs_approximation} and, more importantly, because we need to bound the Lipshitz constant in the exponential function under the expectation sign to use this proposition.
\end{proof}

Proposition \ref{Proposition_F_to_Bessel} covers ``typical'' $\nu$, and we also need two additional bounds in order to produce tail estimates. We do not aim at making these bounds optimal.

\begin{prop}
  Fix one of the four cases \eqref{eq_case1}--\eqref{eq_case4}, corresponding $t$ and $\gamma$, an integer $k=1,2,\dots$, and a compact set $\mathfrak A$ in $\mathbb R^K$. We use the notations \eqref{eq_F_normalized}, \eqref{eq_sigma_formula} and let $\eps>0$ be a small parameter. There exists $C_1>0$, such that for each $\nu=(\nu_1<\dots<\nu_k)$, each $M\in\mathbb R$, and each $(\xi_1,\dots,\xi_k)\in \mathfrak A$,  we have
 \begin{equation}
 \label{eq_F_bound upper}
  \widetilde F_\nu\left(\eps \xi_1,\dots,\eps\xi_k\right)\le C_1 \left[ \prod_{j=1}^k \left(\frac{b\left(t+\eps \xi_j,\gamma\right)a(t,\gamma)}{b(t,\gamma) a\left(t+\eps \xi_j,\gamma\right)}\right)^{M}\right] \exp\left(\bigl(k+o(1)\bigr)\sum_{i=1}^k  \bigl|\sigma \eps (\nu_i-M)\xi_i\bigr| \right),
 \end{equation}
 where $o(1)$ tends to $0$ as $\eps\to 0$, uniformly in $\nu$, $M$, and $(\xi_1,\dots,\xi_k)\in \mathfrak A$.
\end{prop}
\begin{proof} Immediately follows from \eqref{eq_rectangular_function_expectation} and \eqref{eq_x22}.
\end{proof}

\begin{prop} \label{Proposition_F_lower_bound}
  Fix one of the four cases \eqref{eq_case1}--\eqref{eq_case4}, corresponding $t$ and $\gamma$, an integer $k=1,2,\dots$, and a compact set $\mathfrak A$ in $\mathbb R$. We use the notations \eqref{eq_F_normalized}, \eqref{eq_sigma_formula} and let $\eps>0$ be a small parameter. There exists $C_2>0$, such that for each $\nu=(\nu_1<\dots<\nu_k)$, each $M\in\mathbb R$, and each $\xi\in \mathfrak A$,  we have
 \begin{multline}
 \label{eq_F_bound lower}
  \widetilde F_\nu\left(\eps \xi,0^{k-1}\right) \cdot \left(\frac{b\left(t+\eps \xi,\gamma\right)a(t,\gamma)}{b(t,\gamma) a\left(t+\eps \xi,\gamma\right)}\right)^{-M} +\widetilde F_\nu\left(-\eps \xi,0^{k-1}\right)\cdot \left(\frac{b\left(t-\eps \xi,\gamma\right)a(t,\gamma)}{b(t,\gamma) a\left(t-\eps \xi,\gamma\right)}\right)^{-M}\\ \ge  C_2 \exp\Bigl( |\sigma \eps \xi| \frac{|\nu_1-M|+|\nu_k-M|}{8k!} \cdot (1+o(1))\Bigr),
 \end{multline}
 where $o(1)$ tends to $0$ as $\eps\to 0$, uniformly in $\nu$, $M$, and $\xi\in \mathfrak A$.
\end{prop}
\begin{proof}
 Using \eqref{eq_rectangular_function_expectation} and \eqref{eq_x22}, the left-hand side of \eqref{eq_F_bound lower} is
\begin{equation} \label{eq_x23}
  \E_{\Upsilon^{\nu-M;a,b,c,}}  \Biggl[(1+o(1))\exp\Bigl( \sigma \eps \xi \upsilon_1^1 \cdot (1+o(1))\Bigr)+(1+o(1))\exp\Bigl( -\sigma \eps \xi \upsilon_1^1 \cdot (1+o(1))\Bigr)\Biggr].
\end{equation}
There several cases on the configuration of $\nu_1$ and $\nu_k$.
\begin{enumerate}
 \item If both $\nu_1-M$ and $\nu_k-M$ are positive and $\nu_1-M>\frac{\nu_k-M}{2}$, then we use the almost sure inequality (following from the interlacing inequalities in Definition \ref{Def_abc_measure}) for the random array ${\Upsilon^{\nu-M;a,b,c,}}$
     $$
      |\nu_1-M|\le \upsilon_1^1 \le |\nu_k-M|,
     $$
     to lower bound \eqref{eq_x23} by
     $$
      (1+o(1))\exp\Bigl( |\sigma \eps \xi| |\nu_1-M| \cdot (1+o(1))\Bigr) \ge (1+o(1)) \exp\Bigl( |\sigma \eps \xi| \frac{|\nu_1-M|+|\nu_k-M|}{8k!} \cdot (1+o(1))\Bigr).
     $$
 \item If both $\nu_1-M$ and $\nu_k-M$ are negative and $|\nu_k-M|>\frac{|\nu_1-M|}{2}$, then the argument is the same.
 \item In all other cases, we have
 \begin{equation}
 \label{eq_x25}
   \nu_k-\nu_1 \ge  \frac{|\nu_1-M|+|\nu_k-M|}{4}.
 \end{equation}
 We use the following statement:

 {\bf Claim. } There exists a constant $C_3$, which depends on $a,b,c$, and $k$ such that for all $\nu$
 \begin{equation} \label{eq_x24}
  \mathrm{Prob} \left( |\upsilon_1^1|> \frac{\nu_k-\nu_1}{2k!}\right) \ge C_3.
 \end{equation}
 For the case $a=b=c=1$, the inequality \eqref{eq_x24} is proven in \cite[Lemma 8]{Gorin14}. Because by \eqref{eq_abc_measure} the weights for arbitrary $a,b,c$ are uniformly bounded from $0$ and $\infty$, the case of general $(a,b,c)$ is a corollary.
 Using \eqref{eq_x25} and \eqref{eq_x24} we lower-bound \eqref{eq_x23} by
 \begin{equation}
   (C_3+o(1))\exp\Bigl( |\sigma \eps \xi| \frac{\nu_k-\nu_1}{2k!} \cdot (1+o(1))\Bigr)\ge C_2 \exp\Bigl( |\sigma \eps \xi| \frac{|\nu_1-M|+|\nu_k-M|}{8k!} \cdot (1+o(1))\Bigr)
\end{equation}
for another constant $C_2>0$. \qedhere
\end{enumerate}
\end{proof}

We have collected all the ingredients for the next step towards Theorem \ref{Theorem_GUE_corners}.

\begin{prop} \label{Proposition_GUE_limit}
 Choose parameters $a,b,c>0$ such that $\Delta=\frac{a^2+b^2-c^2}{2ab}<1$. Let $(\lambda_i^k)$ be a random monotone triangle corresponding to $(a,b,c)$--random configuration of the six-vertex model with DWBC, as in Definition \ref{Definition_monotone_triangle}. For each fixed $k=1,2,\dots$, we have the distributional convergence of $k$-dimensional vectors
 $$
  \lim_{n\to\infty} \left( \frac{\lambda_i^k- {\mathfrak m}(a,b,c)\cdot n}{{\mathfrak s}(a,b,c) \sqrt{n}} \right)_{i=1}^k \stackrel{d}{=} \bigl(g_i^k\bigr)_{i=1}^k.
 $$
\end{prop}
\begin{proof} We choose $\mathfrak m$ according to formulas \eqref{eq_case2_constants}--\eqref{eq_case4_constants}, and record the result of Proposition \ref{Proposition_inhom_as_generating} using \eqref{eq_F_normalized} as
\begin{multline}\label{eq_x26}
 \ZZ_n\left(\frac{\xi_1}{\sqrt{n}},\dots,\frac{\xi_k}{\sqrt n};\, t,\gamma\right) \prod_{j=1}^k \left[\left(\frac{a\left(t+\frac{\xi_j}{\sqrt n},\gamma\right)}{a(t,\gamma)}\right)^{-n} \left(\frac{b\left(t+\frac{\xi_j}{\sqrt n},\gamma\right)a(t,\gamma)}{b(t,\gamma) a\left(t+ \frac{\xi_j}{\sqrt n},\gamma\right)}\right)^{-n {\mathfrak m}}\right]   = \\\E_{\lambda=(\lambda_1^k,\dots,\lambda_k^k)}\left[ \prod_{j=1}^k \left(\frac{b\left(t+\frac{\xi_j}{\sqrt n},\gamma\right)a(t,\gamma)}{b(t,\gamma) a\left(t+\frac{\xi_j}{\sqrt n},\gamma\right)}\right)^{-n {\mathfrak m}} \widetilde F_\lambda\left(\frac{\xi_1}{\sqrt n},\dots,\frac{\xi_k}{\sqrt n}\right)\right].
\end{multline}
Using Theorems \ref{Theorem_asymptotics_case2}--\ref{Theorem_asymptotics_case4}, we see that the left-hand side of \eqref{eq_x26} factorizes as $\prod_{j=1}^k$ as $n\to\infty$. Hence, using the $k=1$ computation made in Proposition \ref{Proposition_Gaussian_limit}, we convert \eqref{eq_x26} asymptotically as $n\to\infty$ into
\begin{equation}\label{eq_x27}
   \exp\left( {\mathfrak s}^2 \sigma^2 \sum_{j=1}^k \frac{ (\xi_j)^2}{2} + o(1)\right) =  \E_{\lambda=(\lambda_1^k,\dots,\lambda_k^k)}\left[ \prod_{j=1}^k \left(\frac{b\left(t+\frac{\xi_j}{\sqrt n},\gamma\right)a(t,\gamma)}{b(t,\gamma) a\left(t+\frac{\xi_j}{\sqrt n},\gamma\right)}\right)^{-n {\mathfrak m}} \widetilde F_\lambda\left(\frac{\xi_1}{\sqrt n},\dots,\frac{\xi_k}{\sqrt n}\right)\right],
\end{equation}
where ${\mathfrak s}^2$ is given in \eqref{eq_case2_constants}--\eqref{eq_case4_constants} and $\sigma$ is given in \eqref{eq_sigma_formula}.

Combining \eqref{eq_x27} with Proposition \ref{Proposition_F_lower_bound} we conclude that the exponential moments admit a uniform bound: for each $\xi>0$ there exists $C_4>0$ such that
\begin{equation}
\label{eq_x28}
 \sup_{n=1,2,\dots} \E \exp\left(\xi \frac{|\lambda_1^k-{\mathfrak m} n|+|\lambda_k^k-{\mathfrak m} n|}{\sqrt n}\right)<C_4.
\end{equation}
The bound \eqref{eq_x28} implies that $k$-dimensional vectors
\begin{equation}
 \label{eq_x29}
  \left(\frac{\lambda_1^k-{\mathfrak m} n}{ {\mathfrak s} \sqrt{n}},\frac{\lambda_2^k-{\mathfrak m} n}{ {\mathfrak s} \sqrt{n}},\dots, \frac{\lambda_k^k-{\mathfrak m} n}{ {\mathfrak s} \sqrt{n}}\right)
\end{equation}
are tight as $n\to\infty$. Therefore, their distribution has subsequential limits. Let $(\hat \lambda_1,\dots,\hat \lambda_k)$ be the a random vector distributed as one such subsequential limit. In order to identify it, we pass $n\to\infty$ (along the subsequence) in \eqref{eq_x27} under the expectation sign, using Proposition \ref{Proposition_F_to_Bessel} with $M={\mathfrak m} n$, $\eps=\frac{1}{\sqrt{n}}$ for the leading contribution and \eqref{eq_x28} together with \eqref{eq_F_bound upper} for tail bounds. We conclude that
\begin{equation}
  \exp\left( {\mathfrak s}^2 \sigma^2 \sum_{j=1}^k \frac{ (\xi_j)^2}{2}\right)= \E_{\hat \lambda=(\hat \lambda_1,\dots,\hat \lambda_k)}\Bigl[ \B_{\sigma {\mathfrak s} \hat \lambda}(\xi_1,\dots,\xi_k)\Bigr]= \E_{\hat \lambda=(\hat \lambda_1,\dots,\hat \lambda_k)}\Bigl[ \B_{ \hat \lambda}(\sigma {\mathfrak s} \xi_1,\dots,\sigma {\mathfrak s} \xi_k)\Bigr].
\end{equation}
Diving all $\xi_j$ by ${\mathfrak s} \sigma$ and using Proposition \ref{Proposition_GUE_BGF}, we conclude that $(\hat \lambda_1,\dots,\hat \lambda_k)$ has distribution of $k$ eigenvalues of $k\times k$ GUE. Because all subsequential limits are the same, the vectors \eqref{eq_x29} converge in distribution as $n\to\infty$ to this limit.
\end{proof}

\begin{proof}[Proof of Theorem \ref{Theorem_GUE_corners}] Proposition \ref{Proposition_GUE_limit} proves that the (rescaled) $k$-dimensional vector $(\lambda_1^k,\dots,\lambda_k^k)$ converges in distribution as $n\to\infty$ to $(g_1^k,\dots,g_k^k)$. To finish the proof, we will show that conditional distribution of $(\lambda_i^j)_{1\le i \le j<k}$ given $(\lambda_1^k,\dots,\lambda_k^k)$ converges (after rescaling) to the conditional distribution of $(g_i^j)_{1\le i \le j<k}$ given $(g_1^k,\dots,g_k^k)$. Indeed, the former conditional distribution is given in Definition \ref{Def_abc_measure}, as follows from the discussion in Section \ref{Section_F_functions}. On the other hand, the latter distribution is given in Definition \ref{Def_continuous_measure}, see \cite{Baryshnikov01} or \cite{neretin2003rayleigh} for modern proofs or \cite[Section 9.3]{gel1950unitary} for earlier discussions. Hence, the desired convergence of conditional distributions is Proposition \ref{Proposition_Gibbs_approximation} with $\eps=\frac{1}{\sqrt{n}}$.
\end{proof}

\subsection{Proof of Theorem \ref{Theorem_convergence_to_stochastic}}

We follow the same plan as for Theorem \ref{Theorem_GUE_corners} and start from $k=1$ case.

\begin{prop} \label{Proposition_Geometric_limit}
 Choose parameters $a,b,c>0$ such that $a>b$ and $\Delta=\frac{a^2+b^2-c^2}{2ab}>1$. Set
 \begin{equation}
 \label{eq_b1}
  b_1= \frac{a^2+b^2-c^2-\sqrt{(a^2+b^2-c^2)^2 - 4a^2b^2}}{2a^2}.
 \end{equation}
  Let $(\lambda_i^k)$ be a random monotone triangle corresponding to $(a,b,c)$--random configuration of the six-vertex model with DWBC, as in Definition \ref{Definition_monotone_triangle}. We have
 $$
  \lim_{n\to\infty} \left( \lambda_1^1  \right) \stackrel{d}{=} \mathrm{Geom}(b_1),
 $$
 where $\mathrm{Geom}(b_1)$ is the geometric random variable with
 $$
  \mathrm{Prob} \bigl( \mathrm{Geom}(b_1)= \ell\bigr)= b_1^{\ell-1}(1-b_1), \qquad \ell=1,2,\dots.
 $$
\end{prop}
\begin{proof}
 As in \eqref{eq_x17}, we restate the $k=1$ result of Proposition \ref{Proposition_inhom_as_generating} as
 \begin{equation}
\label{eq_x31}
 \ZZ_n\left(\xi_1;\, t,\gamma\right)= \frac{a\left(t+\xi_1,\gamma\right)^{n-1}  }{a(t,\gamma)^{n-1}  }\, \E \left[ \left(\frac{b\left(t+\xi_1,\gamma\right)a(t,\gamma)}{b(t,\gamma) a\left(t+\xi_1,\gamma\right)}\right)^{\lambda_1^1-1}\right].
\end{equation}
Using Theorem \ref{Theorem_asymptotics_case1}, plugging the values of $a$ and $b$ from \eqref{eq_case1}, and noting that $\gamma$ should be negative for $a>b$, we transform \eqref{eq_x31} into
\begin{equation}
\label{eq_x32}
\E \left[ \left(\frac{\sinh\left(t+\xi_1+\gamma\right)\sinh(t-\gamma)}{\sinh(t+\gamma) \sinh\left(t+\xi_1-\gamma\right)}\right)^{\lambda_1^1-1}\right]=\frac{\sinh(t-\ga+\xi_1)}{\sinh(t-\ga)}
\exp\bigl(- \xi_1 \bigr) \cdot (1+o(1)).
\end{equation}
Let us denote
\begin{equation}
\label{eq_x34}
 e^u=\frac{\sinh\left(t+\xi_1+\gamma\right)\sinh(t-\gamma)}{\sinh(t+\gamma) \sinh\left(t+\xi_1-\gamma\right)},\qquad\qquad q= \frac{\sinh(t+\gamma)}{\sinh(t-\gamma)} e^{2\gamma}.
\end{equation}
Then a direct computation shows that
$$
 \frac{1-q}{1-e^u q}= \frac{\sinh(t-\ga+\xi_1)}{\sinh(t-\ga)}
\exp\bigl(- \xi_1 \bigr),
$$
and therefore, \eqref{eq_x32} is equivalent to
\begin{equation}
\label{eq_x33}
 \E e^{u\lambda_1^1} = e^u \frac{1-q}{1-e^uq} \cdot \bigl(1+o(1)\bigr).
\end{equation}
Note that possible values for $\xi_1$ in Theorem \ref{Theorem_asymptotics_case1} form an open neighborhood of $0$. Hence, the possible values for $e^u$ in \eqref{eq_x33} form an open neighborhood of $1$ and possible values for $u$ in \eqref{eq_x33} form an open neighborhood of $0$. Therefore, \eqref{eq_x33} means that the Laplace transform of $\lambda_1^1$ converges in a neighborhood of $0$ to the Laplace transform of geometric random variable with ratio $q$. If follows that $\lambda_1^1$ converges in distribution to $\mathrm{Geom}(q)$, and it remains to show that $q=b_1$, as given in \eqref{eq_b1}. Hence, abbreviating $\sinh$ to $\sh$, we would like to verify that for $0<-\gamma<t$ we have
\begin{multline*}
 \frac{\sh(t+\gamma)}{\sh(t-\gamma)} e^{2\gamma} \stackrel{?}{=} \\ \frac{\sh^2(t-\ga)+\sh^2(t+\ga)-\sh^2(2\ga)-\sqrt{(\sh^2(t-\ga)+\sh^2(t+\ga)-\sh^2(2\ga))^2 - 4\sh^2(t-\ga)\sh^2(t+\ga)}}{2\sh^2(t-\ga)}.
\end{multline*}
We leave this direct check to the reader. 
\end{proof}

For general $k>1$ case of Theorem \ref{Theorem_convergence_to_stochastic}, we need to introduce a new type of generating functions, generalizing the left-hand side of \eqref{eq_x33} and allowing us to identify the answer with the stochastic six-vertex model. Such generating functions can be introduced in terms of different families of weights of the six-vertex model. We chose to work with stochastic weights by two reasons: first, they appear in the answer anyway; second, they make it simpler for us to connect to \cite{borodin2017integrable}, from which we need to use one theorem.

\begin{definition}  Given a parameter $q$ and a $k$-tuple of strictly increasing positive integers, ${\nu=(0<\nu_1<\nu_2< \dots < \nu_k )}$, we define a function of $k$ complex variables $w_1,\dots,w_k$:
 \begin{align}
\F_\nu(w_1,\dots,w_k)=  \prod_{i=1}^k (1-w_i)  \cdot \sum_{\sigma\in \mathfrak S(k)} \sigma \left[ \prod_{1\le \alpha<\beta\le k} \frac{1- w_\alpha(1+ q)+ q  w_\alpha w_\beta }{w_\beta-w_\alpha}  \prod_{i=1}^k  \left(w_i\right)^{\nu_{i}-1} \right], \label{eq_stochastic_functions_new_def}
\end{align}
where $\sigma$ acts by permuting the variables $w_1,\dots,w_k$.
\end{definition}

The superscript $\text{st}$ in $\F$ stays for the stochastic, as an indication that we use the stochastic weights in the identification of $\F$ with a partition function recorded in the following proposition.

\begin{prop}\label{Proposition_F_as_partition} Take $k$-tuple of positive integers $\nu=(0<\nu_1< \dots<\nu_k)$ and parameters $q$ and $w_1,\dots,w_k\in\mathbb C$. Consider a strip domain $\mathbb Z_{>0}\times\{1,2,\dots,k\}$ with the paths entering on the left (at every line) and exiting on top at positions $\{\nu_i\}_{i=1}^k$, as in the right panel of Figure \ref{Fig_domain_for_symfunctions}. We have
\begin{equation}
\label{eq_F_as_partition}
 \F_\nu(w_1,\dots,w_k)=\sum_{\sigma} \prod_{x=1}^{\infty} \prod_{y=1}^k \omega(x,y;\sigma),
\end{equation}
where the sum is over the configurations of the six-vertex model and we use the stochastic weights of Figure \ref{Figure_six_vertices_three_ways} with $w$ depending on the row, $w=w_y$.
\end{prop}
\begin{remark} The proposition implies that $\F_\nu$ are essentially the same functions as $F^{\text{sym}}_\nu$ of \eqref{eq_rectangular_function}, but with stochastic rather than symmetric weights.
\end{remark}

\begin{figure}[t]
\begin{center}
   \includegraphics[width=0.95\linewidth]{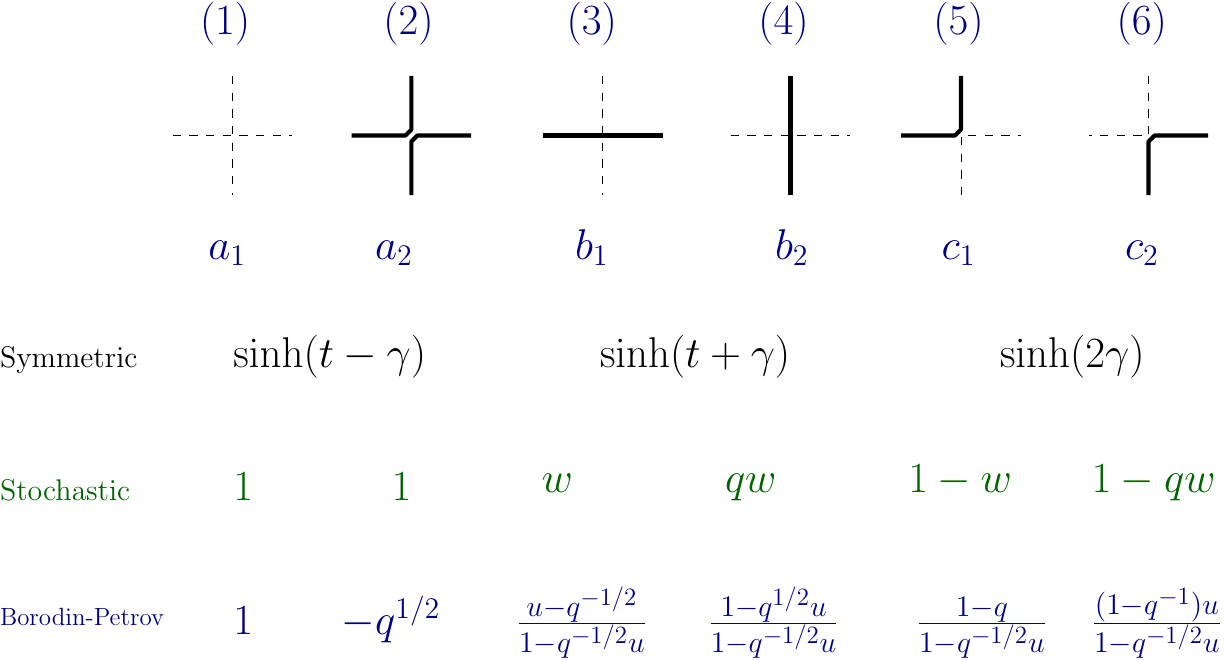}
\end{center}
        \caption{\label{Figure_six_vertices_three_ways} Three types of weights of vertices}
\end{figure}

\begin{proof}[Proof of Proposition \ref{Proposition_F_as_partition}]
 The identification of two definitions of $\F$: through \eqref{eq_F_as_partition} and through \eqref{eq_stochastic_functions_new_def} is one of the manifestations of the Bethe-ansatz solvability of the six-vertex model (see \cite{baxter2016exactly} for general discussions). Our proof relies on \cite[Theorem 4.12]{borodin2017integrable}, which uses different vertex weights given in \cite[(21) and Figure 4]{borodin2017integrable} in which we need to set $s=q^{-1/2}$ and $g-1,g,g+1\in\{0,1\}$; they are copied in the bottom of our Figure \ref{Figure_six_vertices_three_ways}. For the Borodin-Petrov weights,  \cite[Theorem 4.12]{borodin2017integrable} identifies the partition function of the six-vertex model in a strip --- an analogue of \eqref{eq_F_as_partition} --- with
 \begin{equation}
\label{eq_BP_functions}
 \mathsf{F}_\mu(u_1,\dots,u_k)=\frac{(1-q)^k}{\prod_{i=1}^k (1-s u_i)}  \sum_{\sigma\in \mathfrak S(k)} \sigma \left[ \prod_{1\le \alpha<\beta\le k} \frac{u_\alpha-q u_\beta}{u_\alpha-u_\beta}  \prod_{i=1}^k  \left(\frac{u_i-s}{1-s u_i}\right)^{\mu_i} \right],
\end{equation}
where $\sigma$ acts by permuting the variables $u_1,\dots,u_k$ (which correspond to the rows of the strip) and $\mu_i=\nu_{k+1-i}-1$ in order to match our notations. The functions $ \mathsf{F}_\mu$ are related to our $\F_\nu$ by the following transformation:
\begin{equation}
\label{eq_x43}
 \F_\nu(w_1,\dots,w_k)=q^{-k/2} \bigl(-q^{-1/2}\bigr)^{\nu_1+\dots+\nu_k}  \, \mathsf{F}_\mu\left(q^{-1/2}\frac{1-qw_1}{1-w_1},\dots, q^{-1/2}\frac{1-qw_k}{1-w_k}\right),
\end{equation}
with $s=q^{-1/2}$, $\mu_i=\nu_{k+1-i}-1$. The relation \eqref{eq_x43} is seen by directly comparing the formulas  \eqref{eq_stochastic_functions_new_def} and \eqref{eq_BP_functions}. As an intermediate step in this computation, we record
\begin{multline*}
q^{-k/2} \bigl(-q^{-1/2}\bigr)^{\nu_1+\dots+\nu_k}  \, \mathsf{F}_\mu\left(q^{-1/2}/z_1,\dots, q^{-1/2}/z_k\right)
\\= \prod_{i=1}^k \left[\frac{z_i(1-q)}{1-qz_i}\right]   \sum_{\sigma\in \mathfrak S(k)} \sigma \left[ \prod_{1\le \alpha<\beta\le k} \frac{z_\beta-q z_\alpha}{z_\beta-z_\alpha}  \prod_{i=1}^k  \left(\frac{1-z_i}{1-qz_i}\right)^{\nu_{k+1-i}-1} \right],
\end{multline*}
in which one should plug $z_i=\frac{1-w_i}{1-q w_i}$ (which is equivalent to $w_i=\frac{1-z_i}{1-q z_i}$) to get $\F_\nu(w_1,\dots,w_k)$.

It remains to also see the relation \eqref{eq_x43} for the partition functions \eqref{eq_F_as_partition} when changing between Borodin-Petrov and stochastic weights. We plug $u=q^{-1/2}/z$ into the Borodin-Petrov weights in Figure \ref{Figure_six_vertices_three_ways}, resulting in the six-tuple of weights:
\begin{equation} \label{eq_x35}
  1, \quad -q^{1/2},\quad  -q^{1/2}\frac{1-z}{1-qz}, \quad q\frac{1-z}{1-qz}, \quad -q\frac{(1-q)z}{1-qz}, \quad q^{-1/2}\frac{1-q}{1-qz}.
\end{equation}
Further, setting $z=\frac{1-w}{1-qw}$, the six-tuple is converted into
\begin{equation} \label{eq_x35_2}
  1, \quad -q^{1/2},\quad  -q^{1/2}w, \quad qw, \quad -q(1-w), \quad q^{-1/2}(1-qw).
\end{equation}
Comparing \eqref{eq_x35_2} with the stochastic weights of Figure \ref{Figure_six_vertices_three_ways} we see the following match:
\begin{equation}
\label{eq_x36_2}
 (\text{Borodin-Petrov weight})=(\text{stochastic weight}) \cdot (-q)^{\#[\text{line left}]} \cdot (q^{-1/2})^{\#[\text{line right}]},
\end{equation}
where $\#[\text{line left}]$ counts the number of lines entering into the vertex from the left (its values on six types are $0,1,1,0,1,0$) and $\#[\text{line right}]$ counts the number of lines exiting the vertex to the right (values $0,1,1,0,0,1$). Because each line exiting to the right immediately enters its neighboring vertex from the left, \eqref{eq_x36_2} implies a relation between partition functions in the domain of Figure \ref{Fig_domain_for_symfunctions} with stochastic and Borodin-Petrov weights:
\begin{equation}
\label{eq_x37_2}
 (\text{Borodin-Petrov p.f.})=(\text{stochastic p.f.}) \cdot (-q)^{k} \cdot (-q^{1/2})^{(\nu_1-1)+\dots(\nu_k-1)},
\end{equation}
which is precisely the relation between two variants of $F$--functions in \eqref{eq_x43}.
\end{proof}

We summarize useful properties of $\F_\nu(w_1,\dots,w_k)$ in the following lemma.
\begin{lem}
\label{Lemma_F_properties}
 Take a real parameter $q>0$ and $k$-tuple of integers $\nu_1>\dots>\nu_k>0$. We have:
 \begin{enumerate}
  \item $\F_\nu(w_1,\dots,w_k)$ is a symmetric polynomial in $w_1,\dots,w_k$ of degree $\nu_1+\dots+\nu_k + k(k-1)/2$.
  \item $\F_\nu(w_1,\dots,w_k) \prod\limits_{1 \le i<j \le k} (w_i-w_j)$ is a skew-symmetric polynomial in $w_1,\dots,w_k$ with lowest\footnote{We order the monomials by total degree first, and by lexicographic order on multi-degrees (so that $w_1^3 w_2 > w_2^2 w_1^2$) when total degrees are the same.} monomial
     $$
      w_1^{\nu_1-1} w_2^{\nu_2-1}\cdots w_k^{\nu_k-1}
     $$
  and highest monomial
      $$
      \bigl(-1\bigr)^k q^{k(k-1)/2}  \cdot w_1^{\nu_k+k-1} w_2^{\nu_{k-1}+k-1} \cdots w_k^{\nu_1+k-1}.
      $$
  \item $\F_\nu(w_1,\dots,w_k)>0$ whenever $0<w_1,w_2,\dots,w_k<\min(1,q^{-1})$.
  \item For each $0<w<\min(1,q^{-1})$, there exits a positive constant $C(w)$, such that
      \begin{equation}
      \label{eq_F_inequality}
        \left|\F_\nu(w_1,\dots,w_k)\right|< C(w) {\F_\nu}(w,\dots,w), \qquad \text{ for all }\,  w_i\in \mathbb C, \text{ such that } |w_i|\le w.
      \end{equation}
 \end{enumerate}
\end{lem}
\begin{proof}
 First, we prove the second property by noticing that ${\F_\nu}(w_1,\dots,w_k) \prod\limits_{1 \le i<j \le k} (w_i-w_j)$ is a skew-symmetrization of a polynomial:
 \begin{multline} \label{eq_x41}
  {\F_\nu}(w_1,\dots,w_k) \prod\limits_{1\le i<j\le k} (w_i-w_j)\\ =   \sum_{\sigma\in \mathfrak S(k)} (-1)^{\sigma}\sigma \left[ \prod_{1\le \alpha<\beta\le k} (1- w_\beta(1+ q)+ q  w_\alpha w_\beta)\, \prod_{i=1}^k (1-w_i) \left(w_i\right)^{\nu_{k+1-i}-1}  \right].
 \end{multline}
All properties in the second property are immediate from the form \eqref{eq_x41}. We deduce the first property from the second one by noticing that any skew-symmetric polynomial in $w_1,\dots,w_k$ is divisible by $\prod_{i<j} (w_i-w_j)$ and the ratio is a symmetric polynomial.

\smallskip

For the third and forth properties, we recall the formula of Proposition \ref{Proposition_F_as_partition} for ${\F_\nu}(w_1,\dots,w_k)$ and notice that the six-vertex weights for the $i$-th row are expressed in terms of $w_i$ as:
\begin{equation}
\label{eq_x42}
 1,\quad 1, \quad w_i, \quad q w_i, \quad 1-w_i,\quad 1-q w_i.
\end{equation}
If $0<w_1,\dots,w_k<\min(1,q^{-1})$, then all these weights are positive. Hence, $\widehat{\F_\nu}(w_1,\dots,w_k)>0$ as a sum of positive terms. For the last property of Lemma \ref{Lemma_F_properties}, notice that the total number of the vertices of the Types 5 and 6 (they correspond to the corners in the paths) is uniformly bounded: there are at most $k(k+1)/2$ of them in any configuration of the sum \eqref{eq_F_as_partition}. We absorb the factors corresponding to these two types into $C(w)$ factor in \eqref{eq_F_inequality}. The magnitude contribution of the remaining four weights $1,1, w_i, qw_i$ is readily upper bounded by $1,1, w, qw$. In this way, we upper-bound the magnitude of the sum \eqref{eq_F_as_partition} for the weights \eqref{eq_x42}, by a similar sum for the weights $1,1, w, qw, 1-w, 1-qw$, thus arriving at \eqref{eq_F_inequality}.
\end{proof}

Next, we characterize the stochastic six-vertex model, as defined before Theorem \ref{Theorem_convergence_to_stochastic}, in terms of the functions $\F_\nu$. Let us emphasize that there are two distinct meanings of the adjective ``stochastic'' and both are used in this section. When we speak about \emph{stochastic weights}, as in Proposition \ref{Proposition_F_as_partition}, we merely deal with weights of the second line of Figure \ref{Figure_six_vertices_three_ways} and boundary conditions for the model can be arbitrary. But if we speak about the \emph{stochastic six-vertex model}, as in Theorem \ref{Theorem_stochastic_6v_characterization}, then in addition to using the stochastic weights, we also impose special boundary conditions so that random configurations of the models can be sampled using the sequential algorithm described right before Theorem \ref{Theorem_convergence_to_stochastic}.

\begin{theo} \label{Theorem_stochastic_6v_characterization}
 Let $\mathbb P$ be a probability distribution on $k$-tuples of positive integers $0<\nu_1<\nu_2<\dots<\nu_k$. Take two reals $q,w$, such that $0<w<1$ and $0<qw<1$. Fix $\eps>0$ and suppose that
 \begin{equation}
 \label{eq_x38}
   \sum_{\nu} \mathbb P(\nu) \frac{\F_\nu(w_1,\dots,w_k)}{\F_{\nu}(w,\dots,w)}=1, \qquad \text{whenever }\quad w-\eps\le w_i\le w, \quad 1\le i \le k.
 \end{equation}
 Then $\nu$ has the law of the $k$-th row in the stochastic six-vertex model (under the identification of configurations with interlacing arrays as in Definition \ref{Definition_monotone_triangle}) with
 $b_1=w$, $b_2= q w$.
\end{theo}
\begin{remark}
 As we will see in the proof, the left-hand side \eqref{eq_x38} is an analytic function in the domain $\{(w_1,\dots,w_k)\in\mathbb C:\,  |w_i|\le w, \, 1\le i \le k\}$ and, therefore, it is sufficient to check the identity for $(w_1,\dots,w_k)$ in any uniqueness set for this function.
\end{remark}
\begin{proof}[Proof of Theorem \ref{Theorem_stochastic_6v_characterization}]
 {\bf Step 1.} Let us first check that if $\nu$ comes from the $k$-th row of the stochastic six-vertex model, then \eqref{eq_x38} holds. By Proposition \ref{Proposition_F_as_partition} and definition of the stochastic six-vertex model, $ \mathbb P(\nu) = \F_\nu(w,\dots,w)$. Hence, the left-hand side of \eqref{eq_x38} transforms into $\sum_{\nu} \F_\nu(w_1,\dots,w_k)$. Applying Proposition \ref{Proposition_F_as_partition} again, we see that this is a sum over all possible row $k$ configurations in the stochastic six-vertex model with row-dependent vertex weights and, therefore, the sum evaluates to $1$.

 \smallskip

 {\bf Step 2.} In the rest of the proof we show that \eqref{eq_x38} uniquely determines the measure $\mathbb P(\nu)$.
 Clearly, the sum in \eqref{eq_x38} is absolutely convergent at $w_1=w_2=\dots=w$. Hence, by \eqref{eq_F_inequality}, it is uniformly convergent to a holomorphic function of $w_1,\dots,w_k$ for $|w_i|\le w$. Since the sum is $1$ in a neighborhood of $(w,\dots,w)$, by the uniqueness theorem for holomorphic functions, it is also $1$ for all $w_1,\dots,w_k$ in a neighborhood of $(0,\dots,0)$. Therefore, we have
 \begin{equation}
 \label{eq_x40}
   \sum_{\nu} \frac{\mathbb P(\nu)}{\F_\nu(w_1,\dots,w_k)}  \cdot  \F_\nu(w_1,\dots,w_k) \prod\limits_{1 \le i<j \le k} (w_i-w_j) = \prod\limits_{1 \le i<j \le k} (w_i-w_j)
 \end{equation}
 for $(w_1,\dots,w_k)$ in a neighborhood of $(0,\dots,0)$. By uniqueness of the Taylor series we can view \eqref{eq_x40} as an identity of two skew-symmetric power series in $w_1,\dots,w_k$. Because of the second property in Lemma \ref{Lemma_F_properties}, the coefficients of the expansion in terms of $\F_\nu(w_1,\dots,w_k) \prod\limits_{\le i<j \le k} (w_i-w_j)$ are obtained from the coefficients of the expansion in monomials by a triangular transformation (the coefficient for $\nu=(1,2,\dots,k)$ is obtained from the coefficient of the monomial $x_1^{0} x_2^1 \cdots x_k^{k-1}$, the coefficient for $\nu=(1,2,\dots,k-1,k+1)$ is obtained from the coefficients of the monomials $x_1^{0} x_2^1 \cdots x_k^{k-1}$ and $x_1^{0} x_2^1 \cdots x_k^{k}$, etc). Hence, all the numbers $\frac{\mathbb P(\nu)}{\F_\nu(w_1,\dots,w_k)}$ are uniquely determined by \eqref{eq_x40} and $\mathbb P(\nu)$ are identified with the stochastic six-vertex model by Step 1.
\end{proof}
\begin{remark}
 An alternative way to show that $\mathbb P(\nu)$ are uniquely determined by the identity \eqref{eq_x38} could have been by extracting $\mathbb P(\nu)$ from the sum through the orthogonality relations for the functions $\F_\nu$. The orthogonality relations are discussed, e.g., in \cite[Section 7]{borodin2017integrable}, however, the existing literature concentrated on $0<q<1$ case, while we need $q>1$ in this paper.
\end{remark}

In the previous sections we worked with symmetric weights and we need to develop formulas for translating the results into the stochastic weights setting, cf.\ Figure \ref{Figure_six_vertices_three_ways}. Here is the translation for the homogeneous weights.

\begin{prop} \label{Proposition_sym_to_stoch}
 The probability distribution of the configurations of the six-vertex model with Domain Wall Boundary Conditions and homogenous symmetric weights of \eqref{eq_case1} copied in the top line of Figure \ref{Figure_six_vertices_three_ways}, is the same the distribution for DWBC six-vertex model with the stochastic weights in the second line of Figure \ref{Figure_six_vertices_three_ways}, if we set the parameters as:
 \begin{equation}
 \label{eq_weights_match}
  \begin{dcases} q=e^{-4\gamma},
  \\
  w=\frac{1-e^{2t+2\gamma}}{1-e^{2t-2\gamma}},
  \end{dcases} \qquad  \textbf{ or }\qquad  \begin{dcases} q=e^{4\gamma},\\
  w=\frac{1-e^{-2t-2\gamma}}{1- e^{-2t+2\gamma}}.\end{dcases}
 \end{equation}
\end{prop}
\begin{remark}
 In order for all six stochastic weights to be positive, we should have  $0<w<1$, $0<qw<1$. Hence, because $t>|\gamma|$ in \eqref{eq_case1}, only when $\gamma<0$ (which implies $a>b$ for the symmetric weights) the stochastic weights are positive.
\end{remark}
\begin{proof}[Proof of Proposition \ref{Proposition_sym_to_stoch}]
 In any finite planar domain, the probability distribution on the configurations of the six-vertex model with fixed deterministic boundary conditions depends only on two double ratios $\frac{a_1 a_1}{b_1 b_2}$ and $\frac{a_1 a_2}{c_1 c_2}$ rather than on the all six weights, as follows from the existence of the four conservation laws, see, e.g., \cite[Lemma 2.1]{Gorin_Nicoletti_lectures} or \cite[Section 4.2]{Bleher-Liechty14}. For the symmetric weights, these two ratios are:
 \begin{equation}
  \frac{\sinh^2(t-\gamma)}{\sinh^2(t+\gamma)}= \left( \frac{e^{t-\gamma}-e^{\gamma-t}}{e^{t+\gamma}-e^{-t-\gamma}}\right)^2, \qquad\qquad \frac{\sinh^2(t-\gamma)}{\sinh^2(2\gamma)}= \left( \frac{e^{t-\gamma}-e^{\gamma-t}}{e^{2\gamma}-e^{-2\gamma}}\right)^2.
 \end{equation}
For the stochastic weights under either of the two identifications of the parameters \eqref{eq_weights_match} they are
\begin{equation}
\left(\frac{1-e^{2t-2\gamma}}{e^{-2\gamma}(1-e^{2t+2\gamma})} \right)^2, \qquad\qquad \left(\frac{1-e^{2t-2\gamma}}{e^{t+\gamma}(1-e^{-4\gamma})} \right)^2.
\end{equation}
A straightforward comparison shows that these two pairs of the ratios are the same.
\end{proof}

Continuing to the inhomogeneous weights, we need to transform the result of Theorem \ref{Theorem_asymptotics_case1} into the stochastic setting. For that we introduce a corresponding version of the Izergin--Korepin determinant and its normalized version.

\begin{definition} \label{Definition_inhomogeneities_stochastic}
 Fix $n=1,2,\dots$, and $2n+1$ complex parameters $\mathrm X_1,\dots,\mathrm X_n; \mathrm Y_1,\dots,\mathrm Y_n$, and $q$. For a configuration $\sigma$ of the six-vertex model in $n\times n$ square with DWBC, we define the weight $\omega(x,y;\sigma)$ of the vertex at $(x,y)$ to be the stochastic weight of Figure \ref{Figure_six_vertices_three_ways} with
 \begin{equation}
  w=\frac{1- \mathrm X_x \mathrm Y_y}{1-q \mathrm X_x \mathrm Y_y}.
 \end{equation}
 We set
 \begin{equation}
  \mathcal Z_n^{\text{st}}(\mathrm X_1,\dots,\mathrm X_n; \mathrm Y_1,\dots,\mathrm Y_n; q)=\sum_{\sigma} \prod_{x=1}^n \prod_{y=1}^n \omega(x,y;\sigma).
 \end{equation}
The normalized stochastic partition function is defined in terms of complex parameters $w_1,\dots, w_k$ and $w$ as
\begin{equation}
 \ZZ_n^{\text{st}}(w_1,\dots,w_k;\, q,w)=\frac{ \mathcal Z_n^{\text{st}}(1^n;\, \mathrm Y_1,\dots, \mathrm Y_k, \mathrm Y^{n-k};\, q)}{\mathcal Z_n^{\text{st}}(1^n;\, \mathrm Y^n;\, q)}, \qquad \text{ where}
\end{equation}
\begin{equation}
\label{eq_x45}
 w=\frac{1- \mathrm Y}{1-q  \mathrm Y}, \quad \text{ and } \quad   w_i=\frac{1- \mathrm Y_i}{1-q  \mathrm Y_i}, \quad 1\le i \le k.
\end{equation}
\end{definition}
\begin{remark}
 Inverse relations to \eqref{eq_x45} have the same form:
\begin{equation}
\label{eq_x46}
 \mathrm Y=\frac{1- w}{1-q  w}, \quad \text{ and } \quad   \mathrm Y_i=\frac{1- w_i}{1-q  w_i}, \quad 1\le i \le k.
\end{equation}
\end{remark}

\begin{prop} We have
  \begin{equation}
  \label{eq_x44}
  \mathcal Z_n^{\text{st}}(\mathrm X_1,\dots,\mathrm X_n; \mathrm Y_1,\dots,\mathrm Y_n; q)=\frac{\prod_{i,j=1}^n(1-\mathrm X_i \mathrm Y_j)}{\prod_{i<j}(\mathrm X_i-\mathrm X_j)(\mathrm Y_i-\mathrm Y_j)} \det \left[\frac{(1-q) \mathrm X_i \mathrm Y_j}{(1-\mathrm X_i \mathrm Y_j)(1-q \mathrm X_i \mathrm Y_j)}\right]_{i,j=1}^n.
 \end{equation}
\end{prop}
\begin{proof}
 One path to the proof is to transform the weights (as we did in the proofs of Propositions \ref{Proposition_F_as_partition} and  \ref{Proposition_sym_to_stoch}) and then refer to Theorem \ref{Theorem_IK_det}. Alternatively, we can refer to \cite[Theorem 2.2]{Gorin_Nicoletti_lectures} for a direct proof of Izergin-Korepin determinant for the stochastic weights, thus replacing Theorem \ref{Theorem_IK_det}; for matching the notations, in \cite{Gorin_Nicoletti_lectures} we should set $t$ equal to our $q$ and apply to vertices an involution replacing paths by their absences (i.e.\ swapping the types of vertices in pairs: $1\leftrightarrow 2$, $3\leftrightarrow 4$, $5\leftrightarrow 6$).
\end{proof}
\begin{cor} We have
\begin{equation}\label{eq_x47}
  \ZZ_n^{\text{st}}(w_1,\dots,w_k;\, q,w)=\ZZ_n(\xi_1,\dots,\xi_k;\, t,\gamma) \prod_{j=1}^k \left(e^{\xi_j} \left[\frac{e^{-t+\ga}-e^{t-\ga}}{e^{-t-\xi_j+\ga}-e^{t+\xi_j-\ga}}\right]^n\right),
\end{equation}
under identification:
\begin{equation}
 q=e^{-4\gamma},  \qquad w=\frac{1-e^{2t+2\gamma}}{1-e^{2t-2\gamma}}, \qquad w_i =\frac{1-e^{2t+2\gamma+2\xi_i}}{1-e^{2t-2\gamma+2\xi_i} }.
\end{equation}
\end{cor}
\begin{proof}
  We plug
  $$
   q=e^{-4\gamma}, \quad \mathrm X_i=e^{-2\chi_i}, \quad \mathrm Y_i=e^{2\gamma+2\psi_i}, \quad 1\le i \le n,
  $$
  into \eqref{eq_x44} and get
  \begin{equation}
  \frac{\prod\limits_{i,j=1}^n(1-e^{-2\chi_i} e^{2\psi_j+2\gamma})}{\prod_{i<j}(e^{-2\chi_i}-e^{-2\chi_j})(e^{2\psi_i+2\gamma}-e^{2\psi_j+2\gamma})}   \det \left[\frac{(1-e^{-4\gamma}) e^{-2\chi_i} e^{2\psi_j+2\gamma}}{(1-e^{-2\chi_i} e^{2\psi_j+2\gamma})(1-e^{-2\chi_i} e^{2\psi_j-2\gamma})}\right]_{i,j=1}^n.
\end{equation}
On the other hand, \eqref{eq_IK_det} for the weights \eqref{eq_case1} becomes
\begin{align*}
   \mathcal Z_n(\chi_1,\dots,\chi_n;\, \psi_1,\dots,\psi_n;\, \gamma)=& 2^{-n}\dfrac{\prod\limits_{i,j=1}^n \bigl({(e^{\psi_j-\chi_i-\ga}-e^{-\psi_j+\chi_i+\ga}) (e^{\psi_j-\chi_i+\ga}-e^{-\psi_j+\chi_i-\ga}) }
   \bigr)}{\prod\limits_{i<j} \bigl(e^{\chi_i-\chi_j}-e^{\chi_j-\chi_i})(e^{\psi_i-\psi_j}-e^{\psi_j-\psi_i})\bigr)} \\ &\times \det \left[ \frac{2(e^{2\gamma}-e^{-2\gamma})}{(e^{\psi_j-\chi_i-\ga}-e^{-\psi_j+\chi_i+\ga}) (e^{\psi_j-\chi_i+\ga}-e^{-\psi_j+\chi_i-\ga}) } \right]_{i,j=1}^n.
\end{align*}
Comparing the last two formulas, we conclude that
\begin{multline}
\frac{\mathcal Z_n^{\text{st}}\left(e^{-2\chi_1},\dots,e^{-2\chi_n}; e^{2\gamma+2\psi_1},\dots,e^{2\gamma+2\psi_n}; e^{-4\gamma}\right)}
  {\mathcal Z_n(\chi_1,\dots,\chi_n;\, \psi_1,\dots,\psi_n;\, \gamma)}\\= \prod_{i,j=1}^n \frac{e^{\psi_j-\chi_i+\gamma}}{e^{-\psi_j+\chi_i+\ga}-e^{\psi_j-\chi_i-\ga}} \prod_{i<j} \left(-e^{\chi_i+\chi_j} e^{-\psi_i-\psi_j-2\gamma}\right).
\end{multline}
Plugging $\chi_1=\dots=\chi_n=0$, $\psi_i=t+\xi_i$, $\le i\le k$, $\psi_i=t$, $k+1\le i \le n$, we get
\begin{multline}
\frac{\mathcal Z_n^{\text{st}}\left(1^n; e^{2t+2\gamma+2\xi_1},\dots,e^{2t+2\gamma+2\xi_k}, (e^{2t+2\gamma})^{n-k}; e^{-4\gamma}\right)}
  {\mathcal Z_n(0^n;\, t+\xi_1,\dots,t+\xi_k, t^{n-k};\, \gamma)}\\= (-1)^{n(n-1)/2} \prod_{j=1}^k \left(e^{-\xi_j(n-1)} \left[\frac{e^{t+\xi_j+\gamma}}{e^{-t-\xi_j+\ga}-e^{t+\xi_j-\ga}}\right]^n\right) \cdot
   \left[\frac{e^{t+\gamma}}{e^{-t+\ga}-e^{t-\ga}}\right]^{n(n-k)} e^{-(\gamma+t) n(n-1)}.
\end{multline}
 Combining the last equality with its $\xi_1=\dots=\xi_k=0$ version, we get
\begin{multline}
\frac{\mathcal Z_n^{\text{st}}\left(1^n; e^{2t+2\gamma+2\xi_1},\dots,e^{2t+2\gamma+2\xi_k}, (e^{2t+2\gamma})^{n-k}; e^{-4\gamma}\right)}{\mathcal Z_n^{\text{st}}\left(1^n; (e^{2t+2\gamma})^{n}; e^{-4\gamma}\right)}
=\frac{\mathcal Z_n(0^n;\, t+\xi_1,\dots,t+\xi_k, t^{n-k};\, \gamma)}{\mathcal Z_n(0^n;\, t^n;\, \gamma)}\\ \times  \prod_{j=1}^k \left(e^{\xi_j} \left[\frac{e^{-t+\ga}-e^{t-\ga}}{e^{-t-\xi_j+\ga}-e^{t+\xi_j-\ga}}\right]^n\right),
\end{multline}
which is precisely \eqref{eq_x47}.
\end{proof}

Here is the version of Theorem \ref{Theorem_asymptotics_case1} for the stochastic weights.

\begin{theo} \label{Theorem_stochastic_partition_limit} Suppose that $q>1$ and $0<w<q^{-1}$ and fix $k=1,2,\dots$. There exists $\eps>0$, such that
\begin{equation}
 \lim_{n\to\infty}\ZZ_n^{\text{st}}(w_1,\dots,w_k;\, q,w)=1,
\end{equation}
uniformly over $w_i$ in a $\eps$--neighborhood of $w$.
\end{theo}
\begin{remark}
 While in Proposition \ref{Proposition_sym_to_stoch}, there was a symmetry between $0<q<1$ and $q>1$ cases --- it was possible to transform the symmetric weights into both --- this is no longer true in Theorem \ref{Theorem_stochastic_partition_limit}. We could produce its version for $0<q<1$ (exploiting an analogue of \eqref{eq_x47}), but the limit would be more complicated than the constant $1$.
\end{remark}
\begin{proof}[Proof of Theorem \ref{Theorem_stochastic_partition_limit}]
 Theorem \ref{Theorem_asymptotics_case1} for $\gamma<0$ implies that
 \begin{equation}
 \lim_{n\to\infty} \ZZ_n(\xi_1,\dots,\xi_k;\, t,\gamma)  \prod_{j=1}^k\left( e^{\xi_j}\left[\frac{\sinh(t-\ga)}{\sinh(t-\ga+\xi_j)}\right]^{n}\right)=1.
\end{equation}
for $(\xi_1,\dots,\xi_n)$ in a neighborhood of $(0,\dots,0)$. Comparing with \eqref{eq_x47}, we get the desired result.
\end{proof}

We have prepared all the ingredients for the proof of Theorem \ref{Theorem_convergence_to_stochastic}.

\begin{proof}[Proof of Theorem \ref{Theorem_convergence_to_stochastic}] We fix $k$ and let $\mathbb P_{k,n}(\nu),$ $\nu=(0<\nu_1<\nu_2<\dots<\nu_k)$, be the distribution of the $k$-th row in the monotone triangle, i.e.\ of $(\lambda_1^k,\lambda_2^k,\dots,\lambda_k^k)$. Applying Proposition \ref{Proposition_sym_to_stoch}, we can assume that this monotone triangle comes from the six-vertex model with DWBC and stochastic weights with $q=e^{-4\gamma}$ and $w=\frac{1-e^{2t+2\gamma}}{1-e^{2t-2\gamma}}$. (Recall that $a>b$ condition in Theorem \ref{Theorem_convergence_to_stochastic} implies that $\gamma<0$. Therefore the stochastic weights are positive.)

First, let us show that the measures $\mathbb P_{k,n}(\nu)$ are tight as $n\to\infty$. Indeed, conditionally on $(\lambda_1^k,\lambda_2^k,\dots,\lambda_k^k)$, the distribution of $(\lambda_i^j)_{1\le i \le j<k}$ is given by the Gibbs measure of the six-vertex model of Definition \ref{Def_abc_measure}. By \eqref{eq_x24} in the Claim inside the proof\footnote{Although the Claim dealt with symmetric weights, the same transformation as in Proposition \ref{Proposition_sym_to_stoch} allows us to replace them by the stochastic weights.} of Proposition \ref{Proposition_F_lower_bound}, if $\lambda_k^k$ escapes to infinity with positive probability, then so does $\lambda_1^1$. Therefore, tightness of $\mathbb P_{k,n}(\nu)$ for arbitrary $k>1$ is implied by the tightness at $k=1$, which was established in Proposition \ref{Proposition_Geometric_limit}.

Next, we form a generating function
\begin{equation}
\label{eq_x48}
 \mathcal G_{k,n}(w_1,\dots,w_k)=  \sum_{\nu} \mathbb P_{k,n}(\nu) \frac{\F_\nu(w_1,\dots,w_k)}{\F_{\nu}(w,\dots,w)}.
\end{equation}
By the estimate \eqref{eq_F_inequality} and tightness of $\mathbb P_{k,n}$, the series \eqref{eq_x48} is uniformly convergent over $n$ and over $(w_1,\dots,w_k)$ such that $w-\eps\le w_i \le w$ for some $\eps>0$. Therefore, if $\mathbb P_{k,\infty}$ is a subsequential limit point of the measures $\mathbb P_{k,n}$ as $n\to\infty$, then
\begin{equation}
\label{eq_x49}
 \mathcal G_{k,\infty}(w_1,\dots,w_k):=  \sum_{\nu} \mathbb P_{k,\infty}(\nu) \frac{\F_\nu(w_1,\dots,w_k)}{\F_{\nu}(w,\dots,w)}=\lim_{n\to\infty} \mathcal G_{k,n}(w_1,\dots,w_k).
\end{equation}
On the other hand, as in Proposition \ref{Proposition_inhom_as_generating}, we have
\begin{equation}
\label{eq_x50}
  \mathcal G_{k,n}(w_1,\dots,w_k)=\ZZ_n^{\text{st}}(w_1,\dots,w_k;\, q,w).
\end{equation}
Combining \eqref{eq_x49}, \eqref{eq_x50}, and Theorem \ref{Theorem_stochastic_partition_limit}, we conclude that
$$
  \sum_{\nu} \mathbb P_{k,\infty}(\nu) \frac{\F_\nu(w_1,\dots,w_k)}{\F_{\nu}(w,\dots,w)}=1.
$$
Hence, by Theorem \ref{Theorem_stochastic_6v_characterization}, all subsequential limits $\mathbb P_{k,\infty}$ are identified with the $k$-th row of the stochastic six-vertex model. Therefore, $(\lambda_1^k,\lambda_2^k,\dots,\lambda_k^k)$ converges in distribution to the $k$-th row of the stochastic six-vertex model. Because the distribution of $(\lambda_i^j)_{1\le i \le j<k}$ conditional on $(\lambda_1^k,\lambda_2^k,\dots,\lambda_k^k)$ is unchanged as $n\to\infty$ and agrees with the stochastic six-vertex model by a version of Proposition \ref{Proposition_sym_to_stoch}, we conclude that the full triangular array $(\lambda_i^j)_{1\le i \le j}$ converges in distribution to the stochastic six-vertex model in the quadrant and with parameters
\begin{equation}
\label{eq_x51}
 b_1=w=\frac{1-e^{2t+2\gamma}}{1-e^{2t-2\gamma}}, \qquad b_2= qw= e^{-4\gamma}\frac{1-e^{2t+2\gamma}}{1-e^{2t-2\gamma}}.
\end{equation}
It remains to check that given the expressions for $a,b,c$ of \eqref{eq_case1} through $\gamma$ and $t$, the formulas \eqref{eq_x51} match the expressions of $b_1$ and $b_2$ given in the statement of Theorem \ref{Theorem_convergence_to_stochastic}. For $b_1$ this is the same check as at the end of the proof of Proposition \ref{Proposition_Geometric_limit} and for $b_2$ the computation is similar and left to the reader.
\end{proof}

\section{Proofs for $c=0$ case}

\label{Section_c_0}

In this section we prove Theorem \ref{Theorem_Mallows} and Proposition \ref{Proposition_c_0_double}. Our approach is different from the proofs of the previous sections: rather than working with inhomogeneous partition functions, we present and analyze a stochastic algorithm which can be used for sampling the random configurations. (The algorithm works only at $c=0$ and we are not aware of its versions for $c>0$.) We start with presenting this algorithm in a self-contained way and then use it to study the asymptotics of the DWBC six-vertex model.

\bigskip

We fix a real parameter $q>0$ and aim at sampling a random permutation $\tau=(\tau(1),\tau(2),\dots,\tau(n))$ of $n$ letters $\{1,2,\dots,n\}$. Let $\zeta_1, \zeta_2,\dots,\zeta_n$ be $n$ independent random variables, such that $\zeta_k$ has a geometric distribution confined to the integers $\{1,2,\dots n-k+1\}$ and with ratio $q$:
\begin{equation}
\label{eq_confined_Geometric}
 \mathrm{Prob}( \zeta_k=m )= \frac{q^{m-1}}{1+q+\dots+q^{n-k}}, \qquad m\in\{1,2,\dots,n-k+1\}.
\end{equation}
Note that in contrast to the conventional geometric distribution, \eqref{eq_confined_Geometric} makes sense for all values of $q>0$; the restriction $q<1$ is not necessary.

\begin{definition} \label{Definition_q_shuffle_finite} The permutation $\tau=(\tau(1),\tau(2),\dots,\tau(n))$ is obtained by the \emph{$q$-shuffle procedure}:
\begin{itemize}
 \item Set $\tau(1):=\zeta_1$.
 \item Sequentially, for $k=2,3,\dots,n$, set $\tau(k)$ to be the $\zeta_k$-th largest element in the set $\{1,\dots,n\}\setminus\{\tau(1),\dots,\tau(k-1)\}$.
\end{itemize}
\end{definition}

\begin{prop} \label{Proposition_DWBC_c_0}
 Fix $a,b>0$ and consider the Gibbs measure of the DWBC six-vertex model in the $n\times n$ square obtained as $\eps\to 0$ limit of $(a,b,c)$ random configurations with $c=\eps$. The corresponding monotone triangle $(\lambda_i^k)_{1\le i \le k \le n}$, as in Definition \ref{Definition_monotone_triangle}, can be sampled from random permutation $\tau$ of Definition \ref{Definition_q_shuffle_finite} with $q=\frac{b^2}{a^2}$ by setting for each $k=1,2,\dots,n$ the subarray $(\lambda^k_1,\lambda^k_2, \dots,\lambda^k_k)$ to be the rearrangement of $(\tau(1),\dots,\tau(k))$ in the increasing order.
\end{prop}
\begin{proof}
 In the limit $c\to 0$, the Gibbs measure becomes supported on the configuration with the minimum number of $c$-type vertices. By using the third representation in Figure \ref{Figure_six_vertices}, where these vertices correspond to corners of the paths and noticing that we have precisely $n$ paths, which entered the square horizontally and need to exit vertically, we conclude that such a minimal number is $n$. Moreover, the configurations with $n$ $c$--type vertices have no vertices of Type 6 and $n$ vertices of Type 5. Hence, the $(a,b,0)$--Gibbs measure assigns the weight proportional to
 \begin{equation}
 \label{eq_c_0_weight}
  a^{N_1(\sigma)+N_2(\sigma)} b^{N_3(\sigma)+N_4(\sigma)}
 \end{equation}
 to a configuration $\sigma$ with exactly $n$ Type $5$ vertices and no Type $6$ vertices (here $N_i$ is the number of Type $i$ vertices); the weight vanishes if the numbers of Type $5$ and $6$ vertices are different from $n$ and $0$, respectively. In order to better understand the weight \eqref{eq_c_0_weight}, it is helpful to change the meaning of the Type 2 vertices, as shown in Figure \ref{Figure_Type2change}: rather than two turning paths, we now think about two intersecting paths. Then a configuration with minimal number of Type 5 vertices becomes a permutation $(\tau(1),\dots,\tau(n))$, where $\tau(i)$ is the horizontal coordinate of the path when its exits the square vertically, if it entered at height $i$. Equivalently, $(\tau(i),i)$, $1\le i \le n$, are cartesian $(x,y)$--coordinates of Type 5 vertices, and these are the only points where the paths turn. Clearly, the monotone triangle $(\lambda_i^k)_{1\le i \le k \le n}$ is reconstructed from $\tau$ by the procedure described in the proposition.

\begin{figure}[t]
\begin{center}
   \includegraphics[width=0.7\linewidth]{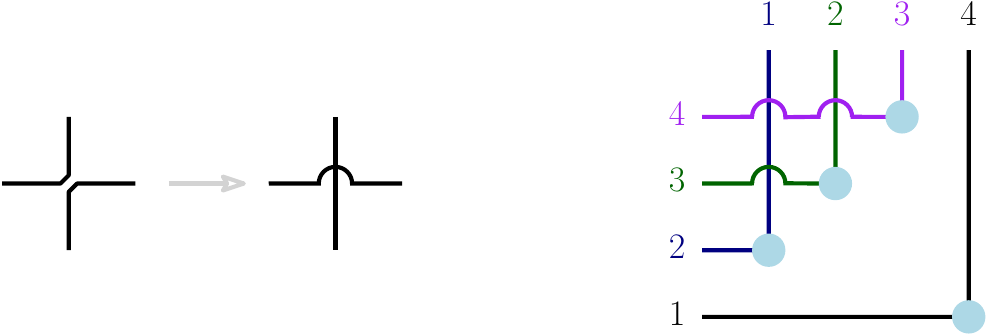}
\end{center}
        \caption{\label{Figure_Type2change} After we change the meaning of Type 2 vertex, each configuration with $n$ Type 5 and no Type 6 vertices is identified with a permutation. On the right we have permutation $4123$.}
\end{figure}

Let us compute the weight \eqref{eq_c_0_weight} corresponding to a permutation $\tau$ by evaluating the product of the weights of vertices in each row of the $n\times n $ square. In the first row we have $\tau(1)-1$ Type 3 vertices, one Type 5 vertex, and $n-\tau(1)$ Type 1 vertices, resulting in the weight
\begin{equation*}
 a^{n-\tau(1)-1} b^{\tau(1)-1}.
\end{equation*}
More generally, in row $k$ there is one Type 5 vertex and:
\begin{itemize}
 \item $n-\tau(k)- \#\{i<k: \tau(i)>k\}$ vertices of Type 1,
 \item $\#\{i<k: \tau(i)<k\}$ vertices of Type 2,
 \item $\tau(k)-1-\#\{i<k: \tau(i)<k\}$ vertices of Type 3,
 \item $\#\{i<k: \tau(i)>k\}$ vertices of Type 4.
\end{itemize}
Noting that $\#\{i<k: \tau(i)<k\}=k-\#\{i<k: \tau(i)>k\}-1$, the product of the weights of the vertices in row $k$ is
\begin{equation*}
 a^{n+k-1-\tau(k)- 2\cdot \#\{i<k:\, \tau(i)>k\}} \,  b^{\tau(k)-k+2\cdot \#\{i<k:\, \tau(i)>k\}}.
\end{equation*}
Multiplying over all $k=1,\dots,n$, and omitting the prefactor which does not depend on the choice of $\tau$, we conclude that the probability of a configuration corresponding to $\tau$ is
$$
 \left(\frac{b^2}{a^2}\right)^{\mathrm{inv}(\tau)}, \qquad \text{ where } \mathrm{inv}(\tau)\text{ is the number of inversion in }\tau.
$$
Such a measure on permutations is called the Mallows measure with $q=\tfrac{b^2}{a^2}$, and the fact that it can be sampled by the $q$--shuffle procedure of Definition \ref{Definition_q_shuffle_finite} is well-known, see, e.g., \cite[Section 3]{gnedin2010q}.
\end{proof}

\begin{proof}[Proof of Theorem \ref{Theorem_Mallows}] We send $n\to\infty$ in the procedure of Definition \ref{Definition_q_shuffle_finite} keeping $q=\tfrac{b^2}{a^2}$ fixed, notice that random variables $\zeta_i$, $i=1,2,\dots$, converge in distribution to i.i.d.\ Geometric distributions, and arrive at the procedure of Definition \ref{Def_q_shuffle}. Hence Theorem \ref{Theorem_Mallows} is the $n\to\infty$ limit of Proposition \ref{Proposition_DWBC_c_0}.
\end{proof}

\begin{proof}[Proof of Proposition \ref{Proposition_c_0_double}]  We send $n\to\infty$ in the procedure of Definition \ref{Definition_q_shuffle_finite} with asymptotic behavior of $q=\frac{b^2}{a^2}$ given by \eqref{eq_q_to_1}. The random variables $\tfrac{1}{n} \zeta_i$ converge in distribution to i.i.d.\ random variables $\eta_i$ of density \eqref{eq_eta_density}. Hence, Proposition \ref{Proposition_c_0_double} is the $n\to\infty$ limit of Proposition \ref{Proposition_DWBC_c_0}.
\end{proof}

\begin{appendix}

\section{Proof of Propositions \ref{prop:OP_asy} and \ref{prop:OP_asyF}}\label{app:OP_asymptotics} 
\subsection{Proof of Proposition \ref{prop:OP_asy}}
The asymptotic properties of the orthogonal polynomials $p_{n,k}(z)$ in the disordered phase $|\De|<1$, the antiferroelectric phase $\De<-1$, and on the boundary $\De = -1$ follow from the nonlinear steepest descent analysis of the $2\times 2$ matrix-valued Riemann--Hilbert problems (RHPs) presented in \cite{Bleher-Fokin06}, \cite{Bleher-Liechty10}, and \cite{Bleher-Bothner12}, respectively, and reviewed in \cite{Bleher-Liechty14}. In each of those works, the first step in the analysis is to rescale the polynomials as we have done in \eqref{eq:OPs_rescaled}.  Our rescaling \eqref{eq:OPs_rescaled} matches the one in \cite[Section 2]{Bleher-Bothner12} for the boundary case $\De=-1$, but differs slightly from the ones in \cite{Bleher-Fokin06} ($|\De|<1$) and \cite{Bleher-Liechty10} ($\De<-1$): whereas we rescale by a factor of $n$, in those papers the rescaling is by a factor of $n/\ga$ (see \cite[Equation (2.4)]{Bleher-Fokin06} and  \cite[Equation (1.25)]{Bleher-Liechty09}).

Throughout this appendix we use the notation $\sg_3  = \begin{pmatrix} 1 & 0 \\ 0 & -1\end{pmatrix}$ to denote the third Pauli matrix, so that a scalar $f$ raised to the power $\sg_3$ is
\[
f^{\sg_3} = \begin{pmatrix} f & 0 \\ 0 & 1/f\end{pmatrix}.
\]

\medskip

{\bf Proof of Proposition \ref{prop:OP_asy} for $|\De| < 1$.} Taking the different scalings into account, the $2\times 2$ matrix-valued RHP presented in \cite[Section 6.4]{Bleher-Liechty14} has the solution given in \cite[Equation (2.21)]{Bleher-Liechty14},
\eq\label{eq:RHP_sol_D}
{\bf Y}_n(\ga z) = \begin{pmatrix} \ga^{n} & 0 \\ 0 & \ga^{-n} \end{pmatrix} \begin{pmatrix} p_{n,n}(z) & C^{\mt_n}(p_{n,n})(z) \\ -2\pi \ii \frac{p_{n,n-1}(z)}{h_{n,n-1}} &  -2\pi \ii \frac{C^{\mt_n}(p_{n,n-1})(z)}{h_{n,n-1}}
\end{pmatrix}.
\eeq

%
%
The asymptotic analysis of the RHP presented in \cite[Section 6.4]{Bleher-Liechty14}   consists of a sequence of explicit transformations to the RHP to arrive at a RHP whose solution is uniformly close to the identity matrix as $n\to\infty$. The analysis of  \cite[Chapter 6]{Bleher-Liechty14}  involves the quantities $\al,\be, g(z)$ and $l$ which differ slightly from the corresponding quantities used in this paper due to the different scalings. In particular, the endpoints of the equilibrium measure $\al$ and $\be$ defined in  \cite[Equation (6.3.7)]{Bleher-Liechty14}  differ from the ones defined in \eqref{eq_case2_alpha_beta} of the current paper by a factor of $\ga$.


For $z$ away from the band $[\alpha, \beta]$ the transformations of the RHP in  \cite[Chapter 6]{Bleher-Liechty14} give
\eq\label{eq:RHunravel}
{\bf Y}_n(\ga z) = e^{(nl_n/2)\sg_3} {\bf X}_n(\ga z) {\bf M}_n(\ga z) e^{n(g_n(\ga z) - l_n/2)\sg_3}.
\eeq
The matrix ${\bf M}_n(z)$ is defined in  \cite[Equations (6.4.18), (6.4.19)]{Bleher-Liechty14} with a cut on the interval $[\al_n, \be_n$], the endpoints of which satisfy by \cite[Proposition 6.3.2]{Bleher-Liechty14} as $n\to\infty$,
\eq\label{eq:aln_al}
\al_n = \ga\al + \bigO(n^{-2}), \quad \be_n = \ga\be + \bigO(n^{-2}).
\eeq
It is clear from the definition  \cite[Equations (6.4.18), (6.4.19)]{Bleher-Liechty14} that as $z\to\infty$,
\eq\label{eq:Mn_zinf}
{\bf M}_n(z) = {\bf I} + \bigO(1/z),
\eeq
and that ${\bf M}_n(z)$ is bounded on closed subsets of $\C \setminus [\al_n,\be_n]$.
The entries of  ${\bf M}_n(z)$ are algebraic functions of $z$ which are uniformly bounded on closed subsets of $\C \setminus [\al_n,\be_n]$, and whose $n$-dependence comes solely from the constants $\al_n$ and $\be_n$. Therefore ${\bf M}_n(z)$ has a limit as $n\to\infty$ which is analytic in $\C \setminus [\ga\al, \ga\be]$. Let us denote this limit as ${\bf M}(z)$,
\eq
{\bf M}(z) := \lim_{n\to\infty} {\bf M}_n(z), \quad z\in \C \setminus [\ga\al,\ga\be],
\eeq
which inherits the property \eqref{eq:Mn_zinf} as $z\to\infty$,
\eq\label{eq:M_zinf}
{\bf M}(z) = {\bf I} + \bigO(1/z).
\eeq
The functions $N_{\rm D}(z)$ and $M_{\rm D}(z)$ appearing in Proposition \ref{prop:OP_asy} are the second row of the matrix ${\bf M}(\ga z)$,
\eq\label{eq:N_M_matrix}
N_{\rm D}(z) = {\bf M}(z)_{21}, \quad M_{\rm D}(z) = {\bf M}(z)_{22},
\eeq
which satisfy \eqref{eq:Nstar_infinity} and \eqref{eq:Mstar_infinity}, respectively, by \eqref{eq:M_zinf}.
According to \eqref{eq:aln_al} $\al_n$ and $\be_n$ differ from $\ga\al$ and $\ga\be$, respectively, by order $\bigO(1/n^2)$ so we have
\eq\label{eq:Mn_M}
{\bf M}_n(\ga z) = {\bf M}(\ga z) + \bigO(n^{-2})
\eeq
uniformly for $z$ in closed subsets of $\C \setminus [\al,\be]$.

The relations of the rest of the quantities in \eqref{eq:RHunravel} to those appearing in Proposition \ref{prop:OP_asy} are presented in the following lemma.

\begin{lem}\label{lem:appendix_misc}
Let $g_n(z)$ be the function defined in \cite[Equation (6.3.93)]{Bleher-Liechty14} using \cite[Equation (6.3.34)]{Bleher-Liechty14}. Let $l_n$ be the constant defined in \cite[Equation (6.3.88)]{Bleher-Liechty14}, and ${\bf X}_n(z)$ the matrix function defined in \cite[Section 6.4.5]{Bleher-Liechty14}. Let $g(z)$ and $l$ be as defined in \eqref{def:g-function} and \eqref{eq:eq_condition_band} of the current paper. Then as $n\to\infty$,
\eq\label{eq:ln_l}
l_n = l+ 2\ln \ga + \bigO(n^{-2}),
\eeq
and for $z$ bounded away from $[\al, \be]$,
\eq\label{eq:g_gn}
g_n(\ga z) =  g(z) + \ln \ga + \bigO(n^{-2}),
\eeq
and
\eq\label{eq:R_largen}
{\bf X}_n(\ga z) = {\bf I} + \bigO(n^{-1}),
\eeq
where both \eqref{eq:g_gn} and \eqref{eq:R_largen} hold uniformly on closed subsets on $\C \setminus [\al,\be]$.
\end{lem}
Assuming this lemma, Proposition \ref{prop:OP_asy}  for $|\De| < 1$ then follows by combining \eqref{eq:RHP_sol_D} and \eqref{eq:RHunravel}, using \eqref{eq:Mn_M} and \eqref{eq:R_largen}.
Taking the $(22)$-entry gives
\eq
\frac{C^{\mt_n}(p_{n,n-1})(z)}{h_{n,n-1}} = -\frac{\ga^n}{2\pi \ii} {\bf Y}_n(\ga z)_{22} = -\frac{\ga^n}{2\pi \ii} e^{-ng_n(\ga z)} {\bf M}(\ga z)_{22} (1+\bigO(1/n)), \\
\eeq
and taking the $(21)$-entry gives
\eq
\frac{p_{n,n-1}(z)}{h_{n,n-1}} =  -\frac{\ga^n}{2\pi i} {\bf Y}_n(\ga z)_{21} = -\frac{\ga^n}{2\pi \ii} e^{n(g_n(\ga z) - l_n)} {\bf M}(\ga z)_{21} (1+\bigO(1/n)),
\eeq
which proves  \eqref{eq:leading_Cauchy} and \eqref{eq:leading_OPs} for $|\De|<1$ after using \eqref{eq:N_M_matrix}, \eqref{eq:ln_l}, and \eqref{eq:g_gn}.

\medskip

To complete the proof of Proposition \ref{prop:OP_asy} for $|\De|<1$, we still need to prove Lemma \ref{lem:appendix_misc}. We start with
\eqref{eq:g_gn}. We first compare $g(z)$ as defined in the current paper with the function denoted $g(z)$ in \cite[Section 6]{Bleher-Liechty14}. To avoid confusion we will denote the function from \cite[Section 6]{Bleher-Liechty14}, given explicitly in \cite[Proposition 6.3.1]{Bleher-Liechty14} as $\tilde g(z)$. The derivative of $\tilde g(z)$ is given in \cite[Equation (6.3.4)]{Bleher-Liechty14}. Comparing with \eqref{eq_case2_G} (and recalling that $\al$ and $\be$ in \cite[Section 6]{Bleher-Liechty14} differ from those in the current paper by a factor of $\ga$), we find
\eq\label{eq:2gprimes}
 \frac{\dd}{\dd z} \tilde g(\ga z) = g'( z),
\eeq
so $\tilde g'(z)$ and $g(\ga z)$ can differ only by a constant. From the conditions $\tilde g(z) = \ln z +\bigO(1/z)$ and $ g(z) = \ln z +\bigO(1/z)$ as $z\to\infty$, we must have
\eq\label{eq:g_gtilde}
\tilde g(\ga z) = g( z) + \ln \ga.
\eeq

We now compare $\tilde g( z)$  and $g_n(z)$. They are the log transforms of their respective equilibrium measures:
\eq\label{eq:g_gn_rho}
\tilde g(z) = \int_{\al\ga}^{\be\ga} \ln(z-x)\, \rho(x)\, \dd x, \qquad g_n(z) = \int_{\al_n}^{\be_n} \ln(z-x)\, \rho_n(x)\, \dd x.
\eeq
 The density $\rho(x)$ (resp. $\rho_n(x)$) is a probability density supported on $[\ga \al, \ga \be]$ (resp. $[\al_n, \be_n]$).
The density $\rho(x)$ is given by \cite[Equation (6.3.9)]{Bleher-Liechty14}
\eq\label{def:rho}
\rho(x) = \frac{2}{\pi^2}\ln\left[\frac{\sqrt{\ga\be(x-\al\ga)}+\sqrt{-\ga\al(\ga\be-x)}}{\sqrt{|x|(\ga\be-\ga\al)}}\right], \quad \ga\al \le x \le \ga\be,
\eeq
and $\rho_n(x)$ is given by \cite[Equations (6.3.78)]{Bleher-Liechty14}
\eq
\rho_n(x) = \rho_n^0(x) -\frac{1}{2\pi^2} k(nx)+\bigO(n^{-2}),
\eeq where
\eq\label{eq:rhon_rho0}
\rho_n^0(x) = \frac{2}{\pi^2}\ln\left[\frac{\sqrt{\be_n(x-\al_n)}+\sqrt{-\al_n(\be_n-x)}}{\sqrt{|x|(\be_n-\al_n)}}\right], \quad \al_n \le x \le \be_n,
\eeq
and
 $k(x)$ is an even real function not depending on $n$, defined in \cite[Equation (6.3.28)]{Bleher-Liechty14}, and satisfying
 \eq\label{eq:k_infty}
k(x) = -\frac{2\pi \ga^2}{3(\pi-2\ga)}x^{-2}+\bigO(x^{-4}), \quad x\to \pm \infty,
\eeq
see \cite[Equations (6.3.29), (6.6.30), (6.3.31)]{Bleher-Liechty14}. Since $\rho_n^0(\al_n) =\rho_n^0(\be_n)=0$, integration by parts gives
\[
\int_t^{\be_n} \rho^0_n(x)\dd x = -t  \rho^0_n(t) + \int_t^{\be_n} x \left(\frac{\dd}{\dd x}\rho_n^0(x)\right)\dd x
\]
for $\al_n \le t \le \be_n$. The latter integral is straightforward to compute, giving
\[
\int_t^{\be_n} \rho^0_n(x)\dd x =   -t  \rho^0_n(t) + \frac{2}{\pi} \arctan\sqrt{\frac{\be_n -t}{t-\al_n}}, \quad \al_n \le t \le \be_n,
\]
as in \cite[Equation (5.35)]{Bleher-Fokin06}. Plugging in $t=\al_n$ then shows $\int_{\al_n}^{\be_n} \rho^0_n(x)\dd x =1$. Since  $\int_{\al_n}^{\be_n} \rho_n(x)\dd x =1$ as well, integrating \eqref{eq:rhon_rho0} over $[\al_n, \be_n]$ and taking $n\to\infty$ implies
\eq\label{eq:k_int}
\int_{-\infty}^\infty k(x)\, \dd x  =0,
\eeq
as in \cite[Equation (5.39)]{Bleher-Fokin06}.

From \eqref{eq:g_gn_rho}, $\tilde g(z)$ and $g_n(z)$ satisfy, as $z\to\infty$,
\eq
\tilde g(z) = \ln(z) + \sum_{k=1}^\infty \frac{\tilde \mu_k}{z^k}, \qquad  g_n(z) = \ln(z) + \sum_{k=1}^\infty \frac{ \mu_k^{(n)}}{z^k},
\eeq
where
\eq
\tilde \mu_k = \int_{\al\ga}^{\be\ga} x^k \rho(x)\,\dd x, \qquad  \mu_k^{(n)} = \int_{\al_n}^{\be_n} x^k \rho_n(x)\,\dd x.
\eeq
Their difference is then given by the convergent series in a neighborhood of $\infty$,
\eq\label{eq:g_n_diff}
\tilde g(z) - g_n(z) = \sum_{k=1}^\infty \frac{\tilde \mu_k - \mu_k^{(n)}}{z^k}.
\eeq
Let $M$ be a large but fixed positive number such that \eqref{eq:g_n_diff} holds for $|z|>M$. We will prove \eqref{eq:g_gn} first for $|z| \le M$, and then for $z>M$.

Assume that $z$ is bounded away from the interval $[\ga\al,\ga\be]$ and $|z|<M$. Since $\ga \al= \al_n+\bigO(n^{-2})$ and  $\ga \be = \be_n+\bigO(n^{-2})$, the difference $\tilde g(z)- g_n(z)$ is
\eq\label{eq:rho_rhon_error}
\tilde g(z) - g_n(z) = \int_{\al_n}^{\be_n} \ln(z-x)(\rho(x) - \rho_n(x))\,\dd x + \bigO(n^{-2}).
\eeq
Again from $\ga \al  = \al_n+\bigO(n^{-2})$ and  $\ga \be = \be_n+\bigO(n^{-2})$, we have $\rho(x) = \rho_n^0(x) + \bigO(n^{-2})$, and we find
\eq\label{eq:integral_diff}
|\tilde g(z) - g_n(z)| = C\int_{\al_n}^{\be_n} k(nx)\, \dd x = \bigO(1/n)\int_{n \al_n}^{n\be_n} k(x)\, \dd x = \bigO(n^{-2}),
\eeq
where we have used properties \eqref{eq:k_infty} and \eqref{eq:k_int} for $k(x)$ in the last step. Combined with \eqref{eq:g_gtilde}, this proves \eqref{eq:g_gn} for $z$ bounded away from $[\ga \al ,\ga \be]$ and $|z|<M$.

Now assume $|z|>M$, so that \eqref{eq:g_n_diff} holds and
\eq\label{eq:g_gn_series}
|\tilde g(z) - g_n(z) |\le  \sum_{k=1}^\infty \frac{|\tilde \mu_k - \mu_k^{(n)}|}{M^k}.
\eeq
Similar to \eqref{eq:rho_rhon_error},
\[
\tilde \mu_k - \mu_k^{(n)} = \int_{\al_n}^{\be_n} x^k(\rho(x) - \rho_n(x))\dd x + \max\{ |\al |,\be\}^k \bigO (n^{-2}),
\]
and similar to \eqref{eq:integral_diff} we have
\begin{multline}\label{eq:mu_k_diff_abs}
\left| \tilde \mu_k - \mu_k^{(n)} \right| \le C\int_{\al_n}^{\be_n} x^k k(nx)\dd x+ \max\{ |\al |,\be\}^k \bigO (n^{-2}) \\
 \le  \frac{C}{n^{k+1}} \int_{n\al_n}^{n\be_n} x^k k(x) \dd x+ \max\{ |\al |,\be\}^k \bigO (n^{-2}).
\end{multline}
Using \eqref{eq:k_infty}, the integral $\int_{n\al_n}^{n\be_n} x^k k(x) \dd x$ is estimated as $\bigO\left(n^{k-1} \max\{ |\al |,\be\}^k\right)$, so \eqref{eq:mu_k_diff_abs} becomes
\eq\label{eq:mu_muk_diff}
\left| \tilde \mu_k - \mu_k^{(n)} \right| \le C \frac{\max\{ |\al |,\be\}^k}{n^2},
\eeq
where the constant $C$ may be different than the one in \eqref{eq:mu_k_diff_abs}. Plugging \eqref{eq:mu_muk_diff} into \eqref{eq:g_gn_series} gives
\[
|\tilde g(z) - g_n(z) |\le \frac{C}{n^2} \sum_{k=1}^\infty  \left(\frac{\max\{ |\al |,\be\}}{M}\right)^k,
\]
which is $\bigO(n^{-2})$ provided $M> \max\{ |\al |,\be\}$. This proves  \eqref{eq:g_gn} for $|z|>M$.

\medskip

The constant $l_n$ is given explicitly in \cite[Equation 6.3.137]{Bleher-Liechty14} as (recall that we replace $\al$ and $\be$ from \cite[Chapter 6]{Bleher-Liechty14} with $\ga\al$ and $\ga\be$, respectively),
\eq\label{eq:ln_On}
l_n = \ln(\be-\al) - 4\ln 2 - 1 + 2\ga + \bigO(n^{-2}).
\eeq
From the equilibrium condition \eqref{eq:eq_condition_band}, the constant $l$ --- as appearing in the current paper --- is
 \eq\label{eq:lD_gV}
 l = 2g(\be) - V(\be) = 2g(\be) - \be(\ga-t).
 \eeq
 Using integration by parts, $g(z)$ is given, up to a constant, as
 \eq\label{eq:gD_ibp}
 g(z) = zg'(z) - \int zg''(z)\dd z = zG_\nu(z) + \int \frac{\dd z}{\sqrt{(z-\al)(z-\be)}},
 \eeq
 where $G_\nu(z)\equiv g'(z)$ is as in \eqref{eq_case2_G} and we have used \eqref{eq:Gp}.
 Integrating \eqref{eq:gD_ibp} and using $g(z) = \ln z + \bigO(1/n)$  to determine the constant, we find
 \eq\label{eq:gD_formula}
 \begin{aligned}
 g(z) &= z G_\nu(z) + 2\ln(\sqrt{z-\al}+\sqrt{z-\be}) -2\ln 2- \frac{\ga\sqrt{-\al\be}}{\pi} \\
& =z G_\nu(z) + 2\ln(\sqrt{z-\al}+\sqrt{z-\be}) -2\ln 2-1,
 \end{aligned}
 \eeq
 where we used $-\al\be = \pi^2/\ga^2$ in the second line.
 Since $G_\nu(\be) = (\ga-t)/2$, plugging $z=\be$ into \eqref{eq:gD_formula} we find
 \[
 g(\be) = \be (\ga-t)/2 + \ln(\be-\al) - 2\ln 2 -1,
\]
and \eqref{eq:lD_gV} becomes
\[
 l = \be (\ga-t)+ \ln(\be-\al)- 4\ln 2 -2 - \be(\ga-t) = \ln(\be-\al)- 4\ln 2 -2.
 \]
 Comparing with \eqref{eq:ln_On}, this proves \eqref{eq:ln_l}.

\medskip

Finally we prove \eqref{eq:R_largen}. To describe the structure of the error ${\bf X}_n(z)$, we first summarize a few results on the Cauchy operator on a contour in the complex plane. For more details see, e.g., \cite[Lecture 2]{Deift19}, \cite[Chapters 4 and 5]{Bottcher-Karlovich97}, and references therein.  For a locally rectifiable oriented contour $\Sg$ in the complex plane and a $2\times 2$ complex matrix-valued H\"older continuous function $f(z)  \in L^1(\Sg)$, the Cauchy transform of $f(z)$ on the contour $\Sg$ is defined as
\[
C_\Sg(f)(z) = -\frac{1}{2\pi \ii}\int_\Sg \frac{f(\mu)}{z-\mu}\, \dd \mu, \quad z\in \C \setminus \Sg.
\]
For $\la \in \Sg$ we can define $C^\pm_\Sg(f)(\la)$ as
\[
C^\pm_\Sg(f)(\la):= \lim_{z\to \la_\pm} C_\Sg(f)(z),
\]
where $ \lim_{z\to \la_\pm}$ refers to the limit as $z$ approaches $\la$ non-tangentially to $\Sg$ from the $\pm$-side of $\Sg$, and the $+$ (resp. $-$) side of the contour is the side to the left (resp. right) when the contour is traversed in the direction of its orientation (local rectifiability of $\Sg$ implies that such a non-tangential limit exists almost everywhere). For a large class of contours $\Sg$ called Carleson contours (including the contour $\Sg_X$ below), it is known that $C^\pm_\Sg$ are bounded operators in $L^p(\Sg)$ for $1<p<\infty$ \cite[Theorem 2.24]{Deift19}.

The matrix ${\bf X}_n(z)$ in \eqref{eq:RHunravel} is given by the convergent series \cite[Equation (2.5.11)]{Bleher-Liechty14} (see also \cite[Equation (9.29)]{Bleher-Fokin06}):
\eq\label{eq:R_series}
{\bf X}_n(z) = {\bf I} + \sum_{k=1}^\infty {\bf X}_{n,k}(z), \quad {\bf X}_{n,k}(z) =  -\frac{1}{2\pi \ii}\int_{\Sg_X} \frac{v_{k-1}(\mu)j_X^0(\mu)}{z-\mu}\,d\mu = C_{\Sg_X}(v_{k-1}\cdot j_X^0)(z),
\eeq
where ${\Sg_X}$ is a contour in $\C$ (see \cite[Section 6.4.5, Figure 6.2]{Bleher-Liechty14}), $j_X^0(\mu)$ is an $n$-dependent function on ${\Sg_X}$, and $v_k(\mu)$ is defined recursively as
\[
v_k(\la) = C^-_{\Sg_X}(v_{k-1}\cdot j_X^0)(\la), \quad \la\in \Sg_X; \quad v_0(\la) = {\bf I}.
\]
We will prove \eqref{eq:R_largen} by estimating ${\bf X}_{n,k}(z)$ for each $k$. The contour ${\Sg_X}$ consists of the components
\[
\Sg_X = \Sg_n^+ \cup \Sg_n^- \cup \d D(\al_n, r) \cup\d D(\be_n, r)  \cup (-\infty, \al_n - r) \cup (\be_n+r, \infty),
\]
for some $r>0$, where $\d D(x, r)$ is a circle of radius $r$ centered at $x$.  $\Sg_n^+$ (resp. $\Sg_n^-$) is a contour in the upper (resp. lower) half-plane which begins on $ \d D(\al_n, r)$ and ends at $\d D(\be_n, r) $, maintaining a distance $r/2$ from the real line except in a small neighborhood of the origin, which it passes at a distance of order $1/n$.   In \cite[Proposition 6.5.1]{Bleher-Liechty14} the following uniform estimates for $j_X^{0}(\mu)$ are given:  for some constants $C, d, \de>0$,
\eq\label{eq:JR_uni_bounds}
||j_X^{0}(\mu)|| \le \left\{
\begin{aligned}
&\frac{C}{n}, \qquad \mu \in \d D(\al_n, r) \cup \d D(\be_n, r), \\
&C e^{-nd|\mu|}, \qquad  \mu \in(-\infty, \al-r) \cup (\be+r, \infty), \\
& \frac{C}{n^\de} e^{- nd |\Im \mu|}, \qquad \mu \in \Sg_n^\pm.
\end{aligned}\right.
\eeq
The above estimate on those intervals implies that for $z$ bounded away from $\Sg_X$, and $k>1$,
\eq\label{eq:Rk_est1}
||{\bf X}_{n,k}(z)||\le \frac{C}{1+|z|} \int_{\Sg_X} ||v_{k-1}(\mu) j_X^{0}(\mu)||\, \dd \mu \le \frac{C}{1+|z|} || v_{k-1} ||_{L_2(\Sg_X)}\cdot ||j_X^{0} ||_{L_2(\Sg_X)},
\eeq
for some constant $C>0$. If $z$ is close to the intervals $(-\infty, \al-\ep) \cup (\be+\ep, \infty)$, the contour $\Sg_X$ may be deformed slightly so that the integration contour in \eqref{eq:R_series} maintains a uniform distance from $z$, and such a deformation does not change the estimates \eqref{eq:JR_uni_bounds}. Thus the estimate \eqref{eq:Rk_est1}. holds for all $z$ bounded away from $[\al,\be]$.

We therefore need to estimate the $L_2$-norms of $v_{k-1}$ and $j_X^{0}$ on $\Sg_X$. Since $\Sg_n^\pm$ pass within an $\bigO(1/n)$ neighborhood of the real line, the estimates \eqref{eq:JR_uni_bounds} give the apriori bound $||j_X^{0} ||_{L_2(\Sg_X)} = \bigO(n^{-\de})$ for some $\de>0$, which can be improved. Using  $\eqref{eq:JR_uni_bounds}$ we find
\eq\label{eq:JR_L2}
||j_X^{0} ||_{L_2(\Sg_X)}^2 = \int_{\Sg_X} ||j_X^0(\mu)||^2\, \dd\mu =  \int_{\Sg_n^{\pm}} ||j_X^0(\mu)||^2\, \dd\mu +\bigO(1/n).
\eeq
The estimate \eqref{eq:JR_uni_bounds} on $\Sg_n^\pm$ is of order $n^{-\de}$ for all $\mu \in\Sg_X^\pm$, and is in fact super-polynomial in $n$ as long as $|\Im \mu|>n^{-\eta}$ for some $0<\eta<1$. The intersection of $\Sg_X^\pm$ with the strip
\[
\mathcal{S}_\eta = \{\mu \in \C : |\Im \mu|< n^{-\eta}\}
\]
has length of order $\bigO(n^{-\eta})$, so the estimate \eqref{eq:JR_L2} becomes
\eq\label{eq:JR_L2a}
\begin{aligned}
||j_X^{0} ||_{L_2(\Sg_X)}^2 &= \int_{\Sg_n^\pm \cap \mathcal{S}_\eta} ||j_X^0(\mu)||^2\, \dd\mu +\bigO(1/n^2) = \int_{\Sg_n^\pm \cap \mathcal{S}_\eta} \bigO(n^{-2\de})\, \dd\mu +\bigO(1/n^2) \\
& = \bigO(n^{-\eta-2\de})+\bigO(1/n^2),
\end{aligned}
\eeq
where we emphasize that the above estimate holds  {\it for some} $\de>0$, and  {\it for any} $0<\eta<1$. We now move on to the estimate of the $L_2$-norm of $v_{k-1}$ for $k>1$. Since $v_{k-1}$ is the Cauchy transform over $\Sg_X$ of the function $v_{k-2} \cdot j_X^0$, the boundedness of the Cauchy transform in $L_2(\Sg_X)$ implies for some constant $C>0$,
\eq\label{eq:vk_L2}
|| v_{k-1} ||_{L_2(\Sg_X)} \le C||v_{k-2} j_X^0||_{L_2(\Sg_X)}.
\eeq
When $k=1$, $v_{k-1} = {\bf I}$, and the estimate \eqref{eq:JR_L2a} gives
\[
|| v_{1} ||_{L_2(\Sg_X)} \le C||j_X^0||_{L_2(\Sg_X)} = \bigO(n^{-\eta/2-\de}) + \bigO(1/n).
\]
Then using the fact that $|| j_X^0||_{L_\infty(\Sg_X) } = \bigO(n^{-\de})$ and using induction in $k$ on \eqref{eq:vk_L2} we find
\eq\label{eq:vk_L2_est}
|| v_{k-1} ||_{L_2(\Sg_X)} = \bigO\left(\left(\frac{C}{n}\right)^{(k-1)\de+\eta/2}\right) + \bigO\left(\left(\frac{C}{n}\right)^{(k-2)\de+1}\right) \\
\eeq
Combining \eqref{eq:Rk_est1} with \eqref{eq:JR_L2a} and \eqref{eq:vk_L2_est}, we find that uniformly for all $k>1$,
\[
||{\bf X}_{n,k}(z)|| =\bigO\left(\frac{1}{1+|z|}\left(\frac{C}{n}\right)^{k\de+\eta}\right) + \bigO\left(\frac{1}{1+|z|}\left(\frac{C}{n}\right)^{(k-1)\de+1+\eta/2}\right).
\]
Taking $\eta > 1-\de$, we find for $k>1$,
\[
||{\bf X}_{n,k}(z)|| = \bigO\left(\frac{1}{n(1+|z|)}\right)\left(\frac{C}{n}\right)^{(k-1)\de}.
\]
We are left only to estimate
\eq\label{eq:Xn1z}
{\bf X}_{n,1}(z) =  -\frac{1}{2\pi \ii} \int_{\Sg_R} \frac{ j_X^0(\mu)}{z-\mu}\dd\mu.
\eeq
The function $j_X^0(\mu)$ has poles accumulating on the imaginary axis, spaced at a distance of order $1/n$, and the integral \eqref{eq:Xn1z} can be estimated by deforming the contours $\Sg_n^\pm$ across these poles and away from the real axis, keeping track of the residues which are accumulated. This is exactly what is done in \cite[Section 6.6]{Bleher-Liechty14} for the integral
\[
-\frac{1}{2\pi \ii} \int_{\Sg_n^\pm}  j_X^0(\mu)\dd\mu.
\]
That calculation is easily adapted to the integral \eqref{eq:Xn1z}. Indeed, similar to \cite[Equation (6.6.14)]{Bleher-Liechty14}, we have
\[
-\frac{1}{2\pi \ii} \int_{\Sg_n^+} \frac{ j_X^0(\mu)}{z-\mu}\dd\mu = -\frac{1}{2\pi \ii} \int_{\Sg_n^+} \frac{e^{-nG_n(\mu)}{\bf M}_n(\mu)\begin{pmatrix} 0 & 0 \\ 1 & 0\end{pmatrix}{\bf M}_n(\mu)^{-1}}{z-\mu}\dd\mu
\]
where $e^{-nG_n(\mu)}$ is a scalar function with poles at the points
\[
z_j = \frac{\ii j \pi}{n(\pi/(2\ga) -1)}, \quad j =  1,  2, \dots,
\]
and residues satisfying
\[
\Res_{\mu = z_j} e^{-nG_n(\mu)} =c_j n^{- \kappa_j}, \quad \kappa_j = 1+ \frac{2j}{\pi/(2\ga) -1}>1.
\]
for some explicit constants $c_j$. Assuming $z$ is bounded away from the origin, this immediately gives the asymptotic expansion similar to \cite[Equation (6.6.22)]{Bleher-Liechty14},
\eq
\begin{aligned}
-\frac{1}{2\pi \ii} \int_{\Sg_n^+} \frac{ j_X^0(\mu)}{z-\mu}\dd\mu &= -\sum_{j: \kappa_j \le 2} \frac{c_j n^{-\kappa_j}{\bf M}_n(z_j)\begin{pmatrix} 0 & 0 \\ 1 & 0\end{pmatrix}{\bf M}_n(z_j)^{-1}}{z-z_j} + \bigO(n^{-2-\ep}) \\
& = -\frac{1}{z}\sum_{j: \kappa_j \le 2} c_j n^{-\kappa_j}{\bf M}_{n+}(0)\begin{pmatrix} 0 & 0 \\ 1 & 0\end{pmatrix}{\bf M}_{n-}(z_j)^{-1}+ \bigO(n^{-2-\ep}),
\end{aligned}
\eeq
for some $\ep>0$, where in the second line we used the fact that $z_j = \bigO(1/n)$ and ${\bf M}_n(z)$ and ${\bf M}_n(z)^{-1}$ are both bounded in a neighborhood of the origin. A nearly identical computation and result applies to $-\frac{1}{2\pi \ii} \int_{\Sg_n^-} \frac{ j_X^0(\mu)}{z-\mu}\dd\mu$, showing that ${\bf X}_{n,1}(z) = \bigO(1/n)$ as $n\to\infty$ uniformly for $z$ bounded away from $[\al,\be]$.

%
%
%

\medskip

{\bf Proof of Proposition \ref{prop:OP_asy} for $\De = -1$.} Here the proof is nearly identical to the $|\De|<1$ case, this time using the Riemann--Hilbert analysis of \cite{Bleher-Bothner12}.  We do not include the details here, but make a few comments regarding the error term in the function ${\bf X}_n(z)$. As in the $|\De|<1$ case, the function ${\bf X}_n(z)$ is given by the series \eqref{eq:R_series}, and the function $j_X^0(\mu)$\footnote{In \cite{Bleher-Bothner12}, the function which we call $j_X^0(\mu)$ is instead denoted $(G_r(\mu) - I)$.}  is estimated as $\bigO(1/n)$ except on the small purely imaginary interval $(-\ii \ep, \ii\ep)$. The proof that  ${\bf X}_n(z) = {\bf I} + \bigO(1/n)$ for $z$ bounded away from $[\al,\be]$ then boils down to estimating integrals of the form
\[
\int_{-\ii\ep}^{\ii\ep} j_X^0(\mu) f(\mu)\dd\mu
\]
for a scalar function $f(\mu)$ which is analytic in a neighborhood of the origin. Such integrals are shown in \cite[Equation (11.13)]{Bleher-Bothner12} to be of order $\bigO(1/n)$. Taking $f(\mu) = \frac{1}{z-\mu}$ then implies the estimate  ${\bf X}(z) = {\bf I} + \bigO(1/n)$ for $z$ bounded away from $[\al,\be]$.

\medskip

{\bf Proof of Proposition \ref{prop:OP_asy} for $\De < -1$.}
Here the proof is very similar to the $|\De|<1$ case, this time using the Riemann--Hilbert analysis of \cite{Bleher-Liechty10}, also described in \cite[Chapter 7]{Bleher-Liechty14}. For the convenience of the reader, we will refer primarily to \cite[Chapter 7]{Bleher-Liechty14} in what follows.
 As in the case $|\De|<1$ described above, the quantities $\al, \al',\be',\be, g(z)$, and $l$ differ slightly from the quantities used in this paper due to the different scalings. The endpoints of the equilibrium measure $\al, \al',\be',$ and $\be$ presented there differ by a factor of $\ga$. We also have the following lemma.
\begin{lem}\label{lem:appendix_misc_2}
Let $\tilde g(z)$ be the function defined in \cite[Equation (7.4.20)]{Bleher-Liechty14} and denoted $g(z)$ there. Let $\tilde l$ be the constant defined in \cite[Equation (7.4.22)]{Bleher-Liechty14} and denoted $l$ there. Let $g(z)$ and $l$ be as defined in \eqref{def:g-function} and \eqref{eq:eq_condition_band} of the current paper. Then
\eq\label{eq:l_ltilde}
\tilde l = l+ 2\ln \ga, \qquad \tilde g(\ga z) =  g(z) + \ln \ga .
\eeq
\end{lem}
\begin{proof}
Note that, as in \eqref{eq:2gprimes}, we have $\frac{\dd}{\dd z} \tilde g(\ga z) = g'( z),$ where  $\frac{\dd}{\dd z} \tilde g(\ga z)$ is given \cite[Equation (7.4.19)]{Bleher-Liechty14} (recall that $\al, \al', \be',$ and $\be$ are all rescaled by a factor of $\ga$), and $g'(z)$ is given in \eqref{eq_case3_G}. Thus, $\tilde g(\ga z) $ and $g'( z)$ can differ only by a constant and the value of the constant follows from the fact that both $g(z)$ and $\tilde g(z)$ behave as $\ln z +\bigO(1/z)$ as $z\to\infty$. Then the relation between $\tilde l$ and $l$ follows from plugging $x=\ga\be$ into the equilibrium condition  \cite[Equation (7.4.22)]{Bleher-Liechty14} for $\tilde g(x)$, and $x=\be$ into the equilibrium condition \eqref{eq:eq_condition_band} for $g(x)$.
\end{proof}

In \cite{Bleher-Liechty14} the discrete RHP, or interpolation problem, is presented in \cite[Section 7.5.1]{Bleher-Liechty14}, and its solution, denoted ${\bf P}_n(z)$, is given in \cite[Equation (7.5.4)]{Bleher-Liechty14}. In terms of the orthogonal polynomials $p_{n,k}(z)$ presented in equation \eqref{def:OPsrescaled_AF} of the current paper (again accounting between the rescaling $x\mapsto \ga x$), that solution is
\eq\label{eq:RHP_sol_AF}
{\bf P}_n(\ga z) = \begin{pmatrix} \ga^{n} & 0 \\ 0 & \ga^{-n} \end{pmatrix} \begin{pmatrix} p_{n,n}(z) & -4\pi \ii C^{\mt_n}(p_{n,n})(z) \\  \frac{p_{n,n-1}(z)}{2h_{n,n-1}} &  -2\pi \ii \frac{C^{\mt_n}(p_{n,n-1})(z)}{h_{n,n-1}}
\end{pmatrix}.	
\eeq
As in the proof above for $|\De|<1$,
unraveling the transformations in the RHP which make up the nonlinear steepest descent analysis of \cite[Chapter 7]{Bleher-Liechty14}, we obtain that for $|\Im z| >\ep$ and $\Re z \notin [\al, \be]$,
\eq\label{eq:RHunravelAF}
{\bf P}_n(\ga z) = {\bf K}^{-1}e^{(n(l+2\ln\ga)/2)\sg_3} {\bf X}_n(\ga z) {\bf M}(\ga z) e^{n(g( z)+\ln \ga - (l+2\ln\ga)/2))\sg_3}{\bf K},
\eeq
where
\[
{\bf K} = \begin{pmatrix} 1 & 0 \\ 0  & -2\pi \ii \end{pmatrix},
\]
and ${\bf M}(z)$ satisfies
\[
{\bf M}(z) = {\bf I} + \bigO(1/z), \quad \textrm{as} \ z\to\infty.
\]
 Note that we have used the relations \eqref{eq:l_ltilde} to write $\tilde g(z)$ and $\tilde l$ in terms of $g(z)$ and $l$.  The matrix ${\bf X}_n(z)$ in \eqref{eq:RHunravelAF} satisfies as $n\to\infty$,
\eq\label{eq:Xn_AF}
{\bf X}_n(z) = {\bf I}+ \bigO(1/n),
\eeq
uniformly on closed subsets of $\C \setminus [\ga\al,\ga\be]$. Unlike in the $|\De|<1$ case\footnote{In some ways the analysis is simpler in this case because the orthogonality weight $w_n^{\rm AF}(x)$ is given by $e^{nV(x)}$ where $V(x)$ is independent of $n$; in the $|\De|<1$ case the weight $w_n^{\rm D}(x)$ is given by $e^{nV_n(x)}$ where $V_n(x)$  depends on $n$. This $n$-dependence of $V_n(x)$ necessitates an $n$-dependent equilibrium measure in the asymptotic analysis, which in turn leads to $n$-dependence of quantities line $g_n(z)$ in \eqref{eq:RHunravel} and ${\bf M}_n(z)$ in \eqref{eq:Mn_zinf}. The corresponding quantities in the $\De<-1$ case are independent of $n$, and there is no need for careful error estimates for them.}, the asymptotic expression \eqref{eq:Xn_AF} follows immediately from the standard $\bigO(1/n)$ estimates on the jump matrix for ${\bf X}_n(z)$, see \cite[Equation (7.5.94)]{Bleher-Liechty14}.
Multiplying out the RHS of \eqref{eq:RHunravelAF}, taking the $(22)$-entry, and comparing with of \eqref{eq:RHP_sol_AF} yields
\[
\frac{C^{\mt_n}(p_{n,n-1})(z)}{h_{n,n-1}} = -  \frac{\ga^n}{2\pi \ii} e^{-n(g(z)+\ln \ga)}{\bf M}(\ga z)_{22}(1+\bigO(1/n)),
\]
which is \eqref{eq:leading_Cauchy} after writing $M(z)$ for ${\bf M}(\ga z)_{22}$. Similarly, comparing the $(21)$-entry of \eqref{eq:RHunravelAF} with \eqref{eq:RHP_sol_AF} yields
\[
\frac{p_{n,n-1}(z)}{h_{n,n-1}} =   \frac{\ga^n}{-\ii \pi} e^{n(g(z)-l-\ln \ga)}{\bf M}(\ga z)_{21}(1+\bigO(1/n)),
\]
which is
\eqref{eq:leading_OPs} after writing $N(z)$ for $-\frac{ {\bf M}(\ga z)_{21}}{\ii \pi}$.

When $|\Im z| <\ep$ and $\Re z \notin [\al, \be]$, there is an additional term in the asymptotic expansion coming from the discreteness of the orthogonality measure. In this region, the transformations of the RHP include a matrix ${\bf D}_{\pm}^u(z)$ whose purpose is to turn a discrete Riemann--Hilbert problem into a continuous one (see \cite[Section 7.5.2]{Bleher-Liechty14}). The result is that for $z$ close to the real axis but still bounded away from $[\al, \be]$, \eqref{eq:RHunravelAF}  becomes
\eq\label{eq:RHunravelAF_real}
{\bf P}_n(\ga z) = {\bf K}^{-1}e^{(n(l+2\ln\ga)/2)\sg_3} {\bf X}_n(\ga z) {\bf M}(\ga z) e^{n(g( z)+\ln \ga - (l+2\ln\ga)/2))\sg_3}{\bf K}{\bf D}^u(\ga z)^{-1},
\eeq
where
\[
{\bf D}^u(\ga z) = \begin{pmatrix} 1 & -  \frac{\pi e^{-nV(z)}}{\sin\left(\frac{n\pi z }{2}\right)}e^{\pm \ii n\pi z/2} \\ 0 & 1 \end{pmatrix} \quad \textrm{for} \ \pm \Im z>0.
\]
This upper-triangular factor does not change the first column on ${\bf P}_n(\ga z)$, but it does affect the second column. In particular, comparing the $(22)$-entry with \eqref{eq:RHP_sol_AF} we find
\eq\label{eq:Cauchy_closetoR}
\frac{C^{\mt_n}(p_{n,n-1})(z)}{h_{n,n-1}} = -\frac{1}{2\pi \ii} e^{-ng(z)} \left({\bf M}(\ga z)_{22} - \frac{e^{n(2g( z) - V(z) + l)} e^{\pm \frac{\ii n\pi z}{2}} {\bf M}(\ga z)_{21}}{2 i \sin\left(\frac{n\pi z}{2}\right)} \right)(1+\bigO(1/n)).
\eeq
The second term in the parentheses above has poles on the lattice $\frac{2}{n} \Z$, but is exponentially small in $n$ for $z$ near the intervals $(-\infty, \al) \cup (\be, \infty)$ and at a distance at least $\ep/n$ from these poles due to \eqref{eq:equilibrium_lessthan}. The second term is not present at all for $z$ away from the real line. This proves \eqref{eq:leading_Cauchy} for  $\De<-1$. 

\subsection{Proof of Proposition \ref{prop:OP_asyF}}\label{app:A2}

Unlike the cases $\De<1$, the asymptotic analysis of the Riemann--Hilbert problem encoding the orthogonal polynomials $p^{\rm F}_{n,n-1}(z)$ was not done in earlier papers. We present the analysis here, following the approach outlined in \cite{Bleher-Liechty11} and \cite[Chapter 3]{Bleher-Liechty14}. The key difference here is that the discrete measure of orthogonality accumulates on the negative real line, rather than the full real line as in \cite{Bleher-Liechty11} and \cite[Chapter 3]{Bleher-Liechty14}. There is, therefore, an additional local solution near the origin using the function $D(z)$ defined in \eqref{eq:def_D} below. The same local transformation is used in \cite[Section 6.6]{Geronimo-Liechty20}, see also \cite{Wang-Wong11}. The rest of the asymptotic analysis presented in this section follows the steps of \cite{Bleher-Liechty11}.

 Letting $p_{n,k}(z)$ be the orthogonal polynomials defined in \eqref{def:OPsrescaled_F}, define the shifted monic polynomials
\eq\label{eq:OPs_F_shifted}
q_{n,k}(z) = p_{n,k}(z-1/n),
\eeq
which satisfy the orthogonality condition
\eq\label{eq:OPs_F_reflected}
\frac{1}{n} \sum_{x\in L_n}q_{n,k}(x)q_{n,j}(x) e^{n(t-|\ga|) x}\left(1-e^{2n|\ga| x}e^{-2|\ga|}\right) =  h^q_{n,k} \de_{j,k}, \quad  h^q_{n,k} =  h_{n,k}e^{(t-|\ga|)},
\eeq
where $t>|\gamma|$ and
\[
L_n = \left\{ -\frac{1}{n}, -\frac{3}{n}, -\frac{5}{n},\dots\right\}.
\]
Define $V_n^q(x)$ as
\begin{equation}\label{em1}
\begin{aligned}
V_n^q(x)&=-n^{-1} \log \left[e^{n(t-|\ga|) x}\left(1-e^{2n|\ga| x}e^{-2|\ga|}\right) \right] \\
&=-(t-\ga) x + \frac{1}{n}\log\left(1-e^{2n|\ga| x}e^{-2|\ga|}\right).
\end{aligned}
\end{equation}
This function extends to a complex analytic function $V_n^q(z)$ on the left half-plane.
In the limit $n\to\infty$ the function $V^q_n(z)$ converges to the function
\begin{equation} \label{eq_V_appendix}
V(z)=-(t-\ga) z.
\end{equation}
More precisely, note that as $n\to\infty$,
\begin{equation}\label{eq:Vnq_V}
e^{-nV_n^q(z)} = e^{-nV(z)} \left(1-e^{2n|\ga| z}e^{-2|\ga|}\right) = \left\{
\begin{aligned}
&e^{-nV(z)} (1+\bigO(e^{-2n|\ga|\,|\Re z|})), \quad \Re z<0, \\
&e^{-nV(z)}\bigO(1), \quad \Re z = 0.
\end{aligned}\right.
\end{equation}

In what follows we will use the properties \eqref{eq:eq_condition_band}, \eqref{eq:equilibrium_lessthanF}, and \eqref{eq:eq_condition_sat}  for the function $g_{\rm F}(z)$, and we will drop the subscript F throughout this Appendix, denoting $g_{\rm F}(z) \equiv g(z)$. We additionally use the property
\eq\label{eq:G_int}
g_+(x) - g_-(x) = 2\pi \ii \int_x^0 \nu(\dd t), \quad x<0,
\eeq
which follows from the definition \eqref{def:g-functionF}.
Denote the density for the equilibrium measure $\nu$ as $\rho(x)$. We will use the following properties of $\rho(x)$ throughout this section.
\begin{itemize}
\item The density $\rho(x)$ is real analytic for $x\in(\be,\al)$, vanishing like a square root as $x\to\be_+$ and approaching the value 1/2 like a square root as $x\to\al_-$.
\item On the interval $(\be,\al)$, $\rho(x)$ satisfies the inequality
\eq
0<\rho(x) < 1/2, \quad \textrm{for} \ \be<x<\al.
\eeq
\item On the interval $[\al,0]$, the density is constant, $\rho(x) \equiv 1/2$.
\end{itemize}
 Note then that for $x\in[\al,0]$, \eqref{eq:G_int} becomes
\eq\label{eq:G_sat}
g_+(x) - g_-(x) = 2\pi \ii \frac{0-x}{2} = -\ii \pi x, \quad \al\le x\le 0.
\eeq
Since $\rho(x)$ is real analytic on the interval $(\be,\al)$, it extends to a complex analytic function $\rho^{a}(z)$ with cuts on the intervals $(-\infty,\be)$ and $(\al,\infty)$. For some $\ep>0$ we therefore may define for $z$ in  $\{[\be-\ep, \al+\ep]\times[-\ii \ep, \ii\ep] \}\setminus  \{(-\infty,\be) \cup (\al,\infty)\}$
\eq\label{eq:Gz_def}
G(z):= -\ii \pi \al - 2\pi \ii \int_\al^z \rho^a(w)\, \dd w,
\eeq
which gives an analytic extension of $g_+(x)- g_-(x)$ for $\be\le x \le \al$.
The above formula implies
\eq\label{eq:G_CR}
\left.\frac{\dd G(x+\ii y)}{\dd y} \right|_{y=0} = 2\pi \rho(x)>0, \quad \textrm{for} \ \be<x<\al.
\eeq
Finally, note that the equilibrium condition \eqref{eq:eq_condition_band} implies that
\eq\label{eq:G_gpgm}
G(x)=g_+(x) - g_-(x) = -(2g_-(x) - V(x) - l) = 2g_+(x) - V(x) - l, \quad \textrm{for} \ \be<x<\al.
\eeq
Furthermore, the above equations extend by analytic continuation into the lower and upper half-planes, so that for $z$ in a complex neighborhood of $(\be,\al)$,
\eq\label{eq:G_upper_lower}
G(z)=\left\{
\begin{aligned}
&2g(z) - V(z) - l \quad \textrm{for} \  \Im z > 0 \\
&-(2g(z) - V(z) - l)\quad \textrm{for} \  \Im z < 0.
\end{aligned}\right.
\eeq

The orthogonal polynomials $q_{n,n}(z)$ and $q_{n,n-1}(z)$ are encoded in the following interpolation problem (IP).
Find a $2\times 2$ matrix-valued function
$\mathbf Q_n(z)=(\mathbf Q_{n}(z)_{ij})_{1\le i,j\le 2}$ with the properties:
\begin{enumerate}
\item
{\it Analyticity}: $\mathbf Q_n(z)$ is an analytic function of $z$ for $z\in\C\setminus L_n$.
\item
{\it Residues at poles}: At each node $x\in L_n$, the elements $\mathbf Q_{n}(z)_{11}$ and
$\mathbf Q_{n}(z)_{21}$ of the matrix $\mathbf Q_n(z)$ are analytic functions of $z$, and the elements $\mathbf Q_{n}(z)_{12}$ and
$\mathbf Q_{n}(z)_{22}$ have a simple pole with the residues,
\begin{equation} \label{IP1}
\underset{z=x}{\rm Res}\; \mathbf Q_{n}(z)_{j2}=\frac{1}{n}e^{-nV_n^q(z)}\mathbf Q_{n}(x)_{j1},\quad j=1,2.
\end{equation}
\item
{\it Asymptotics at infinity}: There exists a function $r(x)>0$ on  $L_n$ such that
\begin{equation} \label{IP2a}
\lim_{x\to\infty} r(x)=0,
\end{equation}
and such that as $z\to\infty$, $\mathbf Q_n(z)$ admits the asymptotic expansion,
\begin{equation} \label{IP2}
\mathbf Q_n(z)\sim \left( {\bf I} + \bigO(1/z)\right)
\begin{pmatrix}
z^n & 0 \\
0 & z^{-n}
\end{pmatrix},\qquad z\in \C\setminus \left[\bigcup_{x\in L_n}^\infty D\big(x,r(x)\big)\right],
\end{equation}
where $D(x,r(x))$ denotes a disk of radius $r(x)>0$ centered at $x$ and ${\bf I}$ is the identity matrix.
\end{enumerate}

The unique solution to the IP is
\begin{equation} \label{IP3}
\mathbf Q_n(z)=
\begin{pmatrix}
 q_{n,n}(z) & \frac{1}{n}\sum_{x\in L_n}\frac{ q_{n,n}(x)e^{-nV_n^q(x)}}{z-x} \\
\frac{1}{h_{n,n-1}^q} q_{n,n-1}(z) & \frac{1}{n h_{n,n-1}^q}\sum_{x\in L_n}\frac{ q_{n,n-1}(x)e^{-nV_n^q(x)}}{z-x}
\end{pmatrix},
\end{equation}
see, e.g., \cite{Bleher-Liechty11}. Note that
\eq\label{eq:Q_p_Cauchy}
{\bf Q}_n(z)_{22} = \frac{1}{n h_{n,n-1} e^{t-|\ga|}} \sum_{x\in \frac{2}{n} \Z_{<0}} \frac{p_{n,n-1}(x) w_n^{\rm F}(x)e^{t-|\ga|}}{z-x-1/n} = -2\pi \ii \frac{C^{\mt_n}(p_{n,n-1})(z-1/n)}{h_{n,n-1}},
\eeq
and
\eq\label{eq:Q_p_OP}
{\bf Q}_n(z)_{21} = \frac{p_{n,n-1}(z-1/n)}{h_{n,n-1} e^{t-|\ga|}}.
\eeq

{\bf Reduction of IP to RHP.} The Interpolation Problem can be reduced to a Riemann--Hilbert Problem (RHP) by multiplying by a matrix-valued function which will cancel the poles of ${\bf Q}_n(z)$, and instead introduce jumps on certain contours in the complex plane. We introduce the function
\begin{equation} \label{redp1}
\Pi(z)=\frac{2\cos (n \pi z/2)}{n \pi}\,.
\end{equation}
Observe that
\begin{equation} \label{redp3}
\Pi(x_k)=0,\quad \Pi'(x_k)=(-1)^{k+1}=\exp\left[\frac{i \pi}{2}\left(n x_k+1\right)\right]
\quad \textrm{for}\quad x_k=\frac{-2k+1}{n}\in L_n\,.
\end{equation}
Introduce the upper and lower triangular matrices,
\begin{equation} \label{redp4}
\mathbf D^u_{\pm}(z)
=\begin{pmatrix}
1 & -\frac{e^{-nV_n^q(z)}}{n\Pi(z)}e^{\pm \frac{i \pi}{2}(n z+1)}   \\
0 & 1
\end{pmatrix}, \quad \mathbf D^l_{\pm}
=\begin{pmatrix}
\Pi(z)^{-1} & 0   \\
-ne^{nV_n^q(z)}e^{\pm \frac{i \pi}{2}(n z+1)} & \Pi(z)
\end{pmatrix},
\end{equation}
and define the matrix-valued functions (two-valued on the real line),
\begin{equation} \label{redp6}
\mathbf R^u_n(z)=\mathbf Q_n(z)\times
\left\{
\begin{aligned}
& \mathbf D^u_+(z),\quad \textrm{when}\quad \Im z\ge0,\\
& \mathbf D^u_-(z),\quad \textrm{when}\quad \Im z\le0,
\end{aligned}
\right.
\end{equation}
and
\begin{equation} \label{redp7}
\mathbf R^l_n(z)=\mathbf Q_n(z)\times
\left\{
\begin{aligned}
& \mathbf D^l_+(z),\quad \textrm{when}\quad \Im z\ge0,\\
& \mathbf D^l_-(z),\quad \textrm{when}\quad \Im z\le0.
\end{aligned}
\right.
\end{equation}
The functions $\mathbf R_n^u(z)$, $\mathbf R_n^l(z)$ are then pole free (see e.g., \cite[Proposition 3.5.1]{Bleher-Liechty14}), and have the following jumps on the negative real axis.  For $x\in \R_-$,
\begin{equation} \label{redp9a}
\mathbf R_{n+}^u(x)=\mathbf R_{n-}^u(x) \begin{pmatrix} 1 & -\ii \pi  e^{-nV_n^q(x)}\, \\
0 & 1
\end{pmatrix},\qquad \mathbf R_{n+}^l(x)=\mathbf R_{n-}^l(x)\begin{pmatrix}
1 & 0 \\
-\ii\pi  e^{nV_n^q(x)} & 1
\end{pmatrix}.
\end{equation}

 {\bf The transformations of the RHP.} For a small fixed $\ep>0$, let $\Om^{\rm b}_\pm$ be the region bounded by the quadrilateral with vertices at the points $\be,\al, \be+\ep\pm \ii\ep,$ and $\al-\ep \pm \ii\ep$; let $\Om^{\rm s}_\pm$ be the region bounded by the quadrilateral with vertices at the points $\al, \al-\ep \pm \ii\ep, \pm\ii \ep,$ and 0; and let  $\Om^{\rm v}_\pm$ be the region bounded by the line segment $(\be, \be + \ep\pm\ii \ep)$ and the two half-lines $(-\infty, \be)$ and $(-\infty\pm \ii \ep, \be \pm \ii \ep)$, see Figure \ref{contour_Sigma_S}\footnote{The superscripts b, s, and v refer to {\it bands}, {\it saturated regions}, and {\it voids} as described in \cite{Bleher-Liechty11}, \cite[Chapter 3]{Bleher-Liechty14}}. Now define the function ${\bf S}_n(z)$ as\footnote{This combines the first and second transformations of the Riemann--Hilbert problem as presented in \cite{Bleher-Liechty11}, \cite[Chapter 3]{Bleher-Liechty14}.}
\begin{equation}\label{stg2}
\mathbf S_n(z)=\left\{
\begin{aligned}
&e^{-\frac{nl}{2}\sg_3} {\bf K} {\bf R}_n^u(z) {\bf K}^{-1} e^{-n(g(z) - \frac{l}{2})\sg_3}  \begin{pmatrix}1 & 0 \\ \mp e^{\mp nG(z)} & 1 \end{pmatrix} \qquad \textrm{for} \quad z\in \Om^{\rm b}_\pm, \\
&e^{-\frac{nl}{2}\sg_3} {\bf K} {\bf R}_n^l(z) {\bf K}^{-1} e^{-n(g(z) - \frac{l}{2})\sg_3}  \begin{pmatrix} \mp \frac{1}{\ii n \pi} e^{\mp \frac{\ii \pi}{2}(nz+1)} & 0 \\ 0 &  \mp \ii n \pi e^{\pm \frac{\ii \pi}{2}(nz+1)} \end{pmatrix} \qquad \textrm{for} \quad z\in\Om^{\rm s}_\pm, \\
&e^{-\frac{nl}{2}\sg_3}  {\bf K} {\bf R}_n^u(z) {\bf K}^{-1} e^{-n(g(z) - \frac{l}{2})\sg_3}  \qquad \textrm{for} \quad z\in\Om^{\rm v}_\pm, \\
&e^{-\frac{nl}{2}\sg_3}  {\bf K} {\bf Q}_n(z) {\bf K}^{-1} e^{-n(g(z) - \frac{l}{2})\sg_3}  \qquad \textrm{otherwise},
\end{aligned}\right.
\end{equation}
where $\mathbf K=\begin{pmatrix} 1 & 0 \\ 0 & -\pi \ii \end{pmatrix}$.
This function satisfies a RHP  with jumps on an oriented contour $\Sigma_S$, consisting of the half-lines $(-\infty, 0]$ and $\{(-\infty, \be+\ep) \pm \ii \ep\}$ as well as the segments $(\be, \be +\ep\pm \ii \ep)$, $(\al- \ep\pm\ii \ep, \al)$, $(\be +\ep\pm \ii \ep,\al- \ep\pm\ii \ep)$, and $(-\ii \ep,  \ii \ep)$, oriented as shown in Figure \ref{contour_Sigma_S}. As we show below, the matrix ${\bf S}_n(z)$ satisfies jump properties which are close to constant as $n\to\infty$ except near the points $\be$, $\al$, and 0. We can then approximate a solution to this Riemann--Hilbert problem by combining solutions to a global Riemann--Hilbert problem with constant jumps (the model problem) and local Riemann--Hilbert problems near the points $\be$, $\al$, and 0.
\begin{figure}\
\begin{center}
\begin{tikzpicture}[scale=1.3]
\draw[very thick] (-6,0) --(5,0)
(5.2,0) node {$0$}
(5,1.2) node {$\ii\ep$}
(5,-1.2) node {$-\ii\ep$};
\draw[very thick] (-1,1) --(5,1);
\draw[very thick] (-6,1) --(-2,1);
\draw[very thick] (-3,0) --(-2,1);
\draw[very thick] (-3,0) --(-2,-1);
\draw[very thick] (-2,1) --(0,1);
\draw[very thick] (-1,-1) --(5,-1);
\draw[very thick] (-6,-1) --(-2,-1);
\draw[very thick] (-2,-1) --(0,-1);
\draw[very thick] (5,1) --(5,-1);
\draw[very thick] (0,1) --(1,0);
\draw[very thick] (0,-1) --(1,0);
\draw[very thick] (-.1,0) --(.1,0)
(1,-0.3) node {$\al$}
(0,1.2) node {$\al-\ep+\ii \ep$}
(0,-1.2) node {$\al-\ep-\ii \ep$}
(-1,0.5) node {$\Om_+^{\rm b}$}
(-1,-0.5) node {$\Om_-^{\rm b}$}
(3,0.5) node {$\Om_+^{\rm s}$}
(3,-0.5) node {$\Om_-^{\rm s}$}
(-5,0.5) node {$\Om_+^{\rm v}$}
(-5,-0.5) node {$\Om_-^{\rm v}$};
\draw[very thick] (-3,0) --(-3.1,0)
(-3,-0.3) node {$\be$}
(-2,1.2) node {$\be+\ep+\ii \ep$}
(-2,-1.2) node {$\be+\ep-\ii \ep$};
\draw[-Straight Barb, very thick] (3,0) --(3.1,0);
\draw[-Straight Barb, very thick] (-1.5,0) --(-1.4,0);
\draw[-Straight Barb, very thick] (-5.1,0) --(-5,0);
\draw[-Straight Barb, very thick] (3,-1) --(3.1,-1);
\draw[-Straight Barb, very thick] (-1.3,-1) --(-1.2,-1);
\draw[-Straight Barb, very thick] (-5.2,-1) --(-5.1,-1);
\draw[-Straight Barb, very thick] (3.1,1) --(3,1);
\draw[-Straight Barb, very thick] (-1.4,1) --(-1.5,1);
\draw[-Straight Barb, very thick] (-5.2,1) --(-5.3,1);
\draw[-Straight Barb, very thick] (5,-0.6) --(5,-0.5);
\draw[-Straight Barb, very thick] (5,0.5) --(5,0.6);
\draw[-Straight Barb, very thick] (-2.6,-0.4) --(-2.5,-0.5);
\draw[-Straight Barb, very thick] (-2.5,0.5) --(-2.6,0.4);
\draw[-Straight Barb, very thick] (0.5,-0.5) --(0.6,-0.4);
\draw[-Straight Barb, very thick]  (0.6,0.4) --(0.5,0.5);

%

    \end{tikzpicture}

\end{center}
\caption{The oriented contour $\Sg_S$ and the regions $\Om_\pm^{\rm v}$,  $\Om_\pm^{\rm b}$, and  $\Om_\pm^{\rm s}$ bounded by it.}\label{contour_Sigma_S}
\end{figure}
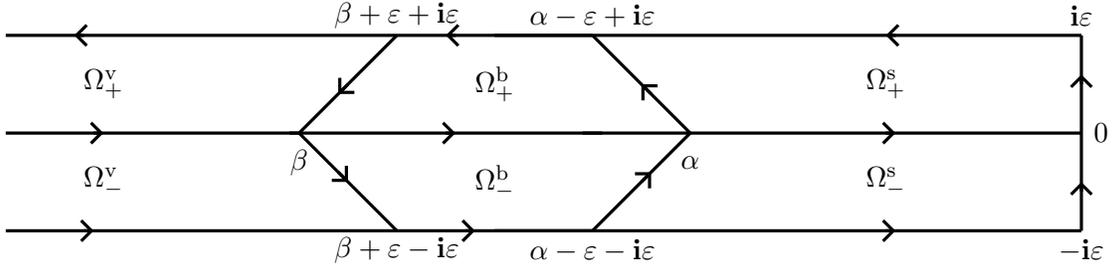

Recall that the $+$ (resp. $-$) side of the contour $\Sigma_S$ is the side on the left (resp. right) of the contour as it is traversed in the the direction of orientation shown in Figure \ref{contour_Sigma_S}. For $z\in \Sg_S$ we denote by $\mathbf S_{n\pm}(z)$ the limiting value of ${\bf S}_n$ as $z$ is approached from the $\pm$-side of the contour.  These values are related by   the jump conditions
\begin{equation}\label{st2a}
\mathbf S_{n+}(z)=\mathbf S_{n-}(z)j_S(z)\,,
\end{equation}
where
\begin{equation}\label{stg3}
j_S(z)=\left\{
\begin{aligned}
& \begin{pmatrix} e^{-nG(z)} - e^{n(2g_-(z) - l - V_n^q(z))} & e^{n(g_+(z) + g_-(z) - l - V_n^q(z))} \\ e^{n(g_+(x) + g_-(z) - l - V_n^q(z))} -2 & e^{n(G(z)} - e^{n(2g_+(z) - l - V_n^q(z))} \end{pmatrix}
 \ \textrm{for} \ z\in (\be,\al), \\
 &\begin{pmatrix} e^{-n(g_+(z) - g_-(z) + \ii \pi z)} & 0 \\ -e^{-n(g_{+}(z)+g_{-}(z)-l-V_n^q(z))} & e^{n(g_+(z) - g_-(z) + \ii \pi z)}  \end{pmatrix} \ \textrm{for} \ z \in \left(\al,0 \right), \\
 &\begin{pmatrix}1 & e^{n(g_{+}(z)+g_{-}(z)-l-V_n^q(z))} \\ 0 &  1 \end{pmatrix}
\qquad \textrm{for} \ z \in(-\infty,\be), \\
& \begin{pmatrix}1-\frac{e^{\mp nG(z)}e^{n(2g(z)-l-V_n^q(z))}}{1+e^{\mp \pi \ii n z}} & \pm \frac{e^{ n(2g(z) - l -V_n^q(z))}}{1+e^{\mp \pi \ii nz}}  \\ \mp e^{\mp nG(z)} & 1\end{pmatrix} \ \textrm{for} \ z\in (\be+\ep,\al-\ep) \pm \ii \ep, \\
& \begin{pmatrix} 1 & \pm \frac{e^{n(2g(z)-l-V_n^q(z))}}{1+e^{\mp \pi \ii nz}} \\ 0 & 1 \end{pmatrix} \ \textrm{for} \ z \in  (-\infty,\be+\ep) \pm \ii \ep, \\
&\begin{pmatrix} (1+e^{\pm \pi \ii n z})^{-1} & 0  \\ \mp e^{-n(2g(z)-V_n^q(z)-l)} &1+e^{\pm \pi \ii n z}\end{pmatrix} \qquad \textrm{for} \ z \in\left\{\left(\al-\ep,0 \right) \pm \ii \ep\right\} \cup \left(0,\pm \ii \ep\right), \\
&\begin{pmatrix} 1+e^{\pm \ii n\pi z}(1-e^{\mp nG(z)}e^{n(2g(z) - l -V_n^q(z))}) & \pm e^{n(2g(z) - V_n^q(z) - l\pm \ii \pi z)}  \\
\pm e^{-n(2g(z) -V_n^q(z) - l)} \mp e^{\mp nG(z)} &1
\end{pmatrix} \qquad \textrm{for} \ z \in\left(\al,\al-\ep\pm \ii\ep \right), \\
&\begin{pmatrix} 1 & 0  \\
\mp e^{\mp nG(z)} &1
\end{pmatrix} \qquad \textrm{for} \ z \in\left(\be,\be+\ep\pm \ii \ep \right). \\
\end{aligned}\right.
\end{equation}
Using \eqref{eq:Vnq_V} we can replace $e^{-nV_n^q(z)}$ with $e^{-nV(z)}$ throughout. We also use the relations \eqref{eq:G_gpgm} and \eqref{eq:G_upper_lower} between $G(z)$ and $g(z)$ to obtain, for some $c>0$,
\begin{equation}\label{stg3a}
j_S(z)=\left\{
\begin{aligned}
& \begin{pmatrix} e^{-nG(z)}\bigO(e^{-cn})  & 1+\bigO(e^{-cn}) \\ -1+\bigO(e^{-cn}) & e^{nG(z)}\bigO(e^{-cn}) \end{pmatrix}
 \qquad \textrm{for} \ z\in (\be,\al), \\
 &\begin{pmatrix} 1& 0 \\ -e^{-n(g_{+}(z)+g_{-}(z)-l-V(z))}(1+\bigO(e^{-2n|\ga|\,| z|})) & 1  \end{pmatrix} \qquad \textrm{for} \ z \in \left(\al,0 \right), \\
 &\begin{pmatrix}1 & e^{n(g_{+}(z)+g_{-}(z)-l-V(z))}\left(1+\bigO(e^{-cn})\right) \\ 0 &  1 \end{pmatrix}
\qquad \textrm{for} \ z \in(-\infty,\be), \\
& \begin{pmatrix}1-\frac{1+\bigO(e^{-cn})}{1+e^{\mp \pi \ii n z}} & \pm \frac{e^{ n(2g(z) - l -V(z))}\left(1+\bigO(e^{-cn})\right)}{1+e^{\mp \pi \ii nz}}  \\ \mp e^{\mp nG(z)} & 1\end{pmatrix} \qquad \textrm{for} \ z\in (\be+\ep,\al-\ep) \pm \ii\ep, \\
& \begin{pmatrix} 1 & \pm \frac{e^{n(2g(z)-l-V(z))}\left(1+\bigO(e^{-cn})\right)}{1+e^{\mp \pi \ii nz}} \\ 0 & 1 \end{pmatrix} \qquad \textrm{for} \ z \in  (-\infty,\be+\ep) \pm \ii \ep, \\
&\begin{pmatrix} (1+e^{\pm \pi \ii n z})^{-1} & 0  \\ \mp e^{-n(2g(z)-V(z)-l)}\bigO(1) &1+e^{\pm \pi \ii n z}\end{pmatrix} \qquad \textrm{for} \ z \in\left\{\left(\al-\ep,0 \right) \pm \ii \ep\right\} \cup \left(0,\pm \ii \ep\right), \\
&\begin{pmatrix} 1+e^{\pm \ii n\pi z}\bigO(e^{-cn}) & \pm e^{n(2g(z) - V(z) - l\pm \ii \pi z)}(1+\bigO(e^{-cn}))  \\
\pm e^{-n(2g(z) -V(z) - l)}\bigO(e^{-cn})&1
\end{pmatrix} \\
&\hspace{3cm} \textrm{for} \ z \in\left(\al,\al-\ep\pm \ii \ep \right), \\
&\begin{pmatrix} 1 & 0  \\
\mp e^{\mp nG(z)} &1
\end{pmatrix} \qquad \textrm{for} \ z \in\left(\be,\be+\ep\pm \ii \ep \right). \\
\end{aligned}\right.
\end{equation}

We choose $\ep>0$ small enough so that
\begin{itemize}
\item 
For $z= x\pm \ii \ep$ with $x\le\be+\ep$, we have $2\,\Re g(z) - \Re V(z) -l <  \pi \ep$. For large negative $x$ this inequality follows from the fact that $g(z) \sim \ln|z|$ and $V(z)$ is linear in $z$, so $V(z)$ dominates for large $|z|$; for $z$ bounded away from $\infty$ and from $\be$ the inequality follows for small enough $\ep$ by extending the inequality\eqref{eq:equilibrium_lessthanF} away from the real line by continuity. For $|x-\be|<\ep$ we use the fact that $2g(x\pm \ii \ep) - V(x\pm \ii \ep) - l = \pm G(x\pm \ii \ep)$, so \eqref{eq:Gz_def} gives
\begin{multline}\label{eq:G_beta_lessthan}
\Re \left(\pm G(x\pm \ii \ep)\right) = \Re\left(\mp 2\pi \ii \int_\be^{x\pm \ii \ep} \rho^a(w)\, \dd w\right) \le 2\pi |\be -(x\pm \ii \ep)| \sup_{w\in (\be, x\pm \ii \ep)}\{|\rho^a(w)|\} \\
\le 2\sqrt{2} \pi \ep\sup_{w\in (\be, x\pm \ii \ep)}\{|\rho^a(w)|\}.
\end{multline}
Since $\rho^a(\be) = 0$ so we may choose $\ep$ small enough so that $|\rho^a(z)|<1/\sqrt{200}$ for all $z$ such that $|z-\be|<\sqrt{2}\ep$, in which case \eqref{eq:G_beta_lessthan} gives
\[
\Re G(x\pm \ii \ep) \le \frac{\pi \ep}{5}, \quad x\in (\be-\ep, \be+\ep).
\]
\item The inequality
\eq\label{eq:G_band_inequalities}
0<\pm  \Re G(x \pm \ii \ep) - \pi \ep<0,
 \eeq
 holds for $x$ in the interval $(\be+\ep,\al-\ep)$. These inequalities are satisfied for small enough $\ep$ by \eqref{eq:G_CR}, which implies $\pm \Re G(x\pm \ii \ep)  = 2\pi \rho(x)\ep + \bigO(\ep^2)<\pi \ep+\bigO(\ep^2)$, as $\rho(x)$ is strictly between 0 and 1/2 on this interval.
 \item The inequality
\eq\label{eq:G_sat_inequalities}
\Re(2g(x\pm \ii \ep) - V(x\pm \ii \ep) - l)>0
 \eeq
 holds for $x$ in the interval $(\al-\ep\pm \ii\ep,\pm \ii\ep)$. For $x\in(\al+\ep,0)$ this inequality follows from extending the inequality \eqref{eq:eq_condition_sat} away from the real line by continuity. For $x\in (\al-\ep,\al+\ep)$ it follows by using \eqref{eq:Gz_def} and \eqref{eq:G_upper_lower} which imply, similar to \eqref{eq:G_beta_lessthan},
 \begin{multline}\label{eq:G_alpha_greater_than}
\Re(2g(x\pm \ii \ep) - V(x\pm \ii \ep) - l)=\Re\left(\pm G(x\pm \ii \ep)\right) = \Re\left(\mp 2\pi \ii \int_x^{x\pm \ii \ep} \rho^a(w)\, \dd w\right) \\
=\Re\left(\pm 2\pi \int_0^{\pm \ep} \rho^a(x+\ii y)\, \dd y\right)
= 2\pi \int_0^{ \ep}\Re \rho^a(x\pm \ii y)\, \dd y
\end{multline}
Since $\rho^a(\al) = 1/2$ so we may choose $\ep$ small enough so that $\Re \rho^a(z)>\frac{1}{4}$ for all $z$ such that $|z-\al|<\sqrt{2}\ep$, in which case \eqref{eq:G_alpha_greater_than} gives
\[
\Re G(x\pm \ii \ep) \ge \frac{\pi \ep}{4}, \quad x\in (\al-\ep, \al+\ep).
\]

 \end{itemize}
The properties \eqref{eq:eq_condition_band}, \eqref{eq:equilibrium_lessthanF}, \eqref{eq:eq_condition_sat}, and \eqref{eq:G_int} -- \eqref{eq:G_gpgm} then imply that $j_S(z)$ is uniformly close to the identity matrix on all parts of the contour bounded away from $(\be, \al) \cup \{0\}$, though not on the diagonally sloped segments near $\al$ and $\be$ or on the vertical segment near 0.   Introduce the function $\tilde j_S(z)$ on the same contour as
\begin{equation}\label{stg3b}
\tilde j_S(z)=\left\{
\begin{aligned}
& \begin{pmatrix} 0 & 1 \\ -1 & 0 \end{pmatrix}
 \ \textrm{for} \ z\in (\be,\al), \\
 &\begin{pmatrix} 1& 0 \\  -e^{-n(g_{+}(z)+g_{-}(z)-l-V(z))} & 1  \end{pmatrix} \ \textrm{for} \ z \in \left(\al,\al+\ep \right), \\
 &\begin{pmatrix}1 & e^{n(g_{+}(z)+g_{-}(z)-l-V(z))} \\ 0 &  1 \end{pmatrix}
\quad \textrm{for} \ z \in(\be-\ep,\be), \\
&\begin{pmatrix} (1+e^{\pm \pi \ii n z})^{-1} & 0  \\ 0 &1+e^{\pm \pi \ii n z}\end{pmatrix} \ \textrm{for} \ z \in \left(0,\pm \ii\ep\right), \\
&\begin{pmatrix} 1  & \pm e^{n(2g(z) - V(z) - l\pm \ii\pi z)}  \\
0 &1
\end{pmatrix}
\quad  \textrm{for} \ z \in\left(\al,\al-\ep\pm \ii \ep \right), \\
&\begin{pmatrix} 1 & 0  \\ \mp e^{\mp nG(z)} & 1 \end{pmatrix}\quad \textrm{for} \ z\in (\be,\be+\ep \pm \ii\ep), \\
& {\bf I} \qquad \textrm{otherwise}.
\end{aligned}\right.
\end{equation}
We then have as $n\to\infty$, uniformly for $z\in \Sg_S$,
\eq\label{eq:S_tildeS}
j_S(z) = \tilde j_S(z) + \bigO(e^{-cn}), \quad c>0.
\eeq

\begin{remark}
 In \eqref{stg3b}, the subleading corrections making $V_n^q(x)$ of \eqref{em1} different from $V(x)$ of \eqref{eq_V_appendix} are no longer visible. So the fact that we managed to convert our problem to the setting with jump matrix \eqref{stg3b} is a rigorous analogue of Conjecture \ref{Conjecture_concentration_case1} in the Riemann--Hilbert-based approach we follow here. In this way, the rest of the proof is a Riemann--Hilbert analogue of the conditional proof of Theorem \ref{Theorem_asymptotics_case1} at the end of Section \ref{Section_asymptotics_through_log_gas}. The Riemann--Hilbert problem with jumps of \eqref{stg3b} can be connected to classical Meixner orthogonal polynomials, and our arguments have parallels with the RH analysis of the latter in \cite{Wang-Wong11}, though their approach is slightly different. In \cite{Wang-Wong11} the authors combine the global and local solutions to the RHP to give an asymptotic solution which is valid on larger regions in $\C$, at the expense of having more complicated formulas.  
\end{remark}

We now introduce a global function ${\bf M}(z)$ with a jump matching $\tilde j_S(z)$ on the interval $(\be,\al)$, and local solutions with jumps matching  $\tilde j_S(z)$ in small neighborhoods of $\be$, $\al$, and 0.

{\bf The solution to the model problem.}
The model RHP is
\begin{enumerate}
  \item
  $\mathbf M(z)$ is analytic in $\C \setminus \left[\be,\al\right]$.
  \item For $z \in(\be,\al)$, the matrix function $\mathbf M(z)$ satisfies the jump condition
  \begin{equation}\label{mg1}
   \mathbf M_{+}(z)=\mathbf M_{-}(z)
  \begin{pmatrix} 0 & 1 \\ -1 & 0 \end{pmatrix}.  \\
  \end{equation}
 \item As $z \to \infty, \
\mathbf M(z)\sim {\bf I} + \bigO(1/z)$.
\end{enumerate}
The solution is
  \begin{equation}\label{mg3}
  \mathbf M(z) = \begin{pmatrix} \frac{\ga(z)+ \ga(z)^{-1}}{2} &  \frac{\ga(z)- \ga(z)^{-1}}{-2i} \\  \frac{\ga(z)- \ga(z)^{-1}}{2i} & \frac{\ga(z)+ \ga(z)^{-1}}{2}\end{pmatrix},
  \end{equation}
where
  \begin{equation}\label{mg3b}
\ga(z) = \left(\frac{z-\al}{z-\be}\right)^{1/4}\,,
\end{equation}
with a cut on $(\be,\al)$ taking the branch such that $\ga(\infty)=1$.

{\bf The local solution at $\al$ and $\be$.}
Consider small disks $D(\al, \ep)$ and $D(\be, \ep)$ around $\al$ and $\be$.  We seek a local parametrix $\mathbf U_n(z)$ in these disks satisfying:
\begin{itemize}
\item $\mathbf U_n(z)$ is analytic in $\{D(\al, \ep) \cup D(\be, \ep)\} \setminus \Sg_S$.
\item For $z\in \{D(\al, \ep) \cup D(\be, \ep)\} \cap \Sg_S$, $\mathbf U_n(z)$ satisfies the jump conditions
\eq\label{eq:U_tildeS}
\mathbf U_{n+}(z)=\mathbf U_{n-}(z)\tilde j_S(z).
\eeq
\item On the boundary of the disks, $\mathbf U_n(z)$ satisfies
\begin{equation}\label{pm1}
\mathbf U_n(z)=\mathbf M(z) \left(I+\bigO(n^{-1})\right)\,, \qquad z\in \d D(\al, \ep) \cup \d D(\be, \ep)\,.
\end{equation}
\end{itemize}
The solution is given explicitly in terms of Airy functions (see \cite[Section 3.9 and 3.10]{Bleher-Liechty14}, and we do not describe it here. We do note that $\mathbf U_n(z)$ is uniformly bounded in $n$ and $z$, since it satisfies a Riemann--Hilbert problem with uniformly bounded jump and boundary conditions.

{\bf The local solution at the origin.}
Introduce the function
\begin{equation}\label{eq:def_D}
D(z):=\left\{
\begin{aligned}
& \frac{\Ga\left(\frac{nz}{2}+\frac{3}{2}\right)e^{nz/2}}{\sqrt{2\pi}\left(\frac{nz}{2}\right)^{nz/2+1}}\,, \quad \textrm{for} \ \Re z>0\,, \\
& \frac{\sqrt{2\pi}e^{nz/2}}{\Ga\left(-\frac{nz}{2}-\frac{1}{2}\right)\left(-\frac{nz}{2}\right)^{nz/2+1}}\,, \quad \textrm{for} \ \Re z<0\,.
\end{aligned}\right.
\end{equation}
This function has the following properties:
\begin{itemize}
\item For $z\in \ii\R$, the function $D$ has the multiplicative jump
\begin{equation} \label{eq:D_jump}
D_+(z)=D_-(z)\times\left\{
\begin{aligned}
&(1+e^{\ii n\pi z})\,, \quad \Im z>0, \\
&(1+e^{-\ii n\pi z})\,, \quad \Im z<0, \\
\end{aligned}\right.
\end{equation}
where the imaginary axis is oriented upward, i.e., the $+$-side of the imaginary axis is to the left.
\item As $n\to \infty$,
\eq\label{eq:D_bound}
D(z) = 1+\bigO(n^{-1}), \quad |z|  \ge \ep.
\eeq
\end{itemize}
The first property follows from the reflection formula for the Gamma function,
and the second follows from Stirling's formula. 

{\bf The final transformation of the RHP}
We now consider the contour $\Sigma_X$, which consists of the circles $\d D(\al, \ep)$, $\d D(\be, \ep)$, and $\d D(0,\ep)$, oriented counterclockwise, together with $\Sigma_S$, see Figure \ref{contour_Sigma_X}.
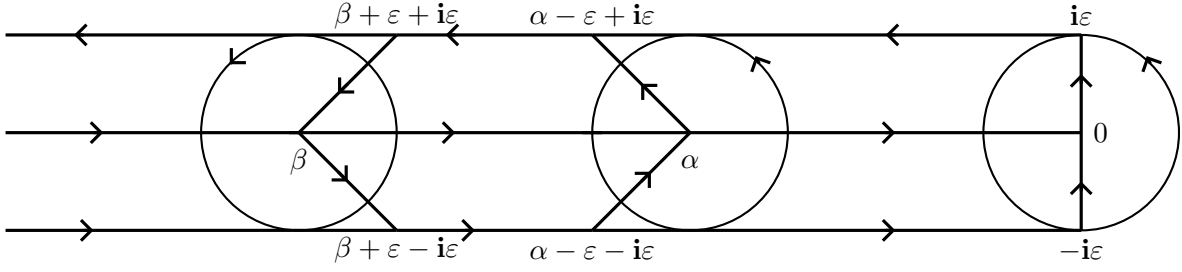
\begin{figure}
\begin{center}
\begin{tikzpicture}[scale=1.3]
\draw[thick] (5,0) circle (1);
\draw[thick] (1,0) circle (1);
\draw[thick] (-3,0) circle (1);
%
%
%

\draw[very thick] (-6,0) --(5,0)
(5.2,0) node {$0$}
(5,1.2) node {$\ii\ep$}
(5,-1.2) node {$-\ii\ep$};
\draw[very thick] (0,1) --(5,1);
\draw[very thick] (-6,1) --(-2,1);
\draw[very thick] (-3,0) --(-2,1);
\draw[very thick] (-3,0) --(-2,-1);
\draw[very thick] (-2,1) --(0,1);
\draw[very thick] (0,-1) --(5,-1);
\draw[very thick] (-6,-1) --(-2,-1);
\draw[very thick] (-2,-1) --(0,-1);
\draw[very thick] (5,1) --(5,-1);
\draw[very thick] (0,1) --(1,0);
\draw[very thick] (0,-1) --(1,0);
\draw[very thick] (-.1,0) --(.1,0)
(1,-0.3) node {$\al$}
(0,1.2) node {$\al-\ep+\ii \ep$}
(0,-1.2) node {$\al-\ep-\ii \ep$};
\draw[very thick] (-3,0) --(-3.1,0)
(-3,-0.3) node {$\be$}
(-2,1.2) node {$\be+\ep+\ii \ep$}
(-2,-1.2) node {$\be+\ep-\ii \ep$};
\draw[-Straight Barb, very thick] (3,0) --(3.1,0);
\draw[-Straight Barb, very thick] (-1.5,0) --(-1.4,0);
\draw[-Straight Barb, very thick] (-5.1,0) --(-5,0);
\draw[-Straight Barb, very thick] (3,-1) --(3.1,-1);
\draw[-Straight Barb, very thick] (-1.3,-1) --(-1.2,-1);
\draw[-Straight Barb, very thick] (-5.2,-1) --(-5.1,-1);
\draw[-Straight Barb, very thick] (3.1,1) --(3,1);
\draw[-Straight Barb, very thick] (-1.4,1) --(-1.5,1);
\draw[-Straight Barb, very thick] (-5.2,1) --(-5.3,1);
\draw[-Straight Barb, very thick] (5,-0.6) --(5,-0.5);
\draw[-Straight Barb, very thick] (5,0.5) --(5,0.6);
\draw[-Straight Barb, very thick] (-2.6,-0.4) --(-2.5,-0.5);
\draw[-Straight Barb, very thick] (-2.5,0.5) --(-2.6,0.4);
\draw[-Straight Barb, very thick] (0.5,-0.5) --(0.6,-0.4);
\draw[-Straight Barb, very thick]  (0.6,0.4) --(0.5,0.5);
\draw[-Straight Barb, very thick] (5.707107,0.707107) --(5.67,0.784);
\draw[-Straight Barb, very thick] (5.707107-4,0.707107) --(5.67-4,0.784);
\draw[-Straight Barb, very thick] (-3-0.7,0.707107) --(-3-0.707107,0.7);

%

    \end{tikzpicture}
\end{center}
\caption{The oriented contour $\Sg_X$.}\label{contour_Sigma_X}
\end{figure}

We let
\begin{equation}\label{tt1}
\mathbf X_n(z)=\left\{
\begin{aligned}
&\mathbf S_n(z) \mathbf M(z)^{-1} \ \textrm{for} \ z \ \textrm{outside the disks }  D(\al, \ep), \ D(\be, \ep), D(0, \ep),\\
&\mathbf S_n(z)  D(z)^{\sg_3}\mathbf M(z)^{-1} \ \textrm{for} \ z \ \textrm{inside the disk }  D(0, \ep), \\
&\mathbf S_n(z) \mathbf U_n(z)^{-1} \ \textrm{for} \ z \ \textrm{inside the disks }  D(\al, \ep), \ D(\be, \ep). \\
\end{aligned}\right.
\end{equation}
Then $\mathbf X_n(z)$ solves the following RHP:
\begin{enumerate}
\item
$\mathbf X_n(z)$ is analytic on $\C \setminus \Sigma_X$.
\item
$\mathbf X_n(z)$ has the jump properties
\begin{equation}\label{tt2}
\mathbf X_{n+}(x)=\mathbf X_{n-}(z)j_X(z),
\end{equation}
where
\begin{equation}\label{tt3}
j_X(z)=\left\{
\begin{aligned}
&\mathbf M(z)\mathbf U_n(z)^{-1} \quad \textrm{for} \ z \ \textrm{on the circles} \ \d D(\al, \ep), \ \d D(\be, \ep), \\
& D(z)^{\sg_3} \quad \textrm{for} \ z \ \textrm{on the circle} \ \d D(0, \ep), \\
&\mathbf U_{n-}(z)j_S(z)\mathbf U_{n+}(z)^{-1} \quad \textrm{for} \ z \in \Sg_S \cap\{ D(\al, \ep)\cup D(\be, \ep) \},\\
& D(z)^{-\sg_3} j_S(z) D(z)^{\sg_3} \quad \textrm{for} \ z \in \Sg_S \cap D(0, \ep), \\
&\mathbf M(z)j_S(z)\mathbf M(z)^{-1} \quad \textrm{otherwise}.
\end{aligned}\right.
\end{equation}
\item
As $z\to\infty, \
\mathbf X_n(z)\sim {\bf I} +  \bigO(1/z)$.
\end{enumerate}
From \eqref{pm1} and \eqref{eq:D_bound} we see $j_X(z) = {\bf I} + \bigO(1/n)$ for $z\in \d D(\al, \ep) \cup \d D(\be, \ep) \cup\d D(0, \ep)$. Also, from \eqref{eq:S_tildeS} and \eqref{eq:U_tildeS}
\eq
j_X(z) = {\bf U}_{n-}(z)\left(\tilde j_S(z) + \bigO(e^{-cn})\right)\mathbf U_{n+}(z)^{-1} =  {\bf I } + \bigO(e^{-cn}), \quad z\in \Sg_S \cap\{ D(\al, \ep)\cup D(\be, \ep) \},
\eeq
where we have used the fact that ${\bf U}_{n}(z)$ is uniformly bounded. Similarly, \eqref{eq:S_tildeS}, \eqref{stg3}, and \eqref{eq:D_jump} together with uniform boundedness of $D(z)$ imply
\eq
j_X(z) =  {\bf I } + \bigO(e^{-cn}), \quad z\in \Sg_S \cap D(0, \ep);
\eeq
and \eqref{eq:S_tildeS}, \eqref{stg3}, and \eqref{mg1} together with uniform boundedness of ${\bf M}(z)$ outside of neighborhoods of $\al$ and $\be$ imply
\eq
j_X(z) =  {\bf I } + \bigO(e^{-cn}), \quad z\in \Sg_S \setminus \{ D(\al, \ep) \cup D(\be, \ep) \cup D(0, \ep)\}.
\eeq
Since
\begin{equation}\label{tt5}
j_X(z)=\left\{
\begin{aligned}
&{\bf I} + \bigO(1/n) \quad \textrm{for} \ z \ \textrm{on the circles} \ \d D(\al, \ep) \cup \ \d D(\be, \ep) \cup \ \d D(0, \ep) \\
& {\bf I } + \bigO(e^{-cn}) \quad \textrm{for} \ z \ \textrm{on the rest of} \ \Sg_X, \\
\end{aligned}\right.
\end{equation}
${\bf X}_n(z)$ is given by a series as in \eqref{eq:R_series} with $j_X^0(\mu) =  j_X(\mu)-{\bf I}$, and the uniform estimate \eqref{tt5} implies that as $n\to \infty$, uniformly for $z\in \C$,
\eq
{\bf X}_n(z)  ={\bf I} + \bigO(1/n).
\eeq

{\bf Extraction of asymptotic formulas \eqref{eq:leading_OPsF} and \eqref{eq:leading_CauchyF}.} For $z$ bounded away from the negative real axis,  \eqref{stg2} and \eqref{tt3} give
\eq\label{eq:Q_unravel}
{\bf Q}_n(z) = {\bf K}^{-1}e^{\frac{nl}{2}\sg_3}{\bf X}_n(z) {\bf M}(z) e^{n(g(z) -l/2)\sg_3} {\bf K}.
\eeq
Using ${\bf X}_n(z)  ={\bf I} + \bigO(1/n)$ and multiplying out the RHS of \eqref{eq:Q_unravel} gives
\eq\label{eq:Q_asy_row2}
{\bf Q}_n(z)_{21} = -\frac{e^{n(g(z) - l)}}{\ii \pi} {\bf M}(z)_{21}\left(1+\bigO(1/n)\right), \quad {\bf Q}_n(z)_{22} = e^{-ng(z)} {\bf M}(z)_{22}\left(1+\bigO(1/n)\right),
\eeq
which, along with \eqref{eq:Q_p_Cauchy}, \eqref{eq:Q_p_OP}, and \eqref{mg3}, proves  \eqref{eq:leading_OPsF} and \eqref{eq:leading_CauchyF} for $z$ bounded away from the negative real axis. When $z$ is near the negative real axis but still bounded away from $[\be,0]$, i.e., $\Re z<\be$, $|\Im z|<\ep$, then \eqref{redp6}, \eqref{stg2}, and \eqref{tt3} give
\eq\label{eq:Q_unravel_real}
{\bf Q}_n(z) = {\bf K}^{-1}e^{\frac{nl}{2}\sg_3}{\bf X}_n(z) {\bf M}(z) e^{n(g(z) -l/2)\sg_3} {\bf K}{\bf D}^u_\pm(z).
\eeq
As in \eqref{eq:Cauchy_closetoR}, the extra factor ${\bf D}^u_\pm(z)$ does not change the first column of  ${\bf Q}_n(z)$ at all, and does not affect the asymptotic formula \eqref{eq:Q_asy_row2} for ${\bf Q}_n(z)_{22}$ provided $z$ remains at a distance at least $\ep/n$ from the set $L_n$. This proves \eqref{eq:leading_OPsF} and \eqref{eq:leading_CauchyF} for all $z$ bounded away from the interval $[\be,0]$.

\end{appendix}

\bibliographystyle{plain}
\bibliography{bibliography.bib}

\def\cydot{\leavevmode\raise.4ex\hbox{.}}
\begin{thebibliography}{10}

\bibitem{Ablowitz-Fokas97}
Mark~J. Ablowitz and Athanassios~S. Fokas.
\newblock {\em Complex variables: introduction and applications}.
\newblock Cambridge Texts in Applied Mathematics. Cambridge University Press,
  Cambridge, 1997.

\bibitem{aggarwal2018current}
Amol Aggarwal.
\newblock Current fluctuations of the stationary asep and six-vertex model.
\newblock {\em Duke Mathematical Journal}, 167(2), 2018.

\bibitem{Aggarwal20}
Amol Aggarwal.
\newblock Arctic boundaries of the ice model on three-bundle domains.
\newblock {\em Invent. Math.}, 220(2):611--671, 2020.

\bibitem{aggarwal2022gaussian}
Amol Aggarwal and Vadim Gorin.
\newblock Gaussian unitary ensemble in random lozenge tilings.
\newblock {\em Probability Theory and Related Fields}, 184(3-4):1139--1166,
  2022.

\bibitem{Anderson-Guionnet-Zeitouni10}
Greg~W. Anderson, Alice Guionnet, and Ofer Zeitouni.
\newblock {\em An introduction to random matrices}, volume 118 of {\em
  Cambridge Studies in Advanced Mathematics}.
\newblock Cambridge University Press, Cambridge, 2010.

\bibitem{ayyer2021goe}
Arvind Ayyer, Sunil Chhita, and Kurt Johansson.
\newblock {GOE} fluctuations for the maximum of the top path in alternating
  sign matrices.
\newblock {\em to appear in Duke Mathematical Journal}, 2021.
\newblock arXiv:2109.02422.

\bibitem{Baryshnikov01}
Yuri Baryshnikov.
\newblock G{UE}s and queues.
\newblock {\em Probab. Theory Related Fields}, 119(2):256--274, 2001.

\bibitem{baxter2016exactly}
Rodney~J. Baxter.
\newblock {\em Exactly solved models in statistical mechanics}.
\newblock Elsevier, 2016.

\bibitem{bekerman2018clt}
Florent Bekerman, Thomas Lebl{\'e}, and Sylvia Serfaty.
\newblock {CLT} for fluctuations of $\beta$-ensembles with general potential.
\newblock {\em Electron. J. Probab}, 23(115):1--31, 2018.

\bibitem{arous1997large}
Gerard Ben~Arous and Alice Guionnet.
\newblock Large deviations for {W}igner's law and {V}oiculescu's
  non-commutative entropy.
\newblock {\em Probability theory and related fields}, 108:517--542, 1997.

\bibitem{Bleher-Bothner12}
Pavel Bleher and Thomas Bothner.
\newblock Exact solution of the six-vertex model with domain wall boundary
  conditions: critical line between disordered and antiferroelectric phases.
\newblock {\em Random Matrices Theory Appl.}, 1(4):1250012, 43, 2012.

\bibitem{bleher2014calculation}
Pavel Bleher and Thomas Bothner.
\newblock Calculation of the constant factor in the six-vertex model.
\newblock {\em Annales de l’Institut Henri Poincar{\'e} D}, 1(4):363--427,
  2014.

\bibitem{Bleher-Fokin06}
Pavel Bleher and Vladimir Fokin.
\newblock Exact solution of the six-vertex model with domain wall boundary
  conditions. {D}isordered phase.
\newblock {\em Comm. Math. Phys.}, 268(1):223--284, 2006.

\bibitem{Bleher-Liechty09}
Pavel Bleher and Karl Liechty.
\newblock Exact solution of the six-vertex model with domain wall boundary
  conditions. {F}erroelectric phase.
\newblock {\em Comm. Math. Phys.}, 286(2):777--801, 2009.

\bibitem{Bleher-Liechty10}
Pavel Bleher and Karl Liechty.
\newblock Exact solution of the six-vertex model with domain wall boundary
  conditions: antiferroelectric phase.
\newblock {\em Comm. Pure Appl. Math.}, 63(6):779--829, 2010.

\bibitem{Bleher-Liechty11}
Pavel Bleher and Karl Liechty.
\newblock Uniform asymptotics for discrete orthogonal polynomials with respect
  to varying exponential weights on a regular infinite lattice.
\newblock {\em Int. Math. Res. Not. IMRN}, (2):342--386, 2011.

\bibitem{Bleher-Liechty14}
Pavel Bleher and Karl Liechty.
\newblock {\em Random matrices and the six-vertex model}, volume~32 of {\em CRM
  Monograph Series}.
\newblock American Mathematical Society, Providence, RI, 2014.

\bibitem{borodin2016stochastic}
Alexei Borodin, Ivan Corwin, and Vadim Gorin.
\newblock Stochastic six-vertex model.
\newblock {\em Duke Mathematical Journal}, 165(3):563--624, 2016.

\bibitem{borodin2019stochastic}
Alexei Borodin and Vadim Gorin.
\newblock A stochastic telegraph equation from the six-vertex model.
\newblock {\em The Annals of Probability}, 47(6):4137--4194, 2019.

\bibitem{borodin2017gaussian}
Alexei Borodin, Vadim Gorin, and Alice Guionnet.
\newblock Gaussian asymptotics of discrete $\beta$-ensembles.
\newblock {\em Publications math{\'e}matiques de l'IH{\'E}S}, 125(1):1--78,
  2017.

\bibitem{borodin2017integrable}
Alexei Borodin and Leonid Petrov.
\newblock Integrable probability: stochastic vertex models and symmetric
  functions.
\newblock In {\em Stochastic processes and random matrices}, pages 26--131,
  2017.

\bibitem{Bottcher-Karlovich97}
Albrecht B\"{o}ttcher and Yuri~I. Karlovich.
\newblock {\em Carleson curves, {M}uckenhoupt weights, and {T}oeplitz
  operators}, volume 154 of {\em Progress in Mathematics}.
\newblock Birkh\"{a}user Verlag, Basel, 1997.

\bibitem{Breuer-Duits13}
Jonathan Breuer and Maurice Duits.
\newblock Central limit theorems for biorthogonal ensembles and asymptotics of
  recurrence coefficients.
\newblock {\em Journal of American Mathematical Society}, (30):27--66, 2017.
\newblock arXiv:1309.6224.

\bibitem{bufetov2019fourier}
Alexey Bufetov and Vadim Gorin.
\newblock Fourier transform on high-dimensional unitary groups with
  applications to random tilings.
\newblock {\em Duke Mathematical Journal}, 168(13):2559--2649, 2019.

\bibitem{Cohn-Elkies-Propp96}
Henry Cohn, Noam Elkies, and James Propp.
\newblock Local statistics for random domino tilings of the {A}ztec diamond.
\newblock {\em Duke Math. J.}, 85(1):117--166, 1996.

\bibitem{Colomo-Pronko08}
Filippo Colomo and Andrei Pronko.
\newblock Emptiness formation probability in the domain-wall six-vertex model.
\newblock {\em Nuclear Phys. B}, 798(3):340--362, 2008.

\bibitem{Colomo-Pronko10}
Filippo Colomo and Andrei Pronko.
\newblock The arctic curve of the domain-wall six-vertex model.
\newblock {\em J. Stat. Phys.}, 138(4-5):662--700, 2010.

\bibitem{Colomo-Pronko-ZinnJustin10}
Filippo Colomo, Andrei Pronko, and Paul Zinn-Justin.
\newblock The arctic curve of the domain wall six-vertex model in its
  antiferroelectric regime.
\newblock {\em J. Stat. Mech. Theory Exp.}, (3):L03002, 11, 2010.

\bibitem{Colomo-Sportiello16}
Filippo Colomo and Andrea Sportiello.
\newblock Arctic curves of the six-vertex model on generic domains: the tangent
  method.
\newblock {\em J. Stat. Phys.}, 164(6):1488--1523, 2016.

\bibitem{Copson65}
Edward~Thomas Copson.
\newblock {\em Asymptotic Expansions}.
\newblock Cambridge University Press, 1965.

\bibitem{corwin2020stochastic}
Ivan Corwin, Promit Ghosal, Hao Shen, and Li-Cheng Tsai.
\newblock Stochastic {PDE} limit of the six vertex model.
\newblock {\em Communications in Mathematical Physics}, 375(3):1945--2038,
  2020.

\bibitem{cuenca2021universal}
Cesar Cuenca.
\newblock Universal behavior of the corners of orbital beta processes.
\newblock {\em International Mathematics Research Notices},
  2021(19):14761--14813, 2021.

\bibitem{dadoun2023asymptotics}
Benjamin Dadoun, Matthieu Fradelizi, Olivier Gu{\'e}don, and P-A Zitt.
\newblock Asymptotics of the inertia moments and the variance conjecture in
  schatten balls.
\newblock {\em Journal of Functional Analysis}, 284(2):109741, 2023.

\bibitem{Deift99}
Percy Deift.
\newblock {\em Orthogonal polynomials and random matrices: a
  {R}iemann-{H}ilbert approach}, volume~3 of {\em Courant Lecture Notes in
  Mathematics}.
\newblock New York University Courant Institute of Mathematical Sciences, New
  York, 1999.

\bibitem{Deift19}
Percy Deift.
\newblock Riemann-{H}ilbert problems.
\newblock In {\em Random matrices}, volume~26 of {\em IAS/Park City Math.
  Ser.}, pages 1--40. Amer. Math. Soc., Providence, RI, 2019.

\bibitem{dimitrov2020six}
Evgeni Dimitrov.
\newblock Six-vertex models and the {GUE}-corners process.
\newblock {\em International Mathematics Research Notices}, 2020(6):1794--1881,
  2020.

\bibitem{dimitrov2023two}
Evgeni Dimitrov.
\newblock Two-point convergence of the stochastic six-vertex model to the
  {Airy} process.
\newblock {\em Communications in Mathematical Physics}, pages 1--103, 2023.

\bibitem{dimitrov2022gue}
Evgeni Dimitrov and Mark Rychnovsky.
\newblock {GUE} corners process in boundary-weighed six-vertex models.
\newblock {\em Annales de l'Institut Henri Poincare (B) Probabilites et
  statistiques}, 58(1):188--219, 2022.

\bibitem{Erdelyi56}
Arthur Erd{\'e}lyi.
\newblock {\em Asymptotic Expansions}.
\newblock Dover, New York, 1956.

\bibitem{feral2008large}
D{\'e}lphine F{\'e}ral.
\newblock On large deviations for the spectral measure of discrete coulomb gas.
\newblock {\em Lecture notes in mathematics}, 1934:19, 2008.

\bibitem{gel1950unitary}
Izrail~Moiseevich Gelfand and Mark~Aronovich Naimark.
\newblock Unitary representations of the classical groups.
\newblock {\em Trudy Matematicheskogo Instituta imeni VA Steklova}, 36:3--288,
  1950.

\bibitem{Geronimo-Liechty20}
J.~S. Geronimo and Karl Liechty.
\newblock The {F}ourier extension method and discrete orthogonal polynomials on
  an arc of the circle.
\newblock {\em Adv. Math.}, 365:107064, 57, 2020.

\bibitem{gnedin2010q}
Alexander Gnedin and Grigori Olshanski.
\newblock q-exchangeability via quasi-invariance.
\newblock {\em Annals of Probability}, 38(6):2103--2135, 2010.

\bibitem{Gorin14}
Vadim Gorin.
\newblock From alternating sign matrices to the {G}aussian unitary ensemble.
\newblock {\em Comm. Math. Phys.}, 332(1):437--447, 2014.

\bibitem{Gorin_Nicoletti_lectures}
Vadim Gorin and Matthew Nicoletti.
\newblock Random matrix asymptotics for the six-vertex model.
\newblock 2023+.

\bibitem{gorin2015asymptotics}
Vadim Gorin and Greta Panova.
\newblock Asymptotics of symmetric polynomials with applications to statistical
  mechanics and representation theory.
\newblock {\em Annals of Probability}, 43(6):3052--3132, 2015.

\bibitem{guhr2002recursive}
Thomas Guhr and Heiner Kohler.
\newblock Recursive construction for a class of radial functions. i. ordinary
  space.
\newblock {\em Journal of Mathematical Physics}, 43(5):2707--2740, 2002.

\bibitem{guionnet2019asymptotics}
Alice Guionnet.
\newblock {\em Asymptotics of random matrices and related models: the uses of
  dyson-schwinger equations}, volume 130.
\newblock American Mathematical Soc., 2019.

\bibitem{gwa1992six}
Leh-Hun Gwa and Herbert Spohn.
\newblock Six-vertex model, roughened surfaces, and an asymmetric spin
  hamiltonian.
\newblock {\em Physical review letters}, 68(6):725, 1992.

\bibitem{Harish-Chandra57}
Harish-Chandra.
\newblock Differential operators on a semisimple {L}ie algebra.
\newblock {\em Amer. J. Math.}, 79:87--120, 1957.

\bibitem{he2022cycles}
Jimmy He, Tobias M{\"u}ller, and Teun Verstraaten.
\newblock Cycles in {Mallows} random permutations.
\newblock {\em Random Structures and Algorithms}, 2022.
\newblock arXiv:2201.11610.

\bibitem{Itzykson-Zuber80}
Claude Itzykson and Jean-Bernard Zuber.
\newblock The planar approximation. {II}.
\newblock {\em J. Math. Phys.}, 21(3):411--421, 1980.

\bibitem{Izergin87}
Anatolii~Georgievich Izergin.
\newblock Partition function of a six-vertex model in a finite volume.
\newblock {\em Dokl. Akad. Nauk SSSR}, 297(2):331--333, 1987.

\bibitem{Izergin-Coker-Korepin92}
Anatolii~Georgievich Izergin, David~Allen Coker, and Vladimir~Evgenievich
  Korepin.
\newblock Determinant formula for the six-vertex model.
\newblock {\em J. Phys. A}, 25(16):4315--4334, 1992.

\bibitem{Jockusch-Propp-Shor98}
William Jockusch, James Propp, and Peter Shor.
\newblock Random domino tilings and the arctic circle theorem.
\newblock 1998.
\newblock Preprint. arXiv:math/9801068.

\bibitem{Johansson98}
Kurt Johansson.
\newblock On fluctuations of eigenvalues of random {H}ermitian matrices.
\newblock {\em Duke Math. J.}, 91(1):151--204, 1998.

\bibitem{Johansson02}
Kurt Johansson.
\newblock Non-intersecting paths, random tilings and random matrices.
\newblock {\em Probab. Theory Related Fields}, 123(2):225--280, 2002.

\bibitem{Johansson05}
Kurt Johansson.
\newblock The arctic circle boundary and the {A}iry process.
\newblock {\em Ann. Probab.}, 33(1):1--30, 2005.

\bibitem{Johansson-Nordenstam06}
Kurt Johansson and Eric Nordenstam.
\newblock Eigenvalues of {GUE} minors.
\newblock {\em Electron. J. Probab.}, 11:no. 50, 1342--1371, 2006.

\bibitem{keating2018random}
David Keating and Ananth Sridhar.
\newblock Random tilings with the {GPU}.
\newblock {\em Journal of Mathematical Physics}, 59(9):091420, 2018.

\bibitem{korepin1982calculation}
Vladimir~E Korepin.
\newblock Calculation of norms of bethe wave functions.
\newblock {\em Communications in Mathematical Physics}, 86:391--418, 1982.

\bibitem{lambert2019quantitative}
Gaultier Lambert, Michel Ledoux, and Christian Webb.
\newblock Quantitative normal approximation of linear statistics of
  $\beta$-ensembles.
\newblock {\em The Annals of Probability}, 47(5):2619--2685, 2019.

\bibitem{LiebWu}
Elliott Lieb and Fa~Yueh Wu.
\newblock Two dimensional ferroelectric models.
\newblock In {\em Phase Transitions and Critical Phenomena, edited by C. Domb
  and M. Green (Academic, 1972), Vol. 1}, pages 331--490, 1972.

\bibitem{lyberg2023fluctuation}
Ivar Lyberg, Vladimir Korepin, and Jacopo Viti.
\newblock Fluctuation of the phase boundary in the six-vertex model with domain
  wall boundary conditions: a monte carlo study.
\newblock {\em arXiv preprint arXiv:2303.14669}, 2023.

\bibitem{mallows1957non}
Colin~L. Mallows.
\newblock Non-null ranking models. i.
\newblock {\em Biometrika}, 44(1/2):114--130, 1957.

\bibitem{meckes2020random}
Elizabeth Meckes and Mark Meckes.
\newblock Random matrices with prescribed eigenvalues and expectation values
  for random quantum states.
\newblock {\em Transactions of the American Mathematical Society},
  373(7):5141--5170, 2020.

\bibitem{neretin2003rayleigh}
Yuri Neretin.
\newblock Rayleigh triangles and non-matrix interpolation of matrix beta
  integrals.
\newblock {\em Sbornik: Mathematics}, 194(4):515, 2003.

\bibitem{novak2015lozenge}
Jonathan Novak.
\newblock Lozenge tilings and hurwitz numbers.
\newblock {\em Journal of Statistical Physics}, 161:509--517, 2015.

\bibitem{okounkov2006birth}
Andrei~Yur'evich Okounkov and Nikolai~Yur'evich Reshetikhin.
\newblock The birth of a random matrix.
\newblock {\em Moscow Mathematical Journal}, 6(3):553--566, 2006.

\bibitem{praehofer2023domain}
Michael Praehofer and Herbert Spohn.
\newblock Domain wall fluctuations of the six-vertex model at the ice point.
\newblock {\em arXiv preprint arXiv:2305.09502}, 2023.

\bibitem{reshetikhin2010lectures}
N.~Reshetikhin.
\newblock Lectures on the integrability of the 6-vertex model.
\newblock In {\em {Exact Methods in Low-dimensional Statistical Physics and
  Quantum Computing}}, pages 197--266. Oxford Univ. Press, 2010.
\newblock arXiv:1010.5031.

\bibitem{shen2019stochastic}
Hao Shen and Li-Cheng Tsai.
\newblock Stochastic telegraph equation limit for the stochastic six vertex
  model.
\newblock {\em Proceedings of the American Mathematical Society},
  147(6):2685--2705, 2019.

\bibitem{Wang-Wong11}
Xiang-Sheng Wang and Roderick Wong.
\newblock Global asymptotics of the {M}eixner polynomials.
\newblock {\em Asymptot. Anal.}, 75(3-4):211--231, 2011.

\bibitem{Watson-Whittaker96}
Edmund Whittaker and George Watson.
\newblock {\em A course of modern analysis}.
\newblock Cambridge Mathematical Library. Cambridge University Press,
  Cambridge, 1996.
\newblock An introduction to the general theory of infinite processes and of
  analytic functions; with an account of the principal transcendental
  functions, Reprint of the fourth (1927) edition.

\bibitem{Zinn_Justin00}
Paul Zinn-Justin.
\newblock Six-vertex model with domain wall boundary conditions and one-matrix
  model.
\newblock {\em Phys. Rev. E (3)}, 62(3, part A):3411--3418, 2000.

\end{thebibliography}

\end{document}